\documentclass[12pt,a4paper]{book}
\usepackage[latin1]{inputenc}
\usepackage{amsmath}
\usepackage{amsfonts}
\usepackage{amssymb}
\usepackage{graphicx}
\usepackage[top=3cm, bottom=3cm, left=3cm, right=3cm]{geometry}
\usepackage {amsthm}
\usepackage{array}
\usepackage{hyperref}

\usepackage{fancyhdr}

\DeclareMathOperator{\li}{li}
\DeclareMathOperator{\ord}{ord}
\DeclareMathOperator{\tmod}{mod}
\DeclareMathOperator{\tdiv}{div}

\DeclareMathOperator{\rk}{rk}
\DeclareMathOperator{\Gal}{Gal}
\DeclareMathOperator{\GL}{GL}
\DeclareMathOperator{\rad}{rad}

\newtheorem{thm}{Theorem}[chapter]
\newtheorem{lem}{Lemma}[chapter]
\newtheorem{conj}{Conjecture}[chapter]
\newtheorem{cor}{Corollary}[chapter]

\newtheorem{exa}{Example}[chapter]
\newtheorem{exe}{Exercise}[chapter]
\newtheorem{rem}{Remark}[chapter]

\theoremstyle{definition}
\newtheorem{dfn}{Definition}[chapter]

\title{\textbf{Asymptotics For Primitive Roots Producing Polynomials And  Primitive Points On Elliptic Curves} }
\author{\textbf{N. A. Carella \footnote{Copyright 2017}
}}
\date{}

\begin{document}

\maketitle
\thispagestyle{empty}
\frontmatter
\thispagestyle{empty}

\chapter{Abstract}
Let $x \geq 1$ be a large number, let $f(n) \in \mathbb{Z}[x]$ be a prime producing polynomial of degree $\deg(f)=m$, and let \(u\neq \pm 1,v^2\) be a fixed integer. Assuming the Bateman-Horn conjecture, an asymptotic counting function for the number of primes $p=f(n) \leq x$ with a fixed primitive root $u$ is derived in this note. This asymptotic result has the form $$\pi_f(x)=\# \{ p=f(n)\leq x:\ord_p(u)=p-1 \}=\left (c(u,f)+ O\left (1/\log  x )\right ) \right )x^{1/m}/\log x$$, where $c(u,f)$ is a constant depending on the polynomial and the fixed integer. Furthermore, new results for the asymptotic order of elliptic primes with respect to fixed elliptic curves $E:f(X,Y)=0$ and its groups of $\mathbb{F}_p$-rational points $E(\mathbb{F}_p)$, and primitive points are proved in the last chapters.\\

\vskip 1 in 
\textit{Mathematics Subject Classifications}: Primary 11A07, 11G07; Secondary 11N37; 11N32; and 11B83. \\
	\textit{Keywords}: Primitive Root; Prime Values of Polynomial.
	Primitive Root Producing Polynomial; Primitive Point; Elliptic Curve.\\
\tableofcontents	
\mainmatter

%cccccccccccccccccccccccccc	
\chapter{Introduction} \label{c1}
The theory primitive roots and elliptic primitive points are extensive area of research in number theory and cryptography. Introductions to some topics in this field are presented in this monograph. The first part covers various topics such as primitive roots producing polynomials, repeated decimals of maximal lengths, density of squarefree primes with a fixed primitive root, and simultaneous primitive roots. And the second part considers questions about the densities of elliptic cyclic primes, elliptic primitive primes, and related questions.

\section{Primitive Root Producing Polynomials}
The literature has some numerical works and theory for a few polynomials such as $f_1(x)=10x^2+7$ and its primitive root $u_1=10$; and $f_2(x)=326x^2+3$ and its primitive root $u_2=326$. More generally the polynomial $f(x)=ux^2+s$ its primitive root $u\ne 0$. These concepts are studied in \cite{LD63}, and \cite[p.\  47]{IR90}. These numerical experiments have shown that $u_1=10$ has a high probability $.870...$ of being a primitive root modulo $p=10n^2+7$, and $u_2=326$ has a higher probability $.9933...$ of being a primitive root modulo $p=326n^2+3$ as $n \to \infty$. A fixed integer $u\ne  \pm 1, v^2$ has an average probability $\prod_{p \geq 2}(1-1/(p(p-1))^{-1})\approx 3/8$ of being a primitive root modulo a random prime $p \geq 2$. The average probability is computed in \cite{GM68}, and \cite{SP69}. The more recent work in \cite{MP07}, and \cite{AS15} have significantly extended the theory and numerical results for primitive roots producing polynomials. \\
	
This note considers the derivation of an asymptotic formula for the primes counting function
\begin{equation}
\pi_f(x)=\#\{p=f(n)\leq x: p \text{ is prime} \text{ and ord}_p(u)=p-1\}
\end{equation}
for primitive roots producing polynomials $f(n) \in \mathbb{Z}[x]$ of degree $\deg(f)=m$.\\  
	
\begin{thm} \label{thm1.1}
Assume the Bateman-Horn conjecture. Let $f(n) \in \mathbb{Z}[x]$ be a prime producing polynomial of degree $\deg(f)=m$. Then, for a large number $x\geq 1$, the number of primes $p=f(n) \leq x$ with a fixed primitive root \(u\ne \pm 1,v^2 \) has the asymptotic formula
\begin{equation} \label{el10}
\sum _{ \substack{p=f(n) \leq x \\ \ord_p(u)=p-1}} 1 =c(u,f) \frac{x^{1/m}}{ \log x}+O \left (\frac{x^{1/m}}{\log^2x}\right ),
\end{equation}
where $c(u,f) \geq 0$ is a constant depending on the polynomial $f(x)$, and the integer $u \in \mathbb{Z}$.\\
\end{thm}
	
A special case of this result deals with the often analyzed polynomial $f(x)=x^2+1$.

\begin{thm} \label{thm1.2}
For a large number $x\geq 1$, the number of primes $p=n^2+1 \leq x$ with a fixed primitive root \(u\ne \pm 1,v^2 \)  has the asymptotic formula
\begin{equation} \label{el20}
\sum _{ \substack{p=n^2+1\leq x\\ \ord_p(u)=p-1}} 1 =c(u,f) \frac{x^{1/2}}{ \log x}+O \left (\frac{x^{1/2}}{\log^2x}\right ),
\end{equation} 
where $c(u,f) \geq 0$ is a constant depending on the polynomial $f(x)=x^2+1$, and the integer $u \in \mathbb{Z}$.\\
\end{thm}

The first few chapters collect the notation and some standard results related to or applicable to the investigation of primitive roots in cyclic groups. Chapter \ref{c5} presents the proofs of Theorem \ref{thm1.1}, and Theorem \ref{thm1.2}. The new results for repeated decimal numbers of maximal periods, and its generalization to $\ell$-adic expansions of maximal periods appear in Chapter \ref{c8}. The results for squarefree totients $p-1$, and simultaneous primitive roots appear in Chapters \ref{c10} and \ref{c11}.

\section{Elliptic Primitive Points}
The group of $\mathbb{F}_p$-rational points $E(\mathbb{F}_p)$ of a  random elliptic curve $E:f(x,y)=0$ has a high probability of being a cyclic group generated by a point $P \in E(\mathbb{F}_p)$. A lower bound for the number of cyclic primes such that $E(\mathbb{F}_p) \cong \mathbb{Z}/n \mathbb{Z}$ is a corollary of Theorem \ref{thm1.3}. An asymptotic formula for the number of elliptic primitive primes such that $<P>=E(\mathbb{F}_p)$, is prove in Theorem \ref{thm16.1}. The precise result is repeated here.

\begin{thm} \label{thm1.3}   
	Let $E:f(X,Y)=0$ be an elliptic curve over the rational numbers $\mathbb{Q}$ of rank $\rk(E(\mathbb{Q})>0$, and let $P$ be a point of infinite order. Then,
	\begin{equation}
	\pi(x,E,P) \gg \frac{x}{\log^2 x}\left (1+O\left ( \frac{x}{\log^2 x} \right )\right ),
	\end{equation}	
	 as $x \to \infty$,
\end{thm}
The corresponding result for groups of prime orders generated by primitive points states the following.
\begin{equation}
\pi(x,E)= \#\{ p \leq x: p \not | \Delta \text{ and } d_E^{-1}\cdot\#E(\mathbb{F}_p)=\text{prime}\}.
\end{equation} 
The parameter $d_E\geq 1$ is a small integer defined in Chapter \ref{c17}, Section \ref{sec77}.\\

\begin{thm} \label{thm1.4}   
	Let $E:f(X,Y)=0$ be an elliptic curve over the rational numbers $\mathbb{Q}$ of rank $\rk(E(\mathbb{Q})>0$. Then, as $x \to \infty$,
	\begin{equation}
	\pi(x,E)\geq \delta(d_E,E)\frac{x}{\log^3 x} \left (1+O\left ( \frac{x}{\log x} \right )\right) ,
	\end{equation}  	
	where $\delta(d_E,E)$ is the density constant.
\end{thm}

The proof of this result is split into several parts. The next sections are intermediate results. The proof of Theorem 1.4 is assembled in Chapter \ref{c17} as Theorem \ref{thm800.1}. Various examples of elliptic curves with infinitely many elliptic groups $E(\mathbb{F}_p)$ of prime orders $n$ are provided here.\\

%22222222222222222222222222222222	
\chapter{Primes Values of Polynomials} \label{c2}
Some basic ideas associated with prime producing polynomials are discussed in this Chapter.

\section{Definitions And Notation}	
\begin{dfn}
The fixed divisor $\tdiv(f)=\gcd(f(\mathbb{Z}))$ of a polynomial $f(x) \in \mathbb{Z}[x]$ over the integers is the greatest common divisor of its image $f(\mathbb{Z})=\{f(n):n \in \mathbb{Z}\}$.
\end{dfn}
The fixed divisor $\tdiv(f)=1$ if and only if the congruence equation $f(n) \equiv 0 \mod p$ has $\nu_f(p)<p$ solutions for every prime $p<\deg(f)$, see \cite[p.\ 395]{FI10}. An irreducible polynomial can represent infinitely many primes if and only if it has fixed divisor $\tdiv(f)=1$.\\
	
The function $\nu_f(q)=\#\{n:f(n) \equiv 0 \bmod q\}$ is multiplicative, and can be written in the form
\begin{equation}
\nu_f(q)=\prod_{p^b||q}\nu_f(p^b),
\end{equation}
where $p^b \,||\,q$ is the maximal prime power divisor of $q$, and $\nu_f(p^b)\leq \deg(f)$, see \cite[p.\ 82]{RD96}.
\begin{exa} \normalfont  
A few well known polynomials are listed here.
\begin{enumerate} 
\item The polynomials $f_1(x)=x^2+1$ and $f_2(x)=x^3+2$ are irreducible over the integers and have the fixed divisors $\tdiv(f_1)=1$, and $\tdiv(f_2)=1$ respectively. Thus, these polynomials can represent infinitely many primes. 
\item The polynomials $f_3(x)=x(x+1)+2$ and $f_4(x)=x(x+1)(x+2)+3$ are irreducible over the integers. But, have the fixed divisors $\tdiv(f_3)=2$, and $\tdiv(f_4)=3$ respectively. Thus, these polynomials cannot represent infinitely many primes. 
\end{enumerate}
\end{exa}

\section{Polynomials Of One Variable}
The quantitative form of the qualitative Hypothesis H, \cite[pp. 386--394]{RP96}, was formulated about fifty years ago in \cite{BH62}. It states the followings.	
\begin{conj} \label{conj2.1} {\normalfont (Bateman-Horn)} \label{conj1.1}   Let $f_1(x), f_2(x), \ldots, f_k(x) \in \mathbb{Z}[x]$ be relatively prime polynomials of degree $\deg(f_i)=d_i \geq 1$. Suppose that each polynomial $f_i(x)$ has the fixed divisor $\tdiv(f_i) = 1$. Then, the number of simultaneously primes $k$-tuples
\begin{equation}
f_1(n),f_2(n),\ldots,f_k(n), \nonumber
\end{equation}
as $n \leq x$ tends to infinity has the equivalent asymptotic formulas
\begin{equation}
\pi_f(x)=s_f \int_2^x\frac{1}{(\log t)^k}dt+O\left(\frac{x^{1/d}}{\log^{k+1} x} \right),
\end{equation}
and
\begin{equation}
\sum_{f_i(n) \leq x} \Lambda(f_1(n)) \cdots \Lambda(f_k(n))=s_f x^{1/d} +O\left(\frac{x^{1/d}}{\log^C x} \right) \nonumber ,
\end{equation}
with $d=d_1+d_2+\cdots+d_k$, and $C>0$ is arbitrary. The density constant is defined by the product
\begin{equation}
s_f=\frac{1}{d_1d_2 \cdots d_k}\prod_{p\geq 2} \left (1-\frac{v_p(f)}{p} \right ) \left (1-\frac{1}{p} \right )^{-k},
\end{equation} 
where the symbol $v_p(f)\geq 0$ denotes the number of solutions of the congruence
\begin{equation}
f_1(x)f_2(x) \cdots f_k(x) \equiv 0 \tmod p.
\end{equation}
\end{conj}
This generalization of prime values of polynomials appears in \cite{BH62}. The constant, known as singular series
\begin{equation} 
\mathfrak{G}(f)=\prod_{p\geq 2} \left (1-\frac{v_p(f)}{p} \right ) \left (1-\frac{1}{p} \right )^{-k},
\end{equation}
can be derived by either the circle method, as done in \cite[Chapter 10]{MT06}, \cite[p.\ 167]{IK04}, \cite{VR73} or by probabilistic means as explained in \cite[p.\ 33]{PJ09}, \cite[p.\ 410]{RP96}, et alii. The convergence of the product is discussed in \cite{BH62}, \cite{RI15}, and other by authors. \\
	
The Siegel-Walfisz theorem, \cite[p.\ 381]{MV07}, \cite{MJ12}, et cetera, extends this conjecture to arithmetic progressions. The corresponding statement has the followings form.	
\begin{thm}  \label{thm2.1}Assume the Bateman-Horn conjecture. Let $f(x)$ be an irreducible polynomial of degree $\deg(f)=d \geq 1$. Suppose that the fixed divisor $\tdiv(f) = 1$. Then, 
\begin{equation}  \label{el2030}
\sum_{\substack{f(n)\leq x \\ f(n) \equiv a \bmod q}}\Lambda(f(n)) =\frac{s_f}{\varphi(q)} x^{1/m}+O \left(\frac{x^{1/m}}{\log^B x}\right)
\end{equation}
where $s_f>0$ constant, and $1 \leq a<q, \; \gcd(a,q)=1$ and $q=O(\log x^C)$ for arbitrary constants $B>C\geq 0$.
\end{thm}
Detailed discussions of the prime values of polynomials appear in \cite[p.\ 387]{RP96}, \cite[p.\ 17]{GL10}, \cite[p.\ 395]{FI10}, \cite[p.\ 405]{NW00}, \cite[p.\ 33]{PJ09}, and related topics in \cite{BR07}, \cite{MA09}, et alii.  \\
	
The Bateman-Horn conjecture was formulated in the 1960's, about half a century ago. A more recent result accounts for the irregularities and oscillations that can occur in the asymptotic formula as the polynomial is varied.	
\begin{thm} \label{thm2.2} {\normalfont (\cite{FG91})}  Let $d \geq 1$ be a fixed integer, and let $B \geq 2$ be a real number. There exist infinitely many irreducible polynomials $f(x)$ of degree $\deg(f)=d$ with nonnegative integer coefficients, such that for some number $\delta_B>0$ depending on $B$, and $x \geq \log^B |f(x)|$, the absolute difference
\begin{equation}
\left | \pi_f(x) -  C_f \frac{x}{\log |f(x)|} \right |>\delta_B C_f \frac{x}{\log |f(x)|}.
\end{equation}
\end{thm}
This result was proved for polynomials of large degrees, but it is claimed to hold for certain structured polynomials of small degrees $d \geq 1$. 
		
\section{Quadratic Polynomials}
There is a vast literature on the topic of almost prime values of the polynomial $f(x)=x^2+1$, see \cite{LR12}. Several results on the squarefree values of $f(x)$ appears in \cite{KR16}.\\
	
The heuristic based on probability yields the Gaussian type analytic formula
\begin{equation}
\pi _{f_2}(x)=\#\{\text{prime } p=f_2(n) \leq x\}=s_{2} \int_{2}^{x} \frac{1}{ t^{1/2} \log t} dt+O\left(\frac{x^{1/2}}{\log ^2 x}\right),
\end{equation}
where $s_{2}$ is a constant. The deterministic approach is based on the weighted prime function, which is basically an extended version of the Chebyshev method for counting primes.
\begin{conj} \label{conj2.2}There are infinitely many quadratic primes $p=n^2+1$ as $n \to \infty$. Further, its counting function has the weighted asymptotic formula
\begin{equation}
\sum_{n^2+1\leq x}\Lambda(n^{2}+1) =s_2 x^{1/2}+O \left(\frac{x^{1/2}}{\log x}\right).
\end{equation} 
\end{conj}
	
In terms of the quadratic symbol $(-1|p)=(-1)^{(p-1)/2}$, the number of roots is $\nu_f(p)=1+(-1|p)$. Thus, the density constant, which depends on the polynomial $f(x)=x^2+1$, is given by  
\begin{equation} \label{quad2}
s_2=\frac{1}{2}\prod_{p>2} \left (1-\frac{(-1|p)}{p-1} \right )=0.6864 \ldots.
\end{equation}
	
The constant is computed in \cite{SD59} and \cite{SD61}.

\section{Cubic Polynomials}
The heuristic based on probability yields the Gaussian type analytic formula
\begin{equation}
\pi _{f_3}(x)=\#\{\text{prime } p=f_3(n) \leq x\}=s_{3} \int_{2}^{x} \frac{1}{ t^{2/3} \log t} dt+O\left(\frac{x^{1/3}}{\log ^2 x}\right),
\end{equation}
where $s_{3}$ is a constant. The deterministic approach is based on the weighted prime function, which is basically an extended version of the Chebyshev method for counting primes.\\	
	
For cubic polynomials, there is the often analyzed irreducible polynomial $f(x)=x^3+2$, for which there is the following prediction.	
\begin{conj} \label{conj2.3}  There are infinitely many cubic primes $p=n^3+2$ as $n \to \infty$. Further, its counting function has the weighted asymptotic formula
\begin{equation}
\sum_{n^3+2\leq x}\Lambda(n^{3}+2) =s_3 x^{1/3}+O \left(\frac{x^{1/3}}{\log x}\right).
\end{equation} 
\end{conj}
The density constant, which depends on the polynomial $f(x)=x^3+2$, is given by  
\begin{equation} \label{el2002}
s_3=\frac{1}{3}\prod_{p\geq2} \left (1-\frac{1-\nu_f(p)}{p-1} \right ) \approx 0.2565157750 \ldots.
\end{equation}
For a pure cubic polynomial $f(x)=x^3-k$, the number of roots modulo $p$ can be computed via the expression
\begin{equation}
\nu_f(p)=
\left \{\begin{array}{ll}
3 & \text{ if } p\equiv 1 \bmod 3 \text{ and } k^{(p-1)/3}\equiv 1 \bmod p,  \\
0 & \text{ if } p \equiv 1 \bmod 3 \text{ and } k^{(p-1)/3}\not \equiv 1 \bmod p,  \\
1 & \text{ if } p\equiv 2 \tmod 3.  \\
\end{array} \right .
\end{equation}

\section{Quartic Polynomials}
The heuristic analysis for the expected number of primes $p=n^4+1 \leq x$ is based on the Bateman-Horn conjecture, described in the earlier section. \\

This heuristic, which is based on probability, yields the Gaussian type analytic formula
\begin{equation}
\pi _{f_4}(x)=\#\{\text{prime } p=f_4(n) \leq x\}=s_{4} \int_{2}^{x} \frac{1}{ t^{3/4} \log t} dt+O\left(\frac{x^{1/4}}{\log ^2 x}\right),
\end{equation}
where $s_{4}$ is a constant. The deterministic approach is based on the weighted prime function, which is basically an extended version of the Chebyshev method for counting primes.\\

\begin{conj} \label{conj2.4} The expected number of quartic primes $p=n^4+1 \leq x$ has the asymptotic weighted counting function
\begin{equation}
\sum_{n^4+1\leq x}\Lambda(n^{4}+1) = s_4 x^{1/4}+O \left(\frac{x^{1/4}}{\log^2 x}\right),
\end{equation}
where the constant, in terms of a few values of the quadratic symbols $(a|p)\equiv a^{(p-1)/2} \text{ mod } p$, is  
\begin{equation}
s_4=\frac{1}{4}\prod_{p>2} \left (1-\dfrac{(-1|p)+(-2|p)+(2|p)}{p-1} \right )=.66974 \ldots.
\end{equation}
\end{conj}
The numerical technique for approximating this constant appears in \cite{SD61}.

\section{Extreme Densities}
The polynomials with very high density of prime values are studied in \cite{JW03}, and earlier papers. These papers cover mostly quadratic polynomials, and seek to optimize the density constant
\begin{equation}
s_2(f)=\frac{1}{2}\prod_{p>2} \left (1-\frac{(-D|p)}{p-1} \right ),
\end{equation}
where $D=b^2-4ac\ne m^2$ is the discriminant of $f(x)=ax^2+bx+c$. The best known case is the Euler prime producing polynomial $f(x)=x^2+x+41$. But, there are many other prime producing quadratic polynomials which have higher densities, see \cite{RP96}, and \cite{JW03}.\\
	
The polynomials with very low density of prime values are studied in \cite{MK84}, \cite{MK86}, and similar papers. A simplified algorithm for constructing prime producing polynomials of very low density and very large least prime is suggested here.\\
	
\textbf{Algorithm 1.}
\begin{enumerate} 
\item \textit{For any real number $z\geq2$, let $ n=\prod_{p\leq z}p$ and $D=n-1$.}
\item \textit{Set $f_1(x)=x^{\varphi(n)}+D$ and $f_2(x)=x^{\lambda(n)}+D$.} 
\end{enumerate} 
\textit{These polynomials are irreducible and have fixed divisors $\tdiv(f_1)=1$ and $\tdiv(f_2)=1$ respectively. Moreover, $f_i(x)$ is composite for $x<e^z$. For any choice of parameter} $n$, $\deg(f_2)=\lambda(n) \leq \deg(f_1)=\varphi(n)$.\\
	
\begin{exa} \normalfont 
Set $n=2\cdot 3 \cdots 43$ and $D=2\cdot 3 \cdots 43-1$. Then, 
\begin{enumerate} 
\item $f_1(x)=x^{\varphi(n)}+D=x^{2\cdot 4\cdot \cdots 42}+2\cdot 3 \cdots 43-1$ is composite for any value $x\leq e^{43}\leq 4727839468229346562$.
\item $f_2(x)=x^{\lambda(n)}+D=x^{55440}+2\cdot 3 \cdots 43-1$ is composite for any value $x\leq e^{43}\leq 4727839468229346562$.
\end{enumerate} 
\end{exa}

\newpage
\section{Problems}
\begin{exe} \normalfont
Let $n=\prod_{p\leq z}p$ and $D=n-1$. Use Eisenstein criteria to show that $f_1(x)=x^{\varphi(n)}+D$ and $f_2(x)=x^{\lambda(n)}+D$ are irreducible and have fixed divisors $\tdiv(f_1)=1$, and $\tdiv(f_2)=1$ respectively. 
\end{exe}
	
\begin{exe} \normalfont Let $n=\prod_{p\leq z}p$ and $D=n-1$. Use Fermat little theorem to show that $f_1(x)=x^{\varphi(n)}+D$ and $f_2(x)=x^{\lambda(n)}+D$ have composite values for $x\leq e^z$.
\end{exe}
	
\begin{exe} \normalfont Let $q\geq3$ be a prime, and $$f_1(x)=2qx+1,f_1(x)=2qx+3,f_2(x)=2qx+5, \ldots f_k(x)=2qx+2[k/2]-1.$$ Set $p_n=f_1(n),p_{n+1}=f_2(n), \ldots, p_{n+k}=f_k(n)$, and assume the $k$-tuple conjecture.\\
(i) Show that $$\liminf_{n \to \infty}(p_{n+k}-p_n)=2k.$$
(ii) Use $p_n \sim n \log n$ to show that $$\limsup_{n \to \infty} (p_{n+k}-p_n) \sim k \log k.$$
Basically, this is $k$ primes equally spaced by the average prime gap $d_n=p_{n+1}-p_n \sim \log n$. 
\end{exe}

\begin{exe} \normalfont
Let $p_n=f(m_n)$, $m_n \in \mathbb{N}$ be a prime producing polynomial of degree $d=\deg f$. Show that the extended Cramer conjecture has the form  $p_{n+1}-p_n \ll (\log |f(m_n)|)^2 \ll d^2 \log n$. 
\end{exe}
	
\begin{exe} \normalfont  Let $f(n)$, $n \in \mathbb{N}$, be a prime producing polynomial of degree $d=\deg f$. Let $$\pi _{f}(x,\mu)=\#\{ p=f(n) \leq x: f(n)=p 
\text{ prime and } \mu(p-1)\ne 0 \}.$$ Show that $\pi_{f}(x, \mu)=t_f x/ \log x+ O(x/( \log x)^2$. 
\end{exe}	

\begin{exe} \normalfont  Let $x \geq 1$ be a large number, and let $B=[\sqrt{x}].$ Prove that $\mathcal{K}=\bigcup_{1 \leq k \leq B}K_k$, where $$K_1=\{p=n^2+1\}, K_2=\{p=n^2+2^2\}, \ldots ,K_k=\{p=n^2+k^2\},$$ contains all the primes $p=4n+1 \leq x$. 
\end{exe}	

%ccccccccccccccccccccccccccccccc
\chapter{Characteristic Functions} \label{c3}
The characteristic function \(\Psi :G\longrightarrow \{ 0, 1 \}\) of primitive elements is one of the standard analytic tools employed to investigate the various properties of primitive roots in cyclic groups \(G\). Many equivalent representations of the characteristic function $\Psi $ of primitive elements are possible. Several of these representations are studied in this chapter.

\section{Primitive Roots Tests}
For a prime $p \geq 2$, the multiplicative group of the finite fields $\mathbb{F}_p$ is a cyclic group for all primes. 

\begin{dfn}
	The order $\min \{k \in \mathbb{N}: u^k \equiv 1 \bmod p \}$ of an element $u \in \mathbb{F}_p$ is denoted by $\ord_p(u)$. An element is a \textit{primitive root} if and only if $\ord_p(u)=p-1$. 
\end{dfn}

\begin{lem} \label{lem3.4b}
	Let $u \in \mathbb{F}_p$. Then, $u$ is a primitive ro0t if and only if $u^{(p-1)/q} \not \equiv 1 \bmod p$ for all prime divisors $q\,|\,p-1$. 
\end{lem}

Basically, this is the classical Lucas-Lehmer primitive root test on the cyclic group $\mathbb{F}_p$. \\

\section{Divisors Dependent Characteristic Function}
A representation of the characteristic function dependent on the orders of the cyclic groups is given below. This representation is sensitive to the primes decompositions $q=p_1^{e_1}p_2^{e_2}\cdots p_t^{e_t}$, with $p_i$ prime and $e_i\geq1$, of the orders of the cyclic groups $q=\# G$. \\

\begin{dfn}
The order of an element in the cyclic group $\mathbb{F}_p^\times$ is defined by $\text{ord}_p (v)=\min\{k:v^k\equiv 1 \bmod p\}$. Primitive elements in this cyclic group have order $p-1=\#G$.
\end{dfn}

\begin{lem} \label{lem3.1}
Let \(G\) be a finite cyclic group of order \(p-1=\# G\), and let \(0\neq u\in G\) be an invertible element of the group. Then
\begin{equation}
\Psi (u)=\frac{\varphi (p-1)}{p-1}\sum _{d \;|\; p-1} \frac{\mu (d)}{\varphi (d)}\sum _{\ord(\chi ) = d} \chi (u)=
\left \{\begin{array}{ll}
1 & \text{ if } \ord_p (u)=p-1,  \\
0 & \text{ if } \ord_p (u)\neq p-1. \\
\end{array} \right .
\end{equation}
\end{lem}

\begin{proof} 
Rewrite the main equation as a product
\begin{equation}
\sum _{d \;|\; p-1} \frac{\mu (d)}{\varphi (d)}\sum _{\ord(\chi ) = d} \chi (u)=\prod_{q \;|\; p-1} \left (1-\frac{\sum_{\ord(\chi ) =q} \chi (u)}{q-1} \right ) ,
\end{equation}
where $q\geq 2$ ranges over the prime divisors of $p-1$. Since there are $q-1$ nontrivial characters $\chi\ne 1$, it is clear that the character sum $\sum _{\ord(\chi ) =q} \chi (u)=q-1$ if and only if the order $\ord_p(u)=q<p-1$. Specifically, $\chi(u)^q=1$.
\end{proof}

The works in \cite{DH37}, and \cite{WR01} attribute this formula to Vinogradov. Other proofs and other details on the characteristic function are given in \cite[p. 863]{ES57}, \cite[p.\ 258]{LN97}, \cite[p.\ 18]{MP07}. The characteristic function for multiple primitive roots is used in \cite[p.\ 146]{CZ98} to study consecutive primitive roots. In \cite{DS12} it is used to study the gap between primitive roots with respect to the Hamming metric. And in \cite{WR01} it is used to prove the existence of primitive roots in certain small subsets \(A\subset \mathbb{F}_p\). In \cite{DH37} it is used to prove that some finite fields do not have primitive roots of the form $a\tau+b$, with $\tau$ primitive and $a,b \in \mathbb{F}_p$ constants. In addition, the Artin primitive root conjecture for polynomials over finite fields was proved in \cite{PS95} using this formula.

\section{Divisors Free Characteristic Function}
It often difficult to derive any meaningful result using the usual divisors dependent characteristic function of primitive elements given in Lemma \ref{lem3.1}. This difficulty is due to the large number of terms that can be generated by the divisors, for example, \(d\,|\,p-1\), involved in the calculations, 	see \cite{ES57}, \cite{DS12} for typical applications and \cite[p.\ 19]{MP04} for a discussion. \\
	
A new \textit{divisors-free} representation of the characteristic function of primitive element is developed here. This representation can overcomes some of the limitations of its counterpart in certain applications. The \textit{divisors representation} of the characteristic function of primitive roots, Lemma \ref{lem3.1}, detects the order \(\text{ord}_p (u)\) of the element \(u\in \mathbb{F}_p\) by means of the divisors of the totient \(p-1\). In contrast, the \textit{divisors-free representation} of the characteristic function, Lemma \ref{lem3.2}, detects the order \(\text{ord}_p(u) \geq 1\) of the element \(u\in \mathbb{F}_p\) by means of the solutions of the equation \(\tau ^n-u=0\) in \(\mathbb{F}_p\), where \(u,\tau\) are constants, and \(1\leq n<p-1, \gcd (n,p-1)=1,\) is a variable. Two versions are given: a multiplicative version, and an additive version.
\begin{lem} \label{lem3.2}
Let \(p\geq 2\) be a prime, and let \(\tau\) be a primitive root mod \(p\). For a nonzero element \(u\in\mathbb{F}_p\), the followings hold:
\begin{flushleft}
\begin{enumerate} 
\item If \(\psi \neq 1\) is a nonprincipal additive character of order \(\ord \psi =p\), then
\begin{equation}
\Psi (u)=\sum _{\gcd (n,p-1)=1} \frac{1}{p}\sum _{0\leq k\leq p-1} \psi \left ((\tau ^n-u)k\right)=\left \{
\begin{array}{ll}
1 & \text{ if } \ord_p(u)=p-1,  \\
0 & \text{ if } \ord_p(u)\neq p-1. \\
\end{array} \right .
\end{equation}
\item If \(\chi \neq 1\) is a nonprincipal multiplicative character of order \(\ord \chi =p-1\), then
\begin{equation}
	\Psi (u)=\sum _{\gcd (n,p-1)=1} \frac{1}{p-1}\sum _{0\leq k<p-1} \chi \left(\left(\tau ^n\bar{u}\right)^k\right)=
	\left \{ 
	\begin{array}{ll}
		1 & \text{ if } \ord_p(u)=p-1,  \\
		0 & \text{ if } \ord_p(u)\neq p-1, \\
	\end{array} \right.
\end{equation}
where \(\bar{u}\) is the inverse of \(u \bmod p\).
\end{enumerate}
\end{flushleft}
\end{lem}

\begin{proof}
 (i) As the index \(n\geq 1\) ranges over the integers relatively prime to \(p-1\), the element \(\tau ^n\in \mathbb{F}_p\) ranges over the primitive roots \(\text{mod } p\). Ergo, the equation
\begin{equation}\label{30-30}
\tau ^n- u=0
\end{equation} has a solution if and only if the fixed element \(u\in \mathbb{F}_p\) is a primitive root. Next, replace \(\psi (z)=e^{i 2\pi  k z/p }\) to obtain
\begin{equation}
\Psi(u)=\sum_{\gcd (n,p-1)=1} \frac{1}{p}\sum_{0\leq k\leq p-1} e^{i 2\pi  (\tau ^n-u)k/p }=\left \{
\begin{array}{ll}
1 & \text{ if } \ord_p (u)=p-1,  \\
0 & \text{ if } \ord_p (u)\neq p-1, \\
\end{array} \right.
\end{equation}
	
This follows from the geometric series identity $\sum_{0\leq k\leq N-1} w^{ k }=(w^N-1)/(w-1)$ with $w \ne 1$, applied to the inner sum. For (ii), use the equation \(\tau ^n\bar{u}=1\), where \(\bar{u}\) is the inverse of \(u\), and apply the geometric series identity.
\end{proof}

\section{$d$th Power Residues}
All the characteristic functions representations for primitive elements can be modified to handle nonprimitive elements of order $\ord_p(v)=(p-1)/d$ in a finite cyclic group $G$ of $p-1=\#G$ elements, and any proper divisor \(d\,|\,p-1\). The the new \textit{divisors-free} representation of the characteristic function for non primitive element is stated here. 

\begin{dfn}
Let $p\geq 2$ be a prime, and let $d \,|\, p-1$. A primitive $d$th power residue 
is an element $u \in (\mathbb{Z}/ p\mathbb{Z})^{\times}$ such that $u^{k}\not \equiv 1 \bmod p$ for all $1 \leq k<d$, and $u^{d}\equiv 1 \bmod p$.
\end{dfn}

Given a primitive root \(\tau\) mod \(p\), the subset 
\begin{equation}\label{30009}
	\{ \tau^{dn}:n\geq 1 \text{ and } \gcd(n,(p-1)/d)=1\}
\end{equation}
consists of all the primitive $d$th power residues. 

\begin{lem} \label{lem3.3}
	Let \(p\geq 2\) be a prime, let \(\tau\) be a primitive root mod \(p\), and let \(\psi \neq 1\) be a nonprincipal additive character of order \(\ord \psi =p\). If \(u\in
	\mathbb{F}_p\) is a nonzero element, and $d|p-1$, then
			\begin{equation}
			\Psi_d (u)=\sum _{\gcd (n,(p-1)/d)=1} \frac{1}{p}\sum _{0\leq k\leq p-1} \psi \left ((\tau ^{dn}-u)k\right)
			=\left \{
			\begin{array}{ll}
			1 & \text{ if } \ord_p(u)=(p-1)/d,  \\
			0 & \text{ if } \ord_p(u)\neq (p-1)/d. \\
			\end{array} \right .
			\end{equation}
			
\end{lem}

\section{Primitive Points Tests}
For a prime $p \geq 2$, the group of points on an elliptic curve $E:y^2=f(x)$ is denoted by $E(\mathbb{F}_p)$. Several definitions and elementary properties of elliptic curves and the $n$-division polynomial $\psi_n(x,y)$ are sketched in Chapter 14. \\

\begin{dfn}
The order $\min \{k \in \mathbb{N}: kP=\mathcal{O} \}$ of an elliptic point is denoted by $\ord_E(P)$. A point is a \textit{primitive point} if and only if $\ord_E(P)=n$. 
\end{dfn}

\begin{lem} \label{lem3.3b}
	If $ E(\mathbb{F}_p)$ is a cyclic group, then it contains a primitive point $P \in E(\mathbb{F}_p)$.  
\end{lem}

\begin{proof}
	By hypothesis, $ E(\mathbb{F}_p) \cong \mathbb{Z}/n\mathbb{Z}$, and the additive group $\mathbb{Z}/n\mathbb{Z}$ contains $\varphi(n)\geq 1$ generators (primitive roots). 
\end{proof}

More generally, there is map into a cyclic group
\begin{equation}
E(\mathbb{Q})/E(\mathbb{Q})_{\normalfont{tors}} \longrightarrow \mathbb{Z}/m\mathbb{Z},  
\end{equation}
for some $m=\#E(\mathbb{F}_p)/d$, with $d \geq 1$; and the same result stated in the Lemma holds in the smaller cyclic group $\mathbb{Z}/m\mathbb{Z}\subset\mathbb{Z}/n\mathbb{Z}$.\\

\begin{lem} \label{lem3.4}
	Let $\# E(\mathbb{F}_p)=n$ and let $P \in E(\mathbb{F}_p)$. Then, $P$ is a primitive point if and only if $(n/q)P\ne \mathcal{O}$ for all prime divisors $q\,|\,n$. 
\end{lem}

Basically, this is the classical Lucas-Lehmer primitive root test applied to the group of points $E(\mathbb{F}_p)$. Another primitive point test intrinsic to elliptic curves is the $n$-division polynomial test.\\

\begin{lem} \label{lem3.5} {\normalfont ($n$-Division primitive point test)}
	Let $\# E(\mathbb{F}_p)=n$ and let $P \in E(\mathbb{F}_p)$. Then, $P$ is a primitive point if and only if $\psi_{n/q}(P) \ne 0$ for all prime divisors $q\,|\,n$. 
\end{lem}

The basic proof stems from the division polynomial relation
\begin{equation}\label{80-99}
mP=\mathcal{O} \Longleftrightarrow \psi_{m}(P)= 0,
\end{equation}
see \cite[Proposition 1.25]{SZ03}. The elliptic primitive point test calculations in the penultimate lemma takes place in the set of integer pairs $\mathbb{Z} \times \mathbb{Z}$, while the calculations for the $n$-division polynomial primitive point test takes place over the set of integers $\mathbb{Z}$. The elementary properties of the $n$-division polynomials are discussed in \cite{SZ03}, and the periodic property of the $n$-division polynomials appears in \cite{SJ05}.\\

\section{Elliptic Pseudoprimes}
Elliptic curves with complex multiplication have many special properties. One of these property is \textit{elliptic pseudoprimality}. 

\begin{dfn}
Let $ E(\mathbb{F}_p)$ be an elliptic group of order $n= \# E(\mathbb{F}_p)$. The order of a point $P \in  E(\mathbb{F}_p)$ is an\textit{ elliptic pseudoprime} if the followings holds.
\begin{enumerate}
\item $\displaystyle \frac{n+1}{2}P=\mathcal{O} \bmod n$,
\item $\displaystyle \frac{n+1}{2}P=Q \in E(\mathbb{Q})_{tors}$.
\end{enumerate} 
Furthermore, the order of a point is a \textit{strong elliptic pseudoprime} if $n+1=2^st$ with $t \geq 1$ odd, and  the followings holds.
\begin{enumerate}
\item $\displaystyle tP=\mathcal{O} \bmod n$,
\item $\displaystyle 2^rtP=Q \in E(\mathbb{Q})_{tors}$ 
for some $0\leq r<s$.
\end{enumerate}  
\end{dfn}
The basic theory for the existence of elliptic pseudoprimes are discussed in \cite{SM09} and by earlier authors.

\section{Additive Elliptic Character}
The discrete logarithm function, with respect to the fixed primitive point $T$, maps the group of points into a cyclic group. The diagram below specifies the  basic assignments.

\begin{equation}
\begin{array} {lll}
E(\mathbb{F}_p)& \longrightarrow & \mathbb{Z}/n\mathbb{Z}, \\
\mathcal{O}& \longrightarrow &\log_T (\mathcal{O})=0, \\
T& \longrightarrow &\log_T(T)=1. \\
\end{array} 
\end{equation}

In view of these information, an important character on the group of $E(\mathbb{F}_p)$-rational points can be specified.

\begin{dfn}
	A nontrivial additive elliptic character $\chi \bmod n$ on the group $E(\mathbb{F}_p)$ is defined by
	\begin{equation} \label{200-09}
	\chi(\mathcal{O})=e^{\frac{i2 \pi}{n}\log_T\mathcal({O})}=1, 
	\end{equation}
	and
	\begin{equation} \label{200-10}
	\chi(mT)=e^{\frac{i2 \pi}{n}\log_T(mT)}=e^{i2 \pi m/n}, 
	\end{equation}
	where $\log_T(mT)=m$ with $m \in \mathbb{Z}$.
\end{dfn}

\section{Divisors Dependent Characteristic Function}
A characteristic function for primitive points on elliptic curve is described in the \cite[p.\ 5]{SV11}; it was used there to derive a primitive point search algorithm.

\begin{lem} \label{lem3.6}
	Let $E$ be a nonsingular elliptic curve, and let $E(\mathbb{F}_p)$ be its group of points of cardinality $n=\#E(\mathbb{F}_p)$. Let $\chi$ be the additive character of order $d$ modulo $d$, and assume $P=(x_0,y_0)$ is a point on the curve. Then, 
	\begin{equation}\label{el33007}
	\Psi_E (P)=\sum _{d \,|\, n} \frac{\mu (d)}{d}\sum _{0 \leq t <d} \chi(tP)=
	\left \{\begin{array}{ll}
	1 & \text{ if } \ord_E (P)=n,  \\
	0 & \text{ if } \ord_E (P)\neq n. \\
	\end{array} \right .
	\end{equation}
\end{lem}

\section{Divisors Free Characteristic Function}
A new \textit{divisors-free} representation of the characteristic function of elliptic primitive points is developed here. This representation can overcomes some of the limitations of its equivalent in \ref{lem3.6} in certain applications. The \textit{divisors representation} of the characteristic function of elliptic primitive points, Lemma \ref{lem3.6}, detects the order \(\text{ord}_E (P)\) of the point \(P\in E( \mathbb{F}_p) \) by means of the divisors of the order \(n=\#E( \mathbb{F}_p) \). In contrast, the \textit{divisors-free representation} of the characteristic function, Lemma \ref{lem3.7}, detects the order \(\text{ord}_E(P) \geq 1\) of a point \(P\in E(\mathbb{F}_p)\) by means of the solutions of the equation $mT-P=\mathcal{O}$, where \(P,T \in E(\mathbb{F}_p)\) are fixed points, \(\mathcal{O}\) is the identity point, and $m $ is a variable such that $0\leq m \leq n-1$, and $\gcd (m,n)=1$. \\

\begin{lem} \label{lem3.7}
	Let \(p\geq 2\) be a prime, and let \(T\) be a primitive point in \(E(\mathbb{F}_p)\). For a nonzero point \(P \in
	E(\mathbb{F}_p)\) of order $n$ the following hold:
	If \(\chi \neq 1\) is a nonprincipal additive elliptic character of order \(\ord \chi =n\), then
	\begin{equation}
	\Psi_E (P)=\sum _{\gcd (m,n)=1} \frac{1}{n}\sum _{0\leq r\leq n-1} \chi \left ((mT-P)r\right)
	=\left \{
	\begin{array}{ll}
	1 & \text{ if } \ord_E(P)=n,  \\
	0 & \text{ if } \ord_E(P)\neq n, \\
	\end{array} \right .
	\end{equation}
	where \(n=\#
	E(\mathbb{F}_p)\) is the order of the rational group of points.
\end{lem}

\begin{proof} As the index \(m\geq 1\) ranges over the integers relatively prime to \(n\), the element \(mT\in E(\mathbb{F}_p)\) ranges over the elliptic primitive points. Ergo, the linear equation 
	\begin{equation}
	mT-P=\mathcal{O},
	\end{equation}
	where \(P,T \in E(\mathbb{F}_p)\) are fixed points, \(\mathcal{O}\) is the identity point, and $m $ is a variable such that $0\leq m \leq n-1$, and $\gcd (m,n)=1$, has a solution if and only if the fixed point \(P\in E(\mathbb{F}_p)\) is an elliptic primitive point. Next, replace \(\chi (t)=e^{i 2\pi  t/n }\) to obtain
	\begin{equation}
	\Psi_E(P)=\sum_{\gcd (m,n)=1} \frac{1}{n}\sum_{0\leq r\leq n-1} e^{i 2\pi  \log_T(mT-P)r/n }=
	\left \{\begin{array}{ll}
	1 & \text{ if } \ord_E (P)=n,  \\
	0 & \text{ if } \ord_E (P)\neq n, \\
	\end{array} \right.
	\end{equation}
	
	This follows from the geometric series identity $\sum_{0\leq k\leq N-1} w^{ k }=(w^N-1)/(w-1)$ with $w \ne 1$, applied to the inner sum.   
\end{proof}

\newpage	
\section{Problems}
\begin{exe} \normalfont Let $p \geq 2$ be a prime, and let $v=\log u$ be the discrete logarithm modulo $p$. Verify the different version of the characteristic function 
$$\Psi (u)=\sum _{d \,|\, p-1} \frac{\mu (d)}{d}\sum _{1\leq m\leq d} e^{i2 \pi mv/p}=
		\left \{\begin{array}{ll}
		1 & \text{ if } \text{ord}_p (u)=p-1,  \\
		0 & \text{ if } \text{ord}_p (u)\neq p-1. \\
		\end{array} \right .$$
		for a finite cyclic group $G$ of order \(p-1=\# G\), and let \(0\neq u\in G\) be an invertible element of the group. This version was used as early as 1960's or before, see \cite{SP69}.
	\end{exe}
	
	\begin{exe} \normalfont Let $0 \ne u \in \mathbb{F}_p$. Verify the characteristic function for quadratic residues in finite fields $$
		\Theta_2 (u)=\sum _{1\leq n \leq p} \frac{1}{p}\sum _{0\leq k <p} e^{i 2\pi (n^2-u)k/p}=
		\left \{\begin{array}{ll}
		1 & \text{ if } u\equiv m^2 \bmod p,  \\
		0 & \text{ if } u\not \equiv m^2 \bmod p. \\
		\end{array} \right . .
		$$
For $u=0$ the equation reduces to $\varTheta_2 (u)=1+\sum _{1\leq n \leq p} \frac{1}{p}\sum_{0\leq k <p} e^{i 2\pi n^2k/p}=1$, use the standard quadratic Gaus sum, see \cite[p.\ 195]{AP98}, to check this.
	\end{exe}
	
\begin{exe} \normalfont Let $(n|p)$ and $\Lambda(n)$ be the quadratic symbol, and the vonMangoldt function. Verify the weighted characteristic function for prime quadratic residues in finite fields $$
	\varTheta_2 (n)=\frac{1+(n\,|\,p)}{2} \Lambda(n)
=
\left \{\begin{array}{ll}
\log n & \text{ if } n \equiv m^2 \bmod p,  \\
0 & \text{ if } n \not \equiv m^2 \bmod p. \\
\end{array} \right . .	$$
	
\end{exe}\begin{exe} \normalfont Let $(n|p)$ and $\Lambda(n)$ be the quadratic symbol, and the vonMangoldt function. Verify the weighted characteristic function for prime quadratic nonresidues in finite fields $$
\Psi_2 (n)=\frac{1-(n\,|\,p)}{2} \Lambda(n)
=
\left \{\begin{array}{ll}
\log n & \text{ if } n\not \equiv m^2 \bmod p,  \\
0 & \text{ if } n \equiv m^2 \bmod p. \\
\end{array} \right . .	$$

\end{exe}

\begin{exe} \normalfont Let $0 \ne u \in \mathbb{F}_p$. Verify the characteristic function for cubic nonresidues in finite fields $$
	\varTheta_3 (u)=\sum _{1\leq n \leq p} \frac{1}{p}\sum _{0\leq k <p} e^{i 2\pi (n^3-u)k/p}=
	\left \{\begin{array}{ll}
	1 & \text{ if } u\equiv m^3 \bmod p,  \\
	0 & \text{ if } u\not \equiv m^3 \bmod p. \\
	\end{array} \right . .
	$$
	For $u=0$ the equation reduces to $\varTheta_3 (u)=1+\sum _{1\leq n \leq p} \frac{1}{p}\sum _{0\leq k <p} e^{i 2\pi n^3k/p}=1$, use the standard cubic Gauss sum, see \cite[p.\ 195]{AP98}, to check this.
\end{exe}

\begin{exe} \normalfont Show that the characteristic function in Lemma \ref{lem3.6} has a product version  
\begin{equation}\label{el33107}
\Psi_E (P)=\prod_{p \,|\, n} \left (1-\frac{\sum _{\ord(\chi)=p} \chi(P)}{p} \right )=
\left \{\begin{array}{ll}
1 & \text{ if } \ord_E (P)=n,  \\
0 & \text{ if } \ord_E (P)\neq n. \\
\end{array} \right .
\end{equation}
\end{exe}
	
	\begin{exe} \normalfont Verify the inverse Euler phi function identity  
		\begin{equation}
		\frac{n}{\varphi(n)}=\sum _{d \,|\, n} \frac{\mu (d)^2}{\varphi(d)} =\prod_{p \,|\, n}\left ( 1+\frac{1}{p-1} \right ).
		\end{equation}
	\end{exe}

%ccccccccccccccccc	
\chapter{Estimates Of Exponential Sums} \label{c4}
Exponential sums indexed by the powers of elements of nontrivial orders have applications in mathematics and cryptography. These applications have propelled the development of these exponential sums. There are many results on exponential sums indexed by the powers of elements of nontrivial orders, the interested reader should consult the literature, and references within the cited papers. \\

\section{Elementary Upperbound}
This subsection provides a simpler estimate of the exponential sum of interest in this analysis. The proof of this estimate is entirely based on 
elementary techniques.\\

\begin{lem}  \label{lem4.1}
Let \(x\geq 1\) be a large number, and let $q= \prod_{r \leq \log \log x}r \asymp \log x$ with $r \geq 2$ prime. Define the subset of primes
\( \mathcal{P}=\{p=qn+1: n \geq 1\}\). If the element \(\theta \in \mathbb{F}_p\) has multiplicative order \(p-1\), then
\begin{equation}
\max_{\gcd(a,p-1)=1} \left|  \sum_{\gcd(m,p-1)=1} e^{i2\pi a \theta ^m/p} \right| \ll\frac{p}{\log \log \log p} 
\end{equation} 
for each $p=qn+1 \in [x^{1/2},x]$. Moreover, $\mathcal{P}(x)=\#\{p=qn+1\leq x\} \gg x/ \log^2  x$. 
\end{lem}

\begin{proof} The maximal absolute value of the exponential sum has the trivial upper bound
\begin{eqnarray} \label{el14040}
	\max_{\gcd(a,p-1)=1} \left |\sum _{\substack{ 1\leq m\leq p-1\\ \gcd(m,p-1)=1}}  e^{i2 \pi a \theta^m/p} \right |
	&\leq& \sum _{\substack{1\leq m <p-1 \\ \gcd(m,p-1)=1}}1 \\
	&=&\varphi(p-1) \nonumber. 
\end{eqnarray}
Fix a large number $x \geq 1$ and let $q= \prod_{r \leq \log \log x}r \asymp \log x$ with $r\geq 2$ prime. Take $p=qn+1$ such that $x^{1/2}\leq p \leq x$. Then, $\log x^{1/2} \leq \log p \leq \log x$, and  
\begin{equation}
	\log q=\sum_{r \leq \log \log x}\log r\ll \log \log x \ll 2\log \log p.
\end{equation}
This follows from the prime number theorem, see \cite[Theorem 414]{HW08}, \cite[p.\ 34]{IK04}, \cite[Theorem 2.4]{MV07}, and similar references. This implies that $r \leq 2\log \log p$.
Furthermore, since $p-1=qn, n \geq 1$, the Euler function satisfies
\begin{eqnarray} \label{el14045}
	\frac{\varphi(p-1)}{p-1}&=&\prod_{r\;|\;p-1}\left (1-\frac{1}{r} \right ) \nonumber \\
	&\leq& \prod_{r\;|\;q}\left (1-\frac{1}{r} \right ) \\
	&\leq&\prod_{r\ll \log \log p}\left (1-\frac{1}{r} \right ) \nonumber \\
	&\ll& \frac{1}{\log \log \log p} \nonumber,
\end{eqnarray}
see \cite[Theorem 429]{HW08}, \cite[p.\ 50]{MV07}, et alii.     
\end{proof}

Some information on the extreme values of the Euler phi function appear in \cite[Theorem 328]{HW08}, \cite[Theorem 2.9]{MV07}, and \cite[p.\ 115]{TG15}.  \\

\begin{cor} For a large prime $p \geq 2$, the subset $\mathcal{A}=\{\tau^n: \gcd(n,p-1)=1\}$, with $\tau \in \mathbb{F}_p$ a primitive root, is uniformly distributed in the interval $(1,p-1)$. 
\end{cor}

\section{Double Exponential Sums}
Various  estimates for exponential sum 
\begin{equation}
\sum _{ n \in \mathcal{X}}  e^{i2 \pi a \tau^n/p} 
\end{equation}
over large arbitrary subsets $\mathcal{X} \subset \mathbb{F}_p$ and subgroups $H \subset \mathbb{F}_p$ are studied in \cite{BJ07}, \cite{BN04},  \cite{KS12}, et alii. A few results for arbitrary subsets \(X\subset \mathbb{F}_p\) and subgroups are also given here. \\
	
\begin{thm} \label{thm4.1}  {\normalfont (\cite[Lemma 4]{FS02})} For integers \(a, k, N\in \mathbb{N}\), assume that \(\gcd (a,N)=c\), and that \(\gcd (k,t)=d\).
\begin{enumerate} 
\item If the element \(\theta \in \mathbb{Z}_N\) is of multiplicative order \(t\geq t_0\), then 
\begin{equation}
\max_{1\leq a\leq p-1} \left| \sum _{ 1\leq x\leq t} e^{i2\pi a \theta ^{k x}/N} \right| <c d^{1/2}N^{1/2} .
\end{equation} 
\item If \(H\subset \mathbb{Z}/N \mathbb{Z}\) is a subset of cardinality \(\# H\geq N^{\delta }, \delta >0\), then 
\begin{equation}
\max_{\gcd (a,\varphi(N))=1}\left|  \sum _{ x\in H} e^{i2\pi a\theta ^x/N} \right| <N^{1-\delta } .
\end{equation}
\end{enumerate}
\end{thm}

Some versions of the next result for double exponential sums were proved in \cite{FS01}, and \cite{GZ05}. 
	
\begin{thm} \label{thm4.2}  {\normalfont (\cite[Theorem 1]{GK05})} Let \(p\geq 2\) be a large prime, and let \(\tau \in \mathbb{F}_p\) be an element of large multiplicative order $T=\ord_p(\tau)$. Let \(X,Y \subset \mathbb{F}_p\) be large subsets of cardinality $\# X  \geq  \# Y \geq p^{7/8+\varepsilon}$ Then, for any given number $k>0$, the following estimate holds.
\begin{equation}
\sum_{y \in Y}\left |  \sum_{ x \in X} \gamma(x) e^{i2\pi a \tau^{xy}/p} \right | \ll \frac{\#Y^{1-\frac{1}{2k}} 
\cdot \#X^{1-\frac{1}{2k+2}} \cdot p^{\frac{1}{2k}+\frac{3}{4k+4}+o(1)}}{T^{\frac{1}{2k+2}}},
\end{equation}
where the coefficients $\gamma(x)$ are complex numbers, $|\gamma(x)| \leq 1$, and $a \geq 1, \gcd(a,p)=1$ is an integer. 
\end{thm}

\begin{thm}  \label{thm4.3} {\normalfont (\cite[Theorem 17]{FS02})} Let \(n\geq 2\) be a large integer, and let \(\tau \in \mathbb{Z}_n\) be an element of large multiplicative order $T=\ord_n(\tau)$. Let \(X,Y \subset \mathbb{Z}_n\) be large subsets of cardinality $\# X $ and $ \# Y$. For any integer \(a\geq 1\) with \(\gcd(a,n)=\delta_a\), the following estimate holds.
\begin{equation}
\sum_{y \in Y}\left |  \sum_{ x \in X} e^{i2\pi a xy/n} \right | \ll \#Y^{1/2} 
\cdot \#X^{21/32} \cdot \delta_a^{1/8} \cdot T^{1/2} \cdot n^{5/16+\varepsilon},
\end{equation}
where $\varepsilon>0$ is an arbitrary small number. 
\end{thm}

\section{Sharp Upperbounds}
The previous established estimates for double exponential sums are modified to fit the requirements of various results as in Theorem \ref{thm1.1}, Theorem \ref{thm17.1}, and Theorem \ref{thm800.1}.   \\

\begin{lem}  \label{lem4.2}
	Let \(p\geq 2\) be a large prime, and let $\tau $ be a primitive root modulo $p$. Then,
\begin{equation}
	\max_{\gcd(s,p-1)=1} \left|  \sum_{\gcd(n,p-1)=1} e^{i2\pi s \tau^n/p} \right| \ll p^{1-\varepsilon} 
\end{equation} 
	for any arbitrary small number $\varepsilon$.  
\end{lem}

\begin{proof} Let $X=Y=\{n:\gcd(n,p-1)=1\}\subset \mathbb{F}_p$ be subsets of cardinalities $\#X=\#Y=\varphi(p-1)$. Since for each fixed $y \in Y$, the map $x \longrightarrow xy$ is a permutation of the subset $X \subset \mathbb{F}_p$, the double sum can be rewritten as a single exponential sum
\begin{equation}
\sum_{y \in Y} \left|  \sum_{x \in X} e^{i2\pi s \tau ^{xy}/p} \right| = \#Y  \left|  \sum_{x \in X} \gamma(x) e^{i2\pi s \tau ^{x}/p} \right|,  
\end{equation}
where $\gamma(x)=1$ if $x \in X$ otherwise it vanishes. As the element $\tau$ is a primitive root, it has order $T=\ord_p(\tau)=p-1$. Setting $k=1$ in Theorem \ref{thm4.2} yields 
\begin{eqnarray}
\left|  \sum_{x \in X} e^{i2\pi s \tau ^{x}/p} \right| &=&\frac{1}{\#Y}\sum_{y \in Y} \left|  \sum_{x \in X} e^{i2\pi s \tau ^{xy}/p} \right| \nonumber \\
&\ll& \frac{1}{\#Y} \left (\#Y^{1/2} \cdot \#X^{3/4}\cdot\frac{p^{7/8+\varepsilon}}{T^{1/4}} \right ) \\
&\ll& \frac{1}{\varphi(p-1)} \left (\varphi(p-1)^{1/2} \cdot \varphi(p-1)^{1/2}\cdot\frac{p^{9/8+\varepsilon}}{(p-1)^{1/4}} \right ) \nonumber\\
&\ll& p^{7/8+\varepsilon} \nonumber .
\end{eqnarray}
This is restated in the simpler notation $p^{7/8+\varepsilon} \leq p^{1-\varepsilon}$ for any arbitrary small number $\varepsilon \in (0,1/16)$. 
\end{proof}

A different proof of this result appears in \cite[Theorem 6]{FS00}. Other exponential sums which have equivalent index sets can be derived from this result. An interesting case is demonstrated in the next result.\\

\begin{lem}  \label{lem4.3}
	Let \(p\geq 2\) be a large prime, and let $s $ be an integer, $1 \leq s <p$. Then,
	\begin{equation}
	\max_{\gcd(s,p-1)=1} \left|  \sum_{\gcd(n,p-1)=1} e^{i2\pi sn/p} \right| \ll p^{1-\varepsilon} 
	\end{equation} 
	for any arbitrary small number $\varepsilon$.  
\end{lem}

\begin{proof} Let $\tau$ be a primitive root in $\mathbb{F}_p$, and let $Z=\{m=\tau^n:\gcd(n,p-1)=1\}\subset \mathbb{F}_p$ be a subset of cardinality $\#Z=\varphi(p-1)$. Since the map $n \longrightarrow m=\tau^n \bmod p$ is one-to-one, the subset $Z$ is a permutation of the subset $X=\{n:\gcd(n,p-1)=1 \} \subset \mathbb{F}_p$. Consequently, this exponential sum can be rewritten as
	\begin{equation}
	\sum_{m \in Z} \left|  e^{i2\pi s m/p} \right| =   \left|  \sum_{n \in X}  e^{i2\pi s \tau ^{n}/p} \right|.  
	\end{equation}
Thus, the claim follows from Lemma \ref{lem4.2}. 
\end{proof}

\section{Composite Moduli}
Given a composite modulo $n=2s$, and the worst case parameter $r=s$, such that $1 \leq r <n$, the exponential sum
\begin{equation}
\sum_{\gcd(m,n)=1} e^{i2\pi r m/n}= \sum_{\gcd(m,n)=1} e^{i\pi m}=-\varphi(n)
\end{equation} 
has a very weak (nearly trivial) upper bound. However, a more general result provides a nontrivial upper bound for certain subsets of integers moduli. These estimates are suitable for many applications.
	
\begin{lem}  \label{lem4.4}
Let \(z\geq 2\) be a number, and let $n \in \mathcal{R}(z) =\{n \geq 1: p|n \Rightarrow p \geq z\}$. Then,
\begin{equation}
\max_{1\leq r<n} \left|  \sum_{\gcd(m,n)=1} e^{i2\pi r m/n} \right| \leq \frac{n}{z}. 
\end{equation}
  
\end{lem}

\begin{proof} Let $n=p_1n_2$ be a large integer, where $p_1=Q(n)$ is the least prime divisor $p_1 | n$, and let $m=a+bp_1$ with $1\leq a<p_1,0\leq b<n_2$. Now, decompose the finite sum as
\begin{equation} \label{300-100}
\sum_{\gcd(m,n)=1} e^{i2\pi r m/n} = \sum_{0\leq b<n_2,} \sum_{1\leq a<p_1}e^{i2\pi r (a+bp_1)/n}
\end{equation}
where $\gcd(a+bp_1,n)=1$. The worse case absolute value of (\ref{300-100}) occurs whenever $\gcd(r,n)$ is maximal. Since $r<n$, and $ n \in \mathcal{R}(z)$, the worst case for the modulo $n$ is $r=r_1n_2$ with $\gcd(r_1,n)=1$. Substitute these quantities to reduce it to 
\begin{eqnarray}
\sum_{0\leq b<n_2,} \sum_{1\leq a<p_1}e^{i2\pi r_1n_2 (a+bp_1)/n}
&=&\sum_{0\leq b<n_2,} \sum_{1\leq a<p_1}e^{i2\pi r_1 (a+bp_1)/p_1} \nonumber \\
&=&\sum_{0\leq b<n_2} \sum_{\substack{1\leq a<p_1 \\ \gcd(ar_1,n)=1}}e^{i2\pi r_1 a/p_1} \\
&\leq &n_2 \left |\sum_{1\leq a<p_1 }e^{i2\pi r_1 a/p_1} \right |\nonumber. 
\end{eqnarray}
As the inner sum $\sum_{1\leq a<p_1 }e^{i2\pi r_1 a/p_1} =-1$ is a geometric series, it yields
\begin{eqnarray}
\max_{1\leq r<n} \left|  \sum_{\gcd(m,n)=1} e^{i2\pi r m/n} \right| 
&\ll& n_2 \cdot|-1| \\ 
&\ll&\frac{n}{p_1} \nonumber, 
\end{eqnarray} 
where $p_1=Q(n) \geq z$. \end{proof}

The restriction to the moduli $n=p_1p_2$ with $p_1,p_2 \asymp n^{1/2}$ has an optimum upper bound. In fact, these moduli exhibit square root cancellations.\\

\begin{cor}  \label{cor4.2}
Let \(n=p_1p_2\geq 2\) be a large integer with $n^{1/2} \ll p_1,p_2 \ll n^{1/2}$. Then,
\begin{equation}
\max_{1\leq r<n} \left|  \sum_{\gcd(m,n)=1} e^{i2\pi r m/n} \right| \ll n^{1/2} 
\end{equation} 
\end{cor}

\begin{proof} Let $m=a+bp_1$ with $1\leq a<p_1,0\leq b<p_2$, where $p_1=Q(n)$ is the least prime divisor $p_1 | n$, and $n=p_1p_2$. Now, rewrite the sum as
\begin{equation}
\sum_{\gcd(m,n)=1} e^{i2\pi r m/n} = \sum_{0\leq b<p_2,} \sum_{1\leq a<p_1}e^{i2\pi r (a+bp_1)/n}
\end{equation}
where $\gcd(a+bp_1,n)=1$. Since $r<n$ the worst case of this parameter is $r=r_1p_2$ with $\gcd(r_1,n)=1$. Substitute these quantities to reduce it to 
\begin{eqnarray}
\sum_{0\leq b<p_2,} \sum_{1\leq a<p_1}e^{i2\pi r_1p_2 (a+bp_1)/n}
&=&\sum_{0\leq b<p_2,} \sum_{1\leq a<p_1}e^{i2\pi r_1 (a+bp_1)/p_1} \nonumber \\
&=&\sum_{0\leq b<p_2} \sum_{\substack{1\leq a<p_1 \\ \gcd(ar_1,n)=1}}e^{i2\pi r_1 a/p_1} \\
&=&p_2\sum_{1\leq a<p_1 }e^{i2\pi r_1 a/p_1} \nonumber. 
\end{eqnarray}
As the inner sum $\sum_{1\leq a<p_1 }e^{i2\pi r_1 a/p_1} =-1$ is a geometric series, it yields
\begin{eqnarray}
\max_{1\leq r<n} \left|  \sum_{\gcd(m,n)=1} e^{i2\pi r m/n} \right| &=&n_2 \left | \sum_{1 \leq a<p_1 }e^{i2\pi r_1a /p_1} \right | \nonumber \\
&\ll& n_2 \cdot|-1| \\ 
&\ll&\frac{n}{p_1} \nonumber .
\end{eqnarray} 
The other possibility $r=r_1p_2$ with $\gcd(r_1,n)=1$ is similar. 
\end{proof}

%ccccccccccccccccccccccccccccccccccccccccccccccccccccc
\chapter{Main Term And Error Term} \label{c5}
The required estimates and asymptotic formulas for the main term $M(x)$ and the error term $E(x)$ for the primitive root producing polynomials problems are assembled in this chapter. The proof of Theorem \ref{thm1.1} appears in the last section.\\

\section{Evaluations Of The Main Terms}
	A pair of different cases are considered here. These cases are general prime producing polynomial $f(x) \in \mathbb{Z}[x]$ and the special polynomial $f(x)=x^2+1$, which has its own literature, are discussed in details.\\
	
	\subsection{General Case}
	\begin{lem} \label{lem5.2}
		Let $f(n)$ be a prime producing polynomial of degree $\deg(f)=m$, and let \(\Lambda (n)\) be the vonMangoldt function. Assume that the Bateman-Horn conjecture holds. Then, for any large number \(x\geq 1\), 
		\begin{equation} \label{e6l59}
		\sum_{f(n)\leq x} \frac{\Lambda(f(n))}{f(n)}\sum_{\gcd(r,f(n)-1=1} 1 =t_f x^{1/m}+O \left (\frac{x^{1/m}}{\log^B x}\right ),
		\end{equation} 
		where the density constant $t_f>0$ depends on $f(x)$, and the fixed primitve root $u\ne 0$. The parameter $B>0$ is an arbitrary constant.
	\end{lem}

\begin{proof}
 Let $f(n)$ be a primitive root producing polynomial of degree $\deg(f)=m \geq 1$. A routine rearrangement gives
	\begin{eqnarray} \label{el6500}
	M(x)&=&\sum_{f(n)\leq x} \frac{\Lambda(f(n))}{f(n)}\sum_{\gcd(m,f(n)-1)=1} 1 \nonumber \\&=&\sum_{f(n)\leq x} \frac{\varphi(f(n)-1)}{f(n)}\Lambda(f(n))\\
	&=&\sum_{f(n)\leq x} \frac{\varphi(f(n)-1)}{f(n)-1}\Lambda(f(n))-\sum_{f(n)\leq x} \frac{\varphi(f(n)-1)}{f(n)(f(n)-1)}\Lambda(f(n)) \nonumber \\
	&=&S_0+S_1 \nonumber.
	\end{eqnarray} 
	Since $n \leq |f(n)|$, and $\Lambda(f(n)) \leq \log |f(n)|$, the second finite sum $S_1$ satisfies
	\begin{equation} \label{el6505}
	S_1=\sum_{f(n)\leq x} \frac{\varphi(f(n)-1)}{f(n)(f(n)-1)}\Lambda(f(n)) \leq \sum_{f(n)\leq x} \frac{\log |f(n)|}{f(n)} =O(\log x),
	\end{equation} 
	(it converges to a constant if $\deg(f) \geq 2$). Next, use $\varphi(N)/N= \sum_{d|N} \mu(d)/d$, see \cite[p.\ 17]{IK04}, and reverse the order of summation to evaluate the first finite sum $S_0$:
	\begin{eqnarray} \label{el6503}
	S_0&=&\sum_{f(n)\leq x} \frac{\varphi(f(n)-1)}{f(n)-1}\Lambda(f(n)) \nonumber \\
	&=& \sum_{f(n)\leq x} \Lambda(f(n)) \sum_{d\:|\:f(n)-1} \frac{\mu(d)}{d} \nonumber \\
	&=& \sum_{d\leq x} \frac{\mu(d)}{d} \sum_{f(n)\leq x,\; \;d\;|\;f(n)-1}\Lambda(f(n)) \\
	&=& \sum_{d\leq x} \frac{\mu(d) \omega_f(d)}{d} \sum_{f(n)\leq x,\; \;d\;|\;f(n)-1}\Lambda(f(n)) \nonumber,
	\end{eqnarray}
	where the indicator function is defined by
	\begin{equation} \label{el6504}
	\omega_f(d)=
	\begin{cases}
	1& \text{ if }f(n)-1\equiv 0\bmod d,\\
	0& \text{ if }f(n)-1\not \equiv 0\bmod d.
	\end{cases}
	\end{equation}
	For an arbitrary constant $B>C+1\geq 1$, the finite sum $S_0$ has a dyadic decomposition as $S_0=S_2+S_3$ over the subintervals $[1,\log^B x]$ and $[\log^Bx,x]$ respectively. \\

In addition, by Conjecture \ref{conj2.1},
\begin{equation} \label{el6523}
\sum_{\substack{f(n)\leq x \\d\;|\;f(n)-1}}\Lambda(f(n)) 
=s_f\frac{x^{1/m}}{d}+O\left (\frac{x^{1/m}}{d \log^B x} \right ).
\end{equation}

Applying the Bateman-Horn conjecture, see Conjecture \ref{conj2.1} in Chapter 2, to the finite sum $S_2$ over the first subinterval yields
	\begin{eqnarray} \label{el6506}
	S_2&=&\sum_{d\leq \log^B x} \frac{\mu(d) \omega_f(d)}{d} \sum_{f(n)\leq x,\; \;d\;|\;f(n)-1}\Lambda(f(n)) \nonumber \\ 
	& = &\sum_{d\leq \log^B x} \frac{\mu(d) \omega_f(d)}{d} \left ( \frac{s_f}{d} x^{1/m}+O \left (\frac{x^{1/m}}{d \log^B x}\right ) \right )\\
	&=& s_f x^{1/m}\sum_{d\leq \log^B x }  \frac{\mu(d)\omega_f(d)}{d^2} 
	+O \left (\frac{x^{1/m}}{ \log^B x} \sum_{d\leq \log^B x }  \frac{1}{d^2 } \right )  \nonumber\\
	&=&  s_f L(1, \omega_f ) x^{1/m}+O \left (\frac{x^{1/m}}{ \log^{B} x}\right ) \nonumber .
	\end{eqnarray} 
	Here $B>C+1 \geq1$, and $ \omega_f(d) =0, 1$, see (\ref{el6504}). The canonical density series 
	\begin{equation} \label{el50021}
	L(s,\omega_f )=\sum_{n\geq 1 } \frac{ \mu(n)\omega_f(n)}{n^{s+1}}
	\end{equation}
	converges for $\Re e(s) \geq1$. In term of the series, the finite sum $\sum_{n\leq z}  \frac{\mu(n)\omega_f(n)}{n^2}=L(1,\omega_f)+O(\log z/z)$. Thus, the density constant is given by $t_f=s_f L(1,\omega_f )>0$. Using the estimate
	\begin{equation}
	\sum_{f(n)\leq x, \;d\;|\;f(n)-1}\Lambda(f(n)) \ll \frac{x^{1/m}}{d},
	\end{equation} 
	the finite sum $S_3$ over the second subinterval has the upper bound
	\begin{eqnarray} \label{el6160}
	S_3&=&\sum_{\log^B x \leq d \leq x} \frac{\mu(d) \omega_f(d)}{d} \sum_{f(n)\leq x,\; \;d\;|\;f(n)-1}\Lambda(f(n)) \nonumber \\ 
	& \ll &x^{1/m} \sum_{\log^B x \leq d \leq x} \frac{ 1}{d^2} \\
	&\ll &\frac{x^{1/m}}{\log^Bx} \sum_{\log^B x \leq d \leq x} \frac{ 1}{d} \nonumber\\
	&\ll &\frac{x^{1/m}}{\log^{B-1} x} \nonumber.
	\end{eqnarray} 
	The third line uses $\omega_f(d)/d \ll 1/ \log^B x$; and the fourth line uses
	\begin{equation}\label{qw44}
	\sum_{n\leq x}\frac{1}{n} \ll \log x.
	\end{equation} 
	Summing expressions in equations (\ref{el6505}), (\ref{el6506}) and (\ref{el6160}) yield
	\begin{eqnarray} \label{el573}
	\sum_{f(n)\leq x} \frac{\varphi(f(n)-1)}{f(n)-1}\Lambda(f(n))
	&=& S_1+S_2+S_3 \nonumber \\
	&=&t_f x^{1/m}+O \left (\frac{x^{1/m}}{ \log^{B} x}\right ) +O \left (\frac{x^{1/m}}{ \log^{B-1} x}\right )  \\
	&=& t_f x^{1/m}+O \left (\frac{x^{1/m}}{ \log^{B-1} x}\right ) \nonumber,
	\end{eqnarray}
	where $B-1>0$.  
\end{proof}

	\subsection{Special Cases}
	The linear primes producing polynomials $f(x)=ax+b$, with $\gcd(a,b)=1$, is the subject of Dirichlet theorem for primes in arithmetic progressions. The next important case is the primes producing quadratic polynomials $f(x)=ax^2+bx+c \in \mathbb{Z}[x]$. \\
	
	\begin{lem} \label{lem5.3}
		Let \(x\geq 1\) be a large number, and let \(\Lambda (n)\) be the vonMangoldt function. Assume the Bateman-Horn conjecture. Then, there is a constant $t_f>0$ such that
		\begin{equation} \label{e600}
		\sum_{n^2+1\leq x} \frac{\Lambda(n^2+1)}{n^2+1}\sum_{\gcd(m,n^2)=1} 1 =t_2 x^{1/2}+O \left (\frac{x^{1/2}}{\log x}\right ),
		\end{equation}
where $t_2>0$ is a constant depending on the polynomial and the fixed primitive root. 
	\end{lem}
	
\begin{proof}  A rearrangement of the main term gives
	\begin{eqnarray} \label{el605}
	M(x)&=&\sum_{n^2+1\leq x} \frac{\Lambda(n^2+1)}{n^2+1}\sum_{\gcd(m,n^2)=1} 1 \nonumber \\
	&=&\sum_{n^2+1\leq x} \frac{\varphi(n^2)}{n^2+1}\Lambda(n^2+1)\\
	&=&\sum_{n^2+1\leq x} \frac{\varphi(n^2)}{n^2}\Lambda(n^2+1)-\sum_{n^2+1\leq x} \frac{\varphi(n^2)}{n^2(n^2+1)}\Lambda(n^2+1) \nonumber \nonumber\\
	&=& S_0+S_1 \nonumber .
	\end{eqnarray} 
	Since $1/\log n \leq \varphi(n^2)/n^2 \leq 1$, the value $n^2 \leq |f(n)|$, and $\Lambda(f(n)) \leq \log |f(n)|$, the finite sum $S_1$ converges to a constant
	\begin{equation} \label{el6570}
	S_1=\sum_{n^2+1\leq x} \frac{\varphi(n^2)}{n^2(n^2+1)}\Lambda(n^2+1) \leq \sum_{n^2+1\leq x} \frac{2\log n}{n^2} =O(1).
	\end{equation} 
	Next, use $\varphi(n)/n= \sum_{d|n} \mu(d)/d$, see \cite{AP98} or \cite[p.\ 17]{IK04}, and reverse the order of summation to evaluate the first finite sum:
	\begin{eqnarray} \label{el615}
	S_0&=&\sum_{n^2+1\leq x} \frac{\varphi(n^2)}{n^2}\Lambda(n^2+1) \nonumber \\
	&=& \sum_{n^2+1\leq x} \frac{\varphi(n)}{n}\Lambda(n^2+1)  \nonumber\\
	&=& \sum_{n^2+1\leq x} \Lambda(n^2+1) \sum_{d\;|\;n} \frac{\mu(d)}{d} \nonumber\\
	&=& \sum_{d\leq x} \frac{\mu(d)}{d} \sum_{n\leq x^{1/2}, \; d\;|\;n}\Lambda(n^2+1)  \\
	&=& \sum_{d\leq x} \frac{\mu(d) \omega_f(d)}{d} \sum_{n\leq x^{1/2}, \; d\;|\;n}\Lambda(n^2+1) \nonumber,
	\end{eqnarray} 
	where the indicator function is defined by
	\begin{equation} \label{el6604}
	\omega_f(d)=
	\begin{cases}
	1& \text{ if }n\equiv 0\bmod d,\\
	0& \text{ if }n\not \equiv 0\bmod d.
	\end{cases}
	\end{equation} 
	For an arbitrary constant $B>0$, the finite sum $S_1$ has a dyadic decomposition as $S_0=S_2+S_3$ over the subintervals $[1,\log^B x]$ and $[\log^Bx,x]$ respectively. \\
	
In addition, by Conjecture \ref{conj2.2},
\begin{equation} \label{el6623}
\sum_{\substack{n^2+1\leq x \\d\;|\;n}}\Lambda(n^2+1) 
=s_2\frac{x^{1/2}}{d}+O\left (\frac{x^{1/2}}{d \log^B x} \right ).
\end{equation}
	
	Applying the Bateman-Horn conjecture or Conjecture \ref{conj2.2} in Chapter 2, to the finite sum $S_2$ over the first subinterval yield
	\begin{eqnarray} \label{el6580}
	S_2&=&\sum_{d\leq \log^B x} \frac{\mu(d) \omega_f(d)}{d} \sum_{n^2+1\leq x, \;d\;|\;n}\Lambda(n^2+1)  \\ 
	& = &\sum_{d\leq \log^B x} \frac{\mu(d) \omega_f(d)}{d} \left ( \frac{s_f}{d} x^{1/2}+O \left (\frac{x^{1/2}}{ d\log^B x}\right ) \right )\nonumber\\
	&=& s_f x^{1/2}\sum_{d\leq \log^B x }  \frac{\mu(d)\omega_f(d)}{d^2} +O \left (\frac{x^{1/2}}{ d\log^B x} \sum_{d\leq \log^B x }  \frac{\omega_f(d)}{d } \right )  \nonumber \\
	&=&  s_f L(1, \omega_f ) x^{1/2}+O \left (\frac{x^{1/2}}{ \log^{B} x}\right ) \nonumber ,
	\end{eqnarray} 
	where the $\omega_f(d)=1$ for all $d\geq 1$ in the last sum. The canonical density series 
	has the simpler expression
	\begin{equation} \label{el50031}
	L(s,\omega_f )=\sum_{n\geq 1 } \frac{ \mu(n)\omega_f(n)}{n^{s+1}}=\prod_{p \geq 2} \left ( 1-\frac{1}{p^2}\right )
	\end{equation}
	which converges for $\Re e(s) \geq1$. In term of the series, the finite sum $\sum_{n\leq z}  \frac{\mu(n)\omega_f(n)}{n \varphi(n)}=L(1,\omega_f )+O(\log z/z)$. Thus, the density constant is given by $t_f=s_f L(1,\omega_f )>0$. Using the estimate $ \sum_{n^2+1\leq x, \;d\;|\;n}\Lambda(n^2+1) \ll x^{1/2}/\varphi(d)$, for the  finite sum $S_3$ over the second subinterval has the upper bound
	\begin{eqnarray} \label{el6590}
	S_3&=&\sum_{\log^B x \leq d \leq x} \frac{\mu(d) \omega_f(d)}{d} \sum_{n^2+1\leq x, \;d\;|\;n^2}\Lambda(n^2+1) \nonumber \\ 
	& \ll &x^{1/2} \sum_{\log^B x \leq d \leq x} \frac{ 1}{d^2} \nonumber\\
	&\ll &\frac{x^{1/2}}{\log^B x} \sum_{ d \leq x} \frac{1}{d} \nonumber\\
	&\ll &\frac{x^{1/2}}{\log^{B-1} x}.
	\end{eqnarray} 
	The third line uses $\omega_f(d)/d \ll 1/ \log^B x$; and the fourth line uses $\sum_{n\leq x}1/n \ll \log x$. Summing expressions in equations in equations (\ref{el6570}), (\ref{el6580}) and (\ref{el6590}) yield
	\begin{eqnarray} \label{el6573}
	\sum_{f(n)\leq x} \frac{\varphi(n^2)}{n^2}\Lambda(n^2+1)
	&=& t_2 x^{1/2}+O \left (\frac{x^{1/2}}{ \log^{B-1} x}\right ) +O \left (\frac{x^{1/2}}{ \log^{B-1} x}\right )  \\
	&=& t_2 x^{1/2}+O \left (\frac{x^{1/2}}{ \log^{B-1} x}\right ) \nonumber,
	\end{eqnarray}
	where $t_2>0$ is a constant, and $B-1>0$ is an arbitrary constant.   \end{proof}

\section{Estimates For The Error Term}
The upper bounds of exponential sums over subsets of elements in finite rings $\left (\mathbb{Z}/N\mathbb{Z}\right )^\times$ stated in the last Section are used to estimate the error term $E(x)$ in the proof of Theorem \ref{thm1.1}. An application of any of the Theorems \ref{thm4.1}, or \ref{thm4.2} or \ref{thm4.3} leads to a sharp result, this is completed below.\\
	
	\begin{lem} \label{lem5.1}
		Let $f(x) \in \mathbb{Z}[x]$ be a prime producing polynomial of degree $\text{deg}(f)=m$, let \(p=f(n)\geq 2\) be a large prime and let \(\tau\) be a primitive root mod \(p\). For \(k,u\in \mathbb{F}_p\), and let \(\psi \neq 1\) be an additive character. If the element
		\(u\ne 0\) is not a primitive root, then
		\begin{equation} \label{el5400}
		\sum_{f(n)\leq x}
		\frac{\Lambda(f(n))}{f(n)}\sum_{\gcd(a,f(n)-1)=1,} \sum_{ 0<k\leq f(n)-1} \psi \left((\tau ^a-u)k\right)\ll x^{(1-\varepsilon)/m},
		\end{equation} 
		where \(\varepsilon >0\) is an arbitrarily small number.
	\end{lem}
	
\begin{proof} Let $\psi(z)=e^{i 2 \pi kz/p}$ with $0< k<p$, and set $p=f(x)$. By hypothesis $u \ne \tau^n$ for any $n \geq 1$ such that $\gcd(n,p-1)=1$,  and a fixed primes producing polynomial $f(n)$ of degree $\deg(f)=m \geq 1$, the interval $[1,x]$ contains at most $2x^{1/m}/\log x$ primes $p=f(n) \leq x$. Therefore, $\sum_{ 0<k\leq p-1} e^{i 2 \pi(\tau^a-u)k/p}=-1$, and the error term has the trivial upper bound
	\begin{eqnarray} \label{el5405}
	\left | \sum_{p\leq x}
	\frac{\log p}{p}\sum_{\gcd(a,p-1)=1,} \sum_{ 0<k\leq p-1} e^{i 2 \pi(\tau^a-u)k/p} \right | 	&\leq& \sum_{p\leq x}
	\frac{\varphi(p-1)\log p}{p} \nonumber\\
	&\leq& \frac{1}{2}x^{1/m}+O\left (\frac{x^{1/m}}{\log x}\right ) .
	\end{eqnarray} 
	
	The contribution from the prime powers $f(n)=p^v \leq x, v\geq 2$, is $\ll x^{1/2m} \log x$, which can be absorbed in the error term.\\
	
	This information implies that the error term $E(x)$ is smaller than the main term $M(x)$, see (\ref{el7704}). Thus, there is a nontrivial upper bound. To sharpen this upper bound, rearrange the triple finite sum in the form
	\begin{eqnarray} \label{el5410}
	E(x) &= &\sum_{p \leq x}\frac{\log p}{p} \sum_{ 0<k\leq p-1,}  \sum_{\gcd(a,p-1)=1} e^{i 2 \pi(\tau^a-u)k/p} \nonumber \\
	&=&  \sum_{p\leq x} \left (
	\frac{1}{p^{1/2}} \sum_{ 0<k\leq p-1} e^{-i 2 \pi uk/p}   \right ) \left (\frac{\log p}{p^{1/2}}\sum_{\gcd(n,p-1)=1} e^{i 2 \pi k\tau ^n/p} \right ).
	\end{eqnarray} 
	And let
	\begin{equation} \label{el5415}
	U_p=\frac{1}{p^{1/2}}\sum_{ 0<k\leq p-1} e^{-i2 \pi uk/p}    \qquad \text{ and } \qquad V_p=\frac{\log p}{p^{1/2}}\sum_{\gcd(n,p-1)=1} e^{i2 \pi k\tau ^n/p} .
	\end{equation} 
	Now consider the Holder inequality $|| AB||_1 \leq || A||_{r} \cdot || B||_s $ with $1/r+1/s=1$. In terms of the components in (\ref{el5415}) this inequality has the explicit 
	form
	\begin{equation} \label{el5420}
	\sum_{p \leq x} | U_p V_p|\leq \left ( \sum_{ p \leq x} |U_p|^r \right )^{1/r} \left ( \sum_{  p\leq x} |V_p|^s \right )^{1/s} .
	\end{equation} 
	
	The absolute value of the first exponential sum $U_p$ is given by
	\begin{equation} \label{el5425}
	| U_p |= \left |\frac{1}{p^{1/2}} \sum_{ 0<k\leq p-1} e^{-i2 \pi uk/p} \right | =\frac{1}{p^{1/2}}  .
	\end{equation} 
	This follows from $\sum_{ 0<k\leq p-1} e^{i 2 \pi uk/p}=-1$ for $u\ne 0$. The corresponding $r$-norm $||U_p||_r^r =  \sum_{p\leq x} |U_p|^r$ has the upper bound
	\begin{equation}\label{el5430}
	\sum_{p\leq x} |U_p|^r=\sum_{ p \leq x}  \left |\frac{1}{p^{1/2}} \right |^r \leq x^{(1-r/2)/m}. 
	\end{equation}
	The finite sum over the primes is estimated using integral
	\begin{equation} \label{el5435}
	\sum_{p\leq x} \frac{1}{p^{r/2}}
	\ll\int_{1}^{x} \frac{1}{t^{r/2}} d \pi(t)=O(x^{(1-r/2)/m}),
	\end{equation}
	where \(\pi(x)=x/\log x+O(x/\log ^{2} x)\) is the prime counting measure. \\
	
	The absolute value of the second exponential sum $V_p=V_p(k)$ is given by
	\begin{equation} \label{el5440}
	|V_p|= \left |\frac{\log p}{p^{1/2}}\sum_{\gcd(n,p-1)=1} e^{i2 \pi k\tau ^n} \right |\ll p^{1/2-\varepsilon} .
	\end{equation} 
	This exponential sum dependents on $k$; but it has a uniform, and independent of $k$ upper bound
	\begin{equation} \label{el88978}
	\max_{1\leq k \leq p-1}   \left |   \sum_{\gcd(n,p-1)=1,} e^{i 2 \pi k\tau ^n/p} \right | \ll  p^{1-\varepsilon} ,
	\end{equation}   
	where \(\varepsilon >0\) is an arbitrarily small number, see Lemma \ref{lem4.2}.  \\
	
	The corresponding $s$-norm $||V_p||_s^s =  \sum_{p\leq x} |V_p|^s$ has the upper bound
	\begin{equation} \label{el5450}
	\sum_{p\leq x} |V_p|^s \leq\sum_{ p \leq x}  \left |p^{1/2-\varepsilon} \right |^s \ll x^{(1+s/2-\varepsilon s)/m}. 
	\end{equation}
	As before, the finite sum over the primes is estimated using integral
	\begin{equation} \label{el5455}
	\sum_{p\leq x} p^{s/2-\varepsilon s}
	\ll\int_{1}^{x} t^{s/2-\varepsilon s} d \pi(t)=O(x^{(1+s/2-\varepsilon s)/m}).
	\end{equation}
	
	Now, replace the estimates (\ref{el5430}) and (\ref{el5450}) into (\ref{el5420}), the Holder inequality, to reach
	\begin{eqnarray} \label{el5460}
	\sum_{p\leq x}
	\left | U_pV_p \right | 
	&\leq & 
	\left ( \sum_{ p \leq x} |U_p|^r \right )^{1/r} \left ( \sum_{ p\leq x} |V_p|^s \right )^{1/s} \nonumber \\
	&\ll & \left ( x^{\frac{1-r/2}{m}}  \right )^{1/r} \left ( x^{\frac{1+s/2-\varepsilon s}{m}}  \right )^{1/s} \nonumber \\
	&\ll & \left ( x^{\frac{1/r-1/2}{m}}  \right )  \left ( x^{\frac{1/s+1/2-\varepsilon}{m}}  \right ) \\
	&\ll &  x^{\frac{1/r+1/s-\varepsilon}{m}}  \nonumber \\
	&\ll &  x^{\frac{1-\varepsilon}{m}} \nonumber .
	\end{eqnarray}
	
	Note that this result is independent of the parameter $1/r+1/s=1$. 
 \end{proof}

	\section{Primitive Roots Producing Polynomials}
	Let $f(x) \in \mathbb{Z}[x]$ be a prime producing polynomial of degree $\text{deg}(f)=m$. The weighted characteristic function for primitive roots of the elements in a finite field \(u\in \mathbb{F}_p\) , satisfies the relation
	\begin{equation}
	\Lambda(f(n))\Psi(u)=
	\left \{\begin{array}{ll}
	\log(f(n)) & \text{ if } f(n)=p^k, k \geq 1, \text{ and } \ord_p (u)=p-1,  \\
	0 & \text{ if } f(n) \ne p^k, k \geq 1, \text{ or } \ord_p (u)\neq p-1. \\
	\end{array} \right.
	\end{equation} 
	The definition of the characteristic function $\Psi (u)$ is given in Lemma \ref{lem3.2}, and the vonMangoldt function is defined by $\Lambda(n)=\log n$ if $n \geq 1$ is a prime power, otherwise it vanishes.\\

\begin{proof}  (Theorem \ref{thm1.1}.) Let $x_0 \geq 1$ be a large constant and suppose that \(u\ne \pm 1,v^2 \) is not a primitive root for all primes \(p=f(n)\geq x_0\). Let \(x>x_0\) be a large number, and suppose that the sum of the weighted characteristic function over the interval \([1,x]\) is bounded. Id est, 
	\begin{equation} \label{el7703}
	x_0^2 \geq\sum _{f(n)\leq x} \Lambda(f(n))\Psi (u).
	\end{equation}
	Replacing the characteristic function, Lemma \ref{lem3.2}, and expanding the nonexistence inequality (\ref{el7703}) yield
	\begin{eqnarray} \label{el7704}
	x_0^2&\geq&\sum _{f(n)\leq x} \Lambda(f(n))\Psi (u)\\
	&=&\sum_{f(n)\leq x}    \Lambda(f(n)) \left ( \frac{1}{f(n)}\sum_{\gcd(r,f(n)-1)=1,} \sum_{ 0\leq k\leq f(n)-1} \psi \left((\tau ^r-u)k\right)  \right )\nonumber\\
	&=&a(u,f)\sum_{f(n)\leq x} \frac{\Lambda(f(n))}{f(n)}\sum_{\gcd(r,p-1)=1} 1\nonumber \\
	& &+ \sum_{ f(n)\leq x}
	\frac{\Lambda(f(n))}{f(n)}\sum_{\gcd(r,f(n)-1)=1,} \sum_{ 0<k\leq f(n)-1} \psi \left((\tau ^r-u)k\right) \nonumber\\
	&=&M(x) + E(x) \nonumber,
	\end{eqnarray} 
	where $a(u,f) \geq0$ is a constant depending on both the polynomial $f(n)$ and the fixed integer $u\ne 0$. \\
	
	The main term $M(x)$ is determined by a finite sum over the trivial additive character \(\psi =1\), and the error term $E(x)$ is determined by a finite sum over the nontrivial additive characters \(\psi =e^{i 2\pi  k/p}\neq 1\).\\
	
	Applying Lemma \ref{lem5.2} to the main term, and Lemma \ref{lem5.1} to the error term yield
	\begin{eqnarray} \label{el7715}
	\sum _{ F(n)\leq x} \Lambda(f(n))\Psi (u)
	&=&M(x) + E(x) \nonumber \\
	&= & c(u,f) x^{1/m}+O \left (\frac{x^{1/m}}{\log x}\right )+O(x^{(1-\varepsilon)/m})\\
	&=& c(u,f) x^{1/m}+O \left (\frac{x^{1/m}}{\log x}\right ) \nonumber,
	\end{eqnarray} 
	where $c(u,f)=a_us_fr(u,f) >0$ is the density constant. But for all large numbers $x \geq (5/ c(u,f))^{2m} x_0^{m} $, the expression
	\begin{eqnarray} \label{el7720}
	x_0^2 &\geq& \sum _{f( n)\leq x} \Lambda(f(n))\Psi (u)
	\nonumber \\
	&= & c(u,f) x^{1/m}+O \left (\frac{x^{1/m}}{\log x}\right ) \\
	&\geq& 5x_0^2\nonumber
	\end{eqnarray} 
	contradicts the hypothesis  (\ref{el7703}). Ergo, the number of primes $p=f(n)\leq x $ in the interval $[1,x]$ with a fixed primitive root $u\ne \pm 1,v^2 $ is infinite as $x \to \infty$.  Lastly, the number of primes has the asymptotic formula
	\begin{equation} \label{el910}
	\pi_f(x)=\sum _{ \substack{p=f(n) \leq x \\ \ord_p(u)=p-1}} 1 =c(u,f) \frac{x^{1/m}}{ \log x}+O \left (\frac{x^{1/m}}{\log^2x}\right ),
	\end{equation}
	which is derived from (\ref{el7715}) by partial summation. 
\end{proof}
	
The constraints required in a fixed primitive root $u\ne \pm 1,v^2 $  modulo $p=f(n)$ in quotient rings $\left (\mathbb{Z}/p\mathbb{Z} \right )^{\times}$ of the rational numbers $\mathbb{Q}$ are extremely simple. The analogous constraints required for a fixed primitive root in  the quotient rings $\left (\mathcal{O}/p\mathcal{O} \right )^{\times}\cong \mathbb{F}_{p^2} \text{ or } \mathbb{F}_{p} \times \mathbb{F}_{p}$ in quadratic numbers fields extensions $\mathbb{Q}(\alpha)$ of the rational numbers $\mathbb{Q}$ are much more rigid, see \cite[Examples 1, 2, 3]{RH02}. 	

\newpage
\section{Problems}
\begin{exe} \normalfont
	Let $x \geq 1$ be a large number. Show that $$\sum_{n \leq x}\frac{\varphi(n^2+1)}{n^2+1}=\frac{1}{2}\prod_{p\equiv 1 \bmod 4} \left ( 1-\frac{2}{p^2} \right )x+O(\log x).$$ A few recent papers have details on this sum, an earlier proof appears in \cite[p.\ 192]{PA88}. 
\end{exe}

\begin{exe} \normalfont
	Let $x \geq 1$ be a large number, and let $f(x)$ be an irreducible polynomial of degree $\deg(f)=d$. Show that $$\sum_{n \leq x}\frac{\varphi(f(n))}{f(n)}=\frac{1}{d}\prod_{f(n)\equiv 0 \bmod p} \left ( 1-\frac{\alpha(p)}{p^2} \right )x+O(\log x),$$ where $\alpha(n)$ is an arithmetic function. An earlier proof was given by Schwarz. 
\end{exe}

%cccccccccccccccccccccccccccccccccccccccccccccccc	
\chapter{Expressions for the Densities} \label{c6}
The calculations of the densities of various subsets of prime numbers has evolved into a complex topic in algebraic number theory. This chapter has an introduction to some of the results and techniques.

\section{Early Results} 
The natural density
\begin{equation} 
a_u= \lim_{x \to \infty} \frac{\#\{ p \leq x: \ord_p(u)=p-1\}}{\#\{ p \leq x: p \text{ prime}\}}
\end{equation}
gives the proportion of primes that have a fixed nonzero integer $u \in \mathbb{Z}$ as a primitve root. The earliest theoretical calculation specified a single constant for all $u\ne \pm 1 , v^2$. Specifically
\begin{equation} \label{el60000}
a_0=\sum_{n\geq 1 } \frac{ \mu(n)}{n \varphi(n)}= \prod_{p \geq 2} \left ( 1-\frac{1}{p(p-1)}\right )=.37739558136192022880547280 \ldots
\end{equation}
is the average density of the primes $p \in \mathbb{P}=\{2,3,5,\ldots \}$ having a fixed primitive root $u\ne \pm 1 , v^2$ in the finite fields $\mathbb{F}_p$; it was proved in \cite{GM68}, and \cite{SP69}. This numerical value was computed in \cite{WJ61}; and a new analysis of this constant appears in \cite{CI09}. Later, discrepancies between the theoretical density and the numerical data were discovered by the Lehmers. The historical details are given in \cite{WJ61}, \cite{SP03}, et alii. \\

\section{Corrected Density}
The density of the subset of primes that split completely in the numbers field $\mathbb{Q}(\zeta_n,
\sqrt[n]{u})$, where $\zeta_n$ is a primitive $n$th root of unity, is the reciprocal of the numbers fields index $[\mathbb{Q}(\zeta_n,
\sqrt[n]{u}):\mathbb{Q}]$. This follows from the Frobenious density theorem. The canonical product (\ref{el60000}) does not account for the dependencies among the fields extensions. For fixed integer $u=(sv^2)^k$ with $s\geq 1$ squarefree, and $k \geq 1$, the corrected density is given by the Kummer infinite series
\begin{equation}\label{60020}
\delta(u)=\sum_{n \geq 1} \frac{\mu(n)}{[\mathbb{Q}(\zeta_n,
	\sqrt[n]{u}):\mathbb{Q}]}=\mathfrak{C}(u) \prod_{p \geq 2} \left (1-\frac{1}{p(p-1)}\right ),
\end{equation}
where correction factor 
\begin{equation} \label{60000}
\mathfrak{C}(u)=1-\mu(s) \prod_{p |s} \left (\frac{1}{[\mathbb{Q}(\zeta_n,
	\sqrt[n]{u}): \mathbb{Q}]-1}\right )=1-\mu(s) \prod_{p |s} \left (\frac{1}{p(p-1)-1}\right )
\end{equation}

is a rational number. This later rational number accounts for the dependencies among the splitting fields $\mathbb{Q}(\zeta_n,
\sqrt[n]{u})$. The derivation is similar to the function field case, which was completed much earlier in the 1900's. The complete corrected formula and the proof for any primitive root of rational primes appears in \cite[p.\ 220]{HC67}.\\ 

\begin{thm} {\normalfont (\cite[p.\ 220]{HC67})} \label{thm6.1} Assume the extended Riemann hypothesis for the Dedekind zeta function over Galois fields of the type $\mathbb{Q}(\zeta_n,
	\sqrt[n]{u})$, with $m,n$ squarefree. Let $u=(sv^2)^k\ne \pm 1, v^2$, with $s\geq 1$ squarefree. Then, the density
\begin{equation} \label{60210}
\delta(u)=
\left \{\begin{array}{ll}
\displaystyle \prod_{p | k} \left ( 1 -\frac{1}{p-1} \right)\prod_{p \nmid k} \left ( 1 -\frac{1}{p(p-1)} \right) & \text{ if } s \equiv 1 \bmod 4 \\
\displaystyle \left (1-\mu(s) \prod_{\substack{p|s \\ p \nmid k }}\frac{1}{p^2-p-1}\prod_{\substack{p|s \\ p | k }}\frac{1}{p^2-p-1} \right )
\prod_{p \nmid k} \left ( 1 -\frac{1}{p(p-1)} \right) & \text{ if } s \not \equiv 1\bmod 4. 
\end{array} \right .
\end{equation}
\end{thm}

\section{Densities Arithmetic Progressions}
The density of primes in the arithmetic progression $\{p=qn+a:n \geq 1\}$ with a fixed primitive $u=(sv^2)^k$, where $s\geq1$ is squarefree, and $k \geq1$, has many parameters associated with the integer $u$. An abridged version of the formula is recorded here.\\

\begin{thm} {\normalfont (\cite[Theorem 1]{MP99})} \label{thm6.2}
Let $q \geq 1$ and $1 \leq a<q, \gcd(a,q)=1$ and $u=(sv^2)^k\ne \pm 1, v^2$, and $s\geq $ squarefree. Set
\begin{equation}\label{60021}
A(q,a,k)= \prod_{p |\gcd(a-1,q) } \left (1-\frac{1}{p}\right )\prod_{\substack{p \nmid q \\ p| k} } \left (1-\frac{1}{p-1}\right )\prod_{\substack{p \nmid q \\ p \nmid k} } \left (1-\frac{1}{p(p-1)}\right ),
\end{equation} 	
if $\gcd(a-1,q)=1$; otherwise set $A(q,a,k)=0$. Then, the density is
\begin{equation}
\delta(u,q,a)=c(u,q,a)\frac{A(a,q,k)}{\varphi(q)}
\end{equation}
where $c(u,q,a)$ is a correction factor.
\end{thm}

The precise expression for the density is a Kummer series
\begin{equation}\label{60022}
\delta(u,q,a)=\sum_{n \geq 1} \frac{\mu(n)c_a(n)}{[\mathbb{Q}(\zeta_q,\zeta_n,
	\sqrt[n]{u}):\mathbb{Q}]},
\end{equation}
where the $\zeta_m$ are $m$th primitive roots of unity, and $c_a=0,1$, see \cite[Theorem  2]{MP99}. 

\section{Densities for Polynomials}
The constant $c(u,f)=a(u)s(f)r(u,f)>0$ is a composite of all the intermediate constants. This number specifies the density of the primes with a fixed primitive root $u$ produced by the primitive root producing polynomial $f(x) \in \mathbb{Z}[x]$. 

\begin{enumerate}
	\item The term $a_u=a(u) \geq 0$ is determined by the fixed integer $u \ne \pm1, v^2$. In some cases, there formulas for computing it. It corresponds to the identity case $f(x)=x$ in \cite[p.\ 218]{HC67} or \cite[p.\ 16]{MP07}. The arithmetic progression case $f(x)=ax+b$ were computed in \cite{MP07} and \cite[p.\ 10]{AS15}, see also \cite{LM11}, \cite{LH77}, et cetera. The calculations of these constants is a topic in algebraic number theory. 
	\item The term $s_f=s(f)$ is determined by the fixed primes producing polynomial $f(n)$ of degree $\deg(f)=m$. It arises from the Bateman-Horn conjecture. 
	\item The other term is the correction factor $r(u,f)$. It arises from the dependencies among all the different parameters. Among these parameter, there is some sort of series such as 
	\begin{equation}
	L(s,\omega_f )=\sum_{n\geq 1 } \frac{ \mu(n)\omega_f(n)}{n^{s+1}}= r(f)\prod_{p \geq 2} \left ( 1-\frac{\omega_f(p)}{p^{s+1}}\right ),
	\end{equation}
	which converges for $\Re e(s) \geq1$, see (\ref{el50021}) and \ref{el50031}).  
\end{enumerate}

There are other definition of the density constant. Given an admissible polynomial, the natural density $c(u,f)$ of primes with a fixed primitive root $u \ne0$ is defined by the a limit
\begin{equation} 
c(u,f)= \lim_{x \to \infty} \frac{\#\{ p=f(n) \leq x: \text{ ord}_p(u)=p-1\}}{\#\{ p \leq x: p \text{ prime}\}}.
\end{equation}
Other analytic and more practical formulas for the densities are discussed in \cite{MP07} and \cite{AS15}.

\section{Primitive Roots And Quadratic Fields}
The densities for certain  generalizations of concept of primitive roots to Quadratic numbers fields and other numbers fields are the subjects of several papers and thesis in \cite{RH00}, \cite{RH02}, \cite{CJ07}, et cetera.

\subsection{Orders of Fundamental Units}
The finite groups generated by the fundamental units modulo a rational prime in a real quadratic field are studied in \cite{RH00}. For a given fixed fundamental unit $\varepsilon \in \mathcal{K}=\mathbb{Q}(\theta)$, the densities for three different subsets of rational primes $p \in \mathbb{Q}$ have been computed. These subsets are:
\begin{enumerate} 
\item
\begin{equation}\label{600400}
S^+=\{p \text{ is split in } \mathcal{K}, \text{ the norm } N(\varepsilon)=1, \text{ and } \ord_p(\varepsilon)=p-1 \}, \nonumber
\end{equation}
\item
\begin{equation}\label{600401}
S^+=\{p \text{ is inert in } \mathcal{K}, \text{ the norm } N(\varepsilon)=1, \text{ and } \ord_p(\varepsilon)=p+1 \},
\end{equation} 
\item
\begin{equation}\label{600402}
S^-=\{p \text{ is inert in } \mathcal{K}, \text{ the norm } N(\varepsilon)=-1, \text{ and } \ord_p(\varepsilon)=2(p+1) \}. \nonumber
\end{equation}
\end{enumerate}

\begin{thm} \label{thm6.3}
	{\normalfont (\cite[Theorem 2]{RH00})}  Suppose the generalized Riemann hypothesis for the numbers fields $\mathbb{Q}(\zeta_n,
	\sqrt[n]{\alpha})$ holds. Then
\begin{equation}\label{60035}
	\delta(S^-)=\left ( 1- \mu(d) \prod_{p|d} \frac{1}{p^2-p-1}\right ) \prod_{p \geq 3} \left ( 1 -\frac{1}{p(p-1)}\right) 
	\end{equation}
	where $d>0$ is the discriminant of the real quadratic field.
\end{thm}

\subsection{Rational Primes in Quadratic Fields}
Given a fixed algebraic integer $\alpha \in \mathcal{K}=\mathbb{Q}(\theta)$, consider the followings subsets of rational primes $p \in \mathbb{Q}$ described by 
\begin{equation}\label{600300}
S^-=\{p \text{ is unramified and inert in } \mathcal{K} \text{ and } \ord_p(\alpha)=p^2-1 \} \nonumber
	\end{equation} 
and 
\begin{equation}\label{600302}
S^+=\{p \text{ is unramified and split in } \mathcal{K} \text{ and } \ord_p(\alpha)=p-1 \}.
\end{equation}

\begin{thm} \label{thm6.4}
{\normalfont (\cite[Theorem 1]{RH02}) } Suppose the generalized Riemann hypothesis for the numbers fields $\mathbb{Q}(\zeta_n,
\sqrt[n]{\alpha})$ holds. Then
\begin{equation}\label{60230}
\delta(S^+)=
\left \{\begin{array}{ll}
r_{\alpha} \prod_{p \geq 2} \left ( 1 -\frac{1}{p^2(p-1)}\right) & \text{ if } \alpha \text{ and } \overline{\alpha} \text{ are multiplicative independent,}  \\
r_{\alpha} \prod_{p \geq 2} \left ( 1 -\frac{1}{p(p-1)}\right) & \text{ otherwise}; \\
\end{array} \right .
\end{equation}
and 
\begin{equation}\label{60235}
\delta(S^-)=r_{\alpha} \prod_{p \geq 3} \left ( 1 -\frac{2}{p(p-1)}\right), 
\end{equation}
where $r_{\alpha} \geq 0$ is a rational number.
\end{thm}

More density results for quartic field extensions of the rational numbers were very recently established in \cite{SM17}.

\section{Densities for Elliptic Groups}
This section provides densities information for certain subsets of primes associated with algebraic groups defined by elliptic curves.
 
\subsection{Cyclic Groups}
The average rational density (\ref{el60000}) for subsets of rational primes is a simpler analog of the average cyclic elliptic density
\begin{equation} \label{el60600}
C_0=\sum_{n\geq 1 } \frac{ \mu(n)}{[\mathcal{K}_n: \mathbb{Q}] }= \prod_{p \geq 2} \left ( 1-\frac{1}{(p^2-1)(p^2+1)}\right )=.8137519 \ldots
\end{equation}
for subsets of elliptic primes. The field extension $\mathcal{K}_n=\mathbb{Q}(E[n])$ is the $n$-division points field of the elliptic curve $E:y^2=f(x)$. For a prime $p\geq 2$, the maximum index is given by $[\mathcal{K}_p:\mathbb{Q}]=\# \GL_2(\mathbb{F}_p)=(p^2-1)(p^2-p)$.\\

The average density over a large subset of elliptic curves was computed in \cite{JN10}.

\begin{thm}  \label{thm6.5}  {\normalfont (\cite[Theorem 8.4]{LS14}) }
Let $E$ be an elliptic curve which has maximal index $[\mathcal{K}_p:\mathbb{Q}]$ for almost all primes $p \geq 2$. Let $\Delta$ be the discriminant of the 2-division point field $\mathcal{K}_2$, and let $D= disc(\mathbb{Q}(\sqrt{\Delta})).$ Then, the density of elliptic primes is
\begin{equation} \label{el60606}
\delta(E)=\mathfrak{C}(E) \prod_{p \geq 2} \left ( 1-\frac{1}{[\mathcal{K}_p:\mathbb{Q}]}\right )=C(E) \prod_{p \geq 2} \left ( 1-\frac{1}{(p^2-1)(p^2+1)}\right ),
\end{equation}
where the correction factor $\mathfrak{C}(E)=1$ for even discriminant $D=2d$, and 
\begin{equation} \label{el608}
\mathfrak{C}(E)= 1+\prod_{p |2D} \frac{-1}{[\mathcal{K}_p:\mathbb{Q}]}=1+ \prod_{p |2D} \frac{-1}{(p^2-1)(p^2+1)}
\end{equation}
for odd discriminant $D=2d+1$.
\end{thm}  

\subsection{Primitive Points}
An early attempt to compute the density of primes for a fixed elliptic curve and a suitable point was in \cite{LT77}. The calculations of the density for elliptic primitive points in the groups of points $E(\mathbb{F}_p)$ have many similarities to the densities for primitive roots in the rational groups $\mathbb{Z}/ p \mathbb{Z}$. 

\begin{conj} \label{conj6.1}  {\normalfont (Lang-Trotter)} 
	For any elliptic curve $E:f(x,y)=0$ of rank $ran(E(\mathbb{Q}))>0$ and a point $P \in E(\mathbb{Q})$ of infinite order, the density 
	\begin{equation}
	\delta(E,P)= \lim_{x \to \infty} \frac{\#\{p \leq x:\ord_E(P)=n\}}{\pi(x)}>0,
	\end{equation}  
	where $n=\#E(\mathbb{F}_p)$, exists, and $P$ is a fixed primitive point for the given elliptic curve.
\end{conj}
The density has a complicated formula
\begin{eqnarray}\label{el60002}
\delta(E,P)&=&\sum_{n \geq 1} \mu(n) \frac{\# \mathcal{C}_{P,n}}{\#\Gal(\mathbb{Q}(E[n],n^{-1}P))/\mathbb{Q})} \nonumber \\ 
&=&\mathfrak{C}(E,P)\prod_{p\geq 2} \left (1 -\frac{p^3-p-1}{p^2(p-1)^2(p+1)}\right ),
\end{eqnarray}
where $\mathfrak{C}(E,P)$ is a correction factor; and $\mathcal{C}_{P,n}$ is the union of certain conjugacy classes in the Galois group $\Gal(\mathbb{Q}(E[n],n^{-1}P))$. A few new papers and theses have further developed this topic, see \cite{LS14}, \cite{BJ17}. The product
\begin{equation}\label{el6999}
\prod_{p\geq 2} \left (1 -\frac{p^3-p-1}{p^2(p-1)^2(p+1)}\right )=0.44014736679205786 \ldots
\end{equation}
 is referred to as the average density of primitive primes $p \in \mathbb{P}=\{2,3,5,\ldots \}$ such that $E(\mathbb{F}_p)$ has a fixed primitive point $P\not \in E(\mathbb{Q})_{\text{tors}}$ in the group of points of the elliptic curve $E$. 

\begin{table}
\begin{center}
\begin{tabular}{||c|c |c||} 
\hline
\textbf{Elliptic Curve} & \textbf{Generating Point}&  \textbf{Rank} \\ [1ex]  
\hline \hline
 $y^2+y=x^3-x$ & $P=(00)$ & $\rk(E)\geq 1$\\ 
 \hline
 $y^2+y=x^3+x^2$ & $P=(00)$ & $\rk(E)\geq 1$\\  
 		\hline
$y^2+xy+y=x^3-x^2$ & $P=(00)$ & $\rk(E)\geq 1$\\ 
\hline
\end{tabular}
 \end{center} 
\caption{\label{9999} Elliptic curves of nontrivial ranks.}
\end{table} 
The average density is studied in \cite{JN10}, and \cite{ZD09}. Other related results are given in \cite{JN09}, \cite{BC11}, et alii.  The density was originally estimated in \cite{LT77} for the three elliptic curves listed below.  \\

%ccccccccccccccccccccccccccccccccccccc
\chapter{Special Cases and Numerical Results} \label{c7}
	This section covers several special cases of primitive roots producing polynomials of small degrees. Some numerical data is also compiled to approximate the density of the corresponding subsets of primes.\\
	
	\section{Quadratic Primes with Fixed Primitive Roots}
	Let $f(x)=ax^2+bx+c \in \mathbb{Z}[x]$ be a prime producing polynomial of degree $\text{deg}(f)=2$. The weighted characteristic function for primitive roots of the elements in a finite \(u\in \mathbb{F}_p\)  field, satisfies the relation
	\begin{equation}
	\Lambda(an^2+bn+c)\Psi(u)=
	\left \{\begin{array}{ll}
	\log(an^2+bn+c) & \text{ if } f(n)=p^k,k \geq 1, \text{ and } \ord_p (u)=p-1,  \\
	0 & \text{ if } f(n) \ne p^k, k\geq 1, \text{ or }\ord_p (u)\neq p-1. \\
	\end{array} \right.
	\end{equation} 
	The definition of the characteristic function $\Psi (u)$ is given in Lemma 3.2, and the vonMangoldt function is defined by $\Lambda(n)=\log n$ if $n \geq 1$ is a prime power, otherwise it vanishes.\\ 
	
\begin{proof} (Theorem \ref{thm1.2}.) Let $x_0 \geq 1$ be a large constant and suppose that \(u\ne \pm 1,v^2 \) is not a primitive root for all primes \(f(n)=n^2+1\geq x_0\). Let \(x>x_0\) be a large number, and suppose that the sum of the weighted characteristic function over the interval \([1,x]\) is bounded. Id est, 
	\begin{equation} \label{el30000}
	x_0^2 \geq\sum _{n^2+1\leq x} \Lambda(n^2+1)\Psi (u).
	\end{equation}
	Replacing the characteristic function, Lemma \ref{lem3.2}, and expanding the nonexistence inequality (\ref{el30000}) yield
	\begin{eqnarray} \label{el30005}
	x_0^2&\geq&\sum _{n^2+1\leq x} \Lambda(n^2+1)\Psi (u)\\
	&=&\sum_{n^2+1\leq x} \Lambda(n^2+1) \left (\frac{1}{n^2+1}\sum_{\gcd(m,p-1)=1,} \sum_{ 0\leq k\leq n^2} \psi \left((\tau ^m-u)k\right)\right ) \nonumber\\
	&=&a(u,f)\sum_{n^2+1\leq x} \frac{\Lambda(n^2+1)}{n^2+1}\sum_{\gcd(m,p-1)=1} 1 \nonumber \\
&& +\sum_{ n^2+1\leq x}
	\frac{\Lambda(n^2+1)}{n^2+1}\sum_{\gcd(m,p-1)=1,} \sum_{ 0<k\leq n^2} \psi \left((\tau ^m-u)k\right) \nonumber\\
	&=&M(x) + E(x) \nonumber,
	\end{eqnarray} 
	where $a(u,f)>0$ is a constant depending on both the polynomial $f(n)=n^2+1$ and the fixed integer $u\ne0$. \\
	
	The main term $M(x)$ is determined by a finite sum over the trivial additive character \(\psi =1\), and the error term $E(x)$ is determined by a finite sum over the nontrivial additive characters \(\psi =e^{i 2\pi  k/p}\neq 1\).\\
	
	Applying Lemma \ref{lem5.2} to the main term, and Lemma \ref{lem5.1} to the error term yield
	\begin{eqnarray} \label{el730}
	\sum _{ n^2+1\leq x} \Lambda(n^2+1)\Psi (u)
	&=&M(x) + E(x) \nonumber \\
	&= & c(u,f) x^{1/2}+O \left (\frac{x^{1/2}}{\log x}\right )+O(x^{1/2-\varepsilon})\\
	&=&c(u,f)x^{1/2}+O \left (\frac{x^{1/2}}{\log x}\right ) \nonumber,
	\end{eqnarray} 
	where $c(u,f)=a(u,f)t_f$, see Section 8. But for all large numbers $x \geq (5/c(u,f))^2 x_0^4 $, the expression
	\begin{eqnarray} \label{el740}
	x_0^2 &\geq& \sum _{ n^2+1\leq x} \Lambda(n^2+1)\Psi (u)
	\nonumber \\
	&= &c(u,f)x^{1/2}+O \left (\frac{x^{1/2}}{\log x}\right ) \\
	&\geq& 5x_0^2\nonumber
	\end{eqnarray} 
	contradicts the hypothesis  (\ref{el30000}). Ergo, the number of primes $p=n^2+1\leq x $ in the interval $[1,x]$ with a fixed primitive root $u\ne \pm 1, v^2$ is infinite as $x \to \infty$. Lastly, the number of primes has the asymptotic formula
	\begin{equation} \label{el725}
	\pi_f(x,u)=\sum _{ \substack{p=f(n) \leq x \\ \text{ord}_p(u)=p-1}} 1 =c(u,f) \frac{x^{1/2}}{ \log x}+O \left (\frac{x^{1/2}}{\log^2x}\right ),
	\end{equation}
	which is derived from (\ref{el730}) by partial summation.  \end{proof}

	\subsection{Numerical Data for the Primitive Root 2}
	A small numerical experiment was conducted to show the abundance of primitive root 2 producing polynomial $f(x)=x^2+1$ and to estimate the density constant $c(2,f)$ in the formula
	\begin{eqnarray}
	\pi _{f}(x,2)&=&\#\{p=n^2+1\leq x: p \text{ prime and } \ord_p(2)=p-1\} \nonumber \\
	&=&c(2,f) \frac{x^{1/2}}{\log x} +O \left (\frac{x^{1/2}}{\log^2 x} \right)
	\end{eqnarray}
	This constant is product of various quantities. These included the basic counting functions
	\begin{eqnarray}
		\pi(x,2)&=&\#\{p\leq x:  \ord_p(2)=p-1\} \nonumber \\
		&=&a_2 \frac{x}{\log x} +O \left (\frac{x}{\log^2 x} \right)
	\end{eqnarray}
	and
	\begin{eqnarray}
	\pi _{f}(x)&=&\#\{p=n^2+1\leq x: p \text{ prime }\} \\
	&=&s_2 \frac{x^{1/2}}{\log x} +O \left (\frac{x^{1/2}}{\log^2 x} \right) \nonumber,
	\end{eqnarray}
	see Conjecture \ref{conj2.2}; 
	respectively. \\
	
	The density constant $c(2,f)=a_2s_2r(2,f)>0$ is expected to be a real multiple of the product of the basic canonical products  
	\begin{equation}
	a_2=\prod_{p>2} \left (1-\frac{1}{p(p-1)} \right )=.3739558136 \ldots,
	\end{equation}
	for the density of primes with a primitive root $u=2$, and  
	\begin{equation}
	s_2= \frac{1}{2}\prod_{p>2} \left (1-\frac{(-1|p)}{p-1} \right )=0.6864 \ldots,
	\end{equation} 
	for the density of primes $p=n^2+1$ in the subset of quadratic primes; respectively.\\
	
	 The composite density is
	\begin{eqnarray}
	c(2,f)&=&r(2,f) \cdot \prod_{p>2} \left (1-\frac{(-1|p)}{p-1} \right ) \prod_{p>2} \left (1-\frac{1}{p(p-1)} \right ) \nonumber \\
	&=& 0.256683267984 \cdot r(2,f).
	\end{eqnarray}
	The real number $r(2,f) \in \mathbb{R}$ is unknown. This seems to serve as a correction factor for the irregularity in the distribution of the primes, and its dependence on the fixed primitive root $u=2$. An advanced numerical analysis for the canonical density of primitive root or Artin constant, the penultimate product, appears in  \cite{CI09}.\\
	
	The fifth column has the relative density defined by the ratio $ \pi_{f}(x,2)/\pi_f(x)$. This ratio depends on the polynomial $f(x)=x^2+1$ and the fixed primitive root $u=2$.\\
\begin{table}	
	\begin{center}
		\begin{tabular}{||c|c |c| c| c|c||} 
			\hline
			$ x$ & $\pi(x)$ & $\pi_{f}(x)$  & $\pi_f(x,2)$ & $ \frac{\pi_{f}(x,2)}{\pi_f(x)}$ &  $c(2,f)=\frac{\pi_{f}(x,2)}{(x^{1/2}/\log x)}$ \\ [1ex]  
			\hline\hline
			$ 10^4$ &1229 &33 &12 &0.363636363636 & 0.552620\\ 
			\hline
			$10^6$ &78498&208 &71 &0.341346153846 & 0.490451  \\
			\hline
			$10^8$ &5761455&1558&608 &0.390243902430 & 0.559989\\
			\hline
			$10^{10}$ &455,052,511&12390 & 4847 & 0.391202582728 & 0.558031 \\ 
			\hline
		\end{tabular}
	\end{center}
\caption{The Polynomial $f(x)=x^2+1$ and the primitive root $u=2$.}
\end{table}

	\subsection{Numerical Data for the Primitive Root 3}
	A small numerical experiment was conducted to show the abundance of primitive root 3 producing polynomial $f(x)=x^2+1$ and to estimate the density constant $c(3,f)$ in the formula
	\begin{eqnarray}
	\pi _{f}(x,3)&=&\#\{p=n^2+1\leq x: p \text{ prime and } \ord_p(3)=p-1\} \nonumber \\
	&=&c(3,f) \frac{x^{1/2}}{\log x} +O \left (\frac{x^{1/2}}{\log^2 x} \right)
	\end{eqnarray}
	This constant is product of various quantities. These included the basic counting functions
		\begin{eqnarray}
	\pi(x,3)&=&\#\{p\leq x:  \ord_p(3)=p-1\} \nonumber \\
	&=&a_3 \frac{x}{\log x} +O \left (\frac{x}{\log^2 x} \right)
	\end{eqnarray}
	and
	\begin{eqnarray}
	\pi _{f}(x)&=&\#\{p=n^2+1\leq x: p \text{ prime }\} \\
	&=&s_2 \frac{x^{1/2}}{\log x} +O \left (\frac{x^{1/2}}{\log^2 x} \right) \nonumber,
	\end{eqnarray}
	see Conjecture \ref{conj2.2}; 
	respectively. \\
	
	The density constant $c(3,f)=a_3s_2r(3,f)>0$ is expected to be a real multiple of the basic canonical products  
	\begin{equation}
	a_3=\prod_{p>2} \left (1-\frac{1}{p(p-1)} \right )=.3739558136 \ldots,
	\end{equation}
	for the density of primes with a primitive root $u=3$, and  
	\begin{equation}
	s_2= \frac{1}{2}\prod_{p>2} \left (1-\frac{(-1|p)}{p-1} \right )=0.6864 \ldots,
	\end{equation} 
	for the density of primes $p=n^2+1$ in the subset of quadratic primes; respectively.\\
	 The composite density is
	\begin{eqnarray}
	c(3,f)&=&r(3,f) \cdot \prod_{p>2} \left (1-\frac{(-1|p)}{p-1} \right ) \prod_{p>2} \left (1-\frac{1}{p(p-1)} \right ) \nonumber \\
	&=& 0.256683267984 \cdot r(3,f).
	\end{eqnarray}
	The real number $r(3,f) \in \mathbb{R}$ is unknown. This seems to serve as a correction factor for the irregularity in the distribution of the primes, and its dependence on the fixed primitive root $u=3$. An advanced numerical analysis for the canonical density of primitive root or Artin constant, the penultimate product, appears in  \cite{CI09}.\\
	
	The real number $r(3,f) \in \mathbb{R}$ is unknown. This seems to serve as a correction factor for the irregularity in the distribution of the primes, and its dependence on the fixed primitive root $u=3$. \\
	
	The fifth column has the relative density defined by the ratio $ \pi_{f}(x,3)/\pi_f(x)$. This ratio depends on the polynomial $f(x)=x^2+1$ and the fixed primitive root $u=3$.\\

\begin{table}	
\begin{center}
\begin{tabular}{||c|c |c| c| c|c||} 
			\hline
			$ x$ & $\pi(x)$& $\pi_{f}(x)$  & $\pi_f(x,3)$ & $ \frac{\pi_{f}(x,3)}{\pi_f(x)}$ &  $c(3,f)=\frac{\pi_{f}(x,3)}{(x^{1/2}/\log x)}$ \\ [1ex]  
			\hline\hline
			$ 10^4$ &1229&33 &12 &0.363636363636 & 0.552620\\ 
			\hline
			$10^6$&78498&208 &74 &0.355769230769 &0.511174  \\
			\hline
			$10^8$&5761455 &1558&519 &0.333119383825&0.478017\\
			\hline
			$10^{10}$&455,052,511 &12390 & 4026 & 0.324939467312 &0.463510 \\ 
			\hline 
		\end{tabular}
	\end{center}
\caption{The Polynomial $f(x)=x^2+1$ and the primitive root $u=3$.}
\end{table}	
	
	\vskip 1 in
	\section{Cubic Primes with Fixed Primitive Roots}
	Let $f(x)=x^3+2 \in \mathbb{Z}[x]$ be a prime producing polynomial of degree $\text{deg}(f)=3$. The weighted characteristic function for primitive roots of the elements in a finite \(u\in \mathbb{F}_p\)  field, satisfies the relation
	\begin{equation}
	\Lambda(n^3+2)\Psi(u)=
	\left \{\begin{array}{ll}
	\log(n^3+2) & \text{ if } f(n)=p^k,k \geq 1, \text{ and }\ord_p (u)=p-1,  \\
	0 & \text{ if } f(n) \ne p^k, k\geq 1, \text{ or }\ord_p (u)\neq p-1. \\
	\end{array} \right.
	\end{equation} 
	The definition of the characteristic function $\Psi (u)$ is given in Lemma \ref{lem3.2}, and the vonMangoldt function is defined by $\Lambda(n)=\log n$ if $n \geq 1$ is a prime power, otherwise it vanishes.\\

	\begin{thm} \label{thm7.1}
		For a large number $x\geq 1$, the number of primes $p=n^3+2 \leq x$ with a fixed primitive root \(u\ne \pm 1,v^2 \)  has the asymptotic formula
		\begin{equation} \label{el1020}
		\sum _{ \substack{p=n^3+2\leq x\\ \ord_p(u)=p-1}} 1 =c(u,f) \frac{x^{1/3}}{ \log x}+O \left (\frac{x^{1/3}}{\log^2x}\right ),
		\end{equation} 
		
		where $c(u,f) \geq 0$ is a constant depending on the polynomial $f(x)=x^3+2$, and the integer $u \in \mathbb{Z}$.
	\end{thm}
	
\begin{proof} Repeat the argument for Theorem \ref{thm1.1} or Theorem \ref{thm1.2}, mutatis mutandis.   \end{proof} 
	
	\subsection{Numerical Data for the Primitive Root 2}
	A small numerical experiment was conducted to show the abundance of primitive root producing polynomial $f(x)=x^3+2$ and to estimate the density constant $c(2,f)$ in the formula
	\begin{eqnarray}
	\pi _{f}(x,2)&=&\#\{p=n^3+2\leq x: p \text{ prime and } \ord_p(2)=p-1\} \nonumber \\
	&=&c(2,f) \cdot \frac{x^{1/3}}{\log x} +O \left (\frac{x^{1/3}}{\log^2 x} \right).
	\end{eqnarray}
	This constant is product of various quantities. These included the basic counting functions
	\begin{eqnarray}
	\pi(x,2)&=&\#\{p=n^3+2\leq x: \ord_p(2)=p-1\} \nonumber \\
	&=&c(2,f) \cdot \frac{x^{1/3}}{\log x} +O \left (\frac{x^{1/3}}{\log^2 x} \right)
	\end{eqnarray}
	for the density of primes with a primitive root $u=2$, and
	\begin{eqnarray}
	\pi _{f}(x)&=&\#\{p=n^3+2\leq x: p \text{ prime }\} \\
	&=&s_3 \cdot \frac{x^{1/3}}{\log x} +O \left (\frac{x^{1/3}}{\log^2 x} \right)\nonumber,
	\end{eqnarray}
	for the density of primes $p=n^3+2$ in the subset of cubic primes, see Conjecture \ref{conj2.3}; 
	respectively. The density constant $c(2,f)=a_2s_3r(2,f)>0$ is expected to be a real multiple $r(2,f)$ of all the product of the canonical products for the density of primes $p=n^3+2$. These are
	\begin{equation}
	a_2=\prod_{p>2} \left (1-\frac{1}{p(p-1)} \right )=.3739558136 \ldots,
	\end{equation}
	for the density of primitive root $u=2$, and
	\begin{equation} \label{el9012}
	s_3=\frac{1}{3}\prod_{p\geq2} \left (1-\frac{1-\nu_f(p)}{p-1} \right ) \approx 0.155943 \ldots,
	\end{equation} 
	see the calculation in (\ref{el2002}), respectively. The composite density is the product 
	\begin{eqnarray}
	c(2,f)&=&r(2,f) \cdot \prod_{p\geq 2} \left (1-\frac{1-\nu_f(p)}{p-1} \right ) \cdot \prod_{p>2} \left (1-\frac{1}{p(p-1)} \right  ) \nonumber \\
	&\approx& 0.058316 \cdot r(2,f).
	\end{eqnarray}
	The real number $r(2,f) \in \mathbb{R}$ is unknown. This seems to serve as a correction factor for the irregularity in the distribution of the primes, and its dependence on the fixed primitve root $u=2$. \\
	
	The fifth column has the relative density defined by the ratio $ \pi_{f}(x,2)/\pi_f(x)$. This ratio depends on the polynomial $f(x)=x^3+2$ and the fixed primitive root $u=2$.\\
\begin{table}	
	\begin{center}
		\begin{tabular}{||c|c |c| c| c|c||} 
			\hline
			$ x$ & $\pi(x)$& $\pi_{f}(x)$  & $\pi_f(x,2)$ & $ \frac{\pi_{f}(x,2)}{\pi_f(x)}$ &  $c(2,f)=\frac{\pi_{f}(x,2)}{(x^{1/3}/\log x)}$  \\ [1ex]  
			\hline\hline
			$ 10^3$ &168&3 &2 &.666667 & 0.460517\\ 
			\hline
			$10^6$&78498&10 &2 &0.200000 &0.092103  \\
			\hline
			$10^9$&50847534 &74&18 &0.243243&0.124340\\
			\hline
			$10^{12}$&37607912018 &520 &110 & 0.211538 &0.101314 \\ 
			\hline
			$10^{15}$&29844570422669 &4059 &865 & 0.213107 &0.099587 \\ 
			\hline
			$10^{18}$&24739954287740860 &33795 &7112 &0.210445 &0.098256 \\ 
			\hline
		\end{tabular}
	\end{center}
\caption{The Polynomial $f(x)=x^3+2$ and the primitive root $u=2$.}
\end{table}	
	\subsection{Numerical Data for the Primitive Root 3}
	A small numerical experiment was conducted to show the abundance of primitive root producing polynomial $f(x)=x^3+2$ and to estimate the density constant $c(3,f)$ in the formula
	\begin{eqnarray}
	\pi _{f}(x,2)&=&\#\{p=n^3+2\leq x: p \text{ prime and } \ord_p(3)=p-1\} \nonumber \\
	&=&c(3,f) \cdot \frac{x^{1/3}}{\log x} +O \left (\frac{x^{1/3}}{\log^2 x} \right).
	\end{eqnarray}
	This constant is product of various quantities. These included the basic counting functions
	\begin{eqnarray}
	\pi(x,2)&=&\#\{p=n^3+2\leq x: \ord_p(3)=p-1\} \nonumber \\
	&=&c(3,f) \cdot \frac{x^{1/3}}{\log x} +O \left (\frac{x^{1/3}}{\log^2 x} \right)
	\end{eqnarray}
	for the density of primes with a primitive root $u=3$, and
	\begin{eqnarray}
	\pi _{f}(x)&=&\#\{p=n^3+2\leq x: p \text{ prime }\} \\
	&=&s_3 \cdot \frac{x^{1/3}}{\log x} +O \left (\frac{x^{1/3}}{\log^2 x} \right)\nonumber,
	\end{eqnarray}
	for the density of primes $p=n^3+2$ in the subset of cubic primes, see Conjecture \ref{conj2.3}; 
	respectively. The density constant $c(3,f)=a_3s_3r(3,f)>0$ is expected to be a real multiple $r(3,f)$ of all the product of the canonical products for the density of primes $p=n^3+2$. These are
	\begin{equation}
	a_3=\prod_{p>2} \left (1-\frac{1}{p(p-1)} \right )=.3739558136 \ldots,
	\end{equation}
	for the density of primitive root $u=3$, and
	\begin{equation} \label{el90012}
	s_3=\frac{1}{3}\prod_{p\geq2} \left (1-\frac{1-\nu_f(p)}{p-1} \right ) \approx 0.155943 \ldots,
	\end{equation} 
	see the calculation in (\ref{el2002}), respectively. The composite density is the product 
	\begin{eqnarray}
	c(3,f)&=&r(3,f) \cdot \prod_{p\geq 2} \left (1-\frac{1-\nu_f(p)}{p-1} \right ) \cdot \prod_{p>2} \left (1-\frac{1}{p(p-1)} \right ) \nonumber \\
	&\approx& 0.058316 \cdot r(3,f).
	\end{eqnarray}
	The real number $r(3,f) \in \mathbb{R}$ is unknown. This seems to serve as a correction factor for the irregularity in the distribution of the primes, and its dependence on the fixed primitive root $u=3$. \\

\begin{table}	
	\begin{center}
		\begin{tabular}{||c|c |c| c| c|c||} 
			\hline
			$ x$ & $\pi(x)$& $\pi_{f}(x)$  & $\pi_f(x,3)$ & $ \frac{\pi_{f}(x,3)}{\pi_f(x)}$ &  $c(3,f)=\frac{\pi_{f}(x,3)}{(x^{1/3}/\log x)}$  \\ [1ex]  
			\hline\hline
			$ 10^3$ &168&3 &1 &.333333 &0.230259\\ 
			\hline
			$10^6$&78498&10 &3 &0.300000 &0.138155  \\
			\hline
			$10^9$&50847534 &74&19 &0.256757&0.131247\\
			\hline
			$10^{12}$&37607912018 &520 &135 &0.259615 &0.124340 \\ 
			\hline
			$10^{15}$&29844570422669&4059 &1058 & 0.260655 &0.121807 \\ 
			\hline
			$10^{18}$&24739954287740860 &33795 &8760&0.259210 &0.121024 \\ 
			\hline
		\end{tabular}
	\end{center}
\caption{The Polynomial $f(x)=x^3+2$ and the primitive root $u=3$.}
\end{table}	

\vskip 1 in
	\section{Quartic Primes with Fixed Primitive Roots}
	Let $f(x)=x^4+1 \in \mathbb{Z}[x]$ be a prime producing polynomial of degree $\text{deg}(f)=4$. The weighted characteristic function for primitive roots of the elements in a finite \(u\in \mathbb{F}_p\)  field, satisfies the relation
	\begin{equation}
	\Lambda(n^4+1)\Psi(u)=
	\left \{\begin{array}{ll}
	\log(n^4+1) & \text{ if } f(n)=p^k,k \geq 1, \text{ and }\ord_p (u)=p-1,  \\
	0 & \text{ if } f(n) \ne p^k, k\geq 1, \text{ or }\ord_p (u)\neq p-1. \\
	\end{array} \right.
	\end{equation} 
	The definition of the characteristic function $\Psi (u)$ is given in Lemma \ref{lem3.2}, and the vonMangoldt function is defined by $\Lambda(n)=\log n$ if $n \geq 1$ is a prime power, otherwise it vanishes.\\

	\begin{thm} \label{thm7.2}
		For a large number $x\geq 1$, the number of primes $p=n^4+1 \leq x$ with a fixed primitive root \(u\ne \pm 1,v^2 \)  has the asymptotic formula
		\begin{equation} \label{el21020}
		\sum _{ \substack{p=n^4+1\leq x\\ \ord_p(u)=p-1}} 1 =c(u,f) \frac{x^{1/4}}{ \log x}+O \left (\frac{x^{1/4}}{\log^2x}\right ),
		\end{equation} 
		
		where $c(u,f) \geq 0$ is a constant depending on the polynomial $f(x)=x^4+1$, and the integer $u \in \mathbb{Z}$.
	\end{thm}
	
\begin{proof} Repeat the argument for Theorem \ref{thm1.1} or Theorem \ref{thm1.2}, mutatis mutandis.  \end{proof}

\subsection{Numerical Data for the Primitive Root 3}
	The quadratic reciprocity calculation 
\begin{equation}\label{qq99}
	\left ( \frac{2}{p}\right )=(-1)^{(p^2-1)/8}=1
\end{equation}
	preempts the existence of a primitive root $u=2$ modulo $p=n^4+1$. In this case, the smallest possible is $u=3$.\\
	
A small numerical experiment was conducted to show the abundance of primitive root producing polynomial $f(x)=x^4+1$ and to estimate the density constant $c(3,f)$ in the formula 
\begin{eqnarray}
	\pi _{f}(x,3)&=&\#\{p=n^4+1\leq x: p \text{ prime and } \ord_p(3)=p-1\} \nonumber \\
	&=&c(3,f) \cdot \frac{x^{1/4}}{\log x} +O \left (\frac{x^{1/4}}{\log^2 x} \right)
	\end{eqnarray}
This constant is product of various quantities. These included the basic counting functions
\begin{eqnarray}
	\pi(x,3)&=&\#\{p=n^4+1\leq x: \ord_p(3)=p-1\} \nonumber \\
	&=&c(3,f) \cdot \frac{x^{1/4}}{\log x} +O \left (\frac{x^{1/4}}{\log^2 x} \right)
	\end{eqnarray}
and
\begin{eqnarray}
\pi _{f}(x)&=&\#\{p=n^4+1\leq x: p \text{ prime }\} \\
&=&s_4 \frac{x^{1/4}}{\log x} +O \left (\frac{x^{1/4}}{\log^2 x} \right) \nonumber ,
\end{eqnarray}
see Conjecture \ref{conj2.4}; respectively.\\ 

The density constant $c(3,f)>0$ is expected to be a real multiple of the product of the basic canonical densities. The basic canonical densities are
\begin{equation}
a_3=\prod_{p>2} \left (1-\frac{1}{p(p-1)} \right )=.3739558136 \ldots,
\end{equation} 
for the density of primes with a primitive root $u=3$, and  
\begin{equation} \label{el9032}
s_4= \frac{1}{4}\prod_{p\geq2} \left (1-\frac{1-\nu_f(p)}{p-1} \right ) =.66974 \ldots,
\end{equation}	
for the density of primes $p=n^4+1$ in the subset of quartic primes, see the calculation in (\ref{el2002}),	respectively. The composite density is product of the the basic densities
	\begin{eqnarray}
	c(3,f)&= &r(3,f) \cdot \prod_{p\geq 2} \left (1-\frac{1-\nu_f(p)}{p-1} \right ) \cdot \prod_{p>2} \left (1-\frac{1}{p(p-1)} \right ) \nonumber \\
	&=&0.250453167 \cdot r(3,f).
	\end{eqnarray}
	The real number $r(3,f) \in \mathbb{R}$ is unknown. This seems to serve as a correction factor for the irregularity in the distribution of the primes, and its dependence on the fixed primitive root $u=3$. \\
	
	The fifth column has the relative density defined by the ratio $ \pi_{f}(x,3)/\pi_f(x)$. This ratio depends on the polynomial $f(x)=x^4+1$ and the fixed primitive root $u=3$.\\

\begin{table}	
	\begin{center}
		\begin{tabular}{||c|c |c| c| c|c||} 
			\hline
			$ x$ & $\pi(x)$& $\pi_{f}(x)$  & $\pi_f(x,3)$ & $ \frac{\pi_{f}(x,3)}{\pi_f(x)}$ &  $c(3,f)=\frac{\pi_{f}(x,3)}{(x^{1/4}/\log x)}$  \\ [1ex]  
			\hline\hline
			$ 10^4$ &1229&4 &2 &0.500000 &0.460517\\ 
			\hline
			$10^8$&5761455&18 &12 &0.666667&0.552620  \\
			\hline
			$10^{12}$&37607912018 &111&63 &0.567568&0.435189\\
			\hline
			$10^{16}$&279238341033925 &790 &463 &0.586076 &0.426439 \\ 
			\hline
                        $10^{20}$&2220819602560918840 &6396 &3866 &0.604440 &0.44509 \\ 
			\hline
		\end{tabular}
	\end{center}
\caption{The Polynomial $f(x)=x^4+1$ and the primitive root $u=3$.}
\end{table}
\subsection{Numerical Data for the Primitive Root 5}
	A small numerical experiment was conducted to show the abundance of primitive root producing polynomial $f(x)=x^4+1$ and to estimate the density constant $c(5,f)$in the formula 
	\begin{eqnarray}
	\pi _{f}(x,5)&=&\#\{p=n^4+1\leq x: p \text{ prime and } \ord_p(5)=p-1\} \nonumber \\
	&=&c(5,f) \cdot \frac{x^{1/4}}{\log x} +O \left (\frac{x^{1/4}}{\log^2 x} \right)
	\end{eqnarray}
	This constant is product of various quantities. These included the basic counting functions
	\begin{eqnarray}
	\pi(x,5)&=&\#\{p=n^4+1\leq x: \ord_p(5)=p-1\} \nonumber \\
	&=&c(5,f) \cdot \frac{x^{1/4}}{\log x} +O \left (\frac{x^{1/4}}{\log^2 x} \right)
	\end{eqnarray}
	and
	\begin{eqnarray}
	\pi _{f}(x)&=&\#\{p=n^4+1\leq x: p \text{ prime }\} \\
	&=&s_4 \frac{x^{1/4}}{\log x} +O \left (\frac{x^{1/4}}{\log^2 x} \right) \nonumber ,
	\end{eqnarray}
	see Conjecture \ref{conj2.4}; respectively.\\ 
	
	The density constant $c(5,f)>0$ is expected to be a real multiple of the product of the basic canonical densities. The basic canonical densities are
	\begin{equation}
	a_5=\prod_{p>2} \left (1-\frac{1}{p(p-1)} \right )=.3739558136 \ldots,
	\end{equation} 
	for the density of primes with a primitive root $u=5$, and  
	\begin{equation} \label{el9038}
	s_4= \frac{1}{4}\prod_{p\geq2} \left (1-\frac{1-\nu_f(p)}{p-1} \right ) =.66974 \ldots,
	\end{equation}	
	for the density of primes $p=n^4+1$ in the subset of quartic primes, see the calculation in (\ref{el2002}),	respectively. The composite density is product of the the basic densities
	\begin{eqnarray}
	c(5,f)&=& r(5,f) \cdot \prod_{p\geq 2} \left (1-\frac{1-\nu_f(p)}{p-1} \right ) \cdot \prod_{p>2} \left (1-\frac{1}{p(p-1)} \right ) \nonumber \\
	&=&0.250453167 \cdot r(5,f).
	\end{eqnarray}
	The real number $r(5,f) \in \mathbb{R}$ is unknown. This seems to serve as a correction factor for the irregularity in the distribution of the primes, and its dependence on the fixed primitive root $u=5$. \\
	
	The fifth column has the relative density defined by the ratio $ \pi_{f}(x,5)/\pi_f(x)$. This ratio depends on the polynomial $f(x)=x^4+1$ and the fixed primitive root $u=2$.\\

\begin{table}	
	\begin{center}
		\begin{tabular}{||c|c |c| c| c|c||} 
			\hline
			$ x$ & $\pi(x)$& $\pi_{f}(x)$  & $\pi_f(x,5)$ & $ \frac{\pi_{f}(x,5)}{\pi_f(x)}$ &  $c(5,f)=\frac{\pi_{f}(x,5)}{(x^{1/4}/\log x)}$  \\ [1ex]  
			\hline\hline
			$ 10^4$ &1229&4 &0 &.000000 &0.000000\\ 
			\hline
			$10^8$&5761455&18 &3 &0.166667&0.138155  \\
			\hline
			$10^{12}$&37607912018 &111&25 &0.396825&0.172694\\
			\hline
			$10^{16}$&279238341033925 &790&169 &0.213924 &0.155655 \\ 
			\hline
                        $10^{20}$& 2220819602560918840&6396 &1214 &0.189806 &0.139767 \\ 
			\hline
		\end{tabular}
	\end{center}
\caption{The Polynomial $f(x)=x^4+1$ and the primitive root $u=5$.}
\end{table}

\newpage
\section{Problems}
	\begin{exe} \normalfont 
		Generalize Theorem \ref{thm1.1} to all the composite and prime values $f(n)$ of a polynomial $f(x)$ of degree $\deg(f)=m$. Is the fixed integer $u\ne \pm 1, v^2$ a primitive root for infinitely many values $\{f(n): n \geq 1\}$?
	\end{exe}
	
	\begin{exe} \normalfont  What is the density inside the set of integers $\mathbb{N}$ of the values $f(n)$ of a polynomial $f(x)$ of degree $\deg(f)=m$, which have a fixed primitive root $u$. In other words, what is the cardinality of the subset $\{f(n): n \leq x \text{ and ord}_{f(n)}( u)=\lambda((f(n)) \}$, where $\lambda(n)$ is the Carmichel function?
	\end{exe}
	
	\begin{exe} \normalfont 
		Determine the real number $r(2,f)>0$ associated with the polynomial $f(x)=x^2+1$ and the density of the fixed primitive root $u=2$ modulo the prime generating polynomial $p=f(n)$, see the numerical data in Section 7.3.
		\end{exe}
		
		\begin{exe} \normalfont 
		Determine the real number $r(2,f)>0$ associated with the polynomial $f(x)=x^3+2$ and the density of the fixed primitive root $u=1$ modulo the prime generating polynomial $p=f(n)$, see the numerical data in Section 7.3.
	\end{exe}

%cccccccccccccccccccccccccccccccccccccccc
\chapter{Densities For The Repeating Decimals} \label{c8}
	Let \(p\geq 2\) be a prime. The properties of the period of the repeating decimal $1/p=0.\overline{x_{d-1} \ldots x_1x_0}$, with $x_i \in \{0,1, 2, \ldots , 9\}$, was investigated by Wallis, Lambert, Gauss and earlier authors centuries ago, see \cite{BM09} for a historical account. As discussed in Articles 14-to-18 in \cite{GC86}, the period, denoted by $\ord_p(10)=d \geq 1$, is a divisor of $p-1$. The problem of computing the densities for the subsets of primes for which the repeating decimals have very large periods such as $d=(p-1)/2$, and $p-1$, is a recent problem.  \\
	
	\begin{thm} \label{thm8.1}
		There are infinitely many primes \(p\geq 2\) with maximal repeating decimal $1/p=0.\overline{x_{p-2}x_{p-3} \cdots x_1x_0}$, where $ 0 \leq x_i\leq 9$. Moreover, the counting function for these primes satisfies the lower bound 
		\begin{equation} \label{el03}
		\pi _{10}(x)=\#\left\{ p\leq x:\ord_p(10)=p-1 \right\} =\delta(10)\frac{x }{\log x}+O\left(\frac{x}{\log^2  x} \right),
		\end{equation}
		where $\delta(10)>0$ is a constant, for all large numbers \(x\geq 1\). 
	\end{thm}
	
	This analysis generalizes to repeating $\ell$-adic expansions $1/p=0.\overline{x_{p-2}x_{p-3} \cdots x_1x_0}$, where $ 0 \leq x_i\leq \ell-1,$ in any numbers system with nonsquare integer base $\ell \geq 2$. The binary case is stated below.\\
	
	\begin{thm} \label{thm8.2}
		There are infinitely many primes \(p\geq 3\) with maximal repeating binary expansion $1/p=0.\overline{x_{p-2}x_{p-3} \cdots x_1x_0}$, where $ 0 \leq x_i\leq 1$. Moreover, the counting function for these primes satisfies the lower bound 
		\begin{equation} \label{el07}
		\pi _{2}(x)=\#\left\{ p\leq x:\ord_p(2)=p-1 \right\} =\delta(2)\frac{x }{\log x}+O\left(\frac{x}{\log^2  x} \right),
		\end{equation}
		where $\delta(2)>0$ is a constant, for all large numbers \(x\geq 1\).
	\end{thm}
	
	The proofs of Theorems \ref{thm8.1} and \ref{thm8.2} are similar. These theorems are corollaries of the more general result for $\ell$-adic expansion considered in Theorem \ref{thm8.3}. The last section a proof of this result. \\
	
	\section{Applications to Normal Numbers}
	The decimal expansions of \textit{normal numbers} (base 10) contains any pattern of decimal digits $ z_nz_{n-1} \ldots z_1z_0$, $z_i \in \{0,1,2,3,4,5,6,7,8,9\}$, of length $n\geq 1$ with probability $1/10^n$. The repeated decimals are used in the construction of normal numbers. As a demonstration, the number $\sum_{n \geq 1}p^{-n}u^{-p^n}$, with $u\ne \pm 1, v^2$, was proved to be a transcendental normal number in \cite{ST76}. Another application appears in \cite{BC02}. \\
	
	\begin{cor} \label{cor8.1}
		The numbers
		\begin{equation}\label{1099}
		\sum_{n \geq 1}\frac{1}{p^n}\frac{1}{2^{p^n}} \qquad \text{ and } \qquad \sum_{n \geq 1}\frac{1}{p^n}\frac{1}{10^{p^n}}
		\end{equation}	
		are normal numbers for infinitely many primes \(p\geq 3\).
	\end{cor}  
Lately, more general results have been determined, consult the literature.	
	
\section{Evaluation Of The Main Term}
	Finite sums and products over the primes numbers occur on various problems concerned with primitive roots. These sums and products often involve the normalized totient function $\varphi(n)/n=\prod_{p|n}(1-1/p)$ and the corresponding estimates, and the asymptotic formulas.\\

\begin{lem} \label{lem8.1}
	{\normalfont (\cite[Lemma 1]{SP69}) } Let \(x\geq 1\) be a large number, and let \(\varphi (n)\) be the Euler totient function. Then
	\begin{equation}
	\sum_{p\leq x }\frac{\varphi(p-1)}{p-1}
	=\li(x)
	\prod_{p \geq 2 } \left(1-\frac{1}{p(p-1)}\right)+
	O\left(\frac{x}{\log ^Bx}\right) ,
	\end{equation}
	where \(\li(x)\) is the logarithm integral, and \(B> 1\) is an arbitrary constant, as \(x \rightarrow \infty\). 
\end{lem}

This result is ubiquitous in various other results in Number Theory. Related discussions are given in \cite[p.\ 16]{MP04}, and, \cite{VR73}. The error term can be reduced to the same order of magnitude as the error term in the prime number theorem 
\begin{equation} \label{900-77}
\pi(x,q,a)=\frac{\text{li}(x)}{\varphi(q)}+O\left(e^{-c \sqrt{\log x}} \right ),
\end{equation} 
where $c>0$ is an absolute constant. But the simpler notation will be used here. 
  \\

\begin{lem}  \label{lem8.2}  Let \(x\geq 1\) be a large number, and let \(\varphi (n)\) be the Euler totient function. Then
	\begin{equation} \label{el88500}
	\sum_{p\leq x} \frac{1}{p}\sum_{\gcd(n,p-1)=1} 1=\li(x)\prod_{p \geq 2 } \left(1-\frac{1}{p(p-1)}\right)+O\left(\frac{x}{\log
		^Bx}\right) .
	\end{equation} 
\end{lem}

\begin{proof}  A routine rearrangement gives 
\begin{eqnarray} \label{el88503}
\sum_{p\leq x} \frac{1}{p}\sum_{\gcd(n,p-1)=1} 1&=&\sum_{p\leq x} \frac{\varphi(p-1)}{p}\\
&=&\sum_{p\leq x} \frac{\varphi(p-1)}{p-1}-\sum_{p\leq x} \frac{\varphi(p-1)}{p(p-1)} \nonumber.
\end{eqnarray} 
To proceed, set $q=2$ and $a=1$ and apply Lemma \ref{lem8.1} to reach
\begin{eqnarray} \label{el88506}
\sum_{p \leq x}\frac{\varphi(p-1)}{p-1}-\sum_{p\leq x} \frac{\varphi(p-1)}{p(p-1)}
&=&a_0 \text{li}(x)+O \left (\frac{x}{\log^Bx}\right )-\sum_{p\leq x} \frac{\varphi(p-1)}{p(p-1)} \nonumber\\
&=&a_0 \text{li}(x)+O \left (\frac{x}{\log^Bx}\right ),
\end{eqnarray} 
where the second finite sum $\sum_{p\leq x} \frac{\varphi(p-1)}{p(p-1)} \ll \log x$ is absorbed into the error term, $B>1$ is an arbitrary constant, and the constant $a_0=\prod_{p \geq 2 } \left(1-\frac{1}{p(p-1)}\right)$.    
\end{proof}

The logarithm integral can be estimated via the asymptotic formula
\begin{equation} \label{el8100}
\li(x)=\int_{2}^{x} \frac{1}{\log z}  d z=\frac{x}{\log x}+O\left (\frac{x}{\log^2 x} \right ).
\end{equation}\\ 

The logarithm integral difference $\li(2x)-\li(x)=x/\log x+O\left(x/\log^2x \right)=\li(x)$ is used in various calculations, see \cite[p.\ 102]{AP98}, or similar reference. 

\section{Estimate For The Error Term}
The upper bounds of exponential sums over subsets of elements in finite fields $\mathbb{F}_p$ stated in Chapter 4 are used to 
estimate the error term $E(x)$ in the proof of Theorems \ref{thm8.1}, \ref{thm8.2}, and \ref{thm8.3}. An application of any of the Lemma \ref{lem4.3} leads to a sharper result, this is completed below.\\

\begin{lem} \label{lem8.3} Let \(p\geq 2\) be a large prime, let \(\psi \neq 1\) be an additive character, and let \(\tau\) be a primitive root mod \(p\). If the element \(u\ne 0\) is not a primitive root, then, 
	\begin{equation} \label{el89400}
	\sum_{x \leq p\leq 2x}
	\frac{1}{p}\sum_{\gcd(n,p-1)=1,} \sum_{ 0<k\leq p-1} \psi \left((\tau ^n-u)k\right)\leq \frac{8x^{1-\varepsilon}}{\log x}
	\end{equation} 
	for all sufficiently large numbers $x\geq 1$ and an arbitrarily small number \(\varepsilon >0\).
\end{lem}

\begin{proof}  By hypothesis $u \ne \tau^n$ for any $n \geq 1$ such that $\gcd(n,p-1)=1$ means that the inner character sum  $\sum_{ 0<k\leq p-1} \psi \left((\tau ^n-u)k\right)=-1$. Therefore,
\begin{eqnarray} \label{el89401}
|E(x)|&=&\left |\sum_{x \leq p\leq 2x}
\frac{1}{p}\sum_{\gcd(n,p-1)=1,} \sum_{ 0<k\leq p-1} \psi \left((\tau ^n-u)k\right)
 \right |  \nonumber \\
&=&\sum_{x \leq p\leq 2x}
\frac{1}{p}\sum_{\gcd(n,p-1)=1}1 \\
&=&\sum_{x \leq p\leq 2x}
\frac{\varphi(p-1)}{p} \nonumber \\
&\leq & \frac{1}{2}\sum_{x \leq p\leq 2x} 1 \nonumber \\
&\leq &\frac{1}{2}\frac{x}{\log x}+O\left (\frac{x}{\log^2 x}\right ),
\end{eqnarray} 

since $\varphi(p-1)/p \leq 1/2$ for all primes $p\geq 2$. This implies that there is a nontrivial upper bound. To sharpen this upper bound, let $\psi(z)=e^{i 2 \pi kz/p}$ with $0< k<p$, and rearrange the triple finite sum in the form
\begin{eqnarray} \label{e89901}
E(x)&=&\sum_{x \leq p \leq 2x}\frac{1}{p} \sum_{ 0<k\leq p-1,}  \sum_{\gcd(n,p-1)=1} \psi ((\tau ^n-u)k) \\  
&= & \sum_{x \leq p \leq 2x}
 \frac{1}{p}\sum_{ 0<k\leq p-1} e^{-i 2 \pi uk/p}   \sum_{\gcd(n,p-1)=1} e^{i 2 \pi k\tau ^n/p} \nonumber ,
\end{eqnarray} 
and let
\begin{equation} \label{el89810}
U_p=\max_{1\leq u<p}\frac{1}{p}\sum_{ 0<k\leq p-1} e^{-i2 \pi uk/p}    \qquad \text{ and } \qquad V_p=\max_{1\leq k<p}\sum_{\gcd(n,p-1)=1} e^{i2 \pi k\tau ^n/p} .
\end{equation}
Then, by the triangle inequality, it follows that
\begin{equation} \label{el89901}
|E(x)| \leq \sum_{x \leq p \leq 2x} | U_p V_p| .
\end{equation}  
Now consider the Holder inequality $|| AB||_1 \leq || A||_{r} \cdot || B||_s $ with $1/r+1/s=1$. In terms of the components in (\ref{el89810}) this inequality has the explicit form
\begin{equation} \label{el89902}
\sum_{x \leq p \leq 2x} | U_p V_p|\leq \left ( \sum_{x \leq p \leq 2x} |U_p|^r \right )^{1/r} \left ( \sum_{x \leq p \leq 2x} |V_p|^s \right )^{1/s} .
\end{equation} 

The absolute value of the first exponential sum $U_p=U_p(u)$ is given by
\begin{equation} \label{el89903}
| U_p |= \left |\max_{1 \leq u <p}\frac{1}{p} \sum_{ 0<k\leq p-1} e^{-i2 \pi uk/p} \right | =\frac{1}{p}  .
\end{equation} 
This follows from $\sum_{ 0<k\leq p-1} e^{i 2 \pi uk/p}=-1$, for $u\ne 0$. The corresponding $r$-norm $||U_p||_r^r =  \sum_{x \leq p \leq 2x} |U_p|^r$ has the upper bound
\begin{eqnarray}\label{el89906}
\sum_{x \leq p \leq 2x} |U_p|^r &=&\sum_{x \leq p \leq 2x}  \left |\frac{1}{p} \right |^r \nonumber \\
 &\leq & \frac{1}{x^r}\sum_{x \leq p \leq 2x}1  \\
 & \leq& \frac{2x^{1-r}}{\log x} \nonumber. 
\end{eqnarray}
The last line uses $\pi(2x)-\pi(x)\leq 2x/\log x$. The absolute value of the second exponential sum $V_p=V_p(k)$ dependents on $k$; but it has a uniform, and independent of $k\ne 0$ upper bound
\begin{equation} \label{el89978}
\max_{1\leq k \leq p-1}   \left |   \sum_{\gcd(m,p-1)=1} e^{i 2 \pi k \tau^m/p} \right | \leq  p^{1-\varepsilon} ,
\end{equation}   
where \(\varepsilon <1/16\) is an arbitrarily small number, see Lemma \ref{lem4.3}. \\

The corresponding $s$-norm $||V_p||_s^s =  \sum_{x \leq p \leq 2x} |V_p|^s$ has the upper bound
\begin{eqnarray} \label{el89980}
\sum_{x \leq p \leq 2x} |V_p|^s  &\leq& \sum_{x \leq p \leq 2x}  \left |p^{1-\varepsilon} \right |^s \nonumber \\
&\leq& (2x)^{(1-\varepsilon) s}\sum_{x \leq p \leq 2x}1 \\
&\leq& \frac{2^{1+s}x^{1+(1-\varepsilon) s}}{\log x} \nonumber 
\end{eqnarray}
since $2^{(1-\varepsilon)s}\leq 2^s$.
Now, replace the estimates (\ref{el89906}) and (\ref{el89980}) into (\ref{el89902}), the Holder inequality, to reach
\begin{eqnarray} \label{el89991}
\sum_{x \leq p \leq 2x}
\left | U_pV_p \right | 
&\leq & 
\left ( \sum_{x \leq p \leq 2x} |U_p|^r \right )^{1/r} \left ( \sum_{x \leq p \leq 2x} |V_p|^s \right )^{1/s} \nonumber \\
&\leq & \left ( \frac{2x^{1-r}}{\log x}  \right )^{1/r} \left ( \frac{2^{1+s}x^{1+(1-\varepsilon) s}}{\log x}  \right )^{1/s} \nonumber \\
&\leq & \left ( \frac{2x^{1/r-1}}{\log ^{1/r}x}  \right )  \left ( \frac{4x^{1/s+1-\varepsilon}}{\log^{1/s} x}  \right ) \\
&\leq &  \frac{8x^{1/r+1/s-\varepsilon}}{\log x}  \nonumber \\
&\leq &  \frac{8x^{1-\varepsilon}}{\log x} \nonumber .
\end{eqnarray}

Note that this result is independent of the parameters $1/r+1/s=1$.   
\end{proof}

\section{Basic Result }
This section is concerned with a general result on the densities of primes with fixed primitive roots, but without reference to arithmetic progressions.

\begin{thm} \label{thm8.3} A fixed integer \(u\neq \pm 1,v^2\) is a primitive root mod \(p\) for infinitely many primes \(p\geq 2\). In addition, the density of these primes satisfies 
	\begin{equation} \label{e8887l3}
	\pi_u(x)=\#\left\{ p\leq x:\text{ord}_p(u)=p-1 \right\}=\delta(u)\frac{x}{\log x}+O\left( \frac{x}{\log^2 x)} \right),
	\end{equation}
	where \(\delta(u) \geq 0\) is a constant depending on the integer $u\ne 0$, for all large numbers \(x\geq 1\).
\end{thm}

\begin{proof}  Suppose that \(u\neq \pm 1,v^2\) is not a primitive root for all primes \(p\geq x_0\), with \(x_0\geq 1\) constant. Let \(x>x_0\) be a large number, and consider the sum of the characteristic function over the short interval \([x,2x]\), that is, \\
\begin{equation} \label{el88740}
0=\sum _{x \leq p\leq 2x} \Psi (u).
\end{equation}\\
Replacing the characteristic function, Lemma \ref{lem3.2}, and expanding the nonexistence equation (\ref{el88740}) yield\\
\begin{eqnarray} \label{el88750}
0&=&\sum _{x \leq p\leq 2x} \Psi (u) \nonumber \\
&=&\sum_{x \leq p\leq 2x} \frac{1}{p}\sum_{\gcd(n,p-1)=1,} \sum_{ 0\leq k\leq p-1} \psi \left((\tau ^n-u)k\right)\\
&=&a_u\sum_{x \leq p\leq 2x} \frac{1}{p}\sum_{\gcd(n,p-1)=1} 1+\sum_{x \leq p\leq 2x}
\frac{1}{p}\sum_{\gcd(n,p-1)=1,} \sum_{ 0<k\leq p-1} \psi \left((\tau ^n-u)k\right)\nonumber\\
&=&a_uM(x) + E(x)\nonumber,
\end{eqnarray} \\
where $a_u \geq0$ is a constant depending on the fixed integer $u\ne0$. \\

The main term $M(x)$ is determined by a finite sum over the trivial additive character \(\psi =1\), and the error term $E(x)$ is determined by a finite sum over the nontrivial additive characters \(\psi =e^{i 2\pi  k/p}\neq 1\).\\

Applying Lemma \ref{lem8.2} to the main term, and Lemma \ref{lem8.3} to the error term yield\\
\begin{eqnarray} \label{el89760}
\sum _{x \leq p\leq 2x} \Psi (u)
&=&a_uM(x) + E(x) \nonumber\\
&=&\delta(u)\left (\text{li}(2x)-\text{li}(x) \right )+O\left(\frac{x}{\log^Bx}\right)+O(x^{1-\varepsilon})\\
&=& \delta(u)\frac{x}{\log x}+O\left(\frac{x}{\log^2x}\right) \nonumber \\
&>&0 \nonumber,
\end{eqnarray} \\
where the constant $\delta(u)=a_ua_0 \geq 0$ depending on $u$. However, $\delta(u)=a_ua_0 > 0$ contradicts the hypothesis  (\ref{el88740}) for all sufficiently large numbers $x \geq x_0$. Ergo, the short interval $[x,2x]$ contains primes with the fixed primitive root $u$.  
 \end{proof} 

Information on the determination of the constant \(\delta(u)=a_ua_0 \geq 0\), which is the density of primes with a fixed primitive root \(u\ne \pm 1,v^2\), is provided in Theorem \ref{thm6.1}, Chapter 6.

\section{Proofs of Theorems 8.1 And  8.2}
The repeating decimal fractions have the squarefree base $u=10$. In particular, the repeated fraction representation 
\begin{equation}
\frac{1}{p}=\frac{m}{10^d}+\frac{m}{10^{2d}}+ \cdots= m \sum_{n \geq 1} \frac{1}{10^{dn}}=\frac{m}{10^d-1}
\end{equation}
has the maximal period $d=p-1$ if and only if 10 has order $\ord_p(10)=p-1$ modulo $p$. This follows from the Fermat little theorem and $10^d-1 =mp$.  The exceptions, known as Abel-Wieferich primes, satisfy the congruence 
\begin{equation}\label{8825}
\ell^{p-1}-1\equiv 0 \bmod p^2,
\end{equation}
confer \cite[p.\ 333]{RP96} for other details. \\ 

The result for repeating decimal of maximal period is a simple corollary of the previous result. \\

\begin{proof} (Theorem 8.1.) This follows from Theorem \ref{thm8.3} replacing the base $\ell=10$. That is,
\begin{equation} \label{el8964}
\pi_{10}(x)=\sum _{p\leq x} \Psi (10)
=\delta(10)\li(x) +O\left( \frac{x}{\log^B x} \right) .
\end{equation}    

This completes the proof.
\end{proof}

Let $\ell=(sv^2)^k \ne \pm 1$ with $s\geq 2$ and $k\geq1$ a squarefree integer. A formula for computing the density appears in Chapter 6, Theorem \ref{thm6.1}, and \cite[p. 220]{HC67}. In the decimal case $\ell=10$, the density is
\begin{equation} \label{667}
\delta(10)= \prod_{p \geq 2} \left ( 1 -\frac{1}{p(p-1)} \right) =0.357479581392 \ldots .
\end{equation}

The argument can be repeated for any other squarefree base $\ell=2, 3, \ldots $.

\newpage
\section{Problems}
	\begin{exe} \normalfont 
		Determine the distribution of the digits of the repeated decimals $1/p=0.\overline{x_{p-2} \cdots x_{1}x_0}$, with $0 \leq x_i \leq \ell-1$, and  maximal periods. References are \cite{MT10} and \cite{ST76}.
	\end{exe}
	
	\begin{exe} \normalfont 
		Compute some information on the correlation function $R(t)=\sum_{0<n<p}x_{n}x_{n+t}$, with $0 \leq x_i \leq \ell-1$, for sequences of maximal periods, where $t \in \mathbb{Z}$ is an integer.
	\end{exe}
	
	\begin{exe} \normalfont 
		Show that primes of the forms $p_1=2^{2^n}+1, n \geq 1$; $p_2=2q+1$ with $q\geq 2$ prime; $p_3=q_xn+3$ with $q_x=\prod_{p\leq x}p$ and $n, x\geq 1$ yield the optimal number of primitive roots. 
	\end{exe}
	
	\begin{exe} \normalfont 
		Show that either $u\ne \pm 1, v^2$ or $p+u$ is a primitive root modulo $p^k$ for all $k \geq 1$.
	\end{exe}
	
	\begin{exe} \normalfont 
		Take $p=487$. Then, the congruences $10^{p-1} -1 \equiv 0 \bmod p$ and $10^{p-1} -1 \equiv 0 \bmod p^2$ are valid. Similarly, for $p=5$, the congruences $7^{p-1} -1 \equiv 0 \bmod p$ and $7^{p-1} -1 \equiv 0 \bmod p^2$ are valid. Can this be generalized to the integers, say, to the congruences $b^{\varphi(5\cdot487)} -1 \equiv 0 \bmod 5\cdot 487$ and $b^{\varphi(5\cdot487)} -1 \equiv 0 \bmod 5^2\cdot 487^2$ for some $b>1$. Hint: Consider Fermat quotients. \end{exe}
	
	\begin{exe} \normalfont 
		Show that primes of the forms $p_1=2^{2^n}+1, n \geq 1$; $p_2=2q+1$ with $q\geq 2$ prime; $p_3=q_xn+3$ with $q_x=\prod_{p\leq x}p$ and $n, x\geq 1$ yield the optimal number of primitive roots. 
	\end{exe}
	
	\begin{exe} \normalfont 
		Show that there are infinitely many reunit primes $p=(10^q-1)/9=11 \cdots1 \geq 3$ with $q\geq 2$ prime. Compute or estimate a lower bound for the counting function $\pi_R(x)=\#\{p \leq x: p \text{ is a prime reunit} \}$. Hint: compute the density of the the subset of primes $p \geq 2$ such that $\ord_p(10)=q$ . 
	\end{exe}
	
	\begin{exe} \normalfont 
		Given an integer $m\geq 1$, what is the least prime $p \geq 2$ for which the repeated decimal $1/p=0.\overline{x_{m-1} \cdots x_{1}x_0}$, with $0 \leq x_i \leq \ell-1$, has period $m$? Can this be generalized to the decimal expansion $1/n$ of any integer $n \geq 1$?
	\end{exe}
	
\begin{exe} \normalfont
Let $x\geq 1$ be a large number and let $q= \prod_{r \leq \log \log x}r \asymp \log x$. Show that
$$
\frac{1}{\varphi(q)} =\frac{1}{q}\prod_{r\leq \log \log x}\left ( 1-\frac{1}{r}\right)^{-1}\geq c_0\frac{\log\log\log x}{\log x},
$$
with $c_0>0$ constant.
\end{exe}

%cccccccccccccccccccccccccccccccccccccccccccccc
\chapter{Class Numbers And Primitive Roots} \label{c9}
	Let \(p\geq 2\) be a prime. The length $m\geq 1$ of the period of the $\ell $-adic expansion of the rational number \(a/p=0.\overline{x_1x_2 \ldots x_d}\), where \(a<p\), and $0\leq x_i<\ell$ was the original motivation for study of primitive roots. \\
	
	\begin{lem} \label{lem9.1}
		{\normalfont (Gauss)} The decimal expansion \(1/p=0.\overline{x_1x_2 \ldots x_d}\) has maximal length
		\(d=p-1\) if and only if \(p\geq 2\) is a prime and 10 is a primitive root \(\tmod  p\). 
	\end{lem}
	
	\section{Recent Formulae}
	The class numbers of quadratic fields have various formulations. Among these is the expression
	\begin{equation} \label{el22100}
	h(p)=
	\begin{array}{ll}
	\left \{  
	\begin{array}{ll}
	\frac{w\sqrt{-p}}{2\pi }L(1,\chi ), & p<0, \\
	\frac{\sqrt{p}}{\log  \epsilon }L(1,\chi ), & p>0, \\
	\end{array} \right .
	\\
	\end{array}
	\end{equation}
	where \(\chi \neq 1\) is the quadratic symbol \(\tmod p\), \(\epsilon =a+b\sqrt{d}\) is a fundamental unit, and \(w=6,4,\text{ or } 2\) according
	to \(d=-3,-4, \text{ or } d<-4\) respectively, see \cite[p.\ 28]{UW00}, \cite[p.\ 391]{MV07} or similar references.\\
	
	\begin{thm}   \label{thm9.1}
		{\normalfont (Girstmair formula)} Let \(p=4m+3\geq 7\) be a prime, and let \(\ell \geq 2\) be a fixed primitive
		root \(\tmod p\). Let 
		\begin{equation}
		\frac{1}{p}=\sum _{n\geq 1} \frac{x_n}{\ell ^n},
		\end{equation}
		where \(x_i\in \{ 0, 1, 2, \text{...}, \ell -1 \}\), be the unique expansion; and let \(h(-p)\geq 1\) be the class number of the quadratic field \(\mathbb{Q}(\sqrt{-p})\).
		Then
		\begin{equation}
		\sum _{1\leq n\leq p-1} (-1)^nx_n=(\ell +1)h(-p) .
		\end{equation}
	\end{thm}
	
	The analysis of the Girstmair formula appears in \cite{GK94}, and \cite{GKS94}. A generalization to base $\ell$ and any orders \(\ord_p(\ell)=(p-1)/r\), with $r\geq 1$, was later proved in
	\cite{MT10}.\\ 
	
	\begin{thm}  \label{thm9.2}
		{\normalfont (\cite{MT10})}  Let \(p\geq 3\) be a prime, and \(r  \,|\,  p-1\). Suppose that a character \(\chi\) of
		order \(r\) is odd and that \(\ell \geq 2\) has order \(\ord_p(\ell)=(p-1)/r\). Let 
		\begin{equation}
		\frac{1}{p}=\sum _{n\geq 1} \frac{x_n}{\ell ^n},
		\end{equation}
		where \(x_i\in \{ 0, 1, 2, \text{...}, \ell -1 \}\), be the unique \(\ell$-expansion. Then
		\begin{equation}
		\sum _{1\leq n\leq (p-1)/r} x_n=\frac{\ell -1}{2}\frac{p-1}{r}+\frac{\ell -1}{r} \sum _{\text{ord$\chi $}=r} B_{1,\chi },
		\end{equation}
		where \(B_{1,\chi }\) is the generalized Bernoulli number.
		\end{thm}
		
		A similar relationship between the class number and the continued fraction of the number \(\sqrt{p}\) was discovered earlier, see \cite[p.\ 241]{HF73}.\\
		
		\begin{thm}   \label{thm9.3}
		{\normalfont (Hirzebruch formula)}  Let \(p=4m+3\geq 7\) be a prime. Suppose that the class number of the quadratic
		field \(\mathbb{Q}(\sqrt{p})\) is \(h(p)=1\), and \(h(-p)\geq 1\) is the class number of the quadratic field \(\mathbb{Q}(\sqrt{-p})\).
		Then 
		\begin{equation}
		\sum _{1\leq n\leq 2t} (-1)^{n+1}a_n=3h(-p) ,
		\end{equation}
		where \(\sqrt{p}=\left[a_0,\overline{a_1,a_2,\ldots, a_{2t}}\right.\)], \(a_i\geq 1\), is the continued fraction of the real number \(\sqrt{p}\).\\
		\end{thm}
		
		Other results concerning this formula are given in \cite{CS73}, \cite{GC92}, \cite{SJ14}, et alii. A few simple and straight forward applications of these results are presented here. One of these applications is related to and reminiscent of the square root cancellation rule. \\

		More generally, the alternating sum of partial quotients is a multiple of 3 for an entire residue class. Specifically, if
		\(\sqrt{N}=\left[a_0,\overline{a_1,a_2,\ldots, a_{2t}}\right.\)], with $N\equiv 3 \bmod 4, N\not \equiv 0 \bmod 3$, then
		\begin{equation}
		\sum_{1 \leq n  \leq d} (-1)^{n+1}a_n=3M.
		\end{equation}
		This was suggested in \cite{CS73}, and proved in \cite{SA74}.\\

		\begin{cor}    \label{cor9.1}
		Let \(p=4m+3\geq 7\) be a prime. Let \(\ell \geq 2\) be a fixed primitive root \(\tmod p\), and
		let \(1/p=0.\overline{x_1x_2 \ldots x_d}\) be the $\ell $-adic expansion of the rational number \(1/p\). Then, for any arbitrarily
		small number \(\epsilon >0\), the alternating sum satisfies 
		\begin{equation}
		(\text{i})\sum_{1\leq n\leq p-1} (-1)^nx_n\ll p^{1/2+\epsilon }. \hskip 1.75 in
		(\text{ii}) \sum _{1\leq n\leq p-1} (-1)^nx_n\gg p^{1/2-\epsilon }.
		\end{equation}
		\end{cor}
		
		\begin{proof}   (i) By Theorem \ref{thm9.2} the alternating sum has the upper bound
		\begin{equation}
		\sum_{1\leq n\leq p-1} (-1)^nx_n=(\ell+1)h(-p) \ll p^{1/2+\epsilon },
		\end{equation}
		where \(h(-p)\ll p^{1/2+\epsilon}\) is an upper bound of the class number of the quadratic field \(\mathbb{Q}\left(\sqrt{-p}\right)\), with \(\epsilon
		>0.\) \\
		
		For (ii), apply Siegel theorem \(L(1,\chi )\gg p^{-\epsilon }\), see \cite[p.\ 372]{MV07}. \end{proof}
		
		The class number formula (\ref{el22100})is utilized below to exhibit two ways of computing the special
		values of the $L$-functions \(L(s,\chi )=\sum _{n\geq 1} \chi (n)n^{-s}\) at \(s=1\).\\
		
		\begin{cor}  \label{cor9.2}
		Let \(p=4m+3\geq 7\) be a prime. Let \(\ell \geq 2\) be a fixed primitive root \(\tmod p\). Let
		\(1/p=0.\overline{x_1 x_2 \cdots x_d}\) be the $\ell $-adic expansion of the rational number \(1/p\), and let \(\sqrt{p}=\left[a_0,\overline{a_1,a_2,\ldots, a_{2t}}\right]\), \(a_i\geq 1\), be the continued fraction of the real number \(\sqrt{p}\). Then, 
		\begin{enumerate} 
		\item For any $h(-p)\geq 1$, the special value 
		\begin{equation} L(1,\chi )=\frac{2 \pi}{(\ell +1)w\sqrt{p}} \sum _{1\leq n\leq p-1} (-1)^nx_n. 
		\end{equation} 
		\item If $h(p)=1$, and $h(-p)\geq 1$ then 
		\begin{equation} L(1,\chi )=\frac{\pi }{3\sqrt{p}}\sum _{1\leq n\leq 2t} (-1)^{n+1}a_n .\end{equation}
		\item   If $h(p)=1$, and $h(-p)\geq 1$ then
		\begin{equation} (\ell +1)\sum _{1\leq n\leq 2t} (-1)^{n+1}a_n =3 \sum _{1\leq n\leq p-1} (-1)^nx_n.
		\end{equation}
		\end{enumerate}
		\end{cor}

		\section{Numerical Data}
		A table listing the numerical data for small primes was compiled. The first column has the primes $p\geq 3$ such that $\ord_p(10)=p-1$; and the second and third columns list the class numbers of the quadratic fields $\mathbb{Q}(\sqrt{-p})$ and $\mathbb{Q}(\sqrt{p})$ respectively. \\
		
		The fourth and fifth column display the alternating sum of digits
		\begin{equation}
		\sum_{1\leq n\leq p-1} (-1)^nx_n
		\end{equation}
		for \(1/p=0.\overline{x_{p-1} x_{p-2} \cdots x_1x_0}\);
		and the alternating sum of partial quotients
		\begin{equation}
		\sum_{1\leq n\leq m} (-1)^{n+1}a_n
		\end{equation}
		for \(\sqrt{p}=a_0.\overline{a_1 a_2 \cdots a_m}\).\\

\begin{table}
		\begin{center}
		\begin{tabular}{||c|c |c| c| c||} 
		\hline
		$ p$ & $h(-p)$& $h(p)$  & $\sum_{1\leq n\leq d}(-1)^{n}x_n$ & $\sum_{1\leq n\leq d}(-1)^{n+1}a_n$    \\ [1ex]  
		\hline\hline
		7&1 &1 &11 & 1\\ 
		\hline
		19&1 &1 &11 &3  \\
		\hline
		23 &3&1 &33&9\\
		\hline
		47 &5 &1 &55 &10 \\ 
		\hline
		59 &3 & 1 & 33 &15 \\ 
		\hline
		131 &5 & 1 & 55 &15 \\ 
		\hline
		167 &11 & 1 & 121 &33 \\ 
		\hline
		179 &5 & 1 & 55 &15 \\ 
		\hline
		223&7 & 3 & 77 &39 \\ 
		\hline
		\end{tabular}
		\end{center}
\caption{\label{909} Class Numbers  and Alternating Sums.}
\end{table}

		\newpage		
		\section{Problems}
		\begin{exe} \normalfont    
		Let \(\pi \in \mathcal{K}=\mathbb{Q}(\sqrt{d})\) be a prime, and let \(\beta \neq \pm \alpha ,\gamma ^2 ,N(\alpha
		)=\pm 1\), be a fixed primitive root \(\mod \pi\). Let \(\frac{1}{\pi }=\sum _{n\geq 1} \frac{x_n}{\beta ^n},\) where \(x_i\in \mathbb{Z}\sqrt{d}/(\beta )\) be the unique expansion. Note: the norm of an element \(\alpha =a+b\sqrt{d}\in
		\mathbb{Q}(\sqrt{d}),\) in the quadratic field is given by \(N(\alpha )=a^2-d b^2\). Generalize the Girstmair formula to quadratic numbers
		fields: Prove some form of relation such as this: If \(h\geq 1\) is the class number of the numbers field \(\mathcal{L}=\mathbb{Q}(\sqrt{\pi
		})\). Then
		$$\sum _{1\leq n\leq p-1} (-1)^nx_n\overset{?}{=}(\beta +1)h .$$
		\end{exe}
		
		\begin{exe} \normalfont 
		Show that Corollary \ref{cor9.1} implies that there exists infinitely many primes \(p=4m+3\geq 7\) for which the class number of
		the quadratic field \(\mathbb{Q}(\sqrt{p})\) is \(h(p)=1\). If not state why not. For example, in Corollary \ref{cor9.2}-iii, it is show that
		$$(\ell +1)\sum _{1\leq n\leq 2t} (-1)^na_n =3\sum _{1\leq n\leq p-1} (-1)^nx_n ,$$ if $h(p)=1$, and $h(-p)\geq 1.$\\
		\end{exe}
		
		\begin{exe} \normalfont 
		The $\ell $-adic expansion of a rational prime \(p=a_m\ell^m+\cdots +a_1\ell+a_0\) has polynomial time complexity $O(\log^c p)$, with $c>0$ constant. What is the time complexity of computing \(1/p=0.\overline{x_1x_2 \ldots x_d}\), where $0 \leq x_i, \ell$ is it exponential?   
		\end{exe}
		
%cccccccccccccccccccccccccccccccccccccccccccccccccccccccccccccccc		
\chapter{Simultaneous Primitive Roots} \label{c10}
For a pair of fixed integers \(u,v\ne \pm 1, r^2\), define the counting function
\begin{equation}\label{el1800}
\pi_{u,v}(x)=\{p\leq x:\ord_p(u)=\ord_p(v)=p-1\}.
\end{equation}
The average number of primes $p\leq x$ such that both $u$ and $v$ are simultaneous primitive roots modulo $p$ was computed in \cite{SP69} as
\begin{equation}\label{el1810}
\frac{1}{N^2}\sum_{u \leq N}\sum_{v \leq N}\pi_{u,v}(x)=B\li(x)+O\left(\frac{x}{\log^Cx}\right ),
\end{equation}
where the density constant is
\begin{equation}\label{el1815}
B=\prod_{p \geq 2} \left(1-\frac{2p-1}{p^2(p-1)}\right) ,
\end{equation}
the index $N \approx x$, and $C>1$ is an arbitrary constant.\\

The conditional proof for simultaneous primitive roots claims
\begin{equation} \label{el1820}
\pi _{u_1\ldots u_m}(x)=c(u_1,\ldots,u_m)\frac{x}{\log x}+O\left(\frac{x}{\log^2  x} (\log \log x)^{2^n-1}\right),
\end{equation}
see \cite[p.\ 114]{MK76}. Similar proof also is given in \cite{PF06}.\\

An unconditional result for the individual case with up to $u_1,u_2, \ldots, u_m$ simultaneous primitive roots, where $m=O(\log \log x)$, and the corresponding counting function
\begin{equation} \label{el1828}
\pi _{u_1\ldots u_m}(x)=\#\left\{ p\leq x:\ord_p(u_1)=\cdots =\ord_p(u_m)=p-1 \right\} ,
\end{equation}
is presented here.  \\

\begin{thm}  \label{thm10.1}
Let $ x\geq 1$ be a large number, and let $u_1,u_2, \ldots, u_m\ne \pm 1, v^2$ be fixed integers, where $m=O(\log \log x)$. Then, the number of primes with the simultaneous fixed primitive roots $u_1,u_2, \ldots, u_m$ has the asymptotic formula 
\begin{equation} \label{el1830}
\pi _{u_1\ldots u_m}(x)=c(u_1,\ldots,u_m)\frac{x}{\log x}+O\left(\frac{x}{\log^D  x} \right),
\end{equation}
where $c(u_1,\ldots,u_m)>0$ is a constant depending on the fixed $m$-tuple $u_1,u^2, \ldots, u_m\ne \pm 1, v^2$, and $D(m)>1$  is a constant depending on $m\ge1$.
\end{thm} 

This analysis generalizes the previous cases for a single fixed primitive root $u_1, u_2 \ne \pm 1, v^2$. The proofs is given in the last section. \\

\section{Evaluation Of The Main Term}
Finite sums and products over the primes numbers occur on various problems concerned with primitive roots. These sums and products often involve the normalized totient function $\varphi(n)/n=\prod_{p|n}(1-1/p)$ and the corresponding estimates, and the asymptotic formulas.\\

\begin{lem} \label{lem10.1}
{\normalfont (\cite[Lemma 4.4]{VR73})}  Let \(x\geq 1\) be a large number, and let \(\varphi (n)\) be the Euler totient function. If \(m\geq 1\) is an integer, then
\begin{equation}
\sum_{p\leq x} \left (\frac{\varphi(p-1)}{p-1} \right)^m
=\frac{x}{\log x}
\prod_{p \geq 2 } \left(1-\frac{p^m-(p-1)^m}{p^m(p-1)}\right)+
O\left(\frac{x}{\log ^2x}e^{cm\log m+m}\right) ,
\end{equation}
where \(c\geq0\) is a constant, as \(x \rightarrow \infty\).
\end{lem}

The related results given in \cite[Lemma 1]{SP69}, and \cite[p.\ 16]{MP04} are special case of this result. A generalization to algebraic numbers fields has the form
\begin{equation}\label{100000}
\sum_{N(\mathcal{P})\leq x} \left (\frac{\varphi(N(\mathcal{P}))-1)}{N(\mathcal{P})-1} \right)^m,
\end{equation}
where $N(\mathcal{P})$ is the norm function, and its proof appears in \cite[Lemma 3.3]{HJ86}. \\

\begin{lem} \label{lem10.2}
Let \(x\geq 1\) be a large number. If $m\geq 1$ is an integer, then
\begin{eqnarray} \label{el1989}
&&\sum_{p\leq x} \frac{1}{p^m} \sum_{\gcd(n_1,p-1)=1} \cdots\sum_{\gcd(n_m,p-1)=1} 1\\
&=&\frac{x}{\log x}\prod_{p \geq 2 } \left(1-\frac{p^m-(p-1)^m}{p^m(p-1)}\right)+O\left(\frac{x}{\log
	^2x}e^{cm\log m+m}\right) \nonumber.
\end{eqnarray} 
\end{lem}

\begin{proof}  A routine rearrangement gives 
\begin{eqnarray} \label{el1983}
\sum_{p\leq x} \frac{1}{p^m}\sum_{\gcd(n_1,p-1)=1} \cdots\sum_{\gcd(n_m,p-1)=1} 1
&=&\sum_{p\leq x} \left (\frac{\varphi(p-1)}{p}\right) ^m\\
&=&\sum_{p\leq x } \left (\frac{\varphi(p-1)}{p-1}\right )^m-\sum_{p\leq x} \left (\frac{\varphi(p-1)}{p(p-1)}\right) ^m \nonumber.
\end{eqnarray} 
To proceed, apply Lemma \ref{lem10.1} to reach
\begin{eqnarray} \label{el1986}
M(x)&=& \sum_{p\leq x } \left (\frac{\varphi(p-1)}{p-1}\right) ^m-\sum_{p\leq x} \left (\frac{\varphi(p-1)}{p(p-1)}\right) ^m \nonumber \\
&=&b_m \frac{x}{\log x}+O \left (\frac{x}{\log^2x} e^{cm\log m+m}\right )-\sum_{p\leq x } \left (\frac{\varphi(p-1)}{p(p-1)}\right )^m \nonumber\\
&=&b_m \frac{x}{\log x}+O \left (\frac{x}{\log^2x} e^{cm\log m+m}\right ),
\end{eqnarray} 
where the second finite sum 
\begin{equation}
\sum_{p\leq x} \left ( \frac{\varphi(p-1)}{p(p-1)}\right )^m \ll (\log \log x)^m 
\end{equation} is absorbed into the error term, and 
\begin{equation}
b_m=\prod_{p \geq 2 } \left(1-\frac{p^{m+1}-p^m}{p^m(p-1)}\right)
\end{equation} is the density constant.   \end{proof}

\section{Estimate For The Error Term}
The upper bound for exponential sum over subsets of elements in finite rings $\left (\mathbb{Z}/N\mathbb{Z}\right )^\times$ stated in Section four is used here to estimate the error term $E_m(x)$ in the proof of Theorem 18.1. \\

\begin{lem} \label{lem10.3}
Let \(x\geq 1\) be a large number, and let $m=O(\log\log x)$. Let \(\psi \neq 1\) be an additive character, and let \(\tau\) be a primitive root mod \(p\). If the fixed integers $u_1,u_2, \ldots, u_m$ are not primitive roots, then, 
\begin{eqnarray} \label{el1900}
E(x)&=& \sum_{x \leq p\leq 2x}
\frac{1}{p^m}\sum_{\gcd(n,p-1)=1,} \sum_{ 0<k_1\leq p-1} \psi \left((\tau ^n-u_1)k_1\right) \nonumber\\
&& \cdots\sum_{\gcd(n,p-1)=1,} \sum_{ 0<k_m\leq p-1} \psi \left((\tau ^n-u_m)k_m\right)\\
&\ll& x^{1-2\varepsilon} \nonumber,
\end{eqnarray} 
where $\varepsilon>0$ is an arbitrarily small number.
\end{lem}

\begin{proof}  To simplify the notation, assume that $m=2$. Let $\psi(z)=e^{i 2 \pi kz/p}$ with $0< k<p$. By hypothesis $u_1,u_2 \ne \tau^n$ for any $n \geq 1$ such that $\gcd(n,p-1)=1$. Therefore, $\sum_{1 \leq k <p}e^{i2\pi (\tau ^n-u_1)k/p}\sum_{1 \leq k <p}e^{i2\pi (\tau ^n-u_2)k/p}=(-1)^2$ (with $u=u_1,u_2\ne0$) and the error term has the trivial upper bound
\begin{eqnarray} \label{el1901}
|E_m(x)| &= & \left |\sum_{x \leq p\leq 2x}
\frac{1}{p^2}\sum_{\gcd(n,p-1)=1,} \sum_{ 0<k_1\leq p-1} e^{i2\pi (\tau ^n-u_1)k_1/p} \sum_{\gcd(n,p-1)=1,}\sum_{ 0<k_2\leq p-1} e^{i2\pi (\tau ^n-u_2)k_2/p}\right | \nonumber \\
&\leq &\sum_{x \leq p\leq 2x}
\frac{\varphi(p-1)^2}{p^2} \\
&\leq& \frac{1}{4}\frac{x}{\log x}+O\left (\frac{x}{\log^2 x}\right ) \nonumber,
\end{eqnarray} 
since $\varphi(n)/n \leq 1/2$ for all $n \geq 1$. This information implies that the error term $E_m(x)$ is smaller than the main term $M_m(x)$, see (\ref{el20730}). Thus, there is a nontrivial upper bound. To sharpen this upper bound, rearrange the triple finite sum in the following way:
\begin{eqnarray} \label{el1903}
E_m(x)  &=&\sum_{x \leq p \leq 2x}\frac{1}{p^2} \sum_{\gcd(n,p-1)=1,}  \sum_{ 0<k_1\leq p-1} e^{i2\pi (\tau ^n-u_1)k_1/p}  \sum_{\gcd(n,p-1)=1,}\sum_{ 0<k_2\leq p-1}  e^{i2\pi (\tau ^n-u_2)k_2/p}\nonumber \\
&=&  \sum_{x \leq p\leq 2x} U_p V_p ,
\end{eqnarray} 
where
\begin{equation} \label{el1906}
U_p=\frac{1}{p}\sum_{ 0<k_1\leq p-1} e^{-i 2 \pi u_1k_1/p} \sum_{ 0<k_2\leq p-1} e^{-i 2 \pi u_2k_2/p} ,
\end{equation} 
and
\begin{equation} \label{el1907}
V_p=\frac{1}{p} \sum_{\gcd(n,p-1)=1} e^{i 2 \pi k_1\tau ^n/p}\sum_{\gcd(n,p-1)=1} e^{i 2 \pi k_2\tau ^n/p}.
\end{equation}
Now consider the Holder inequality $|| AB||_1 \leq || A||_{r} \cdot || B||_s $ with $1/r+1/s=1$. In terms of the components in (\ref{el1906}) and (\ref{el1907}), this inequality has the explicit form
\begin{equation} \label{el1902}
\sum_{x \leq p\leq 2x} | U_p V_p|\leq \left ( \sum_{x \leq p\leq 2x} |U_p|^r \right )^{1/r} \left ( \sum_{x \leq p\leq 2x} |V_p|^s \right )^{1/s} .
\end{equation} 

The absolute value of the first exponential sum $U_p$ is given by
\begin{equation} \label{el1909}
| U_p |= \left |\frac{1}{p} \sum_{ 0<k_1\leq p-1} e^{-i 2 \pi u_1k_1/p} \sum_{ 0<k_2\leq p-1} e^{-i 2 \pi u_2k_2/p} \right | \leq\frac{1}{p}  .
\end{equation} 
This follows from $\sum_{ 0<k\leq p-1} e^{i 2 \pi uk/p}=-1$ for $u=u_1,u_2\ne0$. The corresponding $r$-norm $||U_p||_r^r =  \sum_{p\leq x} |U_p|^r$ has the upper bound
\begin{equation}\label{el1908}
\sum_{x \leq p\leq 2x} |U_p|^r=\sum_{x \leq p\leq 2x}  \left |\frac{1}{p} \right |^r \leq x^{1-r}. 
\end{equation}
Here the finite sum over the primes is estimated using integral
\begin{equation} \label{el1910}
\sum_{x \leq p\leq 2x} \frac{1}{p^{r}}
\ll\int_{x}^{2x} \frac{1}{t^{r}} d \pi(t)=O(x^{1-r}),
\end{equation}
where \(\pi(x)=x/\log x+O(x/\log ^{2} x)\) is the prime counting measure. \\

The absolute value of the second exponential sum $V_p=V_p(k)$ has the upper bound
\begin{equation} \label{e1l943}
|V_p|= \left |\frac{1}{p}\sum_{\gcd(n,p-1)=1} e^{i2 \pi k_1\tau ^n/p} \sum_{\gcd(n,p-1)=1} e^{i2 \pi k_2\tau ^n/p}\right |\ll p^{1-2\varepsilon} .
\end{equation} 
Observe that each exponential sum dependents on the parameter $k=k_1,k_2$; but it has a uniform, and independent of $k$ upper bound
\begin{equation} \label{el88978}
\max_{1\leq k \leq p-1}   \left |   \sum_{\gcd(n,p-1)=1,} e^{i 2 \pi k\tau ^n/p} \right | \ll  p^{1-\varepsilon} ,
\end{equation}   
where \(\varepsilon >0\) is an arbitrarily small number, see Lemma \ref{lem4.2}. \\

The corresponding $s$-norm $||V_p||_s^s =  \sum_{p\leq x} |V_p|^s$ has the upper bound
\begin{equation} \label{el1980}
\sum_{x \leq p\leq 2x} |V_p|^s \leq\sum_{x \leq p\leq 2x}  \left |p^{1-2\varepsilon} \right |^s \ll x^{1+s-2\varepsilon s}. 
\end{equation}
The finite sum over the primes is estimated using integral
\begin{equation} \label{el1990}
\sum_{x \leq p\leq 2x} p^{s-2\varepsilon s}
\ll\int_{x}^{2x} t^{s-2\varepsilon s} d \pi(t)=O(x^{1+s-2\varepsilon s}).
\end{equation}

Now, replace the estimates (\ref{el1908}) and (\ref{el1980}) into (\ref{el1902}), the Holder inequality, to reach
\begin{eqnarray} \label{el1991}
\sum_{x \leq p\leq 2x} | U_p V_p| 
&\leq & 
\left ( \sum_{x \leq p\leq 2x} |U_p|^r \right )^{1/r} \left ( \sum_{x \leq p\leq 2x} |V_p|^s \right )^{1/s} \nonumber \\
&\ll & \left ( x^{1-r}  \right )^{1/r} \left ( x^{1+s-2\varepsilon s}  \right )^{1/s} \nonumber \\
&\ll & \left ( x^{1/r-1}  \right )  \left ( x^{1/s+1-2\varepsilon}  \right ) \\
&\ll &  x^{1/r+1/s-2\varepsilon}  \nonumber \\
&\ll &  x^{1-2\varepsilon} \nonumber .
\end{eqnarray}

Note that this result is independent of the parameter $1/r+1/s=1$. \end{proof}

\section{Primes With Simultaneous Primitive Roots}
The results of the last two sections are used here to prove Theorem \ref{thm10.1} by means of a reductio ad absurdum argument.\\

\begin{proof} ( Theorem \ref{thm10.1}.) Suppose that $u_1,u_2, \ldots,u_m \ne \pm 1, v^2$ are not simultaneous primitive roots for all primes \(p\geq x_0\), with \(x_0\geq 1\) constant. Let \(x>x_0\) be a large number, and consider the sum of the characteristic function over the short interval \([x,2x]\), that is, 
\begin{equation} \label{el20720}
0=\sum _{x \leq p\leq 2x} \Psi (u_1)\cdots \Psi (u_m).
\end{equation}
Replacing the characteristic function, Lemma \ref{lem3.2}, and expanding the nonexistence equation (\ref{el20720}) yield
\begin{eqnarray} \label{el20730}
0&=&\sum _{x \leq p\leq 2x} \Psi (u_1)\cdots \Psi (u_m)  \\
&=&\sum_{x \leq p\leq 2x} \left (\frac{1}{p^m}\sum_{\gcd(n_1,p-1)=1,} \sum_{ 0\leq k_1\leq p-1} \psi \left((\tau ^{n_1}-u_1)k_1\right) \cdots \sum_{\gcd(n_m,p-1)=1,}\sum_{ 0\leq k_m\leq p-1} \psi \left((\tau ^{n_m}-u_m)k_m\right)\right ) \nonumber\\
&=&a_{u_1\cdots u_m}\sum_{x \leq p\leq 2x} \frac{1}{p^m}\sum_{\gcd(n_1,p-1)=1} \cdots \sum_{\gcd(n_m,p-1)=1}1 \nonumber \\
& &+\sum_{x \leq p\leq 2x}
\frac{1}{p^m}\sum_{\gcd(n_1,p-1)=1,} \sum_{ 0<k_1\leq p-1} \psi \left((\tau ^n_1-u_1)k_1\right) \cdots \sum_{\gcd(n_m,p-1)=1,}\sum_{ 0<k_1\leq p-1} \psi \left((\tau ^{n_m}-u_1)k_1\right) \nonumber\\
&=&M_m(x) + E_m(x)\nonumber,
\end{eqnarray} 
where $a_{u_1 \cdots u_m} >0$ is a constant depending on the fixed integers $u_1, \ldots, u_m \ne \pm 1, v^2$. \\

The main term $M_m(x)$ is determined by a finite sum over the trivial additive character \(\psi =1\), and the error term $E_m(x)$ is determined by a finite sum over the nontrivial additive characters \(\psi =e^{i 2\pi  k/p}\neq 1\).\\

Applying Lemma \ref{lem10.2} to the main term, and Lemma \ref{lem10.3} to the error term yield
\begin{eqnarray} \label{el20760}
0&=&\sum _{x \leq p\leq 2x} \Psi (u_1) \cdots \Psi (u_m) \nonumber \\
&=&M_m(x) + E_m(x) \nonumber\\
&=&a_{u_1\cdots u_m}\left (b_m \left (\frac{2x}{\log 2x}- \frac{x}{\log x} \right )\right )+O\left(\frac{x}{\log^2x} e^{cm\log m +m}\right)+O\left(x^{1-2 \varepsilon} \right) \nonumber\\
&=& c_{u_1\cdots u_m}\frac{x}{\log x}+O\left( \frac{x }{\log^2 x}e^{cm\log m +m} \right)   \nonumber \\
&>&0,
\end{eqnarray} 
where $c_{u_1\cdots u_m}=a_{u_1\cdots u_m}b_m>0$. This contradict the hypothesis  (\ref{el20720}) for all sufficiently large numbers $x \geq x_0$. Ergo, the short interval $[x,2x]$ contains primes $p \geq 2$ such that the $u_1, \ldots, u_m \ne \pm 1, v^2$, are simultaneous primitive roots modulo $p$. Specifically, the counting function is given by
\begin{equation} \label{el20764}
\pi_{u_1\cdots u_m}(x)=\sum_{p\leq x} \Psi (u_1)\cdots \Psi (u_m)
=c_{u_1\cdots u_m}\frac{x}{\log x}+O\left( \frac{x }{\log^2 x}e^{cm\log m +m} \right) .
\end{equation} 
This completes the verification. \end{proof}

\newpage	
\section{Problems}
\begin{exe} \normalfont 
Develop an algorithm to compute a pair of consecutive primitive roots. What is the time complexity?
\end{exe}

\begin{exe} \normalfont 
Find a formula for the density of two consecutive primitive roots. 
\end{exe}

%ccccccccccccccccccccccccccccccc
\chapter{Densities of Squarefree Totients $p-1$} \label{c11}
	Let $x \geq 1$ be a large number, and fix an integer $u \ne \pm 1, v^2$. This Section is concerned with the analysis for the density of the subset of primes with squarefree totients $p-1$ and a fixed primitive root $u$. More precisely, the subset of such primes is defined by 
	\begin{equation}
	\mathcal{P}_{*}=\{ p \in \mathbb{P}:\mu(p-1) \ne 0 \text{ and  } \ord_p(u)=p-1\}.
	\end{equation} 
	The corresponding counting function is defined by  
	\begin{equation}
	\pi_*(x)= \# \{ p \leq x:\mu(p-1) \ne 0 \text{ and } \ord_p(u)=p-1\}.
	\end{equation}
	An asymptotic formula for this function is derived here.  \\
	
	\begin{thm}  \label{thm11.1}
		Let $x \geq 1$ be a large number,and let $u \ne 0$ be a fixed integer. Then, 
		\begin{equation} \label{el1000}
		\pi_*(x)=c_u\li(x)+O \left (\frac{x}{\log^B x}\right ),
		\end{equation} 
		where the density $c_u \geq 0$ depends on $u$, and $B>1$ is an arbitrary constant.
	\end{thm}
	
	The density constant has the interesting form
	\begin{equation}\label{el1499}
	c_u=a_ua_0^2=a_u\prod_{p \geq 2} \left (1-\frac{1}{p(p-1)} \right )^2.
	\end{equation}
	The factor $a_u\geq 0$ depends only on the integer $u\ne 0$; a formula for it appears in  \cite[p.\ 218]{HC67}, \cite[p.\ 16]{MP04}, and similar references; the paper \cite{CI09} is devoted to the numerical analysis of the constant $a_0$. \\

	The proof is completed in the last section after the preliminary preparation in Sections.\\

\section{Evaluation Of The Main Term}
	Let $\mathbb{P}_{sf}=\{ p \in \mathbb{P}: \mu(p-1)\ne 0\}$ be the subset of primes with squarefree totients. The corresponding primes counting function
	\begin{equation}
	\pi_{sf}(x)=\# \{ p \leq x: \mu(p-1)\ne 0\}= a_0 \li(x)+O\left ( \frac{x}{\log^B x}\right ),
	\end{equation}
	where $a_0=\prod_{p \geq 2 } \left(1-1/p(p-1)\right)$, and $B>1$ is arbitrary, was derived in \cite{ML49}, in fact a more general result for $k$-free numbers $n=p+m$ is given. Similar analysis also appear in \cite[Theorem 11.22]{MV07}, and \cite[Theorem 2]{GM08}. \\
	
The Siegel-Walfisz theorem extends this result to arithmetic progressions $p=qn+a$ with squarefree $p-1=qn$, see \cite{MJ12}, \cite[p.\ 381]{MV07}, \cite[p.\ 124]{IK04}, \cite[p.\ 376]{TG15}, and similar references. The extended counting is
\begin{equation} \label{el1229}
\pi_{sf}(x,q,a)=  \# \{ p=qn+a \leq x: \mu(p-1) \ne 0 \}.
\end{equation}
The extended result is the followings. \\

\begin{thm} \label{thm11.2}
		Let $x \geq 1$ be a large number, and let $C\geq 0$ be a constant. Then, 
		\begin{equation} \label{el1230}
		\pi_{sf}(x,q,a)= a_0 \frac{\li(x)} {\varphi(q)}+O\left ( \frac{x}{\log^B x}\right ),
		\end{equation}
		where $\gcd(a,q)=1$ and $q=O(\log^C x)$, with $B> C+1\geq 1$ arbitrary.
	\end{thm}
	
	The previous result is used to obtain an asymptotic formula for the average order of normalized Euler function over the squarefree totients.\\
	
	\begin{thm} \label{thm11.3}
		Let $x \geq 1$ be a large number, and let $B >$ be a constant. Then, 
		\begin{equation} \label{el1232}
		\sum_{p\leq x} \frac{\varphi(p-1)}{p-1}\mu(p-1)^2=a_0^2 \li(x)+O\left ( \frac{x}{\log^B x}\right ),
		\end{equation}
		where $a_0=\prod_{p \geq 2 } \left(1-1/p(p-1)\right)$ is the density constant.
	\end{thm}

\begin{proof}  Use the identity $\varphi(n)/n= \sum_{d|n} \mu(d)/d$, see \cite[p.\ 236]{HW08}, and reverse the order of summation to evaluate sum:
	\begin{eqnarray} \label{el1264}
	S_0&=&\sum_{p\leq x} \frac{\varphi(p-1)}{p-1}\mu(p-1)^2 \nonumber \\
	&=& \sum_{p\leq x}\mu(p-1)^2 \sum_{d\;|\;p-1} \frac{\mu(d)}{d} \nonumber \\
	&=& \sum_{d\leq x} \frac{\mu(d)}{d} \sum_{ p\leq x, \;d\;|\;p-1}\mu(p-1)^2  \\
	&=& \sum_{d\leq x} \frac{\mu(d) }{d}  \cdot \pi_{sf}(x,d,1) \nonumber .
	\end{eqnarray}
	An application of the Mirsky-Siegel-Walfisz theorem, see Theorem \ref{thm11.2}, requires a dyadic decomposition
	\begin{eqnarray} \label{el1274}
	S_0&=&\sum_{d\leq x} \frac{\mu(d) }{d}  \cdot \pi_{sf}(x,d,1) \nonumber\\
	&=& \sum_{d\leq \log^B x} \frac{\mu(d) }{d}  \cdot \pi_{sf}(x,d,1) +\sum_{\log^B x \leq d\leq x} \frac{\mu(d) }{d}  \cdot \pi_{sf}(x,d,1)\\
	&=& S_1+S_2 \nonumber ,
	\end{eqnarray}
	where $B>1$ are arbitrary constants. Use Theorem \ref{thm11.2} to complete the evaluation of the finite sum $S_1$:
	\begin{eqnarray} \label{el40506}
	S_1&=&\sum_{d\leq \log^B x} \frac{\mu(d)}{d}  \cdot \pi_{sf}(x,d,1) \nonumber\\
	& = &\sum_{d\leq \log^B x} \frac{\mu(d) }{d} \left ( \frac{a_0\text{li}(x)}{\varphi(d)} +O \left (\frac{x}{\log^B x}\right ) \right )\nonumber\\
	&=&a_0\text{li}(x) \sum_{d\leq \log^B x} \frac{\mu(d) }{d\varphi(d)} 
	+O \left (\frac{x}{\log^B x}\right ) \sum_{d\leq \log^B x} 
	\frac{1}{d}  \\
	&=& a_0  L(1, \mu \varphi) \text{li}(x)+O \left (\frac{x \log \log x}{\log^B x}\right ) \nonumber  .
	\end{eqnarray} 
	Here the canonical density series is  
	\begin{equation}
	L(s,\mu \varphi)=\sum_{n\geq 1 }  \frac{\mu(n)}{n^{s}\varphi(n)}=\prod_{p \geq 2} \left ( 1-\frac{1}{p^s(p-1)}\right )
	\end{equation}
	converges for $\Re e(s)\geq 1$. Thus, the sum $\sum_{n\leq z}  \mu(n)/n \varphi(n)=a_0+O(\log z/z)$, with $a_0=L(1, \mu \varphi)=\prod_{p \geq 2} \left ( 1-1/p(p-1)\right ).$\\
	
	The estimate for the finite sum $S_2$ uses Brun-Titchmarsh theorem, that is, \begin{equation}
	\pi(x,q,a) \leq \frac{(2+o(1))x}{\varphi(q) \log(x/q)} \leq \frac{3x}{\varphi(q) \log x},
	\end{equation} 
	for $1\leq q<x$, see \cite[p.\ 167]{IK04}, \cite[p.\ 157]{HG07}, \cite{MJ12}, and \cite[p.\ 83]{TG15}. Since $\pi_{sf}(x,q,a) \leq \pi(x,q,a)$, the estimated upper bound is
	\begin{eqnarray} \label{el40510}
	|S_2|&=&\left | \sum_{\log^B x \leq d \leq x} \frac{\mu(d) }{d}  \cdot \pi_{sf}(x,d,1) \right | \nonumber \\
	& \leq &\frac{3x}{\log x} \sum_{\log^B x \leq d \leq x} \frac{1}{d\varphi(d)} \nonumber\\
	&\ll&\frac{3x}{\log^{B+1} x} \sum_{\log^B x \leq d \leq x} \frac{1 }{\varphi(d)}\\
	&\leq& \frac{3x}{\log^{B} x}  \nonumber.
	\end{eqnarray}     
	The third line uses $1/d \leq 1/\log^Bx$, and fourth line uses
	\begin{equation}\label{k33}
	\sum_{n \leq x} \frac{1}{\varphi(n)} \ll \log x,
	\end{equation}
	 see \cite[p.\ 42]{MV07}, \cite[p.\ 362]{NW00}. Combine the expressions (\ref{el40506}) and (\ref{el40510}) to obtain
	\begin{eqnarray} \label{el1294}
	\sum_{d\leq x} \frac{\mu(d) }{d}  \cdot \pi_{sf}(x,d,1)
	&=& S_1+S_2 \nonumber \\
	&=& \left (a_0^2 \text{li}(x) +O\left (\frac{x}{\log^{B} x} \right ) \right ) +O\left (\frac{x}{\log^{B} x} \right )\\
	&=& a_0^2 \text{li}(x) +O\left (\frac{x}{\log^{B} x} \right ) \nonumber,
	\end{eqnarray}
	where $B>1$ are arbitrary constants.
\end{proof}
	
	Related analysis appears in \cite[Lemma 1]{SP69}.\\
	
	\begin{lem} \label{lem11.1}
		Let $x \geq 1$ be a large number, and let \(\mu(n)\) be the Mobius function. Then, 
		\begin{equation} \label{el1259}
		\sum_{p\leq x} \frac{\mu(p-1)^2}{p}\sum_{\gcd(n,p-1=1} 1 =a_0^2 \li(x)+O \left ( \frac{x}{\log^B x}\right ),
		\end{equation} 
		where the constant $a_0>0$ is the density of the squarefree totients $p-1$, and $B>1$ is arbitrary.
	\end{lem}

\begin{proof}   A routine rearrangement gives 
	\begin{eqnarray} \label{el1300}
	M_{sf}(x)&=&\sum_{ p\leq x} \frac{\mu(p-1)^2}{p}\sum_{\gcd(n,p-1)=1} 1 \nonumber \\
	&=&\sum_{ p\leq x} \frac{\varphi(p-1)}{p}\mu(p-1)^2\\
	&=&\sum_{p\leq x} \frac{\varphi(p-1)}{p-1}\mu(p-1)^2-\sum_{p\leq x} \frac{\varphi(p-1)}{p(p-1)}\mu(p-1)^2 \nonumber\\
	&=&T_0+T_1 \nonumber .
	\end{eqnarray} 
	The second finite sum has the trivial upper bound
	\begin{equation} \label{el1310}
	T_1=\sum_{p\leq x} \frac{\varphi(p-1)}{p(p-1)}\mu(p-1)^2 \leq \sum_{n\leq x} \frac{1}{n}=O(\log x).
	\end{equation} 
	Applying Theorem \ref{thm11.3}, and summing the expressions lead to
	\begin{eqnarray} \label{el1330}
	T_0+T_1
	&=& \left (a_0^2 \text{li}(x) +O\left (\frac{x}{\log^{B} x} \right ) \right ) +O\left (\frac{x}{\log^{B} x} \right )\\
	&=& a_0^2 \text{li}(x) +O\left (\frac{x}{\log^{B} x} \right ) \nonumber,
	\end{eqnarray}
	where $B>1$ are arbitrary constants.
\end{proof}

	\begin{thm} \label{thm11.4}
		Let $x \geq 1$ be a large number, and let $C\geq 0$ be a constant. Then,
		\begin{equation} \label{el1231}
		\sum_{\substack{p\leq x \\ p \equiv 1 \bmod q}} \frac{\varphi(p-1)}{p-1}\mu(p-1)^2=a_0^2 \frac{\li(x)} {\varphi(q)}+O\left ( \frac{x}{\log^B x}\right ),
		\end{equation}
		where $\gcd(a,q)=1$ and $q=O(\log^C x)$, with $B>C+1\geq 1$ arbitrary. 
	\end{thm}

\section{Estimates For The Error Term}
	The upper bounds of exponential sums over subsets of elements in finite rings $\left ( \mathbb{Z}/N \mathbb{Z} \right )^ \times$ can be  used to estimate the error term $E_{sf}(x)$ in the proof of Theorem \ref{thm11.1}. The estimate in Lemma \ref{lem11.2} is derived from any one of the Theorem \ref{thm4.1} or \ref{thm4.2} or \ref{thm4.3}. \\
	
	\begin{lem} \label{lem11.2}
		Let \(p\geq 2\) be a large prime, let \(\psi \neq 1\) be an additive character, and let \(\tau\) be a primitive root mod \(p\). If the element \(u\ne 0\) is not a primitive root, then, 
		\begin{equation} \label{el8400}
		\sum_{x \leq p\leq 2x}
		\frac{\mu(p-1)^2}{p}\sum_{\gcd(n,p-1)=1,} \sum_{ 0<k\leq p-1} \psi \left((\tau ^n-u)k\right)\ll x^{1-\varepsilon}
		\end{equation} 
		for all sufficiently large numbers $x\geq 1$ and an arbitrarily small number \(\varepsilon >0\).
	\end{lem}
	
\begin{proof}    Let $\psi(z)=e^{i 2 \pi kz/p}$ with $0< k<p$. By hypothesis $u \ne \tau^n$ for any $n \geq 1$ such that $\gcd(n,p-1)=1$. Therefore, $\sum_{ 0<k\leq p-1} e^{i 2 \pi(\tau^n- u)k/p}=-1$, and the error term has the trivial upper bound
	\begin{eqnarray} \label{el5401}
	|E_{sf}|&=& \left |\sum_{x \leq p\leq 2x}
	\frac{\mu(p-1)^2}{p}\sum_{\gcd(n,p-1)=1,} \sum_{ 0<k\leq p-1} \psi \left((\tau ^n-u)k\right) \right | \nonumber \\
	&\leq&\sum_{x \leq p\leq 2x}
	\frac{\varphi(p-1)}{p} \\ 
	&\leq& \frac{1}{2}\frac{x}{\log x}+O\left (\frac{x}{\log^2 x}\right ) \nonumber. 
	\end{eqnarray} 
	The third line uses $\varphi(n)/n \leq 1/2$This information implies that the error term is smaller than the main term. Thus, there is a nontrivial upper bound. To sharpen this upper bound, rearrange the triple finite sum in the form
	\begin{eqnarray} \label{el5901}
	E_{sf}(x)&=&\sum_{x \leq p\leq 2x}\frac{\mu(p-1)^2}{p} \sum_{ 0<k\leq p-1,}  \sum_{\gcd(n,p-1)=1} \psi ((\tau ^n-u)k) \\ 
	&=&  \sum_{x \leq p\leq 2x}
	\left (\frac{\mu(p-1)}{p^{1/2}}\sum_{ 0<k\leq p-1} e^{-i 2 \pi uk/p} \right ) \left ( \frac{\mu(p-1)}{p^{1/2}} \sum_{\gcd(n,p-1)=1} e^{i 2 \pi k\tau ^n/p} \right )\nonumber.
	\end{eqnarray} 
	And let
	\begin{equation} \label{el5810}
	U_p=\frac{\mu(p-1)}{p^{1/2}}\sum_{ 0<k\leq p-1} e^{-i2 \pi uk/p}    \qquad \text{ and } \qquad V_p=\frac{\mu(p-1)}{p^{1/2}}\sum_{\gcd(n,p-1)=1} e^{i2 \pi k\tau ^n/p} .
	\end{equation} 
	Now consider the Holder inequality $|| AB||_1 \leq || A||_{r} \cdot || B||_s $ with $1/r+1/s=1$. In terms of the components in (\ref{el5810}) this inequality has the explicit form
	\begin{equation} \label{el5902}
	\sum_{x \leq p\leq 2x} | U_p V_p|\leq \left ( \sum_{x \leq p\leq 2x} |U_p|^r \right )^{1/r} \left ( \sum_{x \leq p\leq 2x} |V_p|^s \right )^{1/s} .
	\end{equation} 
	
	The absolute value of the first exponential sum $U_p$ is given by
	\begin{equation} \label{el5903}
	| U_p |= \left |\frac{\mu(p-1)}{p^{1/2}} \sum_{ 0<k\leq p-1} e^{-i2 \pi uk/p} \right | \leq\frac{1}{p^{1/2}}  .
	\end{equation} 
	This follows from $\sum_{ 0<k\leq p-1} e^{i 2 \pi uk/p}=-1$ for $u\ne0$. The corresponding $r$-norm $||U_p||_r^r =  \sum_{p\leq x} |U_p|^r$ has the upper bound
	\begin{equation}\label{el5906}
	\sum_{x \leq p\leq 2x} |U_p|^r=\sum_{x \leq p\leq 2x}  \left |\frac{1}{p^{1/2}} \right |^r \leq x^{1-r/2}. 
	\end{equation}
	Here the finite sum over the primes is estimated using integral
	\begin{equation} \label{el5907}
	\sum_{x \leq p\leq 2x} \frac{1}{p^{r/2}}
	\ll\int_{x}^{2x} \frac{1}{t^{r/2}} d \pi(t)=O(x^{1-r/2}),
	\end{equation}
	where \(\pi(x)=x/\log x+O(x/\log ^{2} x)\) is the prime counting measure. \\
	
	The absolute value of the second exponential sum $V_p=V_p(k)$ has the upper bound
	\begin{equation} \label{el5943}
	|V_p|= \left |\frac{\mu(p-1)}{p^{1/2}}\sum_{\gcd(n,p-1)=1} e^{i2 \pi k\tau ^n} \right |\ll p^{1/2-\varepsilon} .
	\end{equation} 
	This exponential sum dependents on $k$; but it has a uniform, and independent of $k$ upper bound
	\begin{equation} \label{el88978}
	\max_{1\leq k \leq p-1}   \left |   \sum_{\gcd(n,p-1)=1,} e^{i 2 \pi k\tau ^n/p} \right | \ll  p^{1-\varepsilon} ,
	\end{equation}   
	where \(\varepsilon >0\) is an arbitrarily small number, see Lemma \ref{lem4.2}.  \\
	
	The corresponding $s$-norm $||V_p||_s^s =  \sum_{x \leq p\leq 2x} |V_p|^s$ has the upper bound
	\begin{equation} \label{el5980}
	\sum_{x \leq p\leq 2x} |V_p|^s \leq\sum_{x \leq p\leq 2x}  \left |p^{1/2-\varepsilon} \right |^s \ll x^{1+s/2-\varepsilon s}. 
	\end{equation}
	The finite sum over the primes is estimated using integral
	\begin{equation} \label{el5990}
	\sum_{x \leq p\leq 2x} p^{s/2-\varepsilon s}
	\ll\int_{x}^{2x} t^{s/2-\varepsilon s} d \pi(t)=O(x^{1+s/2-\varepsilon s}).
	\end{equation}
	
	Now, replace the estimates (\ref{el5906}) and (\ref{el5980}) into (\ref{el5902}), the Holder inequality, to reach
	\begin{eqnarray} \label{el5991}
	\sum_{x \leq p\leq 2x} | U_p V_p| 
	&\leq & 
	\left ( \sum_{x \leq p\leq 2x} |U_p|^r \right )^{1/r} \left ( \sum_{x \leq p\leq 2x} |V_p|^s \right )^{1/s} \nonumber \\
	&\ll & \left ( x^{1-r/2}  \right )^{1/r} \left ( x^{1+s/2-\varepsilon s}  \right )^{1/s} \nonumber \\
	&\ll & \left ( x^{1/r-1/2}  \right )  \left ( x^{1/s+1/2-\varepsilon}  \right ) \\
	&\ll &  x^{1/r+1/s-\varepsilon}  \nonumber \\
	&\ll &  x^{1-\varepsilon} \nonumber .
	\end{eqnarray}
	
	Note that this result is independent of the parameter $1/r+1/s=1$.  \end{proof}

\section{Squarefree Totients $p-1$ and Fixed Primitive Root}
	Let $p \in \mathbb{P}=\{2,3,5, \ldots \}$ be a large prime, and let $\mu: \mathbb{N} \longrightarrow \{-1,0,1\}$ be the Mobius function. The weighted characteristic function for primitive roots in a 
	finite field \(u\in \mathbb{F}_p\), with squarefree totient $p-1$ satisfies the relation
	\begin{equation}
	\mu(p-1)^2\Psi(u)=
	\left \{\begin{array}{ll}
	1 & \text{ if } \mu(p-1) \ne 0 \text{ and ord}_p(u)=p-1,  \\
	0 & \text{ if } \mu(p-1)=0 \text{ or ord}_p(u)\neq p-1. \\
	\end{array} \right.
	\end{equation} 
	The definition of the characteristic function $\Psi (u)$ is given in Lemma \ref{lem3.2}, and the Mobius function is defined by $\mu(n)=\pm 1$ if $n \geq 1$ is a 
	squarefree integer, otherwise it vanishes.\\ 
	
\begin{proof}   (Theorem 11.1.) 	Let $x_0 \geq 1$ be a large constant and suppose that \(u\neq \pm 1,v^2\) is not a primitive root for all primes 
	\(p \geq x_0\) such that $\mu(p-1) \ne 0$. Let \(x>x_0\) be a large number, and consider the sum of the weighted characteristic function over the short interval \([x,2x]\), that is, 
	\begin{equation} \label{el20703}
	0=\sum _{x \leq p\leq 2x} \mu(p-1)^2\Psi (u).
	\end{equation}
	Replacing the characteristic function, Lemma \ref{lem3.2}, and expanding the nonexistence equation (\ref{el20703}) yield
	\begin{eqnarray} \label{el20704}
	&0=&\sum _{x \leq p\leq 2x} \mu(p-1)^2\Psi (u)\\
	&=&\sum_{x \leq p\leq 2x}    \mu(p-1)^2 \left ( \frac{1}{p}\sum_{\gcd(n,p-1)=1,} \sum_{ 0\leq k\leq p-1} \psi \left((\tau^n-u)k\right)  \right )\nonumber\\
	&=&a_u\sum_{x \leq p\leq 2x} \frac{\mu(p-1)^2}{p}\sum_{\gcd(n,p-1)=1} 1 \nonumber \\
&& + \sum_{x \leq p\leq 2x}
	\frac{\mu(p-1)^2}{p}\sum_{\gcd(n,p-1)=1,} \sum_{ 0<k\leq p-1} \psi \left((\tau^n-u)k\right) \nonumber\\
	&=&M_{sf}(x) + E_{sf}(x) \nonumber,
	\end{eqnarray} 
	where $a_u > 0$ is a constant depending on $u$ and the density of squarefree totients $p-1$. \\
	
	The main term $M_{sf}(x)$ is determined by a finite sum over the trivial additive character \(\psi =1\), and the error term $E_{sf}(x)$ is determined by a finite sum over the nontrivial additive characters \(\psi =e^{i 2\pi  k/p}\neq 1\).\\
	
	Let $B>2$ be an arbitrary constant. Applying Lemma \ref{lem11.1}
	to the main term, and Lemma \ref{lem11.2} to the error term yield
	\begin{eqnarray} \label{el20715}
	\sum _{x \leq p\leq 2x} \mu(p-1)^2\Psi (u)
	&=&M_{sf}(x) + E_{sf}(x) \nonumber \\
	&= & c_u  \left ( \text{li}(2x)- \text{li}(x)+O \left (\frac{x}{\log^B x}\right ) \right )+O(x^{1-\varepsilon})\\
	&=& c_u \text{li}(x)+O \left (\frac{x}{\log^B x}\right ) \nonumber\\
	&>&0,
	\end{eqnarray} 
	where the constant $c_u=a_ua_0^2 > 0$. \\
	
	Clearly, this contradicts the hypothesis  (\ref{el20703}) for all sufficiently large number $x >x_0$. Ergo, the number of primes $p\leq x $ in the short interval  $[x,2x]$ with squarefree totient $p-1$ and a fixed primitive root $u \ne \pm 1, v^2$ is unbounded as $x \to \infty$. In particular, the number of such primes has the asymptotic formula
	\begin{equation} \label{el20725}
	\pi_*(x)=c_u \li(x)+O \left (\frac{x}{\log^B x}\right ).
	\end{equation}
	This completes the verification. \end{proof}
	
	The logarithm integral difference $\li(2x)-\li(x)=x/\log x+O\left(x/\log^2x \right)=\li(x)$ is evaluated via (\ref{el8100}) in Section 8.3. A formula for computing the density constant $a_{u}$ appears in \cite[p.\ 218]{HC67}, \cite[p.\ 16]{MP04}, and similar references. \\
	
\newpage
\section{Problems}
\begin{exe} \normalfont Prove Theorem \ref{thm11.5}.
\end{exe}

	\begin{exe} \normalfont Generalize Theorem \ref{thm11.1} to the subset of squarefree integers $\mathcal{Q}=\{n \in \mathbb{N}:\mu(n)\ne 0\}$.
	\end{exe}
	
	\begin{exe} \normalfont Compute the density $c_u$ for the subset of squarefree totients $p-1$ with a fixed primitive root $u=2$ in the set of primes $\mathcal{P}=\{2,3,5,7,11, \ldots\}$.
	\end{exe}

%ccccccccccccccccccccccccccccccccccccccc	
\chapter{Densities For Primes In Arithmetic Progressions} \label{c12}

The basic result for the densities of primes with fixed primitive roots is generalized to primes in arithmetic progressions.

\section{Description of the Result}
Given a fixed integer $u\ne \pm1, v^2$, the precise primes counting function is defined by 
\begin{equation}\label{el8880}
\pi_{u}(x,q,a)=\#\left\{ p\leq x:p\equiv a \tmod q \text{ and } \ord_p(u)=p-1 \right\}
\end{equation}
for $1 \leq a <q$ and $\gcd(a,q)=1$.  The limit
\begin{equation} \label{el8888}
\delta(u,q,a)=\lim_{x \to \infty} \frac{\pi_{u}(x,q,a)}{\pi(x,q,a)}=a_{u} \frac{A_q}{\varphi(q)}
\end{equation}
is the density of the subset of primes with a fixed primitive root $\ell\ne \pm 1, v^2$.

\begin{thm} \label{thm12.1} Let $1 \leq a <q$ be integer such that $\gcd(a,q)=1$ and $q=O(\log^C)$, with $C\geq 0$ constant. Then, there are infinitely many primes $p=qn+a$ such that $u$ is a fixed primitive root modulo $p$. In addition, the counting function for these primes satisfies
\begin{equation} \label{e7l3}
\pi _{u}(x,q,a)= \delta(u,q,a) \frac{x}{\log x}+O\left( \frac{x}{\log^B x} \right),
\end{equation}
where \(\delta(u,q,a) \geq 0\) is the density constant depending on the integers $u,q$ and $a$, with $B>C+1$ constant, for all large numbers \(x\geq 1\).
\end{thm}

\section{Main Term For Arithmetic Progressions}	
\begin{lem} {\normalfont (\cite[Lemma 5]{MP99})} \label{lem12.1}
Let \(x\geq 1\) be a large number, and let \(\varphi (n)\) be the Euler totient function.
If \(q\leq \log^C x\), with $C\geq 0$ constant, an integer $1\leq a< q$ such that $\gcd(a,q)=1$, then
\begin{equation}
\sum_{\substack{p\leq x  \\ p \equiv a \bmod q} }\frac{\varphi(p-1)}{p-1}
=A_q\frac{\li(x)}{\varphi(q)}
+
O\left(\frac{x}{\log ^Bx}\right) ,
\end{equation}
where \(\li(x)\) is the logarithm integral, and \(B> C+1\geq1\) is an arbitrary constant, as \(x \rightarrow \infty\), and 
\begin{equation} \label{80000}
A_q=\prod_{p |\gcd(a-1,q) } \left(1-\frac{1}{p}\right)\prod_{p \nmid q } \left(1-\frac{1}{p(p-1)}\right).
\end{equation} 
\end{lem}

Related discussions for $q=2$ are given in \cite[Lemma 1]{SP69}, \cite[p.\ 16]{MP04}, and, \cite{VR73}. The error term can be reduced to the same order of magnitude as the error term in the prime number theorem 
\begin{equation} \label{el1377}
\pi(x,q,a)=\frac{\text{li}(x)}{\varphi(q)}+O\left(e^{-c \sqrt{\log x}} \right ),
\end{equation} 
where $c>0$ is an absolute constant. But the simpler notation will be used here. 
The case \(q=2\) is ubiquitous in various results in Number Theory.  \\

\begin{lem} \label{lem12.2}
Let \(x\geq 1\) be a large number, and let \(\varphi (n)\) be the Euler totient function.
If \(q\leq \log^C x\), with $C\geq 0$ constant, an integer $1\leq a< q$ such that $\gcd(a,q)=1$, then
\begin{equation} \label{el8859}
\sum_{\substack{p\leq x \\ p \equiv a \bmod q}} \frac{1}{p}\sum_{\gcd(n,p-1)=1} 1=A_q \frac{\li(x)}{\varphi(q)}+O\left(\frac{x}{\log
	^Bx}\right) ,
\end{equation} 
where \(\li(x)\) is the logarithm integral, and \(B> C+1\geq1\) is an arbitrary constant, as \(x \rightarrow \infty\), and $A_q$ is defined in (\ref{80000}). 
\end{lem}

\begin{proof}  A routine rearrangement gives 
\begin{eqnarray} \label{el500}
\sum_{\substack{p\leq x \\ p \equiv a \bmod q}} \frac{1}{p}\sum_{\gcd(n,p-1)=1} 1&=&\sum_{\substack{p\leq x \\ p \equiv a \bmod q}} \frac{\varphi(p-1)}{p}\\
&=&\sum_{\substack{p\leq x \\ p \equiv a \bmod q}} \frac{\varphi(p-1)}{p-1}-\sum_{\substack{p\leq x \\ p \equiv a \bmod q}} \frac{\varphi(p-1)}{p(p-1)} \nonumber.
\end{eqnarray} 
To proceed, apply Lemma \ref{lem12.1} to reach
\begin{eqnarray} \label{el501}
\sum_{\substack{p\leq x \\ p \equiv a \bmod q}}\frac{\varphi(p-1)}{p-1}-\sum_{\substack{p\leq x \\ p \equiv a \bmod q}} \frac{\varphi(p-1)}{p(p-1)}
&=&A_q \frac{\li(x)}{\varphi(q)}+O \left (\frac{x}{\log^Bx}\right )-\sum_{\substack{p\leq x \\ p \equiv a \bmod q}} \frac{\varphi(p-1)}{p(p-1)} \nonumber\\
&=&A_q \frac{\li(x)}{\varphi(q)}+O \left (\frac{x}{\log^Bx}\right ),
\end{eqnarray} 
where the second finite sum 
\begin{equation}
\sum_{\substack{p\leq x \\ p \equiv a \bmod q}} \frac{\varphi(p-1)}{p(p-1)} \ll \log \log x
\end{equation} 
is absorbed into the error term, $B>C+1\geq 1$ is an arbitrary constant, and $A_q$ is defined in (\ref{80000}).   \end{proof}

The logarithm integral can be estimated via the asymptotic formula
\begin{equation} \label{el12100}
\li(x)=\int_{2}^{x} \frac{1}{\log z}  d z=\frac{x}{\log x}+O\left (\frac{x}{\log^2 x} \right ).
\end{equation}\\ 

The logarithm integral difference $\li(2x)-\li(x)=x/\log x+O\left(x/\log^2x \right)=\li(x)$ is used in various calculations, see \cite[p.\ 102]{AP98}, or similar reference.

\section{Estimates For The Error Term}
The upper bounds of exponential sums over subsets of elements in finite rings $\left (\mathbb{Z}/N\mathbb{Z}\right )^\times$ stated in Chapter \ref{c4} are used to estimate the error term $E_{q}(x)$ in the proof of Theorem \ref{thm12.1}. Two estimates will be considered. The first one in Lemma \ref{lem12.3} is based on the weaker upper bound in Lemma \ref{lem4.1}; and the second estimate in Lemma \ref{lem12.4} is based the nontrivial result in Lemma \ref{lem4.3}.\\

\subsection{Elementary Error Term}	
This is a simpler estimate but it is restricted to certain residue classes $a \bmod q$ for which $\gcd(a-1,q)$ has one or two primes divisors, refer to (\ref{80000}) for the exact density formula. Otherwise, the error term is the same order of magnitude as the main term.\\

\begin{lem}\label{lem12.3}
Let \(x\geq 1\) be a large number, and let $q= \prod_{r \leq \log \log x}r \asymp \log x$. Define the subset of primes
	\( \mathcal{P}=\{p=qn+a :n\geq 1\}\). Let \(\psi \neq 1\) be an additive character, and let \(\tau\) be a primitive root mod \(p\). If the element \(u\ne 0\) is not a primitive root, then, 
	\begin{equation} \label{el8800}
		\sum_{\substack{x \leq p\leq 2x\\ p\equiv a \bmod q}}
		\frac{1}{p}\sum_{\gcd(n,p-1)=1,} \sum_{ 0<k\leq p-1} \psi \left((\tau ^n-u)k\right)\ll  \frac{1}{\varphi(q)}\frac{x}{(\log \log \log x)\log x} 
	\end{equation} 
	for all sufficiently large numbers $x\geq 1$ and $p\in \mathcal{P}$.
\end{lem}

\begin{proof}    Let $\psi(z)=e^{i 2 \pi kz/p}$ with $0< k<p$. By hypothesis $u \ne \tau^n$ for any $n \geq 1$ such that $\gcd(n,p-1)=1$. Therefore, $\sum_{1 \leq k <p}e^{i2\pi (\tau ^n-u)k/p}=-1$ and the error term has the trivial upper bound
\begin{eqnarray} \label{el8001}
	|E_{q}(x)| &= & \left |\sum_{\substack{x \leq p\leq 2x\\ p\equiv a \bmod q}}
	\frac{1}{p}\sum_{\gcd(n,p-1)=1,} \sum_{ 0<k\leq p-1} e^{i2\pi (\tau ^n-u)k/p} \right | \nonumber \\
	&\leq &\sum_{\substack{x \leq p\leq 2x\\ p\equiv 1 \bmod q}}
	\frac{\varphi(p-1)}{p} \\ &\ll &\frac{1}{\varphi(q)}\frac{x}{\log x} \nonumber,
\end{eqnarray} 
where $\varphi(n)/n \leq 1/2$ for all $n \geq 1$. This information implies that the error term $E_{\ell}(x)$ is smaller than the main term $M_{q}(x)$, see (\ref{el8730}). Thus, there is a nontrivial upper bound. To sharpen this upper bound, rearrange the triple finite sum in the following way:
\begin{eqnarray} \label{el8003}
	E_{q}(x)  &=&\sum_{\substack{x \leq p\leq 2x\\ p\equiv a \bmod q}}\frac{1}{p} \sum_{ 0<k\leq p-1,}  \sum_{\gcd(n,p-1)=1} e^{i2\pi (\tau ^n-u)k/p} \nonumber \\
	&=&  \sum_{\substack{x \leq p\leq 2x\\ p\equiv 1 \bmod q}}
	\left (\sum_{ 0<k\leq p-1} e^{-i 2 \pi uk/p} \right ) \left ( \frac{1}{p} \sum_{\gcd(n,p-1)=1} e^{i 2 \pi k\tau ^n/p} \right ) .
\end{eqnarray} 
And let
\begin{equation} \label{el8005}
	U_p=\sum_{ 0<k\leq p-1} e^{-i2 \pi uk/p}    \qquad \text{ and } \qquad V_p=\frac{1}{p}\sum_{\gcd(n,p-1)=1} e^{i2 \pi k\tau ^n/p} .
\end{equation} 
In this application, the Holder inequality $|| AB||_1 \leq || A||_{\infty}  \cdot || B||_1 $ takes the form
\begin{equation} \label{el8013}
	\sum_{\substack{x \leq p\leq 2x\\ p\equiv a \bmod q}} | U_p V_p|\leq \max_{x \leq p \leq 2x} |U_p| \cdot \sum_{x \leq p\leq 2x} |V_p| .
\end{equation} 
The supremum norm $|| U_p||_{\infty}$ of the exponential sum $U_p$ is given by
\begin{equation} \label{el8026}
	\max_{\substack{x \leq p\leq 2x\\ p\equiv a \bmod q}}| U_p |= \max_{\substack{x \leq p\leq 2x\\ p\equiv a \bmod q}} \left | \sum_{ 0<k\leq p-1} e^{-i2 \pi uk/p} \right | =1.
\end{equation} 
This follows from $\sum_{ 0<k\leq p-1} e^{i 2 \pi uk/p}=-1$ for $u \ne 0$. \\

By Lemma \ref{lem4.1}, the absolute value of the exponential sum $V_p=V_p(k)$ has the trivial upper bound
\begin{eqnarray} \label{el8039}
	|V_p|&=& \left |\frac{1}{p}\sum_{\gcd(n,p-1)=1} e^{i2 \pi k\tau ^n} \right | \nonumber \\
	&\leq& \frac{1}{p} \varphi(p-1) \\
	&\ll& \frac{1}{p} \cdot \frac{p}{\log \log \log p} \nonumber \\
	&=&\frac{1}{\log \log \log p}  \nonumber.
\end{eqnarray} 
By the Brun-Titchmarsh theorem, the number of primes $p=qn+a$ in the interval $[x,2x]$ satisfies
\begin{equation}
	\pi(2x,q,a)-\pi(x,q,a) \ll \frac{1}{\varphi(q)}\frac{x}{ \log x},
\end{equation}
see \cite[p.\  167]{IK04}, \cite[p.\  157]{HG07}, \cite{MJ12}, and \cite[p.\  83]{TG15}. Thus, the corresponding $1$-norm
$||V_p||_1 =  \sum_{x \leq p \leq 2x} |V_p|$ has the upper bound
\begin{eqnarray} \label{el8041}
	\sum_{\substack{x \leq p\leq 2x\\ p\equiv a \bmod q}} |V_p| & =& \sum_{\substack{x \leq p\leq 2x\\ p\equiv a \bmod q}}  \frac{1}{\log \log \log p}   \nonumber \\
	& \leq&  \frac{1}{\log \log \log x} \sum_{\substack{x \leq p\leq 2x\\ p\equiv a \bmod q}} 1   \\
	& \ll & \frac{1}{\varphi(q)}\frac{x}{(\log \log \log x)\log x}\nonumber . 
\end{eqnarray}

Now, replace the estimates (\ref{el8039}) and (\ref{el8041}) into (\ref{el8013}), the Holder inequality, to reach
\begin{eqnarray} \label{el8080}
	|E_{q}(x)| &\leq & \sum_{\substack{x \leq p\leq 2x\\ p\equiv a \bmod q}}
	\left | U_p V_p \right | \nonumber \\
	&\leq &	\max_{\substack{x \leq p\leq 2x\\ p\equiv a \bmod q}} |U_p| \cdot \sum_{\substack{x \leq p\leq 2x\\ p\equiv 1 \bmod q}} |V_p|  \\
	&\ll & 1 \cdot \frac{1}{\varphi(q)} \frac{x}{(\log \log \log x)\log x} \nonumber \\
	&\ll &  \frac{1}{\varphi(q)}\frac{x}{(\log \log \log x)\log x} \nonumber .
\end{eqnarray}

This completes the verification.  \end{proof}

\subsection{Sharp Error Term} An application of the Lemma \ref{lem4.3} leads to a sharp result, this is completed below.\\

\begin{lem} \label{lem12.4} Let \(p\geq 2\) be a large prime, let \(\psi \neq 1\) be an additive character, and let \(\tau\) be a primitive root mod \(p\). If the element \(u\ne 0\) is not a primitive root, then, 
\begin{equation} \label{el88400}
\sum_{\substack{x \leq p\leq 2x \\ p \equiv a \bmod q}}
\frac{1}{p}\sum_{\gcd(n,p-1)=1,} \sum_{ 0<k\leq p-1} \psi \left((\tau ^n-u)k\right)\ll  \frac{1}{\varphi(q)}\frac{x^{1-\varepsilon }}{\log x},
\end{equation} 
where $1 \leq a <q, \, \gcd(a,q)=1$ and $O(\log^C x)$ with $C>0$ constant, for all sufficiently large numbers $x\geq 1$ and an arbitrarily small number \(\varepsilon >0\).
\end{lem}

\begin{proof}   By hypothesis $u \ne \tau^n$ for any $n \geq 1$ such that $\gcd(n,p-1)=1$. Therefore,
\begin{eqnarray} \label{el88401}
\sum_{\substack{x \leq p\leq 2x \\ p \equiv a \bmod q}}
\frac{1}{p}\sum_{\gcd(n,p-1)=1,} \sum_{ 0<k\leq p-1} \psi \left((\tau ^n-u)k\right)
&\leq &\sum_{\substack{x \leq p\leq 2x \\ p \equiv a \bmod q}}
\frac{\varphi(p-1)}{p} \nonumber\\
&\leq & \frac{1}{2}\sum_{\substack{x \leq p\leq 2x \\ p \equiv a \bmod q}}
1\\
&\leq &\frac{1}{2} \frac{3}{\varphi(q)}\frac{x}{\log x}+O\left (\frac{x}{\log^2 x}\right )   \nonumber,
\end{eqnarray} 
where $\varphi(n)/n \leq 1/2$ for any integer $n \geq 1$. This implies that there is a nontrivial upper bound. To sharpen this upper bound, let $\psi(z)=e^{i 2 \pi kz/p}$ with $0< k<p$, and rearrange the triple finite sum in the form
\begin{eqnarray} \label{e88901}
E(x) &=&   \sum_{\substack{x \leq p\leq 2x \\ p \equiv a \bmod q}}\frac{1}{p} \sum_{ 0<k\leq p-1,}  \sum_{\gcd(n,p-1)=1} \psi ((\tau ^n-u)k)  
\nonumber \\ &=&  \sum_{\substack{x \leq p\leq 2x \\ p \equiv a \bmod q}}
 \left (\frac{1}{p^{1/2}}\sum_{ 0<k\leq p-1} e^{-i 2 \pi uk/p}  \right ) \left (\frac{1}{p^{1/2}}\sum_{\gcd(n,p-1)=1} e^{i 2 \pi k\tau ^n/p} \right ),
\end{eqnarray}  
and let
\begin{equation} \label{el88810}
U_p=\max_{1 \leq u<p}\frac{1}{p^{1/2}}\sum_{ 0<k\leq p-1} e^{-i2 \pi uk/p}    \qquad \text{ and } \qquad V_p=\max_{1 \leq k<p}\frac{1}{p^{1/2}}\sum_{\gcd(n,p-1)=1} e^{i2 \pi k\tau ^n/p} .
\end{equation}

Then,
\begin{equation} \label{el88901}
|E(x)| \leq \sum_{\substack{x \leq p\leq 2x \\ p \equiv a \bmod q}} | U_p V_p|.
\end{equation}  
Now consider the Holder inequality $|| AB||_1 \leq || A||_{r} \cdot || B||_s $ with $1/r+1/s=1$. In terms of the components in (\ref{el88810}) this inequality has the explicit form
\begin{equation} \label{el88902}
\sum_{\substack{x \leq p\leq 2x \\ p \equiv a \bmod q}} | U_p V_p|\leq \left ( \sum_{\substack{x \leq p\leq 2x \\ p \equiv a \bmod q}} |U_p|^r \right )^{1/r} \left ( \sum_{\substack{x \leq p\leq 2x \\ p \equiv a \bmod q}} |V_p|^s \right )^{1/s} .
\end{equation} 

The absolute value of the first exponential sum $U_p=U_p(u)$ is given by
\begin{equation} \label{el88903}
| U_p |= \left |\frac{1}{p^{1/2}} \sum_{ 0<k\leq p-1} e^{-i2 \pi uk/p} \right | =\frac{1}{p^{1/2}}  .
\end{equation} 
This follows from $\sum_{ 0<k\leq p-1} e^{i 2 \pi uk/p}=-1$ for $u \ne 0$ and summation of the geometric series. The corresponding $r$-norm has the upper bound
\begin{eqnarray}\label{el88906}
||U_p||_r^r &= &\sum_{\substack{x \leq p\leq 2x \\ p \equiv a \bmod q}} |U_p|^r \nonumber \\ &=&\sum_{\substack{x \leq p\leq 2x \\ p \equiv a \bmod q}}  
\left |\frac{1}{p^{1/2}} \right |^r \nonumber \\
&\leq &\frac{1}{x^{r/2}}\sum_{\substack{x \leq p\leq 2x \\ p \equiv a \bmod q}}  1 \nonumber \\
&\leq &\frac{3}{\varphi(q)}\frac{x^{1-r/2}}{\log x}. 
\end{eqnarray}
Here the finite sum over the primes is estimated using the Brun-Titchmarsh theorem; this result states that the number of primes $p=qn+1$ in the interval $[x,2x]$ satisfies the inequality
\begin{equation} \label{el8040}
\pi(2x,q,a)-\pi(x,q,a) \leq \frac{3}{\varphi(q)}\frac{x}{ \log x},
\end{equation}
see \cite[p.\  167]{IK04}, \cite[p.\  157]{HG07}, \cite{MJ12}, and \cite[p.\  83]{TG15}. \\

The absolute value of the second exponential sum $V_p=V_p(k)$ has the upper bound
\begin{equation} \label{el88943}
|V_p|= \left |\frac{1}{p^{1/2}}\sum_{\gcd(n,p-1)=1} e^{i2 \pi k\tau ^n} \right |\ll p^{1/2-\varepsilon} .
\end{equation} 
This exponential sum dependents on $k$; but it has a uniform, and independent of $k$ upper bound
\begin{equation} \label{el88979}
\max_{1\leq k \leq p-1}   \left |   \sum_{\gcd(n,p-1)=1,} e^{i 2 \pi k\tau ^n/p} \right | \ll  p^{1-\varepsilon} ,
\end{equation}   
where \(\varepsilon >0\) is an arbitrarily small number, see Lemma \ref{lem4.2}. A similar application appears in \cite[p.\ 1286]{MT10}.\\

In light of this estimate, the corresponding $s$-norm has the upper bound
\begin{eqnarray}\label{el88980}
||V_p||_s^s &= &\sum_{\substack{x \leq p\leq 2x \\ p \equiv a \bmod q}} |V_p|^s \nonumber \\ &\ll&\sum_{\substack{x \leq p\leq 2x \\ p \equiv a \bmod q}}  
\left |p^{1/2-\varepsilon} \right |^s \nonumber \\
&\ll &\left ((2x)^{1/2-\varepsilon} \right )^s\sum_{\substack{x \leq p\leq 2x \\ p \equiv a \bmod q}}  1 \nonumber \\
&\ll &\frac{1}{\varphi(q)}\frac{x^{1+s/2-\varepsilon s}}{\log x} . 
\end{eqnarray}
The finite sum over the primes is estimated using the Brun-Titchmarsh theorem, and $2^{1/s+1/2-\varepsilon } \leq	 2$ for $s \geq 1$.\\

Now, replace the estimates (\ref{el88906}) and (\ref{el88980}) into (\ref{el88902}), the Holder inequality, to reach
\begin{eqnarray} \label{el88991}
\sum_{\substack{x \leq p\leq 2x \\ p \equiv a \bmod q}}
\left | U_pV_p \right | 
&\leq & 
\left ( \sum_{\substack{x \leq p\leq 2x \\ p \equiv a \bmod q}} |U_p|^r \right )^{1/r} \left ( \sum_{\substack{x \leq p\leq 2x \\ p \equiv a \bmod q}} |V_p|^s \right )^{1/s} \nonumber \\
&\ll & \left ( \frac{1}{\varphi(q)}\frac{x^{1-r/2}}{\log x}   \right )^{1/r} \left ( \frac{1}{\varphi(q)}\frac{x^{1+s/2-\varepsilon s}}{\log x}   \right )^{1/s} \\
&\ll & \frac{1}{\varphi(q)}\frac{x^{1-\varepsilon }}{\log x} \nonumber .
\end{eqnarray}

The estimate is independent of the parameters $1/r+1/s=1$.  \end{proof}

\section{Result For Arithmetic Progressions}

\begin{proof}   (Theorem \ref{thm12.1}.) Suppose that $u \ne \pm 1, v^2$ is not a primitive root for all primes \(p\geq x_0\), with \(x_0\geq 1\) constant. Let \(x>x_0\) be a large number, and $q=O(\log^Cx)$. Consider the sum of the characteristic function over the short interval \([x,2x]\), that is, 
\begin{equation} \label{el8720}
0=\sum_{\substack{x \leq p\leq 2x \\ q \equiv a \bmod q}} \Psi (u).
\end{equation}
Replacing the characteristic function, Lemma \ref{lem3.2}, and expanding the nonexistence equation (\ref{el8720}) yield
\begin{eqnarray} \label{el8730}
0&=&\sum _{\substack{x \leq p\leq 2x\\
		p \equiv a \bmod q		}} \Psi (u) \nonumber \\ 
&=&\sum_{\substack{x \leq p\leq 2x\\
		p \equiv a \bmod q		}} \left (\frac{1}{p}\sum_{\gcd(n,p-1)=1,} \sum_{ 0\leq k\leq p-1} \psi \left((\tau ^n-u)k\right) \right )\\
&=&a_{u}\sum_{\substack{x \leq p\leq 2x\\
		p \equiv a \bmod q		}} \frac{1}{p}\sum_{\gcd(n,p-1)=1} 1+\sum_{\substack{x \leq p\leq 2x\\
		p \equiv a \bmod q		}}
\frac{1}{p}\sum_{\gcd(n,p-1)=1,} \sum_{ 0<k\leq p-1} \psi \left((\tau ^n-u)k\right)\nonumber\\
&=&a_uM_{q}(x) + E_{q}(x)\nonumber,
\end{eqnarray} 
where $a_{u} \geq 0$ is a constant depending on the integers $\ell\ne \pm 1, v^2$ and $q\geq 2$. \\

The main term $M_{q}(x)$ is determined by a finite sum over the trivial additive character \(\psi =1\), and the error term $E_{q}(x)$ is determined by a finite sum over the nontrivial additive characters \(\psi =e^{i 2\pi  k/p}\neq 1\).\\

Take $B >C+1$. Applying Lemma \ref{lem12.2} to the main term, and Lemma \ref{lem12.4} to the error term yield
\begin{eqnarray} \label{el8762}
0&=&\sum _{\substack{x \leq p\leq 2x\\
		p \equiv a \bmod q		}} \Psi (u) \nonumber \\
&=&a_uM_{q}(x) + E_{q}(x) \nonumber\\
&=&a_{u}\left (A_q\frac{\li(2x)-\li(x)}{\varphi(q)} \right )+O\left(\frac{x}{\log^Bx}\right)+O\left(\frac{1}{\varphi(q)} \frac{x^{1-\varepsilon}}{\log x} \right) \nonumber\\
&=& \delta(u,q,a)\frac{x}{\log x}+O\left( \frac{x }{\log^B} \right)  ,
\end{eqnarray} 
where $\delta(u,q,a)=a_{u}A_q /\varphi(q) \geq 0$, and $a_{u}\geq 0$ is a correction factor depending on $u$. \\

But $\delta(u,q,a) > 0$ contradicts the hypothesis (\ref{el8720}) for all sufficiently large numbers $x \geq x_0$. Ergo, the short interval $[x,2x]$ contains primes $p=qn+a$ such that $u$ is a fixed primitive root modulo $p$. Specifically, the counting function is given by
\begin{equation} \label{el8967}
\pi_u(x,q,a)=\sum _{\substack{p\leq x\\
		p \equiv a \bmod q		}} \Psi (\ell)
=\delta(u,q,a)\li(x) +O\left( \frac{x}{\log^B x} \right) .
\end{equation} 
This completes the verification. 
\end{proof}

A formula for computing the density $\delta(u,q,a)=c(u,q,a)A_q  /\varphi(q) \geq 0$ appears in Theorem \ref{thm6.2}; and the earliest case $q=2$ appears in Theorem \ref{thm6.1}, Chapter 6, see also \cite[p.\ 218]{HC67}, and a general discussion is giving in \cite[p.\ 16]{MP04}. The determination of the correction factor $c(u,q,a)$ in primes counting problem is a complex problem, some cases are discussed in \cite{SP03}, and \cite{LS14}. \\

The logarithm integral difference $\li(2x)-\li(x)=x/\log x+O\left(x/\log^2x \right)=\li(x)$ is evaluated via (\ref{el8100}) in Section 8.1.\\

%ccccccccccccccccccccccccccccccccccccccc
\chapter{Relative Order Of Primitive Elements} \label{c13}
		Some statistical information on the relative orders of primitive roots modulo prime numbers are discussed in this section. Let $x \geq 1$ be a large number. The relative order of a random integer $U\leq x$ is defined by
		
		\begin{equation}\label{el22000}
		\frac{\ord_p(U)}{\varphi(p)}=\frac{\ord_p(U)}{p-1}. 
		\end{equation}
		
		This is a pseudorandom point in the unit interval $[0,1]$. The pseudorandom function $p \longrightarrow \ord_p(U)$ does not appears to be dense in the unit interval $[0,1]$. The average relative order of a random integer $U \in [2,N]$ with respect to a random prime $p \leq x$ was computed sometimes ago.\\
		
		\begin{thm}  \label{thm13.1}
		(\cite{SP69})  Let $x\geq 1$ be a large number, and let $N \asymp x$. Then, the relative order of a random integer $U \leq N$ has the asymptotic order
		\begin{equation}
		\frac{1}{N} \sum_{U\leq N} \sum_{p\leq x} \frac{\ord_p(U)}{\varphi(p)}=C \li(x)+O\left ( \frac{x}{\log^D x}\right),
		\end{equation}
		where $D>1$ is an arbitrary constant, and the density is given by
		\begin{equation}
		C=\prod_{p \geq 2}\left ( 1-\frac{p}{p^3-1}\right)\approx 0.647731.
		\end{equation}
		\end{thm}
		
\begin{figure} [h]
\includegraphics[width=.80\textwidth]{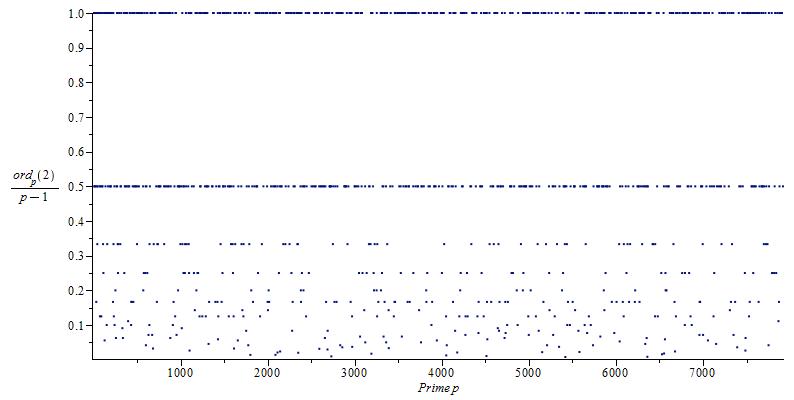}
\centering
\caption{Plot of the relative order $\ord_p(2)/(p-1)$ of $2$ verse the prime $p$.}
\end{figure}	

\begin{figure} [h]
\includegraphics[width=.80\textwidth]{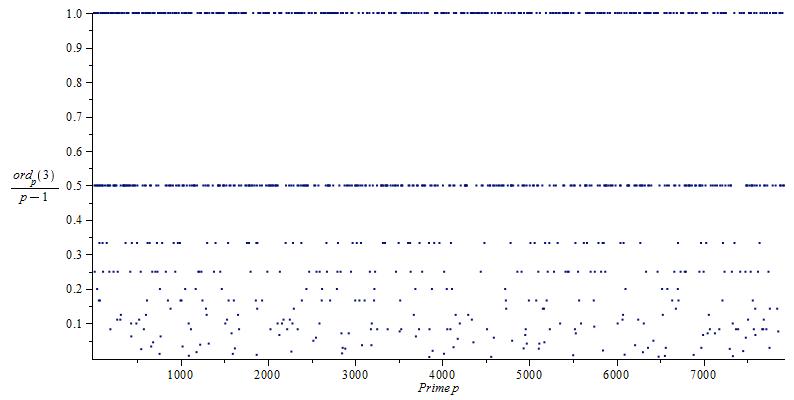}
\centering
\caption{Plot of the relative order $\ord_p(3)/(p-1)$ of $3$ verse the prime $p$.}
\end{figure}		
		
\newpage		
\section{Problems}
		\begin{exe} \normalfont  Let $x\geq 1$, let $N\asymp x$,and let \(U \in [2,N]$. Simplify the calculation of the average relative order $$E(\ord_p(U)/\varphi(p))=\frac{1}{N}\sum _{ U\leq N} \sum _{p \leq x} \frac{\ord_p(U)}{\varphi(p)}. $$
		\end{exe}
		
		\begin{exe} \normalfont  Let $x\geq 1$, let $N\asymp x$,and let $U \in [2,N]$. Compute the variance 
		$$Var(\ord_p(U)/\varphi(p))=\frac{1}{N}\sum _{U\leq N} \sum _{p \leq x}\left ( \frac{\ord_p(U)}{\varphi(p)}-C\li(x)\right)^2.$$
		\end{exe}

%ccccccccccccccccccccccccccccccccccc
\chapter{Orders of Quadratic Residues} \label{c14}
The theory of primitive roots is essentially a theory for quadratic nonresidues of maximal order modulo $p$. Similarly, the theory of quadratic residues modulo $p$ of maximal order $\ord_p(u)= (p-1)/2$ is an interesting problem. In general the order of a quadratic residue is a divisor of $(p-1)/2$. The analysis of the orders of quadratic residues is essentially the same as the analysis for primitive roots modulo $p$.\\

\section{Results For Quadratic Residues}
Consider the subset of primes defined by 
	\begin{equation}
	\mathcal{P}_{2}(u)=\{ p \in \mathbb{P}: \ord_p(u)=(p-1)/2\},
	\end{equation} 
	tand he corresponding counting function  
	\begin{equation}
	\pi(x,u)= \# \{ p \leq x:p \text{ prime and } \ord_p(u)=(p-1)/2\}.
	\end{equation}
There is a conditional result for the density of $d$-power residues in \cite{LH77}. The goal of this chapter is to achieve some unconditional results for the density of quadratic residues of maximal orders.
\begin{dfn}
A quadratic residue $u \geq 2$ of maximal order $\ord_p(u)=(p-1)/2$ is called a primitive quadratic residue.
\end{dfn}

\begin{thm}  \label{thm14.1}
		Let $ x\geq 1$ be a large number, and let $u= v^2>1$ be fixed integers, where $v=p_1p_2 \cdots p_t$ is squarefree integer. Then, the number of primes with a fixed quadratic residue $u$ of maximal order the asymptotic formula 
		\begin{equation} \label{el12830}
		\pi (x,u)=d_u\frac{x}{\log x}+O\left(\frac{x}{\log^2  x} \right),
		\end{equation}
		where $d_u>0$ is a constant depending on the fixed quadratic residue $u>1$.
	\end{thm}

\section{Evaluation Of The Main Term}
	Finite sums and products over the primes numbers occur on various problems concerned with primitive roots. These sums and products often involve the normalized totient function $\varphi(n)/n=\prod_{p|n}(1-1/p)$ and the corresponding estimates, and the asymptotic formulas.\\
	
\begin{lem} \label{lem14.1}
		Let \(x\geq 1\) be a large number, and let \(\varphi (n)\) be the Euler totient function.
		If \(q\leq \log^C x\), with $C\geq 0$ constant, an integer $1\leq a< q$ such that $\gcd(a,q)=1$, then
		\begin{equation} \label{el59}
		\sum_{\substack{p\leq x \\ p \equiv a \bmod q}} \frac{1}{p}\sum_{\gcd(n,(p-1)/2)=1} 1=\frac{\li(x)}{2\varphi(q)}\prod_{p \geq 2 } \left(1-\frac{1}{p(p-1)}\right)+O\left(\frac{x}{\log
			^Bx}\right) ,
		\end{equation} 
		where \(\li(x)\) is the logarithm integral, and \(B> C+1\geq1\) is an arbitrary constant, as \(x \rightarrow \infty\).
	\end{lem}
	
	\begin{proof}   Same as Lemma \ref{lem8.3}.   \end{proof}
	
\section{Estimate For The Error Term}
	The upper bound for exponential sum over subsets of elements in finite rings $\left (\mathbb{Z}/N\mathbb{Z}\right )^\times$ stated in Section four is used here to estimate the error term $E(x)$ in the proof of Theorem \ref{thm13.1}. \\
	
	\begin{lem}\label{lem14.2}
		Let \(x\geq 1\) be a large number, and let \(\tau\) be a primitive root mod \(p\). If the element \(u\ne 0\) is not a primitive quadratic residue, then, 
		\begin{equation} \label{el8000}
		\sum_{x \leq p\leq 2x}
		\frac{1}{p}\sum_{\gcd(n,(p-1)/2)=1,} \sum_{ 0<k\leq p-1} \psi \left((\tau ^{2n}-u)k\right)\ll  \frac{x^{1-\varepsilon}}{\log x} 
		\end{equation} 
		for all sufficiently large numbers $x\geq 1$ and $\varepsilon>0$ an arbitrary small number.
	\end{lem}
	
\begin{proof}   Consider the proof of Lemma \ref{lem8.3} mutatis mutandis.   \end{proof}	
	
\section{Result For Arithmetic Progressions}
The simple and elementary estimate for the error term given in Lemma \ref{lem14.1} leads to a weak but effective result for primes in the arithmetic progression $p=qn+1, n\geq 1$. However, it is restricted to a certain $q \geq 2$.\\
	
	Given a fixed quadratic residue $u= v^2$, where $v=p_1p_2 \cdots p_k$ is squarefree, the precise primes counting function is defined by 
	\begin{equation}\label{el888}
	\pi _{u}(x,q,a)=\#\left\{ p\leq x:p\equiv a \tmod q \text{ and } \ord_p(u)=(p-1)/2 \right\}.
	\end{equation}
Recall that in Lemma \ref{lem3.3}, the characteristic function for primitive quadratic residue has the form
\begin{equation}
\Psi_2 (u)=\sum _{\gcd (n,(p-1)/2)=1} \frac{1}{p}\sum _{0\leq k\leq p-1} \psi \left ((\tau ^{2n}-u)k\right)
=\left \{
\begin{array}{ll}
1 & \text{ if } \ord_p(u)=(p-1)/2,  \\
0 & \text{ if } \ord_p(u)\neq (p-1)/2. \\
\end{array} \right .
\end{equation}
This corresponds to the case $d=2$.\\

	\begin{thm} \label{thm14.2} Let \(x\geq 1\) be a large number, and let $q= \prod_{r \leq \log \log x}r \asymp  \log x$. Define the subset of primes
		\( \mathcal{P}=\{p=qn+1: n \leq 1\}\). Let $u= v^2>1$ be fixed integers, where $v=p_1p_2 \cdots p_t$ is squarefree integer. Then, the number of primes with a fixed quadratic residue $u$ of maximal order has the asymptotic formula 
		\begin{equation} \label{el2830}
		\pi (x,u,q,1)=d_u\frac{1}{\varphi(q)}\frac{x}{\log x}+O\left(\frac{x}{\log^2  x} \right),
		\end{equation}
		where $d_u>0$ is a constant depending on the fixed quadratic residue $u>1$.
	\end{thm}
	
\begin{proof}  Suppose that $u= v^2$ is not a quadratic residue of maximal order for all primes \(p\geq x_0\), with \(x_0\geq 1\) constant. Let \(x>x_0\) be a large number, and consider the sum of the characteristic function over the short interval \([x,2x]\), that is, 
\begin{equation} \label{el720}
	0=\sum_{\substack{x \leq p\leq 2x \\ q \equiv 1 \bmod q}} \Psi_2 (u)
\end{equation}
Replacing the characteristic function, Lemma \ref{lem3.3}, and expanding the nonexistence equation (\ref{el720}) yield
\begin{eqnarray} \label{el13730}
	0&=&\sum _{\substack{x \leq p\leq 2x\\
			p \equiv 1 \bmod q		}} \Psi_2 (u) \nonumber \\
	&=&\sum_{\substack{x \leq p\leq 2x\\
			p \equiv 1 \bmod q		}} \left (\frac{1}{p}\sum_{\gcd(n,(p-1)/2)=1,} \sum_{ 0\leq k\leq p-1} \psi \left((\tau ^{2n}-\ell)k\right) \right )\\
	&=&d_{u}\sum_{\substack{x \leq p\leq 2x\\
			p \equiv 1 \bmod q		}} \frac{1}{p}\sum_{\gcd(n,(p-1)/2)=1} 1+\sum_{\substack{x \leq p\leq 2x\\
			p \equiv 1 \bmod q		}}
	\frac{1}{p}\sum_{\gcd(n,(p-1)/2)=1,} \sum_{ 0<k\leq p-1} \psi \left((\tau ^{2n}-\ell)k\right)\nonumber\\
	&=&M_{2}(x) + E_{2}(x)\nonumber,
	\end{eqnarray} 
	where $d_{u}\geq 0$ is a constant depending on the integers $u= v^2$ and $q\geq 2$. \\
	
	The main term $M_{2}(x)$ is determined by a finite sum over the trivial additive character \(\psi =1\), and the error term $E_{2}(x)$ is determined by a finite sum over the nontrivial additive characters \(\psi =e^{i 2\pi  k/p}\neq 1\).\\
	
	Take $B >2$. Applying Lemma \ref{lem14.2} to the main term, and Lemma \ref{lem14.1} to the error term yield
	\begin{eqnarray} \label{el13760}
	\sum _{\substack{x \leq p\leq 2x\\
			p \equiv 1 \bmod q		}} \Psi_2 (u)
	&=&M_{2}(x) + E_{2}(x) \nonumber\\
	&=&d_{u}\left (d_u\frac{\li(2x)-\li(x)}{\varphi(q)} \right )+O\left(\frac{x}{\log^Bx}\right)+O\left(\frac{1}{\varphi(q)} \frac{x}{(\log \log \log  x)\log x} \right) \nonumber\\
	&=& d_{u}\frac{1}{\varphi(q)}\frac{x}{\log x}+O\left( \frac{1}{\varphi(q)}\frac{x }{(\log \log \log x)\log x} \right)  \\
	&>&0 \nonumber,
	\end{eqnarray} 
where $d_u>0$. This contradict the hypothesis  (\ref{el720}) for all sufficiently large numbers $x \geq x_0$. Ergo, the short interval $[x,2x]$ contains primes $p=qn+1$ such that $u >1$  is a quadratic residue of maximal order. Specifically, the counting function is given by
\begin{equation} \label{el8764}
\pi_u(x,q,1)=\sum _{\substack{p\leq x\\
			p \equiv 1 \bmod q		}} \Psi_2 (u)
=d_{u}\frac{\li(x)}{\varphi(q)} +O\left( \frac{1}{\varphi(q)}\frac{x}{(\log \log \log  x)\log x} \right) .
	\end{equation} 
	This completes the verification.\end{proof}

There are formulas for computing the densities $d_{u}>0$ of primitive quadratic residues $u=v^2$; these appear in \cite{LH77} and \cite{WS82}. \\
	
\section{Numerical Data}
The two graphs displayed below show the similarities of the distribution of the primes $p \geq 3$ with a fixed primitive root $u=2$ and the distribution of the primes $p \geq 3$ with a fixed primitive quadratic residue $u^2=4$ modulo $p$ as the prime varies.\\

\vskip .5 in 
\begin{figure}[h]
\includegraphics[width=.70\textwidth]{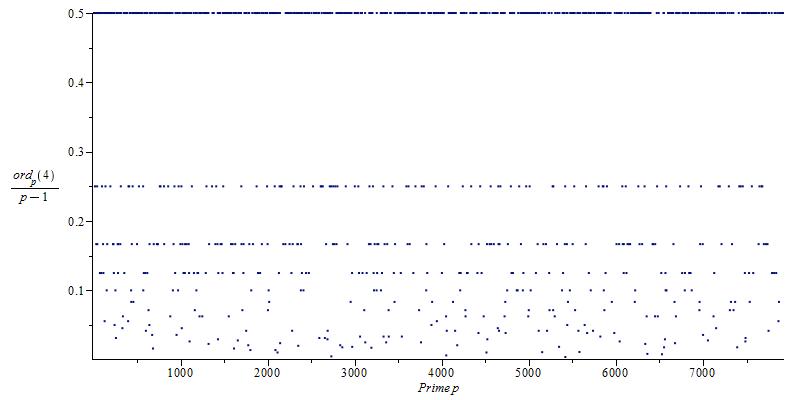}
\centering
\caption{Plot of the relative order $\ord_p(4)/(p-1)$ of $4$ verse the prime $p$.}
\end{figure}
\vskip .5 in 	
\begin{figure}[h]
\includegraphics[width=.70\textwidth]{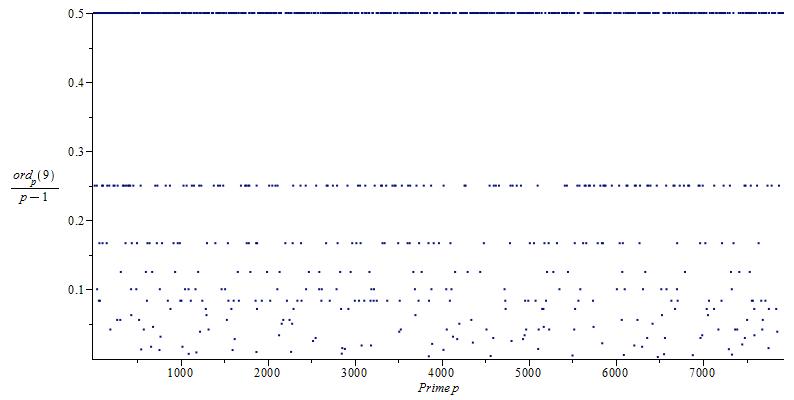}
\centering
\caption{Plot of the relative order $\ord_p(9)/(p-1)$ of $9$ verse the prime $p$.}
\end{figure}
\vskip .5 in

\newpage	
\section{Problems}
\begin{exe} \normalfont  Let $\tau \geq 2$ be a primitive root modulo $p \geq 3$, let $u=4$ be a quadratic residue modulo $p$. Compute the density 
$d_4 > 0$ of the primes for which $ \ord_p(4)=(p-1)/2$. 
\end{exe}

\begin{exe} \normalfont  Let $\tau \geq 2$ be a primitive root modulo $p \geq 3$, let $u=v^2$, where $v=p_1p_2 \cdots p_k$ is squarefree, be a quadratic residue modulo $p$. Find a formula for the density 
$d_u >0$ of the primes for which $ \ord_p(4)=(p-1)/2$. 
\end{exe}
			
\begin{exe} \normalfont  Let $\tau \geq 2$ be a primitive root modulo $p \geq 3$, let $u=\tau^{2^am}$ be a quadratic residue modulo $p$, where 
$a \geq 1, m \geq 1$. Explain how the parameter $a \geq 1$ changes the order $\ord_p(u) | (p-1)/2$ of the element $u$, and the density $d_u \geq 0$. 
\end{exe}

\begin{exe} \normalfont  Let $u=v^2$, where $v=p_1p_2 \cdots p_k$ is squarefree. Show that the subset of rationals $\{ \ord_p(u) | (p-1)/2: p \text { prime} \}$ is not dense in the interval $[0,1/2]$. 
\end{exe}		

\begin{exe} \normalfont Let $u=4$, where $v=2$. Prove or disprove that the subset of primes $\{ p \text { prime}:\ord_p(4)=(p-1)/2 \}$ is contained in the subset of primes $\{ p \text { prime}:\ord_p(2)=p-1 \}$. 
\end{exe}	

\begin{exe} \normalfont  Let $u=v^2$, where $v=p_1p_2 \cdots p_k$ is squarefree. Prove  or disprove that the subset of primes $\{ p \text { prime}:\ord_p(u)=(p-1)/2 \}$ is contained in the subset of primes $\{ p \text { prime}:\ord_p(v)=p-1 \}$. 
\end{exe}	

\begin{exe} \normalfont  Show that the relative order $\ord_p(u)/(p-1)=1/2,1/3,/1/4, \ldots, 1/n$ for some $n \geq 2$, and $\ord_p(u)/(p-1) \ne3/4,2/3,5/8, \ldots, m/n$ for any $m,n \ge 2$. 
\end{exe}

%cccccccccccccccccccccccccccccccccccccccccc
\chapter{Elliptic Curves Over Finite Fields} \label{c15}
This chapter provides some details on the theory of elliptic curves over finite fields. It establishes the notation and various results useful in the later Chapters. These concepts are centered on the structure of the groups of rational points. 

\section{Descriptions of Elliptic Curves}
An elliptic curve is algebraic curve $E:f(x,y)=0$ over a field $\mathcal{K}$. In fields of characteristic $char(\mathcal{K})\ne 2,3$ it has a short form equation
\begin{equation}
E:y^2=x^3+ax+b.
\end{equation}

\begin{dfn}
The elliptic curve is nonsingular if the \textit{discriminant} $\Delta(E)=-16(4a^3+27b^2)\ne 0.$ 
\end{dfn}
The \textit{conductor} of the elliptic curve $N \geq 1$ is a rational multiple of the discriminant. \\

For any nonsingular algebraic curve $E:f(x,y)=0$ over a field $\mathcal{K}$, the group of $\mathcal{K}$-rational points is defined by
\begin{equation}
E(\mathcal{K})=\{(x,y)\in \mathcal{K}:f(x,y)=0\} \cup \{\mathcal{O}\}.
\end{equation}
The element $\mathcal{O}$ is the group identity, also known as the point at infinity.

\section{Cardinality of the Group}
A basic result on the size of the elliptic group uses an exponential sum with respect to the defining equation to estimate the best error term possible.
\begin{thm}[Hasse]
The cardinality of the group of rational points satisfies 
\begin{equation}
\left | p+1-\#E(\mathbb{F}_p) \right | \leq 2 \sqrt{p}.
\end{equation} 
\end{thm}
\begin{proof}  For any odd prime $p \geq 3$, and each $x \in\mathbb{F}_p$, a nonsingular elliptic curve $E:y^2=f(x)$ has $1+\left (\frac{f(x)}{p}\right )=0,1,2$ $\mathbb{F}_p$-rational points. Therefore, the total count is
\begin{equation}
\sum_{0\leq x<p} \left (1+\left (\frac{f(x)}{p}\right ) \right )=p+1+\sum_{0\leq x<p} \left (\frac{f(x)}{p}\right )=p+1-a_p, 
\end{equation}
where $\left (\frac{f(x)}{p}\right )$ is the quadratic symbol. This tally includes the point at infinity. To complete the proof, use the Weil bound $\sum_{0\geq x<p} \left (\frac{f(x)}{p}\right ) \leq (d-1) \sqrt{p}$, where $ d=\deg(f)=3$, and $f(x)$ is a squarefree polynomial.  \end{proof}

\subsection{Frobenious Trace} 
The Frobenious map is defined by
\begin{equation}
\begin{array} {lll}
\phi:E(\overline{\mathbb{F}_p)}& \longrightarrow & E(\overline{\mathbb{F}_p)}, \\
P=(x,y)& \longrightarrow &(x^p,y^p). \\
\end{array} 
\end{equation}
The fixed subset $\{(x,y)=(x^p,y^p):(x,y) \in E(\overline{\mathbb{F}_p)}\}=E(\mathbb{F}_p)$, and the kernel $\#ker(\phi-1)=\deg(\phi-1)=\#E(\mathbb{F}_p)$ is the cardinality of the group of points.\\
	
The characteristic polynomial of the map $\phi$ is
\begin{equation}\label{800=23}
\phi^2-Tr(\pi) \phi+N(\pi),
\end{equation}
where the trace and norm are defined by
\begin{equation}\label{800-25}
Tr(\pi)=\pi+\overline{\pi}=a_p \text{ and } N(\pi)=\pi\overline{\pi}=p
\end{equation}
for $\pi \in \mathbb{Q}(\sqrt{p^2-4a_p})$.\\

The subset of traces of Frobenious $\{a_p(E):p \text{ is prime}\}$ is a basis for the complete set of Frobenious numbers $\{a_n(E):n \geq 1 \text{ is an integer}\}$.

\begin{lem}
Let $E$ be a curve of genus $g\geq 1$, and $p,q \in \mathbb{Z}$ be integers. Then,
\begin{enumerate}
\item $a_{pq}=a_pa_q$ if and only if $\gcd(p,q)=1$.
\item $ a_{p^{n+1}}=a_{p^{n}}-p	a_{p^{n-1}} $ for any prime $p \leq 2$, and $n \geq 1$.
\end{enumerate}
\end{lem}

\subsection{Hasse-Weil $L$-Function}
The Frobenious numbers are the coefficients of the $L$-function attached to an elliptic curve. The series and its product representation are given by 
\begin{equation}
L(s)=\sum_{n\geq 1} \frac{a_n}{n^s}=\prod_{p|N} \left( 1-\frac{a_p}{p^{2s-1}} \right )^{-1}\prod_{p \nmid N} \left( 1-\frac{a_p}{p^s}+\frac{1}{p^{2s-1}} \right )^{-1},
\end{equation}
where $N$ is the conductor of the curve, and $a_p=0,\pm1$ for $p\,|\,N$. It is absolutely convergent on the complex half plane $\{ s \in \mathbb{C}:\Re e(s)>3/2\}$. The theory of this function is widely available in the literature, see \cite[Definition 5.3]{RS02}.\\

Roughly speaking, the local data $E(\mathbb{F}_p)$ for finitely many primes determines the torsion group $E(\mathbb{Q}_{\text{tors}})$, while the local data $E(\mathbb{F}_p)$ for infinitely many primes determines arithmetic and analytic rank $\rk(E)=r \geq 0$ of the group $E(\mathbb{Q})\cong \mathbb{Z}^r \times E(\mathbb{Q}_{\text{tors}})$.

\section{Structure of the Group}
The group of points form a finite group of order $n=\#E(\mathbb{F}_p)$ and has the form 
\begin{equation}
E(\mathbb{F}_p) \cong \mathbb{Z}_d \times \mathbb{Z}_{e}.
\end{equation}
Here the integer $n=de$ and $d|e$. This information seems to show that $E(\mathbb{F}_p)$ is a cyclic group if and only if $d=1$. \\

The integer $e\geq 1$ is called the exponent of the group. It has a value in the range $[1,p+2\sqrt{p}]$.\\

The algorithms for computing the precise structure of these groups are studied in \cite{MV09}, \cite{BP75} and similar references.

\subsection{Average Exponent}
Let $E(\mathbb{F}_p) \cong \mathbb{Z}_{d_p} \times \mathbb{Z}_{e_p}$. The average exponent is estimated to be $e_p \approx p$. There is a recent result for the product.
  
\begin{thm} \label{thm14.1}  {\normalfont (\cite{FK13})} Given an elliptic curve defined over the rational numbers $\mathbb{Q}$, there exists a number $C_E \in (0,1)$ such that on the Generalized Riemann hypothesis, 
\begin{equation}
\sum_{p \leq x, p \nmid N} d_pe_p=C_E \li(x^2)+O\left (x^2\frac{\log \log x}{(\log x)^{9/8}} \right)
\end{equation}	  
for $x\geq 2$. The constant is given by the Kummer series
\begin{equation}
C_E=\sum_{n \geq 1}
\frac{(-1)^{\omega(n)}\varphi(\rad(n))}{[\mathbb{Q}(E[n]):\mathbb{Q}]}.
\end{equation}
\end{thm}
 
\section{$n$-Division Points And Torsion Groups}
Let $\mathcal{K}$ be a numbers field. A point $P \in E(\mathcal{K})$ has finite order if and only if $nP=\mathcal{O}$ for some integer $n \geq 1$.\\

In general, given a nonsingular elliptic curve $E:f(x,y)=0$, and any integer $n \geq 1$, the subset of $n$-division points is defined by 
\begin{equation}
E[n]=\{ P \in E(\mathbb{\overline{Q}}): nP=\mathcal{O}\}
\end{equation}
and 
\begin{equation}
E[2]=\{ (e_1,0),(e_2,0),(e_3,0):f(e_i,0)=0\}.
\end{equation}

\begin{lem}
The group of n-division points is as follows.
\begin{enumerate}
\item $E[n] \cong \mathbb{Z}_n \times \mathbb{Z}_n$   if $\gcd(n,p)=1$ or $p \nmid n$.
\item $E[n] \cong \mathbb{Z}_m \times \mathbb{Z}_m$    if $n=mp^k$with $k \geq 1$. 
\end{enumerate}
\end{lem}

The torsion group $E(\mathcal{K})_{\text{tors}}=\{P \in E(\mathcal{K})\}: nP=\mathcal{O} \text{ some } n \}$ is the set of points of finite orders. \\
	
The size of the torsion group $E(\mathbb{Q})_{\text{tors}}$ is bounded by 16, see \cite[Theorem 7.5]{SJ09}, \cite[Theorem 6.4]{SZ03}, et cetera. \\ 

\section{Distribution of the Orders}
The order of a nonsingular elliptic curve is an integer $n \geq 1$ in the Hasse interval $[p-2\sqrt{p},p+2\sqrt{p}]$. The normalized Frobenious trace
\begin{equation}
z=\frac{a_E(p)}{2\sqrt{p}} \in [-1,1]
\end{equation}
is equidistributed and has the density function
\begin{equation}
f(t)= \left \{\begin{array}{ll }
\frac{2}{\pi} \sin^2t & \text{ if }E \text{ has CM}\\
\frac{1}{2\pi} \frac{1}{\sin^2t} & \text{ if }E \text{ has nonCM}
\end{array} \right . .
\end{equation}
The corresponding cumulative function is
\begin{eqnarray}
F(a,b)&=& \lim_{x \to \infty}\frac{\#\{p \leq x: p \not | \Delta \text{ and } z \in [a,b]   \}}{\pi(x)} \nonumber \\
&=&\int_a^bf(t)dt.
\end{eqnarray}
References for the proofs and other details on the Deuring-Sato-Tate density function appears in \cite[p.\ 5]{BJ05}. Some corrections to the old Deuring-Sato-Tate density function are discussed \textit{ibidem}.\\

\begin{cor}
The orders $n=\# E(\mathbb{F}_p)$ of a nonsingular elliptic curve $E$ is equidistributed over the interval $[p-2\sqrt{p},p+2\sqrt{p}]$, and has a Deuring-Sato-Tate density function. 
\end{cor}

\section{Problems}
\begin{exe} \normalfont
	Let $E:f(x,y)=0$ be a nonsingular elliptic curve. Prove that the maximal subexponent $d$ of the group $
	E(\mathbb{F}_p) \cong \mathbb{Z}_d \times \mathbb{Z}_{e}
	$ is $d \leq \sqrt{p+2\sqrt{p}}$.
\end{exe}

%161616161616161161616161
\chapter{Elliptic Primitive Points} \label{c16}
The short interval $[p-2\sqrt{p},p-2\sqrt{p}]$ contains $(6 \pi^{-2})(4\sqrt{p})+o(\sqrt{p})$ squarefree integers, and each group $E(\mathbb{F}_p) $ of squarefree integer order $n \geq$ is cyclic. Moreover, the probability of finding a generator for an elliptic cyclic group $E(\mathbb{F}_p) $ of cardinality $n=\#E(\mathbb{F}_p) $ is $ \gg \varphi(n)/n \gg 1/\log n$. In light of the fact, is plausible to state that a random elliptic curve over a finite field has a good chance of having a primitive point. This chapter studies the asymptotic counting function for the number of elliptic primitive primes such that the given elliptic curve has a primitive point. \\

\section{Results For Primitive Points} A classical result in number theory states that a cyclic group $\left (\mathbb{Z} / n\mathbb{Z} \right )^{\times}$ has $\varphi(n)$ primitive roots (generators) if and only if $n=2, 4, p^k,2p^k$ with $p\geq 3$ prime and $k\geq 1$. \\

Similarly, the  cyclic group $E(\mathbb{F}_p)$ of cardinality $n \geq 1$ can have up to $\varphi(n)$ primitive points (generators), but it is unknown which forms the integers $n$ must have, and what is the precise dependence on the elliptic curve. \\

Since the ratio $n/\varphi(n) \gg 1/\log n$, a representation $\langle P\rangle=E(\mathbb{F}_p)$ can be determined in probabilistic polynomial time. The deterministic algorithm have exponential time complexity $O\left (p^{1/2+\varepsilon}\right )$, see \cite{KS00}.\\

For fixed nonsingular elliptic curve $E:f(x,y)=0$ over the rational numbers $\mathbb{Q}$, of rank $\rk(E)\geq 1$, and a point $P \in E(\mathbb{Q})$ of infinite order, the group of rational points is defined by
 
\begin{equation}
E(\mathbb{F}_p)=\{(x,y) \in \mathbb{F}_p \times \mathbb{F}_p :f(x,y)=0 \},
\end{equation}
and the subset of elliptic primitive primes is defined by
\begin{equation}
\mathcal{P}(E,P)=\{p \in \mathbb{P}: E(\mathbb{F}_p) =\langle P \rangle   \},
\end{equation}
where $\mathbb{P}=\{2,3,5, \ldots\}$ is the set of prime numbers. The corresponding counting function is defined by
\begin{equation}
\pi(x,E,P)=\#\{p \leq x: E(\mathbb{F}_p) =\langle P\rangle   \}.
\end{equation}

The finite number of prime divisors $p\,|\, N$ of the conductor $N$ of the elliptic curve are not included in the tally.\\

The early analysis and conjectures for primitive points in the groups of points of elliptic curves were undertaken in \cite{BP75}, and \cite{LT77}. The heuristic arguments concluded with the following observation.

\begin{conj} [Lang-Trotter] \label{conj16.1} 
For any elliptic curve $E:f(x,y)=0$ of rank $\rk(E(\mathbb{Q}))\geq 1$ and a point $P \in E(\mathbb{Q})$ of infinite order, the elliptic primitive prime counting function has the asymptotic formula 
\begin{equation}
\pi(x,E,P)=\delta(E,P)\frac{x}{\log x}+O\left ( \frac{x}{\log^2x}\right),
\end{equation}  
where $\delta(E,P)\geq0$ is the density of elliptic primes. 
\end{conj}
The natural density has a complicated formula
\begin{eqnarray}\label{60002}
\delta(E,P)&=&   \lim_{x \to \infty} \frac{\#\{p \leq x:\ord_E(P)=n\}}{\pi(x)} 	\nonumber   \\
&=&\sum_{n \geq 1} \mu(n) \frac{\# \mathcal{C}_{P,n}}{\#\Gal(\mathbb{Q}(E[n],n^{-1}P)/\mathbb{Q})} \nonumber \\ 
&=&\mathfrak{C}(E,P)\prod_{p\geq 2} \left (1 -\frac{p^3-p-1}{p^2(p-1)^2(p+1)}\right ),
\end{eqnarray}
where $\mathfrak{C}(E,P)$ is a correction factor; and $\mathcal{C}_{P,n}$ is the union of certain conjugacy classes in the Galois group $\Gal(\mathbb{Q}(E[n],n^{-1}P))$. The symbol $\ord_E(P)=\min \{ m \geq 1: mP= \mathcal{O} \}$ denotes the order of the point $P \in E(\mathbb{F}_p)$, and $n=\#E(\mathbb{F}_p)$ is the cardinality of the group, see \cite{LS14}, \cite{BJ17} for the derivations. The basic product
\begin{equation}\label{60999}
\prod_{p\geq 2} \left (1 -\frac{p^3-p-1}{p^2(p-1)^2(p+1)}\right )=0.44014736679205786 \ldots
\end{equation}
is referred to as the average density of primitive primes $p \in \mathbb{P}=\{2,3,5,\ldots \}$ such that $E(\mathbb{F}_p)$ has a fixed primitive point $P\not \in E(\mathbb{Q})_{\text{tors}}$ in the group of points of the elliptic curve $E$. The average density is studied in \cite{JN09}, \cite{LS14}, and \cite{ZD09}. Other related results are given in \cite{JN10}, \cite{BC11}, \cite{BJ17}, et alii. \\

Under the generalized Riemann hypothesis, several specialized cases of the elliptic primitive point conjecture were proved in \cite{GM86}, and generalized to algebraic varieties in \cite{VV12}. These cases cover various types of curves, which have complex multiplication, without complex multiplication, and curves of large ranks.\\

A completely different approach, applicable to both CM and nonCM elliptic curves, is devised here. This technique does not require the detailed analysis for the different types of CM and nonCM elliptic curves; and the complicated ranks information. The new method requires some information on the local properties of the point $P$. The main result has the following unconditional claim.

\begin{thm} \label{thm16.1}   
Let $E:f(X,Y)=0$ be an elliptic curve over the rational numbers $\mathbb{Q}$ of rank $\rk(E(\mathbb{Q})\geq 1$, and let $P \in E(\mathbb{Q})$ be a point of infinite order. Suppose that $\langle P\rangle=E(\mathbb{F}_p)$ for at least one large prime. Then, 
\begin{equation}
\pi(x,E,P) \gg \frac{x}{\log x} \left (1 +O\left ( \frac{x}{\log x} \right ) \right )
\end{equation}	
as $x \to \infty$.
\end{thm}

In synopsis, for any suitable, nonsingular elliptic curve $E:f(X,Y)=0$, the corresponding group of rational points $E(\mathbb{F}_p)$ is generated by a point $P \in E(\mathbb{Q})$ of infinite order, and has the cyclic representation $\langle P\rangle=E(\mathbb{F}_p)$ for infinitely many primes. The necessary condition $\langle P\rangle=E(\mathbb{F}_p)$ for at least one large prime $p$ sieves off those points $P \in E(\mathbb{F}_p)$ for which 
\begin{align} \label{el16009}
P=2Q \text{ for some } Q \in E[2], \hskip .5 in &  P=3Q \text{ for some } Q \in E[3], \\
P=5Q \text{ for some } Q \in E[5], \hskip .5 in &  P=7Q \text{ for some } Q \in E[7], \ldots \nonumber.
\end{align}

\begin{exa} \normalfont \label{exa1.1} Consider $E:Y^2=X^3-X$ and its the group of $\mathbb{F}_p$-rational points $E(\mathbb{F}_p) \cong \mathbb{Z}_d \times \mathbb{Z}_e$, where $d=2,4$ for all primes. In this case there are no primitive points since every point $P= dQ$ for some $Q \in E[d]=\{Q \in E(\mathbb{Q}):dQ= \mathcal{O} \}$.
\end{exa}

A more general analysis applies to this example, see \cite{JJ08}, and \cite{ZD09}. Other numerical examples illustrating the condition (\ref{el16009}) for various elliptic curves are explained in \cite[Chapter 2]{MG15}. The least prime $p$ such that $\langle P\rangle=E(\mathbb{F}_p)$ seems to be completely determined by the conductor $N$ of the elliptic curve, some information appears in \cite[Corollary 1.2]{CA03}.\\

The preliminary background results and notation are discussed in Sections two to six. Section seven presents a proof of Theorem \ref{thm16.1}. After some preliminary sections, the proof is given in the last section. The focus is on the determination of the asymptotic part of the problem.
 
\section{Evaluation Of The Main Term}
The exact asymptotic form of main term $M(x)$ in the proof of Theorem \ref{thm16.1} is evaluated here. The analysis is restricted to the subset of integers $\mathcal{R}=\{ n \geq 1 \text{ and } Q(n)\geq n^{\alpha}\}$, where $Q(n)=\min \{p\,|\,n\}$ is the smallest prime divisor of $n$, and $\alpha >0$ is a small number. The subset of nonsmooth integers \begin{equation}\label{50000}
R= [p-2\sqrt{p},p+2\sqrt{p}] \cap \mathcal{R}
\end{equation} 
has nontrivial cardinality $R(x)\gg \sqrt{p}$, see Remark \ref{r500}.\\

\begin{lem} \label{lem16.1}
	For any large number \(x\geq 1\), and let $p \in [x,2x]$ be prime. Then
	\begin{equation} \label{el4959}
	\sum_{x \leq p \leq 2x} \frac{1}{4 \sqrt{p} }  \sum_{\substack{p-2 \sqrt{p}\leq n\leq p+2\sqrt{p}\\ n=p_1n_2} } \frac{1}{n}\sum_{\gcd(m,n)=1}1  \gg \frac{x}{\log x} \left (1 +O\left ( \frac{x}{\log x} \right ) \right ),
	\end{equation} 
	where the second finite sum is restricted to the subset of integers $n=p_1n_2$ for which $p_1=Q(n)=\min \{p\,|\,n\} \geq n^{\alpha}$,
	and $\alpha >0$ is a small number.
\end{lem}

\begin{proof} The value of the phi function $\varphi(n)=\#\{m:\gcd(m,n)=1\}$. Thus, main term can be rewritten in the form
	\begin{equation} \label{el14500}
	\sum_{x \leq p \leq 2x} \frac{1}{4 \sqrt{p} }  \sum_{\substack{p-2 \sqrt{p}\leq n\leq p+2\sqrt{p}\\ n=p_1n_2} } \frac{1}{n}\sum_{\gcd(m,n)=1}1\\
	=\sum_{x \leq p \leq 2x} \frac{1}{4 \sqrt{p} }  \sum_{\substack{p-2 \sqrt{p}\leq n\leq p+2\sqrt{p}\\ n=p_1n_2} } \frac{\varphi(n)}{n} .
	\end{equation} 
	The restriction to large integers of the form $n=p_1n_2$ with $p_1=Q(n) \geq n^{\alpha}$, and $\alpha>0$, leads to the lower estimate
	\begin{equation}\label{800-90}
	\frac{\varphi(n)}{n} = 1 -\frac{1}{p_1}-\frac{1}{p_2}-
	\cdots \geq 1-\frac{\log n}{n^{\alpha}} \geq \frac{1}{2}.
	\end{equation}
	Replacing this, and applying Lemma \ref{lem250.60} to the inner finite sum yield
	\begin{eqnarray} \label{el14503}
	\sum_{x \leq p \leq 2x} \frac{1}{4 \sqrt{p} }  \sum_{\substack{p-2 \sqrt{p}\leq n\leq p+2\sqrt{p}\\ n=p_1n_2} } \frac{\varphi(n)}{n}
	&\geq&\frac{1}{2}\sum_{x \leq p \leq 2x} \frac{1}{4 \sqrt{p} }  \sum_{\substack{p-2 \sqrt{p}\leq n\leq p+2\sqrt{p}\\ n=p_1n_2} } 1  \nonumber \\
	&\geq&\sum_{x \leq p \leq 2x}\frac{1}{ \sqrt{p} } \left ( c(\alpha,\beta) \sqrt{p} +o(\sqrt{p})    \right )\nonumber\\
	&\gg&\sum_{x \leq p \leq 2x}1 \\	&\gg&\frac{x}{\log x}\left (1+  O  \left (  \frac{x}{\log x}  \right ) \right )  \nonumber.
	\end{eqnarray}
	This proves the lower bound.    \end{proof}

\begin{rem} \normalfont \label{r500} The second line above uses the nonsmooth numbers estimate
\begin{equation}
\sum_{\substack{p-2 \sqrt{p}\leq n\leq p+2\sqrt{p}\\ n=p_1n_2} } 1  =\Theta(x,x^{\alpha},x^{\beta})= c(\alpha,\beta)\sqrt{p} +o(\sqrt{p}),
\end{equation}
for the number of integers $n \in (p-2 \sqrt{p}, p+2\sqrt{p}) $, where $p
_1=Q(n) \geq n^{\alpha}$ is the least prime divisor of $n$. This is derived from the results for smooth integers in short intervals, see Lemma \ref{lem250.60}. Currently, there are many rigorous proofs in the literature for the parameter $\alpha \geq 1/6$, and $\beta \geq \alpha$. \\

The weaker estimate for $\Theta(x,x^{\alpha},x^{\beta})$ in (\ref{250-600}) yields 
\begin{equation}
\sum_{\substack{p-2 \sqrt{p}\leq n\leq p+2\sqrt{p}\\ n=p_1n_2} } 1  =\Theta(x,x^{\alpha},x^{\beta}) \gg \frac{\sqrt{p}}{\log p} .
\end{equation}
\end{rem}

\section{Estimate For The Error Term}
The analysis of the upper bound for the error term $E(x)$, which occurs in the proof of Theorem \ref{thm1.1}, is split into two parts. The first part in Lemma \ref{lem16.2} is an estimate for the triple inner sum. And the final upper bound is assembled in Lemma \ref{lem16.3}. The analysis is restricted to the subset of integers $\mathcal{R}=\{ n \geq 1 \text{ and } Q(n)\geq n^{\alpha}\}$, where $Q(n)=\min \{p\,|\,n\}$ is the smallest prime divisor of $n$, and $\alpha >0$ is a small number. The subset of nonsmooth integers $R= [p-2\sqrt{p},p+2\sqrt{p}] \cap \mathcal{R}$ has nontrivial cardinality $R(x)\gg \sqrt{p}$.\\

\begin{lem} \label{lem16.2}
	Let $E$ be a nonsingular elliptic curve over rational number, let $P \in E(\mathbb{Q})$ be a point of infinite order. Let \(x\geq 1\) be a large number. For each prime $p\geq 3$, fix a primitive point $T$, and suppose that $P \in E(\mathbb{F}_p)$ is not a primitive point for all primes $p\geq 2$, then
	
	\begin{equation} \label{el14400}	 
	\sum_{\substack{p-2 \sqrt{p}\leq n\leq p+2\sqrt{p}\\ n =p_1n_2} }\frac{1}{n}\sum_{\gcd(m,n)=1,}  
	\sum_{ 1 \leq r <n} \chi((mT-P)r) \leq 4p^{1/2-\alpha} ,
	\end{equation} 
	where $p_1=Q(n) \geq n^{\alpha}$,  and $\alpha>0$ is a small number.
\end{lem}

\begin{proof} The nontrivial additive character $\chi \ne 1$ evaluates to
	\begin{equation}
	\chi(rmT)=e^{\frac{i2 \pi}{n}\log_T(rmT)}=e^{i2 \pi rm/n},
	\end{equation}
and 
\begin{equation}
\chi(-rP)=e^{\frac{i2 \pi}{n}\log_T(-rP)}=e^{-i2 \pi rk/n},
\end{equation}
	respectively, see (\ref{300-30}). To derive a sharp upper bound, rearrange the triple inner finite sum as a product 
	\begin{eqnarray} \label{el14417}
	T(p)&=&\sum_{\substack{p-2 \sqrt{p}\leq n\leq p+2\sqrt{p}\\ n=p_1n_2} }\frac{1}{n}\sum_{\gcd(m,n)=1,}  
	\sum_{ 1 \leq r <n} \chi((mT-P)r) \nonumber \\
	&= &\sum_{\substack{p-2 \sqrt{p}\leq n\leq p+2\sqrt{p}\\ n=p_1n_2}} \frac{1}{n}\sum_{ 1 \leq r <n} \chi(-rP)\sum_{\gcd(m,n)=1} \chi(rmT)\\
	&= &\sum_{\substack{p-2 \sqrt{p}\leq n\leq p+2\sqrt{p}\\ n=p_1n_2}}\left ( \sum_{ 1 \leq r <n} e^{-i2 \pi rk/n} \right ) \left ( \frac{1}{n}\sum_{\gcd(m,n)=1} e^{i2 \pi rm/n} \right ) \nonumber.
	\end{eqnarray}
	Let
	\begin{equation} \label{el88005}
	U_n=\max_{1 \leq k<n} \sum_{ 0<r\leq n-1} e^{-i2 \pi rk/n}    \qquad \text{ and } \qquad V_n=\max_{1 \leq k<n}\frac{1}{n}\sum_{\gcd(m,n)=1} e^{i2 \pi rm/n} .
	\end{equation} 
Then, by the triangle inequality,
\begin{equation} \label{el87000}
|T(p)| \leq \sum_{\substack{p-2 \sqrt{p}\leq n\leq p+2\sqrt{p}\\ n=p_1n_2}} | U_n V_n| \leq \sum_{\substack{p-2 \sqrt{p}\leq n\leq p+2\sqrt{p}\\ n=p_1n_2}} | U_n || V_n| .
	\end{equation}

	In this application, the Holder inequality $|| AB||_1 \leq || A||_{\infty}  \cdot || B||_1 $ takes the form
	\begin{equation} \label{el87001}
	\sum_{\substack{p-2 \sqrt{p}\leq n\leq p+2\sqrt{p}\\ n=p_1n_2}} | U_n V_n|\leq \max_{\substack{p-2 \sqrt{p}\leq n\leq p+2\sqrt{p}\\ n=p_1n_2} } |U_n| \cdot \sum_{\substack{p-2 \sqrt{p}\leq n\leq p+2\sqrt{p}\\ n=p_1n_2}} |V_n| .
	\end{equation} 
	The supremum norm $|| U_n||_{\infty}$ of the exponential sum $U_n=U_n(k)$ is given by
	\begin{equation} \label{el87026}
	\max_{\substack{p-2 \sqrt{p}\leq n\leq p+2\sqrt{p}\\ n=p_1p_2} }| U_n |= \max_{\substack{p-2 \sqrt{p}\leq n\leq p+2\sqrt{p}\\ n=p_1n_2} } \left | \sum_{ 0<r\leq n-1} e^{-i2 \pi rk/n} \right | =1.
	\end{equation} 
	This follows from $\sum_{ 0<r\leq n-1} e^{i 2 \pi rk/n}=-1$, $1 \leq k <n$. \\
	
	The absolute value of the second exponential sum $V_n=V_n(r)$ has the upper bound
	\begin{equation} \label{el85943}
	|V_n|=\frac{1}{n} \left | \max_{1 \leq r<n}\sum_{\gcd(m,n)=1} e^{i2 \pi rm/n} \right |\leq \frac{1}{n^{\alpha}} ,
	\end{equation} 
	where $\alpha>0$ is a small number. In general, this exponential sum, modulo a composite integer $n$, dependents on both $r$ and $n$. But for the restricted moduli $n=p_1n_2$, where $p_1 =Q(n)\geq n^{\alpha}$ is the least prime divisor of $n$, it has a uniform, and independent of $r$ upper bound, see Lemma \ref{lem4.4}. \\
	
	The corresponding $1$-norm $||V_n||_1 $ has the upper bound
	\begin{eqnarray} \label{el87080}
	\sum_{\substack{p-2 \sqrt{p}\leq n\leq p+2\sqrt{p}\\ n=p_1n_2} } |V_n| &\leq& \sum_{\substack{p-2 \sqrt{p}\leq n\leq p+2\sqrt{p}\\ n=p_1n_2} }  \left |  \frac{1}{n^{\alpha}} \right | \nonumber \\
	&\leq& \frac{1}{p^{\alpha}}\sum_{\substack{p-2 \sqrt{p}\leq n\leq p+2\sqrt{p}\\ n=p_1n_2}}  1 \\
	&\leq& 4p^{1/2-\alpha} \nonumber. 
	\end{eqnarray}
	The last line uses the trivial estimate $\sum_{\substack{p-2 \sqrt{p}\leq n\leq p+2\sqrt{p}\\ n=p_1n_2}}  1 \leq 4\sqrt{p}$. \\
	
	Now, replace the estimates (\ref{el87026}) and (\ref{el87080}) into (\ref{el87000}), the Holder inequality, to reach
	\begin{eqnarray} \label{el88080}
	\sum_{\substack{p-2 \sqrt{p}\leq n\leq p+2\sqrt{p}\\ n=p_1n_2} }
	\left | U_n V_n \right | 
	&\leq &	\max_{\substack{p-2 \sqrt{p}\leq n\leq p+2\sqrt{p}\\ n=p_1n_2} } |U_n| \cdot \sum_{\substack{p-2 \sqrt{p}\leq n\leq p+2\sqrt{p}\\ n=p_1n_2} } |V_n|  \\
	&\leq & 1 \cdot 4p^{1/2-\alpha} \nonumber .
	\end{eqnarray}

	This completes the verification.   \end{proof}

\begin{lem} \label{lem16.3}
	Let $E$ be a nonsingular elliptic curve over rational number, let $P \in E(\mathbb{Q})$ be a point of infinite order. For each large prime $p\geq 3$, fix a primitive point $T$, and suppose that $P \in E(\mathbb{F}_p)$ is not a primitive point for all large primes $p\geq x$, then
	\begin{equation} \label{el14403}
	\sum_{x\leq p\leq 2x}\sum_{\substack{p-2 \sqrt{p}\leq n\leq p+2\sqrt{p}\\ n=p_1n_2} }   \frac{1}{n}\sum_{\gcd(m,n)=1} 1
	\sum_{ 1 \leq r <n} \chi((mT-P)r) =O\left ( \frac{x^{1-\alpha}}{\log x} \right ),
	\end{equation} 
	where $p_1=Q(n) \geq n^{\alpha}$,  and $\alpha>0$ is a small number.
\end{lem}

\begin{proof} By assumption $P \in E(\mathbb{F}_p)$ is not a primitive point. Hence, the linear equation $mT-P= \mathcal{O}$ has no solution $m \in \{m:\gcd(m,n)=1\}$. This implies that the discrete logarithm $\log_T(mT-P)\ne0$, and $ \sum_{ 1 \leq r <n} \chi((mT-P)r) =-1$. This in turns yields   
	\begin{eqnarray} \label{el14404}
	|E(x)|&=&\left |\sum_{x \leq p \leq 2x } \frac{1}{4 \sqrt{p} } \sum_{\substack{p-2 \sqrt{p}\leq n\leq p+2\sqrt{p}\\ n=p_1n_2} }\frac{1}{n} \sum_{\gcd(m,n)=1,}  
	\sum_{ 1 \leq r <n} \chi((mT-P)r) \right |\nonumber \\
	&\leq &\sum_{x \leq p \leq 2x } \frac{1}{4 \sqrt{p} } \sum_{\substack{p-2 \sqrt{p}\leq n\leq p+2\sqrt{p}\\ n=p_1n_2} }\frac{1}{n} \sum_{\gcd(m,n)=1} 1 \\
	&\ll &\sum_{x \leq p \leq 2x } \frac{1}{4 \sqrt{p} } \sum_{\substack{p-2 \sqrt{p}\leq n\leq p+2\sqrt{p}\\ n=p_1n_2} }\frac{1}{2} \nonumber \\
	&\ll& \frac{x}{\log x}+O\left (\frac{x}{\log^2 x} \right ) \nonumber.
	\end{eqnarray}
	
	Here, the inequality $(1/n)\sum_{\gcd(m,n)=1}1=\varphi(n)/n \leq 1/2$ for all integers $n\geq 1$ was used in the third line. Hence, there is a nontrivial upper bound for the error term. \\
	
	To derive a sharper upper bound, take absolute value, and apply Lemma \ref{lem16.2} to the inner triple sum to obtain this:
	\begin{eqnarray} \label{el14444}
	|E(x)|&=&\left |\sum_{x \leq p \leq 2x } \frac{1}{4 \sqrt{p} } \sum_{\substack{p-2 \sqrt{p}\leq n\leq p+2\sqrt{p}\\ n=p_1n_2} }\frac{1}{n} \sum_{\gcd(m,n)=1,}  
	\sum_{ 1 \leq r <n} \chi((mT-P)r) \right |\nonumber \\
	&\leq&\sum_{x \leq p \leq 2x } \frac{1}{4 \sqrt{p} } \left |  \sum_{\substack{p-2 \sqrt{p}\leq n\leq p+2\sqrt{p}\\ n=p_1n_2} }\frac{1}{n} \sum_{\gcd(m,n)=1,}  
	\sum_{ 1 \leq r <n} \chi((mT-P)r) \right |\nonumber \\
	&\leq &\sum_{x \leq p \leq 2x } \frac{1}{4 \sqrt{p} } \left (4p^{1/2-\alpha} \right ) \nonumber \\
	&\leq& \sum_{x \leq p \leq 2x } \frac{1}{p^{\alpha}}  \\
	&\leq & \frac{1}{x^{\alpha}} \left ( \frac{x}{\log x}+O\left (\frac{x}{\log^2 x} \right ) \right ) \nonumber \\
	&=&O\left (\frac{x^{1-\alpha}}{\log x} \right ) \nonumber,
	\end{eqnarray}
	where $\alpha>0$ is a small number    \end{proof}

\section{Primitive Points On Elliptic Curves}
The torsion subgroup $E(\mathbb{Z})_{\text{tors}}=\{P \in E(\mathbb{Q}): nP=\mathcal{O} \text{ some } n \geq 1 \}$ of an elliptic curve is sort of analogous to the subset of quadratic residues $\{x^2:x \in \mathbb{F}_p\}$ in a finite field $\mathbb{F}_p$. Consequently, the constraints for primitive points $P \not  \in E(\mathbb{Z})_{\text{tors}}$ is analogous to the primitive roots constraints $u\ne \pm 1, v^2$. The size of the torsion group $E(\mathbb{Q})_{\text{tors}}$ is bounded by 16, see \cite[Theorem 7.5]{SJ09}, \cite[Theorem 6.4]{SZ03}, et cetera. More generally, the cardinality  $\#E(\mathcal{K})_{\text{tors}}=O(d^2)$ for an algebraic numbers field extension of degree $d=[\mathcal{K}:\mathbb{Q}]$ remains bounded as $x \to \infty$, see \cite{CP16}. Hence, a torsion point cannot be a primitive point, with finite number of exceptions.\\ 

The other constraint $\langle P\rangle=E(\mathbb{F}_p)$ for at least one large prime $p$ ensures that  
\begin{equation}
P\ne 2Q \text{ for some } Q \in E[2], \qquad  P\ne 3Q \text{ for some } Q \in E[3], \ldots ,
\end{equation}
see Example \ref{exa1.1}. A few other cases that fail this requirement are demonstrated in \cite[Chapter 2]{MG15}.\\

The characteristic function for primitive points in the group of points $E(\mathbb{F}_p)$ of an elliptic curve $E:f(x,y)=0$ has the representation
\begin{equation}
\Psi_E(P)=
\left \{\begin{array}{ll}
1 & \text{ if } \ord_E (P)=n,  \\
0 &  \text{ if } \ord_E (P) \ne n. \\
\end{array} \right.
\end{equation} 
Here, $\ord_E(P)=\min\{n \geq 1 :nP= \mathcal{O}\}$ is the order of the point in the group of $\mathbb{F}_p$-rational points. The parameter $n=\# E(\mathbb{F}_p)$ is the size of the group of points; and the exact formula for the characteristic function $\Psi_E (P)$ is given in Lemma \ref{lem3.7}.\\ 

Since for each prime $p \geq 2$ the order $n \in[p-2\sqrt{p}, p+2\sqrt{p}]$ is unique, the weighted sum
\begin{equation} \label{el147703}
\frac{1}{4 \sqrt{p}}\sum _{p-2\sqrt{p}\leq n\leq p+2\sqrt{p}} \Psi_E (P)=\left \{\begin{array}{ll}
\displaystyle \frac{1}{4 \sqrt{p}} & \text{ if } \ord_E (P)=n ,  \\
0 &  \text{ if } \ord_E (P) \ne n. \\
\end{array} \right.
\end{equation}
defines a discrete measure for the density of elliptic primitive primes $p \geq 2$ such that $P \in E(\mathbb{F}_p)$ is a primitive points. \\

The analysis is restricted to the subset of integers $n=p_1n_2$, where $p_1=Q(n) \geq n^{\alpha}$ is the least prime divisor of $n$, and $\alpha>0$ is a small number. This restriction arises in the calculations for the error term and the corresponding exponential sums.

\begin{proof}  (Theorem \ref{thm1.1}.) Let $\langle P \rangle= E(\mathbb{F}_p)$ for at least one large prime $p \leq x_0$. Let $x > x_0 \geq 1$ be a large number, and suppose that \(P \in E(\mathbb{Q})/E(\mathbb{Q})_{\text{tors}} \) is not a primitive point in $E(\mathbb{F}_p)$ for all primes \(p\geq x\). Then, the sum of the elliptic primitive primes measure over the short interval \([x,2x]\) vanishes. Id est, 
	\begin{equation} \label{el47703}
	0=\sum _{x \leq p\leq 2x} \frac{1}{4 \sqrt{p}}\sum _{\substack{p-2 \sqrt{p}\leq n\leq p+2\sqrt{p}\\ n=p_1p_2} } \Psi_E (P).
	\end{equation}
	Replacing the characteristic function, Lemma \ref{lem3.7}, and expanding the nonexistence equation (\ref{el47703}) yield
	
	\begin{eqnarray} \label{el14704}
	0&=&\sum _{x \leq p\leq 2x}  \frac{1}{4 \sqrt{p}} \sum_{\substack{p-2 \sqrt{p}\leq n\leq p+2\sqrt{p}\\ n=p_1n_2} } \Psi_E (P)  \\
	&=&\sum _{x \leq p\leq 2x}  \frac{1}{4 \sqrt{p}} \sum_{\substack{p-2 \sqrt{p}\leq n\leq p+2\sqrt{p}\\ n=p_1p_2} } \left (\frac{1}{n} \sum_{\gcd(m,n)=1,} \sum_{0 \leq r \leq n-1} \chi ((mT-P)r ) \right ) \nonumber \\
	&=& \delta_0(E,P)\sum_{x \leq p\leq 2x}  \frac{1}{4 \sqrt{p}} \sum_{\substack{p-2 \sqrt{p}\leq n\leq p+2\sqrt{p}\\ n=p_1n_2} } \frac{1}{n} \sum_{\gcd(m,n)=1} 1 \nonumber \\
	& & \qquad + \sum_{x \leq p\leq 2x}  \frac{1}{4 \sqrt{p}} \sum _{\substack{p-2 \sqrt{p}\leq n\leq p+2\sqrt{p}\\ n=p_1n_2} } \frac{1}{n} \sum_{\gcd(m,n)=1,} 
	\sum_{1 \leq r \leq n-1} \chi ((mT-P)r ) \nonumber \\
	&=& \delta(E,P)M(x) + E(x) \nonumber,
	\end{eqnarray} 
	
	where $\delta(E,P)=\delta_0(E,P)\mathfrak{C}(E,P) \geq 0$ is the density depending on both the fixed elliptic curve $E:f(x,y)=0$, the fixed point $P$, and some correction factor $\mathfrak{C}(E,P) \geq 0$. \\

	The main term $M(x)$ is determined by a finite sum over the principal additive character \(\chi_0((mT-P)r )=1\), and the error term $E(x)$ is determined by a finite sum over the nontrivial additive characters \(\chi ((mT-P)r )\neq 1\).\\
	
	Applying Lemma \ref{lem16.1} to the main term, and Lemma \ref{lem16.3} to the error term yield
	\begin{eqnarray} \label{el47715}
	\sum _{x \leq p\leq 2x} \frac{1}{4 \sqrt{p}}\sum _{\substack{p-2 \sqrt{p}\leq n\leq p+2\sqrt{p}\\ n=p_1n_2} } \Psi_E (P)
	&=&\delta(E,P)M(x) + E(x) \\
	&\gg & \left ( \frac{x}{ \log x}+O \left (\frac{x}{\log^2 x} \right ) \right )+O\left (x^{1-\alpha}\right) \nonumber \\
	&\gg& \frac{x}{ \log x}\left (1 +O \left (\frac{1}{\log x} \right ) \right ) \nonumber.
	\end{eqnarray} 
	
	But for $\delta(E,P)>0$ and all large numbers $x > x_0$, the expression
	\begin{eqnarray} \label{el7740}
	\sum _{x \leq p\leq 2x} \frac{1}{4 \sqrt{p}}\sum _{\substack{p-2 \sqrt{p}\leq n\leq p+2\sqrt{p}\\ n=p_1n_2} } \Psi_E (P)
	&\gg& \frac{x}{ \log x}\left (1 +O \left (\frac{1}{\log x} \right ) \right ) \nonumber\\
	&>&0,
	\end{eqnarray} 
	contradicts the hypothesis  (\ref{el47703}). Therefore, the number of elliptic primitive primes $p\geq x $ in the short interval $[x,2x]$ for which the fixed elliptic curve of rank $\rk(E)\geq 1$ has a primitive point $P$ that generates the elliptic groups of orders $\# E(\mathbb{F}_p)=n=p_1n_2$, with $p_1=Q(n) \geq n^{\alpha}$,  and $\alpha>0$ is a small number, is infinite as $x \to \infty$.  \\
	
	Lastly, the number of elliptic primitive primes has the asymptotic formula
	\begin{eqnarray} \label{el7741}
	\pi(x,E,P)	&=&\sum_{ \substack{p \leq x \\ \#E(\mathbb{F}_p)=n \text{ and } \ord_E(P)=n}} 1 \nonumber \\
	&\geq &\sum_{p\leq x} \frac{1}{4 \sqrt{p}} \sum_{\substack{p-2 \sqrt{p}\leq n\leq p+2\sqrt{p}\\ n=p_1n_2} } \Psi_E(P)     \\
	&\gg& \frac{x}{ \log x}\left (1 +O \left (\frac{1}{\log x} \right ) \right ) \nonumber,
	\end{eqnarray} 
	which is obtained from summation of the elliptic primitive primes density function over the interval $[1,x]$.   
\end{proof}

\section{Problems}
\begin{exe} \normalfont \normalfont 
Prove the Lange-Trotter for elliptic curves of rank $\rk(E)=0$.
\end{exe}

\begin{exe} \normalfont
	Let $E:f(x,y)=0$ be a nonsingular elliptic curve,  let $P \in E(\mathbb{Q})$ be a point of infinite order, and let $\delta(E,P)>0$ be the density of elliptic primitive primes. Show that the sum $$\sum_{p\leq x, \text{E primitive}} \frac{1}{p}=\delta(E,P)\log \log x +\beta(E,P)+O(1/\log x),$$ where $\beta(E,P) \in \mathbb{R}$ is a constant. The term $E$ primitive refers to the relation $<P>= E(\mathbb{F}_p)$.
\end{exe}

\begin{exe} \normalfont
	Let $E:y^2=x^3-x-1$, which is a nonsingular elliptic curve of rank $\rk(E)=1$, and let $\delta(E,P)=.4401473667 \ldots $ be the average density of elliptic primitive primes with respect to the point a
	$P$ of infinite order. Compute or estimate the elliptic Merten constant $$\beta(E,P)=\sum_{p\leq x, \text{E primitive}} \frac{1}{p} -\delta(E,P)\log \log x $$
	for $p \leq 10^3$. The term $E$ primitive refers to the relation $<P>= E(\mathbb{F}_p)$.
\end{exe}

\begin{exe} \normalfont
Let $E:y^2=x^3+1$; a nonsingular curve of rank $\rk(E) =0$. Classify the primes $p \geq 2$ for which the elliptic group $E(\mathbb{F}_p) \cong \mathbb{Z}_n$ is cyclic.
\end{exe}

\begin{exe} \normalfont
Let $E:y^2=x^3+3$; a nonsingular curve of rank $\rk(E) \geq 1$. Classify the primes $p \geq 2$ for which the elliptic group $E(\mathbb{F}_p) \cong \mathbb{Z}_n$ is cyclic.
\end{exe}

\begin{exe} \normalfont Show that the number of primes $p\geq 2$ such that $\langle P\rangle=E(\mathbb{Z}_p)$, but $\langle P\rangle \ne E(\mathbb{Z}_{p^2})$  is finite.
\end{exe}

\begin{exe} \normalfont Show that $\langle P\rangle=E(\mathbb{Z}_p)$, and $\langle P\rangle = E(\mathbb{Z}_{p^2})$ imply that $\langle P\rangle = E(\mathbb{Z}_{p^k})$ for all $k \geq 1$.
\end{exe}

%CHAPTERCHAPTER17171717
\chapter{Elliptic Groups of Prime Orders} \label{c17}
The applications of elliptic curves in cryptography demands elliptic groups of certain orders $n=\#E(\mathbb{F}_p)$, and certain factorizations of the integers $n$. The extreme cases have groups of $\mathbb{F}_p$-rational points $E(\mathbb{F}_p)$ of prime orders, and small multiples of large primes. \\

\begin{conj}[Koblitz] \label{conj800.1}  
	Let $E:f(X,Y)=0$ be an elliptic curve  of of discriminant $\Delta \ne 0$ defined over the integers $\mathbb{Z}$ which is not $\mathbb{Q}$-isogenous to a curve with nontrivial $\mathbb{Q}$-torsion and does not have CM. Then, 
	\begin{eqnarray}
	\pi(x,E)&=& \#\{ p \leq x: p \not | \Delta \text{ and } \#E(\mathbb{F}_p)=\text{prime}\} \\ 
	&=&\frac{x}{\log^2x}\prod_{p\geq 2} \left (1 -\frac{p^2-p-1}{(p-1)^3(p+1)}\right )+O\left ( \frac{x}{\log^3x}\right) \nonumber,
	\end{eqnarray}  
for large $x\geq 1$.
\end{conj}

The product expression appearing in the above formula is basically the average density of prime orders, some additional details are given in Section \ref{sec8}. A result for groups of prime orders generated by primitive points is proved here. Let
\begin{equation}
\pi(x,E)= \#\{ p \leq x: p \not | \Delta \text{ and } d_E^{-1}\cdot\#E(\mathbb{F}_p)=\text{prime}\}.
\end{equation} 
The parameter $d_E\geq 1$ is a small integer defined in Section \ref{sec77}.\\

\begin{thm} \label{thm800.1}   
	Let $E:f(X,Y)=0$ be an elliptic curve over the rational numbers $\mathbb{Q}$ of rank $\rk(E(\mathbb{Q})>0$. Then, as $x \to \infty$,
	\begin{equation}
	\pi(x,E)\geq \delta(d_E,E)\frac{x}{\log^3 x} \left (1+O\left ( \frac{x}{\log x} \right )\right) ,
	\end{equation}  	
	where $\delta(d_E,E)$ is the density constant.
\end{thm}

The proof of this result is split into several parts. The next sections are intermediate results. The proof of Theorem \ref{thm800.1} is assembled in the    penultimate section, and the last section has examples of elliptic curves with infinitely many elliptic groups $E(\mathbb{F}_p)$ of prime orders $n$.\\ 

\section{Evaluation Of The Main Term}
A lower bound for the main term $M(x)$ in the proof of Theorem \ref{thm800.1} is evaluated here. \\

\begin{lem} \label{lem800.1}
	Let \(x\geq 1\) be a large number, and let $p \in [x,2x]$ be prime. Then, 
	\begin{equation} \label{800-100}
	\sum_{x \leq p \leq 2x} \frac{1}{4 \sqrt{p} }  \sum_{p-2 \sqrt{p}\leq n\leq p+2\sqrt{p} } \frac{\Lambda(n)}{\log   n}\sum_{\gcd(m,n)=1}\frac{1}{n} \gg \frac{x}{\log^3 x} \left (1+  O  \left ( \frac{x}{\log x}  \right ) \right ).
	\end{equation} 
\end{lem}

\begin{proof} The Euler phi function $\varphi(n)=\#\{1\leq m<n:\gcd(m,n)=1\} $, has the lower bound $\varphi(n)/n \gg 1/ \log n$ for any integer $n \geq 1$, better estimates are proved in \cite[p.\ 118]{DL12}, \cite[p.\ 55 ]{TG15}, et cetera. Thus, main term can be rewritten in the form
	\begin{eqnarray} \label{800-110}
	M(x)&=&\sum_{x \leq p \leq 2x} \frac{1}{4 \sqrt{p} }  \sum_{p-2 \sqrt{p}\leq n\leq p+2\sqrt{p}} \frac{\Lambda(n)}{\log n} \sum_{\gcd(m,n)=1}\frac{1}{n} \nonumber\\
	&=&\sum_{x \leq p \leq 2x} \frac{1}{4 \sqrt{p} }  \sum_{p-2 \sqrt{p}\leq n\leq p+2\sqrt{p}} \frac{\Lambda(n)}{\log n}\frac{\varphi(n)}{n}  \\
	&\gg&\sum_{x \leq p \leq 2x} \frac{1}{4 \sqrt{p} }  \sum_{p-2 \sqrt{p}\leq n\leq p+2\sqrt{p}} \frac{\Lambda(n)}{\log^2 n} \nonumber \\
	&\gg& \sum_{x \leq p \leq 2x} \frac{1}{4 \sqrt{p} } \cdot \frac{1}{\log^2 p} \sum_{p-2 \sqrt{p}\leq n\leq p+2\sqrt{p}} \Lambda(n) .
	\end{eqnarray} 
	Now, observe that as the prime $p \in [x,2x]$ varies, the number of short intervals $[p-2 \sqrt{p},p+2\sqrt{p}]$ is the same as the number of primes in the interval $[x,2x]$, namely, 
	\begin{equation}\label{800-140}
	\pi(2x)-\pi(x)=\frac{x}{\log x} \left (1+ O \left (\frac{1}{\log x}  \right )\right )>\frac{x}{2\log x} 	
	\end{equation}
	for large $x\geq1$. By Theorem \ref{thm4.1}, the finite sum over the short interval satisfies
	\begin{equation}\label{800-115}
	\sum_{p-2 \sqrt{p}\leq n\leq p+2\sqrt{p}} \Lambda(n)>\frac{4\sqrt{p}}{2},
	\end{equation} 
	with $O \left (x\log^{-C}x  \right )$ exceptions $p \in [x,2x]$, where $1<C<4$. 
	Take $C=3+\varepsilon$, where $\varepsilon >0$ is a small number. Then, the number of exceptions is small in comparison to the number of intervals $\pi(2x)-\pi(x) > x/2\log x$ for large $x \geq 1$. Hence, an application of this Theorem yields
	\begin{eqnarray} \label{800-120}
	M(x)&\gg&\sum_{x \leq p \leq 2x} \frac{1}{4 \sqrt{p} } \cdot \frac{1}{\log^2 p} \sum_{p-2 \sqrt{p}\leq n\leq p+2\sqrt{p}} \Lambda(n) \nonumber \\
	&\gg&\sum_{x \leq p \leq 2x} \frac{1}{4 \sqrt{p} } \cdot \frac{1}{\log^2 p} \cdot \left( \frac{4\sqrt{p}}{2} \right ) +O \left (\frac{x}{\log^{3+\varepsilon} x}  \right )\nonumber \\
	&\gg&\frac{1}{\log^2 x}\sum_{x \leq p \leq 2x} 1 +O \left (\frac{x}{\log^{3+\varepsilon} x}  \right ) \\
	&\gg&\frac{1}{\log^2 x}\cdot  \frac{x}{\log x} \left (1+ O \left (\frac{1}{\log^{\varepsilon} x}  \right )\right )  \nonumber, 
\end{eqnarray} 
where all the errors terms are absorbed into one term. 
\end{proof}

\begin{rem} \normalfont The exceptional intervals $[p-2 \sqrt{p},p+2\sqrt{p}]$, where $p \in [x,2x]$, contain fewer primes, that is,
\begin{equation}\label{800-210}
\sum_{p-2 \sqrt{p}\leq n\leq p+2\sqrt{p}} \Lambda(n)=o(\sqrt{p}).
\end{equation}
This shortfalls is accounted for in the correction term 
\begin{eqnarray}\label{800-213}
\sum_{x \leq p \leq 2x} \frac{1}{4 \sqrt{p} } \cdot \frac{1}{\log^2 p} \sum_{p-2 \sqrt{p}\leq n\leq p+2\sqrt{p}} \Lambda(n)&=&o \left(\sum_{x \leq p \leq 2x} \frac{1}{4 \sqrt{p} } \cdot \frac{1}{\log^2 p} \cdot \sqrt{p} \right ) \nonumber \\ &=&O \left (\frac{x}{\log^C x}  \right )
\end{eqnarray}
in the previous calculation. 
\end{rem}

\section{Estimate For The Error Term}
The analysis of an upper bound for the error term $E(x)$, which occurs in the proof of Theorem \ref{thm800.1}, is split into two parts. The first part in Lemma \ref{lem800.2} is an estimate for the triple inner sum. And the final upper bound is assembled in Lemma \ref{lem800.3}. \\

\begin{lem} \label{lem800.2}
	Let $E$ be a nonsingular elliptic curve over rational number, let $P \in E(\mathbb{Q})$ be a point of infinite order. Let \(x\geq 1\) be a large number. For each prime $p\geq 3$, fix a primitive point $T$, and suppose that $P \in E(\mathbb{F}_p)$ is not a primitive point for all primes $p\geq 2$, then
	
	\begin{equation} \label{800-220}	 
	\sum_{p-2 \sqrt{p}\leq n\leq p+2\sqrt{p} }\frac{\Lambda(n)}{n \log n}\sum_{\gcd(m,n)=1}  
	\sum_{ 1 \leq r <n} \chi((mT-P)r) \leq 2 .	\end{equation} 
\end{lem}

\begin{proof} Let $\log_T:E(\mathbb{F}_p) \rightarrow \mathbb{Z}_n$ be the discrete logarithm function with respect to the fixed primitive point $T$, defined by $\log_T(mT)=m$, $\log_T(P)=k$, and $\log_T(\mathcal{O})=0$. Then, the nontrivial additive character evaluates to
	\begin{equation}
	\chi(rmT)=e^{\frac{i2 \pi}{n}\log_T(rmT)}=e^{i2 \pi rm/n},
	\end{equation} and 
\begin{equation}
\chi(-rP)=e^{\frac{i2 \pi}{n}\log_T(-rP)}=e^{-i2 \pi rk/n},
\end{equation}
	respectively. To derive a sharp upper bound, rearrange the inner double sum as a product 
	\begin{eqnarray} \label{800-230}
	T(p)&=&\sum_{p-2 \sqrt{p}\leq n\leq p+2\sqrt{p}}\frac{\Lambda(n)}{n \log n}\sum_{\gcd(m,n)=1,}  
	\sum_{ 1 \leq r <n} \chi((mT-P)r) \nonumber \\
	&= &\sum_{p-2 \sqrt{p}\leq n\leq p+2\sqrt{p}} \frac{\Lambda(n)}{n \log n}\sum_{ 1 \leq r <n} \chi(-rP)\sum_{\gcd(m,n)=1} \chi(rmT)\\
	&= &\sum_{p-2 \sqrt{p}\leq n\leq p+2\sqrt{p}}\frac{\Lambda(n)}{n \log n}\left (\sum_{ 1 \leq r <n} e^{-i2 \pi rk/n} \right ) \left (\sum_{\gcd(m,n)=1} e^{i2 \pi rm/n} \right ) \nonumber.
	\end{eqnarray}
The hypothesis $mT-P\ne \mathcal{O}$ for $m\geq 1$ such that $\gcd(m,n)=1$ implies that the two inner sums in (\ref{800-230}) are complete geometric series, except for the terms for $m=0$ and $r=0$. Moreover, since $n\geq 2$ is a prime, the two geometric sums have the exact evaluations
	\begin{equation} \label{800-240}
	U_n=\sum_{ 0<r\leq n-1} e^{-i2 \pi rk/n}=-1    \qquad \text{ and } \qquad V_n=\sum_{\gcd(m,n)=1} e^{i2 \pi rm/n}=-1
	\end{equation} 
	for $1 \leq k <n$, and $1\leq r < n$ respectively. 
\end{proof}
	
\begin{rem} \normalfont The evaluation of the finite sum $U_n=U_n(k)=-1$ is independent of $k\geq 1$ because $n \geq 2$ is prime and $1 \leq k <n$. Similarly, the evaluation of the finite sum $V_n=V_n(r)=-1$ is independent of $r\geq 1$ because $n \geq 2$ is prime and $1 \leq r <n$. \\ 
	
Therefore, it reduces to
\begin{eqnarray} \label{800-233}
T(p)&=&\sum_{p-2 \sqrt{p}\leq n\leq p+2\sqrt{p}}\frac{\Lambda(n)}{n \log n} \nonumber \\
&\leq & \frac{1}{\log p}\sum_{ p-2 \sqrt{p}\leq n\leq p+2\sqrt{p}}\frac{\Lambda(n)}{n } \\
&\leq & \frac{1}{ \log p}\sum_{ n\leq p+2\sqrt{p}}\frac{\Lambda(n)}{n} \nonumber\\
&\leq  & 2 \nonumber.
\end{eqnarray}
The last line uses the finite sum $\sum_{ n\leq x}\Lambda(n)/n\ll \log x$, refer to \cite[Theorem 2.7]{MV07} for additional information. 
\end{rem}

\begin{lem} \label{lem800.3}
	Let $E$ be a nonsingular elliptic curve over rational number, let $P \in E(\mathbb{Q})$ be a point of infinite order. For each large prime $p\geq 3$, fix a primitive point $T$, and suppose that $P \in E(\mathbb{F}_p)$ is not a primitive point for all primes $p\geq 2$, then
	
	\begin{equation} \label{800-400}
	\sum_{x\leq p\leq 2x}\sum_{p-2 \sqrt{p}\leq n\leq p+2\sqrt{p} }   \frac{\Lambda(n)}{n \log n}\sum_{\gcd(m,n)=1,} 
	\sum_{ 1 \leq r <n} \chi((mT-P)r) =O( x^{1/2} ).
	\end{equation} 
\end{lem}

\begin{proof} By assumption $P \in E(\mathbb{F}_p)$ is not a primitive point. Hence, the linear equation $mT-P= \mathcal{O}$ has no solution $m \in \{m:\gcd(m,n)=1\}$. This implies that the discrete logarithm $\log_T(mT-P)\ne0$, and $ \sum_{ 1 \leq r <n} \chi((mT-P)r) =-1$. This in turns yields   
	\begin{eqnarray} \label{800-410}
	|E(x)|&=&\left |\sum_{x \leq p \leq 2x } \frac{1}{4 \sqrt{p} } \sum_{p-2 \sqrt{p}\leq n\leq p+2\sqrt{p}}\frac{\Lambda(n)}{n\log n} \sum_{\gcd(m,n)=1,}  
	\sum_{ 1 \leq r <n} \chi((mT-P)r) \right |\nonumber \\
	&\leq &\sum_{x \leq p \leq 2x } \frac{1}{4 \sqrt{p} } \sum_{p-2 \sqrt{p}\leq n\leq p+2\sqrt{p} }\frac{\Lambda(n)}{n \log n} \sum_{\gcd(m,n)=1} 1 \\
	&\ll &\sum_{x \leq p \leq 2x } \frac{1}{4 \sqrt{p} } \sum_{p-2 \sqrt{p}\leq n\leq p+2\sqrt{p}}\frac{\Lambda(n)}{2 \log n} \nonumber \\
	&\ll& \frac{x}{\log^3 x}+O\left (\frac{x}{\log^4 x} \right ) \nonumber.
	\end{eqnarray}
	
	Here, the inequality $(1/n)\sum_{\gcd(m,n)=1}1=\varphi(n)/n \leq 1/2$ for all integers $n\geq 1$ was used in the second line. Hence, there is a nontrivial upper bound for the error term. To derive a sharper upper bound, take absolute value, and apply Lemma \ref{lem800.2} to the inner triple sum to obtain this:
\begin{eqnarray} \label{800-420}
|E(x)|&\leq&\sum_{x \leq p \leq 2x } \frac{1}{4 \sqrt{p} } \left |  \sum_{p-2 \sqrt{p}\leq n\leq p+2\sqrt{p} }\frac{\Lambda(n)}{n\log n} \sum_{\gcd(m,n)=1,}  
\sum_{ 1 \leq r <n} \chi((mT-P)r) \right |\nonumber \\
&\leq &\sum_{x \leq p \leq 2x } \frac{1}{4 \sqrt{p} } \left (2  \right ) \nonumber \\
&\leq &  \frac{1}{\sqrt{x} }\sum_{x \leq p \leq 2x }1 \\
&=&O\left (x^{1/2} \right ) \nonumber,
\end{eqnarray}
The last line uses the trivial estimate $\sum_{x \leq p \leq 2x }1\leq x$.
    \end{proof}

\section{Elliptic Divisors} \label{sec77}
The divisor $\text{div}(f)=\gcd(f(\mathbb{Z}))$ of a polynomial $f(x)$ is the greatest common divisor of all its values over the integers, confer \cite[p.\ 395]{FI10}. Basically, the same concept extents to the setting of elliptic groups of prime orders, but it is significantly more complex. 

\begin{dfn} Let $\mathcal{O}_{\mathcal{K}}$ be the ring of integers of a quadratic numbers field $\mathcal{K}$. The elliptic divisor is an integer $d_E \geq1$ defined by
\begin{equation}\label{800-590}
d_E=\gcd \left (\{ \#E(\mathbb{F}_p): p\geq 2 \text{ and }p \text{ splits in } \mathcal{O}_{\mathcal{K}} \}  \right ).
\end{equation}
\end{dfn}
Considerable works, \cite{CA05}, \cite{MG15}, \cite{IJ08}, \cite{JJ08}, have gone into determining the elliptic divisors for certain classes of elliptic curves.\\
 
\begin{thm} \label{thm800-20} {\normalfont (\cite[Proposition 1]{JJ08})}
The divisor of an elliptic curve $E:y^2=x^3+ax+b$ over the rational numbers $\mathbb{Q}$ with complex multiplication by $\mathbb{Q}(\sqrt{D})$ and conductor $N$ satisfies $d_E |24$. The complete list, with $c,m \in \mathbb{Z}-\{0\}$, is the following.\\

\begin{center}
\begin{tabular}{|c|c|c|c|c|c|}
\hline 
\rule[-1ex]{0pt}{2.5ex} $D$ & $(a,b)$ & $d_E$ & $D$ & $(a,b)$ & $d_E$ \\ 
	\hline 
	\rule[-1ex]{0pt}{2.5ex} $-3$ & $(0,m)$ & $1$ & $-7$ & $(-140c^2,-784c^3)$ & $4$ \\ 
	\hline 
	\rule[-1ex]{0pt}{2.5ex} $-3$ & $(0,m^2);(0,-27m^2)$ & $3$ & $-8$ & $(-30c^2,-56c^3)$ & $2$ \\ 
	\hline 
	\rule[-1ex]{0pt}{2.5ex} $-3$ & $(0,m^3)$ & $4$ & $-11$ &$(-1056c^2,-13552c^3)$  & $1$ \\ 
	\hline 
	\rule[-1ex]{0pt}{2.5ex} $-3$ & $(0,c^6);(0,27c^6)$ & $12$ & $-19$ &$(-608c^2,-5776c^3)$  & $1$ \\ 
	\hline 
	\rule[-1ex]{0pt}{2.5ex} $-4$ & $(m,0)$ & $2$ & $-43$ &$(-13760c^2,-621264c^3)$  & $1$ \\ 
	\hline 
	\rule[-1ex]{0pt}{2.5ex} $-4$ & $(m^2,0);(-m^2,0)$ & $4$ & $-67$ &$(-117920c^2,-15585808c^3)$  & $1$ \\ 
	\hline 
	\rule[-1ex]{0pt}{2.5ex} $-4$ & $(-c^4,0);(4c^4,0)$ & $8$ & $-163$ &$(-34790720c^2,-78984748304c^3)$  & $1$ \\ 
	\hline 
	\end{tabular}   
\end{center}
\end{thm}

\vskip .25 in

\section{Densities Expressions} \label{sec8}
The product expression appearing in Conjecture \ref{conj800.1}, id est,
\begin{equation}
P_0=\prod_{p\geq 2} \left (1 -\frac{p^2-p-1}{(p-1)^3(p+1)}\right ) \approx 0.505166168239435774,
\end{equation}  
is the basic the average density of prime orders, it was proved in \cite{KN88}, and very recently other proofs are given in \cite[Theorem 1 ]{BC11}, \cite{JN10}, \cite{LS14}, et alii. The actual density has a slight dependence on the elliptic curve $E$ and the point $P$. The determination of the dependence is classified into several cases depending on the torsion groups $E(\mathbb{Q})_{\text{tors}}$, and other parameters. \\

\begin{lem} \label{lem800.11} {\normalfont (\cite[Proposition 4.2]{ZD09})} Let $E:f(x,y)=0$ be a Serre curve over the rational numbers. Let $D$ be the discriminant of the numbers field 
$\mathbb{Q}(\sqrt{\Delta})$, where $\Delta$ is the discriminant of any Weierstrass model of $E$ over $\mathbb{Q}$. If $d_E=1$, then

\begin{equation} 
\delta(1,E)=	
\begin{cases}
\displaystyle \left (1+ \prod_{q|D} \frac{1}{q^3-2q^2-q+3}\right )\prod_{p\geq 2} \left (1 -\frac{p^2-p-1}{(p-1)^3(p+1)}\right ) 
& \text{ if } D \equiv 1 \bmod 4;\\
\displaystyle \prod_{p\geq 2} \left (1 -\frac{p^2-p-1}{(p-1)^3(p+1)}\right ) & \text{ if } D \equiv 0 \bmod 4.
\end{cases}
\end{equation}
\end{lem}

\section{Elliptic Brun Constant}
The upper bound 
\begin{eqnarray}\label{800-650}
\pi(x,E)&=& \#\{ p \leq x: p \not | \Delta \text{ and } \#E(\mathbb{F}_p)=\text{prime}\} \nonumber \\
&\ll& \frac{x}{(\log x)(\log \log \log x)}
\end{eqnarray}
for elliptic curves with complex multiplication was proved in \cite[Proposition 7]{CA05}. The same upper bound for elliptic curves without complex multiplication was proved in \cite[Theorem 1.3]{ZD08}. An improved version for all elliptic curves follows easily from methods used here.\\

\begin{lem}  \label{800-22} For any large number $x \geq 1$ and any elliptic curves $E:f(X,Y)=0$ of discriminant $\Delta\ne 0$,
	\begin{equation}\label{800-700}
	\pi(x,E) \leq \frac{6x}{\log^2x}.
	\end{equation}
\end{lem}

\begin{proof} The number of such elliptic primitive primes has the asymptotic formula
	\begin{eqnarray} \label{800-860}
	\pi(x,E)&=&\sum_{ \substack{p \leq x \\ \ord_E(P)=n \text{ prime}}} 1
	\\
	&=&\sum _{p\leq x} \frac{1}{4 \sqrt{p}}\sum _{p-2\sqrt{p}\leq n\leq p+2\sqrt{p}}\frac{\Lambda(n)}{\log n}\cdot  \Psi_E (P) \nonumber,
	\end{eqnarray} 
	where $\Psi_E(P)$ is the charateritic function of primitive points $P \in E(\mathbb{Q})$. This is obtained from the summation of the elliptic primitive primes density function over the interval $[1,x]$, see (\ref{800-500}).\\
	
	Since $\Psi_E (P)=0,1$, the previous equation has the upper bound
	\begin{eqnarray} \label{800-870}
	\pi(x,E)
	&\leq &\sum_{p\leq x} \frac{1}{4 \sqrt{p}}\sum _{p-2\sqrt{p}\leq n\leq p+2\sqrt{p}}\frac{\Lambda(n)}{\log n}\\
	&\leq &\sum_{p\leq x} \frac{1}{4 \sqrt{p}}\sum _{p-2\sqrt{p}\leq q\leq p+2\sqrt{p}}1\nonumber,
	\end{eqnarray} 
	where $q$ ranges over the primes in the short interval $[p-2\sqrt{p}, p+2\sqrt{p}]$. The inner sum is estimated using either the explicit formula or Brun-Titchmarsh theorem. The later result states that the number of primes $p$ in the short interval $[x,x+4\sqrt{x}]$ satisfies the inequality
	\begin{equation} \label{800-40}
	\pi(x+4\sqrt{x})-\pi(x) \leq \frac{3 \cdot 4\sqrt{x}}{ \log x},
	\end{equation}
	see \cite[p.\  167]{IK04}, \cite[Theorem 3.9]{MV07}, and \cite[p.\  83]{TG15}, and similar references. Replacing (\ref{800-40}) into (\ref{800-870}) yields
	
	\begin{eqnarray} \label{800-880}
	\sum_{p\leq x} \frac{1}{4 \sqrt{p}}\sum _{p-2\sqrt{p}\leq q\leq p+2\sqrt{p}}1 &\leq & \sum_{p\leq x} \frac{1}{4 \sqrt{p}} \left ( \frac{3 \cdot 4\sqrt{p}}{\log p} \right )\\
	&\leq & 3\sum_{p\leq x} \frac{1}{\log p}\nonumber\\ 
	&\leq&6 \frac{x}{\log^2 x}\nonumber.
	\end{eqnarray} 
	The last inequality follows by partial summation and the prime number theorem $\pi(x)= x/\log x+O(x/\log^2 x)$.
\end{proof}

This result facilitates the calculations of a new collection of constants associated with elliptic curves.\\

\begin{cor} For any elliptic curve $E$, the elliptic Brun constant
	\begin{equation}\label{800-760}
	\sum_{\substack{p \geq 2\\ \#E(\mathbb{F}_p)=\text{prime}}} \frac{1}{p}< \infty
	\end{equation}
	converges.
\end{cor}

\begin{proof} Use the prime counting measure $\pi(x,E)\leq 6 x/\log^2 x+O(x/\log^3 x)$ in Lemma \ref{800-22} to evaluate the infinite sum
	\begin{eqnarray}\label{800-333}
	\sum_{\substack{p \geq 2\\ \#E(\mathbb{F}_p)=\text{prime}}} \frac{1}{p} &=& \int_2^{\infty}\frac{1}{t} d \pi(t,E) \nonumber \\
	&=&O(1)+\int_2^{\infty}\frac{\pi(t,E)}{t^2} d t \\
	&<& \infty \nonumber 
	\end{eqnarray}
	as claimed.
\end{proof} 

\subsection{Examples} 
The elliptic Brun constants and the numerical data for the prime orders of a few elliptic curves were compiled. The last example shows the highest density of prime orders. Accordingly, it has the largest constant.\\   

\begin{exa}  \normalfont The nonsingular Bachet elliptic curve $E: y^2=x^3+2$ over the rational numbers has complex multiplication by $\mathbb{Z}[\rho]$, and nonzero rank $\rk(E)=1$. The data for $p \leq 1000$ with prime orders $n$ are listed on the Table \ref{t801}; 
	and the elliptic Brun constant is
	\begin{equation}\label{800-41}
	\sum_{\substack{p \geq 2\\ \#E(\mathbb{F}_p)=\text{prime}}} \frac{1}{p}=.520067922 \ldots.
	\end{equation}
	
	\begin{table} \label{t801}
		\begin{center}
			\begin{tabular}{||c|c|c|c|c|c|c|c|c|c|c||}
				\hline 
				\rule[-1ex]{0pt}{2.5ex} $p$ & 3 & 13 & 19 & 61 & 67 & 73 & 139 & 163 & 211 & 331 \\ 
				\hline 
				\rule[-1ex]{0pt}{2.5ex} $n$ & 3 & 19 & 13 & 61 & 73 & 81 & 163 & 139 & 199 & 331 \\ 
				\hline 
				\rule[-1ex]{0pt}{2.5ex} $p$ & 349 & 541 & 547 & 571 & 613 & 661 & 757 & 829 & 877 &  \\ 
				\hline 
				\rule[-1ex]{0pt}{2.5ex} $n$ & 313 & 571 & 571 & 541 & 661 & 613 & 787 & 823 & 937 &  \\ 
				\hline 
			\end{tabular} 
		\end{center}
		\caption{\label{859} Prime Orders $n=\#E(\mathbb{F}_p)$  modulo $p$ for $y^2=x^3+2$.}
	\end{table}	
	
\end{exa}

\begin{exa} \normalfont The nonsingular elliptic curve $E: y^2=x^3+ 6x-2$ over the rational numbers has no complex multiplication and zero rank $\rk(E)=0$. The data for $p \leq 1000$ with prime orders $n$ are listed on the Table \ref{t804}; 
	and the elliptic Brun constant is
	\begin{equation}\label{800-72}
	\sum_{\substack{p \geq 2\\ \#E(\mathbb{F}_p)=\text{prime}}} \frac{1}{p}=.186641187 \ldots.
	\end{equation}
	
	\begin{table} \label{t804}
		\begin{center}
			\begin{tabular}{||c|c|c|c|c|c|c|c|c|c|c||}
				\hline 
				\rule[-1ex]{0pt}{2.5ex} $p$ & 3 & 7 & 97 & 103 & 181 & 271 & 313 & 367 & 409 & 487 \\ 
				\hline 
				\rule[-1ex]{0pt}{2.5ex} $n$ & 3 & 7 & 97 & 107 & 163 & 293 & 331 & 383& 397 & 499 \\ 
				\hline 
				\rule[-1ex]{0pt}{2.5ex} $p$ & 883 &  967&  & &  &  &  &  &  &  \\ 
				\hline 
				\rule[-1ex]{0pt}{2.5ex} $n$ & 853 & 941 &  & &  &  & &  &  &  \\ 
				\hline 
			\end{tabular} 
		\end{center}
		\caption{\label{869} Prime Orders $n=\#E(\mathbb{F}_p)$  modulo $p$ for $y^2=x^3+6x-2$.}
	\end{table}	
	
\end{exa}

\begin{exa} \normalfont The nonsingular elliptic curve $E: y^2=x^3-x$ over the rational numbers has complex multiplication by $\mathbb{Z}[\rho]$, and nonzero rank $\rk(E)=0$. The data for $p \leq 1000$ with prime orders $n/4$ are listed on the Table \ref{t806}; 
	and the elliptic Brun constant is
	\begin{equation}\label{800-73}
	\sum_{\substack{p \geq 2\\ \#E(\mathbb{F}_p)/4=\text{prime}}} \frac{1}{p}=.549568584 \ldots.
	\end{equation}
	
	\begin{table} \label{t806}
		\begin{center}
			\begin{tabular}{||c|c|c|c|c|c|c|c|c|c|c||}
				\hline 
				\rule[-1ex]{0pt}{2.5ex} $p$ & 5 & 7 & 11 & 19 & 43 & 67 & 163 & 211 & 283 & 331 \\ 
				\hline 
				\rule[-1ex]{0pt}{2.5ex} $n$ & 2 & 2 & 3 & 5 & 11 & 17 & 41 & 53& 71 & 83 \\ 
				\hline 
				\rule[-1ex]{0pt}{2.5ex} $p$ & 523 & 547&691  &787 &907  &  &  &  &  &  \\ 
				\hline 
				\rule[-1ex]{0pt}{2.5ex} $n$ & 131 & 137 &173  &197 &227  &  & &  &  &  \\ 
				\hline 
			\end{tabular}  
		\end{center}
		\caption{Prime Orders $n/4=\#E(\mathbb{F}_p)/4$  modulo $p$ for $y^2=x^3-x$.}
	\end{table}	
	
\end{exa}

\begin{exa} \normalfont The nonsingular elliptic curve $E: y^2=x^3-x$ over the rational numbers has complex multiplication by $\mathbb{Z}[\rho]$, and nonzero rank $\rk(E)=0$. The data for $p \leq 1000$ with prime orders $n/8$ are listed on the Table \ref{t808}; 
	and the elliptic Brun constant is
	\begin{equation}\label{800-77}
	\sum_{\substack{p \geq 2\\ \#E(\mathbb{F}_p)/8=\text{prime}}} \frac{1}{p}=.2067391731 \ldots.
	\end{equation}
	
	\begin{table} \label{t808}
		\begin{center}
			\begin{tabular}{||c|c|c|c|c|c|c|c|c|c|c||}
				\hline 
				\rule[-1ex]{0pt}{2.5ex} $p$&17 & 23 &29 &37 & 53 & 101 & 103 &109&149 &151  \\ 
				\hline 
				\rule[-1ex]{0pt}{2.5ex} $n$ & 2 & 3 & 5 & 5 & 5 & 13 & 13 & 13 & 17 & 19 \\ 
				\hline 
				\rule[-1ex]{0pt}{2.5ex} $p$ &157 &277  &293  &317  &389  &487 &541  &631  &661  &701  \\ 
				\hline 
				\rule[-1ex]{0pt}{2.5ex} $n$ &17  &37  &37  &41  &37  &53  &61  &73  &79  &89  \\ 			\hline 
				\rule[-1ex]{0pt}{2.5ex} $p$ &757 & 773 &797  &821  &823  &829 &853  &  &  &  \\ 
				\hline 
				\rule[-1ex]{0pt}{2.5ex} $n$ & 97 &101  &97  &109&103  &97  &101 &  &  &  \\ 
				\hline 
				\hline 
			\end{tabular} 
		\end{center}
		\caption{Prime Orders $n/8=\#E(\mathbb{F}_p)/8$  modulo $p$ for $y^2=x^3-x$.}
	\end{table}	
	
\end{exa}

\begin{exa} \normalfont The nonsingular Bachet elliptic curve $E: y^2=x^3+1$ over the rational numbers has complex multiplication by $\mathbb{Z}[\rho]$, and nonzero rank $\rk(E)=?1$. The data for $p \leq 1000$ with prime orders $n/12$ are listed on the Table \ref{t820}; 
	and the elliptic Brun constant is
	\begin{equation}\label{800-71}
	\sum_{\substack{p \geq 2\\ \#E(\mathbb{F}_p)/12=\text{prime}}} \frac{1}{p}=.5495685884 \ldots.
	\end{equation}
	
	\begin{table} \label{t820}
		\begin{center}
			\begin{tabular}{||c|c|c|c|c|c|c|c|c|c|c||}
				\hline 
				\rule[-1ex]{0pt}{2.5ex} $p$ & 31 & 43 & 59 & 67 & 73 & 79 & 97 & 103 &131 & 139 \\ 
				\hline 
				\rule[-1ex]{0pt}{2.5ex} $n$ & 3 & 3 &5 & 7 & 7 & 7 & 7 &7 & 11 & 13 \\ 
				\hline 
				\rule[-1ex]{0pt}{2.5ex} $p$ & 151 & 163 & 181 & 199 & 227 & 241 & 337 & 367 & 379 &409  \\ 
				\hline 
				\rule[-1ex]{0pt}{2.5ex} $n$ & 13 &13 &13 & 19 & 19 & 19 & 31 & 31 & 31 &31  \\ 
				\hline 
				\rule[-1ex]{0pt}{2.5ex} $p$ & 421 &443 & 463 & 487 & 491 & 523 & 563 & 709 & 751 &787  \\ 
				\hline 
				\rule[-1ex]{0pt}{2.5ex} $n$ & 37 &37 &37 &37 &41 &43 & 47 & 61 & 67 &61  \\ 
				\hline 
				\rule[-1ex]{0pt}{2.5ex} $p$ & 823 & 829 & 859& 883 & 907 & 947 & 967 & 991 &  &  \\ 
				\hline 
				\rule[-1ex]{0pt}{2.5ex} $n$ & 73 &73 &67 & 73 & 79 & 79 & 79 & 79 &  &  \\ 
				\hline   
			\end{tabular} 
		\end{center}
		\caption{ Prime Orders $n/12=\#E(\mathbb{F}_p)/12$  modulo $p$ for $y^2=x^3+1$.}
	\end{table}	
	
\end{exa}	

\section{Prime Orders $n$} \label{sec7}
The characteristic function for primitive points in the group of points $E(\mathbb{F}_p)$ of an elliptic curve $E:f(X,Y)=0$ has the representation
\begin{equation} \label{800-29}
\Psi_E(P)=
\left \{\begin{array}{ll}
1 & \text{ if } \ord_E (P)=n,  \\
0 &  \text{ if } \ord_E (P) \ne n. \\
\end{array} \right.
\end{equation} 
The parameter $n=\# E(\mathbb{F}_p)$ is the size of the group of points, and the exact formula for $\Psi_E (P)$ is given in Lemma \ref{lem3.7}.\\ 

Since each order $n$ is unique, the weighted sum
\begin{equation} \label{800-500}
\frac{1}{4 \sqrt{p}}\sum _{p-2\sqrt{p}\leq n\leq p+2\sqrt{p}}\frac{\Lambda(n)}{\log n} \cdot \Psi_E (P) 
=\left \{\begin{array}{ll}
\displaystyle \frac{1}{4 \sqrt{p}} \cdot \frac{\Lambda(n)}{\log n} & \text{ if } \ord_E (P)=n \text{ and } n=q^k,  \\
0 &  \text{ if } \ord_E (P) \ne n \text{ or } n\ne q^k, \\
\end{array} \right.
\end{equation}
where $n=q^k,k\geq 1$, is a prime power, is a discrete measure for the density of elliptic primitive primes $p \geq 2$ such that $P \in E(\mathbb{\overline{Q}})$ is a primitive point of prime power order $n \in [p-2\sqrt{p}, p+2\sqrt{p}]$.\\

\begin{proof} \text{(Theorem \ref{thm800.1}).} Let $\langle P \rangle= E(\mathbb{F}_p)$ for at least one large prime $p \leq x_0$, and let $x \geq x_0 \geq 1$ be a large number. Suppose that \(P\not \in E(\mathbb{Q})_{\text{tors}} \) is not a primitive point of prime order $\ord_E(P)=n$ in $E(\mathbb{F}_p)$ for all primes \(p\geq x\). Then, the sum of the elliptic primes measure over the short interval \([x,2x]\) vanishes. Id est, 
\begin{equation} \label{800-510}
0=\sum _{x \leq p\leq 2x} \frac{1}{4 \sqrt{p}}\sum _{p-2\sqrt{p}\leq n\leq p+2\sqrt{p}} \frac{\Lambda(n)}{\log n} \cdot\Psi_E (P).
\end{equation}
Replacing the characteristic function, Lemma \ref{lem3.7}, and expanding the nonexistence equation (\ref{800-510}) yield
\begin{eqnarray} \label{800-520}
0&=&\sum _{x \leq p\leq 2x}  \frac{1}{4 \sqrt{p}}\sum _{p-2\sqrt{p}\leq  n\leq p+2\sqrt{p}}
 \frac{\Lambda(n)}{\log n} \cdot \Psi_E (P) \nonumber\\
&=&\sum _{x \leq p\leq 2x}  \frac{1}{4 \sqrt{p}} \sum _{p-2\sqrt{p}\leq  n\leq p+2\sqrt{p}}  \frac{\Lambda(n)}{\log n} \left (\sum_{\gcd(m,n)=1}\frac{1}{n} 
\sum_{ 0 \leq r \leq n-1} \chi ((mT-P)r ) \right ) \nonumber \\
&=&\delta(d_E,E)\sum _{x \leq p\leq 2x}  \frac{1}{4 \sqrt{p}} \sum _{p-2\sqrt{p}\leq  n\leq p+2\sqrt{p}} \frac{\Lambda(n)}{\log n}\sum_{\gcd(m,n)=1}\frac{1}{n}
\\
& & \qquad+ \sum _{x \leq p\leq 2x}  \frac{1}{4 \sqrt{p}} \sum _{p-2\sqrt{p}\leq  n\leq p+2\sqrt{p}}  \frac{\Lambda(n)}{ \log n}\sum_{\gcd(m,n)=1} 
\frac{1}{n}\sum_{ 1 \leq r \leq n-1} \chi ((mT-P)r )\nonumber \\
&=&\delta(d_E,E)M(x) + E(x) \nonumber,
\end{eqnarray} 
where $\delta(d_E,E)\geq0$ is a constant depending on both the fixed elliptic curve $E:f(X,Y)=0$ and the integer divisor $d_E$. \\

The main term $M(x)$ is determined by a finite sum over the principal character \(\chi =1\), and the error term $E(x)$ is determined by a finite sum over the nontrivial multiplicative characters \(\chi \neq 1\).\\

Applying Lemma \ref{lem800.1} to the main term, and Lemma \ref{lem800.3} to the error term yield
\begin{eqnarray} \label{800-530}
F(x)&=&\sum _{x \leq p\leq 2x} \frac{1}{4 \sqrt{p}}\sum _{p-2\sqrt{p}\leq n\leq p+2\sqrt{p}} \frac{\Lambda(n)}{\log n}\cdot \Psi_E (P) \nonumber \\
&=&\delta(d_E,E)M(x) + E(x) \\
&\gg & \delta(d_E,E)\frac{x}{ \log^3 x} \left (1 +O \left (\frac{x}{\log x} \right ) \right )+O\left (x^{1/2 }\right) \nonumber \\
&\gg&  \delta(d_E,E)\frac{x}{ \log^3 x} \left (1 +O \left (\frac{x}{\log x} \right ) \right ) \nonumber.
\end{eqnarray} 
But, if $\delta(d_E,E)>0$, the expression 
\begin{eqnarray} \label{800-540}
F(x)&=&\sum _{x \leq p\leq 2x} \frac{1}{4 \sqrt{p}}\sum _{p-2\sqrt{p}\leq n\leq p+2\sqrt{p}} \frac{\Lambda(n)}{\log    n} \cdot \Psi_E (P) \nonumber \\
&\gg&  \delta(d_E,E) \frac{x}{ \log^3 x} \left (1 +O \left (\frac{x}{\log x} \right ) \right ) \\
&>&0\nonumber,
\end{eqnarray} 
contradicts the hypothesis  (\ref{800-510}) for all large numbers $x \geq x_0$. Ergo, there are infinitely many primes $p\geq x $ such that a fixed elliptic curve of rank $\rk(E)>0$ with a primitive point $P$ of infinite order, for which the corresponding groups $E(\mathbb{F}_p)$ have prime orders. Lastly, the number of such elliptic primitive primes has the asymptotic formula
\begin{eqnarray} \label{800-560}
\pi(x,E)&=&\sum_{ \substack{p \leq x \\ \ord_E(P)=n \text{ prime}}} 1
 \nonumber \\
&=&\sum _{p\leq x} \frac{1}{4 \sqrt{p}}\sum _{p-2\sqrt{p}\leq n\leq p+2\sqrt{p}}\frac{\Lambda(n)}{\log n}\cdot  \Psi_E (P)    \\
&\geq&\delta(d_E,E) \frac{x}{ \log^3 x} \left (1+O \left (\frac{x}{\log x} \right )\right )\nonumber,
\end{eqnarray} 
which is obtained from the summation of the elliptic primitive primes density function over the interval $[1,x]$. 
\end{proof}

The exact evaluation
\begin{equation} \label{800-570}
\pi(x,E)=\delta(d_E,E) \frac{x}{ \log^2 x} \left (1+O \left (\frac{x}{\log x} \right )\right ),
\end{equation} 
as claimed in Conjecture \ref{conj800.1}, might requires more precise information for primes in short intervals $[p-2\sqrt{p},p-2\sqrt{p}]$.

\section{Examples Of Elliptic Curves}
 The densities of several elliptic curves have been computed by several authors. Extensive calculations for some specific densities are given in \cite{ZD09}.\\

\begin{exa}  \normalfont   The nonsingular Bachet elliptic curve $E: y^2=x^3+2$ over the rational numbers has complex multiplication by $\mathbb{Z}[\rho]$, and nonzero rank $\rk(E)=1$. It is listed as 1728.n4 in \cite{LMFDB}. The numerical data shows that $\# E(\mathbb{F}_p)=n$ is prime for at least one prime, see Table \ref{t801}. Hence, by Theorem \ref{thm800.1}, the corresponding group of $\mathbb{F}_p$-rational points $\# E(\mathbb{F}_p)$ has prime orders $n=\# E(\mathbb{F}_p)$ for infinitely many primes $p \geq 3$. \\

Since $\Delta=-2^6 \cdot 3^3$, the discriminant of the quadratic field $\mathbb{Q}(\sqrt{\Delta})$ is $D=-3$. Moreover, the integer divisor $d_E=1$ since $\# E(\mathbb{F}_p)$ is prime for at least one prime, see Table \ref{t801}. Thus, applying Lemma \ref{lem800.11}, gives the natural density
\begin{equation}
\delta(1,E)=\frac{10}{9} P_0 \approx 0.5612957424882619712979385 \ldots,
\end{equation}
The predicted number of elliptic primes $p \nmid 6N$ such that $\# E(\mathbb{F}_p)$ is prime has the asymptotic 
formula
\begin{equation}
\pi(x,E)=\delta(1,E)\int_2^x \frac{1}{\log(t+1)} \frac{dt}{t}.
\end{equation}

A lower bound for the counting function is
\begin{equation}
\pi(x,E)\geq \delta(1,E) \frac{x}{\log^ 3 x} \left (1+ O\left (\frac{1}{\log x} \right ) \right ),
\end{equation}
see Theorem \ref{thm800.1}.\\

\begin{table}
\begin{center}
\begin{tabular}{||l|r c l||} 
		\hline
		\textbf{Invariant} & \textbf{Value}&&   \\ [1ex]  
		\hline\hline
		Discriminant  &$\Delta$&=&$-16(4a^3+27b^2)=-1728$\\ 
		\hline
		Conductor& $N$&=&$1728 $\\
		\hline
		j-Invariant &$j(E)$&=&$ (-48b)^3/\Delta=0$\\
		\hline
		Rank &$\rk(E)$&=& $1$ \\ 
		\hline
Special $L$-Value &$L^{
'}(E,1)$&$\approx$&$ 2.82785747365$\\
	\hline
		Regulator &$R$&=&$.754576$ \\ 
		\hline
		Real Period &$\Omega$&=&$5.24411510858$ \\ 
		\hline
		Torsion Group &$E(\mathbb{Q})_{\text{tors}}$&$\cong$&$\{ \mathcal{O} \}$ \\ 
		\hline
		Integral Points &$E(\mathbb{Z})$&=&$ \{ \mathcal{O} ,(-1,1);(-1,1)\}$ \\
		\hline
		Rational Group &$E(\mathbb{Q})$&$=$&$ \mathbb{Z}$ \\
		\hline
		Endomorphims Group &$End(E)$&=&$\mathbb{Z}[(1+\sqrt{-3})/2]$, CM \\
\hline
Integer Divisor  &$d_E$&=&$1$ \\
		\hline
\end{tabular}
\end{center}
\caption{\label{9600} Data for $y^2=x^3+2$.}
\end{table}	

The associated weight $k=2$ cusp form, and $L$-function are
\begin{equation}
f(s)=\sum_{n \geq 1}a_n q^n=q-q^7-5q^{13}+7q^{19}+ \cdots ,
\end{equation}

and
\begin{eqnarray}
L(s)&=&\sum_{n \geq 1}\frac{a_n}{n^s} \nonumber \\ &=& \prod_{p |N} \left ( 1-\frac{a_p}{p^s} \right )^{-1} \prod_{p \nmid N} \left ( 1-\frac{a_p}{p^s} +\frac{1}{p^{2s-1}}\right )^{-1} \\ 
&=& 1-\frac{1}{7^s}-\frac{5}{13^s}+\frac{7}{19^s}+ \cdots  \nonumber ,
\end{eqnarray}
where $q=e^{ i 2 \pi} $, respectively. The coefficients are generated using $a_p=p+1-\# E(\mathbb{F}_p)$, and the formulas
\begin{enumerate}
\item $a_{pq}=a_pa_q$  if $\gcd(p,q)=1$;
\item $a_{p^{n+1}}=a_{p^n}a_p-pa_{p^{n-1}}$  if $n \geq 2$.
\end{enumerate}

The corresponding functional equation is 
\begin{equation}
\Lambda(s)=\left ( \frac{\sqrt{N}}{2 \pi} \right ) ^s \Gamma(s) L(s)  \qquad \text{ and } \qquad \Lambda(s)=\Lambda(2-s),
\end{equation}
where $N=1728$, see \cite[p.\ 80]{KN93}.
\end{exa}

\begin{exa}     \normalfont  The nonsingular elliptic curve $E: y^2=x^3+ 6x-2$ over the rational numbers has no complex multiplication and zero rank, it is listed as 1728.w1 in \cite{LMFDB}. The numerical data shows that $\# E(\mathbb{F}_p)=n$ is prime for at least one prime, see Table \ref{t804}. Hence, by Theorem \ref{thm800.1}, the corresponding group of $\mathbb{F}_p$-rational points $\# E(\mathbb{F}_p)$ has prime orders $n=\# E(\mathbb{F}_p)$ for infinitely many primes $p \geq 3$. \\

Since $\Delta=-2^6 \cdot 3^5$, the discriminant of the quadratic field $\mathbb{Q}(\sqrt{\Delta})$ is $D=-3$. Moreover, the integer divisor $d_E=1$ since $\# E(\mathbb{F}_p)$ is prime for at least one prime, see Table \ref{t804}. Thus, applying Lemma \ref{lem800.11}, gives the natural density
\begin{equation}
\delta(1,E)=\frac{10}{9} P_0 \approx 0.5612957424882619712979385 \ldots,
\end{equation}
The predicted number of elliptic primes $p \nmid 6N$ such that $\# E(\mathbb{F}_p)$ is prime has the asymptotic 
formula
\begin{equation}
\pi(x,E)=\delta(1,E)\int_2^x \frac{1}{\log(t+1)} \frac{dt}{t}.
\end{equation}

A table for the prime counting function $\pi(1,E)$, for $2 \times 10^7 \leq x \leq 10^9$, and other information on 
this elliptic curve appears in \cite{ZD09}. \\

A lower bound for the counting function is
\begin{equation}
\pi(x,E)\geq \delta(1,E) \frac{x}{\log^ 3 x} \left (1+ O\left (\frac{1}{\log x} \right ) \right ),
\end{equation}
see Theorem \ref{thm800.1}.\\

\begin{table}
		\begin{center}
		\begin{tabular}{||l|r c l||} 
		\hline
		\textbf{Invariant} & \textbf{Value}&&   \\ [1ex]  
		\hline\hline
		Discriminant  &$\Delta$&=&$-16(4a^3+27b^2)=-2^6 \cdot 3^5$\\ 
		\hline
		Conductor& $N$&=&$2^6 \cdot 3^3  $\\
		\hline
		j-Invariant &$j(E)$&=&$ (-48b)^3/\Delta=2^9 \cdot 3$\\
		\hline
		Rank &$\rk(E)$&=& $0$ \\ 
		\hline
Special $L$-Value &$L(E,1)$&$\approx$&$ 2.24402797314$\\
	\hline
		Regulator &$R$&=&$ 1$ \\ 
		\hline
		Real Period &$\Omega$&=&$2.2440797314$ \\ 
		\hline
		Torsion Group &$E(\mathbb{Q})_{\text{tors}}$&=&$\{ \mathcal{O} \}$ \\ 
		\hline
		Integral Points &$E(\mathbb{Z})$&=&$ \{ \mathcal{O} \}$ \\
\hline
Rational Group &$E(\mathbb{Q})$&$=$&$ \{\mathcal{O}\}$ \\
		\hline
		Endomorphims Group &$End(E)$&=&$\mathbb{Z}$, nonCM \\
\hline
Integer Divisor  &$d_E$&=&$1$ \\
		\hline
		\end{tabular}
		\end{center}
\caption{\label{999} Data for $y^2=x^3+6x-2$.}
\end{table}

The associated weight $k=2$ cusp form, and $L$-function are
\begin{equation}
f(s)=\sum_{n \geq 1}a_n q^n=q+2q^5+q^7+2q^{11}-q^{13}-6q^{17}+5q^{19}+ \cdots ,
\end{equation}

and
\begin{eqnarray}
L(s)&=&\sum_{n \geq 1}\frac{a_n}{n^s} \nonumber \\ &=& \prod_{p |N} \left ( 1-\frac{a_p}{p^s} \right )^{-1} \prod_{p \nmid N} \left ( 1-\frac{a_p}{p^s} +\frac{1}{p^{2s-1}}\right )^{-1} \\ 
&=& 1+\frac{2}{5^s}+\frac{1}{7^s}+\frac{2}{11^s}-\frac{1}{13^s}-\frac{6}{17^s}+\frac{5}{19^s}+ \cdots  \nonumber ,
\end{eqnarray}
where $q=e^{ i 2 \pi} $, respectively. The coefficients are generated using $a_p=p+1-\# E(\mathbb{F}_p)$, and the formulas
\begin{enumerate}
\item $a_{pq}=a_pa_q$  if $\gcd(p,q)=1$;
\item $a_{p^{n+1}}=a_{p^n}a_p-pa_{p^{n-1}}$  if $n \geq 2$.
\end{enumerate}
The corresponding functional equation is 
\begin{equation}
\Lambda(s)=\left ( \frac{\sqrt{N}}{2 \pi} \right ) ^s \Gamma(s) L(s)  \qquad \text{ and } \qquad \Lambda(s)=\Lambda(2-s),
\end{equation}
where $N=1728$, see \cite[p.\ 80]{KN93}.
\end{exa}

\begin{exa}   \normalfont  The nonsingular elliptic curve $E: y^2=x^3-x$ over the rational numbers has complex multiplication by $\mathbb{Z}[i]$, and zero rank, it is listed as 32.a3  in \cite{LMFDB}. The numerical data shows that $\# E(\mathbb{F}_p)/8=n$ is prime for at least one prime, see Table \ref{t808}. Hence, by Theorem \ref{thm800.1}, the corresponding group of $\mathbb{F}_p$-rational points $\# E(\mathbb{F}_p)$ has prime orders $n=\# E(\mathbb{F}_p)$ for infinitely many primes $p \geq 3$. \\

Since $\Delta=-2^6$, the discriminant of the quadratic field $\mathbb{Q}(\sqrt{\Delta})$ is $D=-4$. Moreover, the integer divisor $d_E=8$ since $\# E(\mathbb{F}_p)/8$ is prime for at least one prime, see Table \ref{t808}. Thus, Lemma \ref{lem800.11} is not applicable. The natural density
\begin{equation}
\delta(8,E)=\frac{1}{2}\prod_{p\geq 3} \left (1 -\chi(p)\frac{p^2-p-1}{(p-\chi(p))(p-1)^2}\right )  \approx 0.5336675447 \ldots,
\end{equation}
where $\chi(n)=(-1)^{(n-1)/2}$, is computed in \cite[Lemma 7.1]{ZD09}. The predicted number of elliptic primes $p \nmid 6N$ such that $\# E(\mathbb{F}_p)$ is prime has the asymptotic 
formula
\begin{equation}
\pi(x,E)=\delta(8,E)\int_9^x \frac{1}{\log(t+1)- \log 8} \frac{dt}{\log t}.
\end{equation}

A table for the prime counting function $\pi(x,E)$, for $2 \times 10^7 \leq x \leq 10^9$, and other information on 
this elliptic curve appears in \cite{ZD09}. \\

A lower bound for the counting function is
\begin{equation}
\pi(x,E)\geq \delta(8,E) \frac{x}{\log^ 3 x} \left (1+ O\left (\frac{1}{\log x} \right ) \right ),
\end{equation}
see Theorem \ref{thm800.1}.\\

\begin{table}
\begin{center}
\begin{tabular}{||l|r c l||} 
		\hline
		\textbf{Invariant} & \textbf{Value}&&   \\ [1ex]  
		\hline\hline
		Discriminant  &$\Delta$&=&$-16(4a^3+27b^2)=-2^6$\\ 
		\hline
		Conductor& $N$&=&$2^5 $\\
		\hline
		j-Invariant &$j(E)$&=&$ (-48b)^3/\Delta=2^6 \cdot 3^3$\\
		\hline
		Rank &$\rk(E)$&=& $0$ \\ 
		\hline
Special $L$-Value &$L(E,1)$&$\approx$&$ .655514388573$\\
	\hline
		Regulator &$R$&=&$ 1$ \\ 
		\hline
		Real Period &$\Omega$&=&$5.24411510858$ \\ 
		\hline
		Torsion Group &$E(\mathbb{Q})_{\text{tors}}$&$\cong$&$\mathbb{Z}_2 \times \mathbb{Z}_2$ \\ 
		\hline
		Integral Points &$E(\mathbb{Z})$&=&$ \{ \mathcal{O} ,(\pm1,0);(0,0)\}$ \\
\hline
Rational Group &$E(\mathbb{Q})$&$=$&$ E(\mathbb{Q})_{\text{tors}}$ \\		
		\hline
		Endomorphims Group &$End(E)$&=&$\mathbb{Z}[i]$, CM \\
\hline
Integer Divisor  &$d_E$&=&$2,4,8$ \\
		\hline
\end{tabular}
\end{center}
\caption{\label{959} Data for $y^2=x^3-x$.}
\end{table}	

The associated weight $k=2$ cusp form, and $L$-function are
\begin{equation}
f(s)=\sum_{n \geq 1}a_n q^n=q+2q^5+q^7+2q^{11}-q^{13}-6q^{17}+5q^{19}+ \cdots ,
\end{equation}

and
\begin{eqnarray}
L(s)&=&\sum_{n \geq 1}\frac{a_n}{n^s} \nonumber \\ &=& \prod_{p |N} \left ( 1-\frac{a_p}{p^s} \right )^{-1} \prod_{p \nmid N} \left ( 1-\frac{a_p}{p^s} +\frac{1}{p^{2s-1}}\right )^{-1} \\ 
&=& 1+\frac{2}{5^s}+\frac{1}{7^s}+\frac{2}{11^s}-\frac{1}{13^s}-\frac{6}{17^s}+\frac{5}{19^s}+ \cdots  \nonumber ,
\end{eqnarray}
where $q=e^{ i 2 \pi} $, respectively. The coefficients are generated using $a_p=p+1-\# E(\mathbb{F}_p)$, and the formulas
\begin{enumerate}
\item $a_{pq}=a_pa_q$  if $\gcd(p,q)=1$;
\item $a_{p^{n+1}}=a_{p^n}a_p-pa_{p^{n-1}}$  if $n \geq 2$.
\end{enumerate}

The corresponding functional equation is 
\begin{equation}
\Lambda(s)=\left ( \frac{\sqrt{N}}{2 \pi} \right ) ^s \Gamma(s) L(s)  \qquad \text{ and } \qquad \Lambda(s)=\Lambda(2-s),
\end{equation}
where $N=1728$, see \cite[p.\ 80]{KN93}.
\end{exa}

\begin{exa}   \normalfont  The nonsingular Bachet elliptic curve $E: y^2=x^3+1$ over the rational numbers has complex multiplication by $\mathbb{Z}[\rho]$, and zero rank, it is listed as 36.a4  in \cite{LMFDB}. The numerical data shows that $\# E(\mathbb{F}_p)/12=n$ is prime for at least one prime, see Table \ref{t820}. Since this is an elliptic curve of zero rank, there is neither conjecture nor result similar to Theorem \ref{thm800.1}, to predict the order of magnitute of the corresponding group of $\mathbb{F}_p$-rational points $\# E(\mathbb{F}_p)$ that has prime orders $n/12=\# E(\mathbb{F}_p)/12$ for infinitely many primes $p \geq 3$. \\

Since $\Delta=-2^4 \cdot 3^3$, the discriminant of the quadratic field $\mathbb{Q}(\sqrt{\Delta})$ is $D=-3$. Moreover, the integer divisor $d_E=12$ since $\# E(\mathbb{F}_p)/12$ is prime for at least one prime, see Table \ref{t820}. But Lemma \ref{lem800.11} is not applicable. The natural density should be of the form
\begin{equation}
\delta(12,E)=\mathcal{C}(E,P)\prod_{p\geq 2} \left (1 -\frac{p^3-p-1}{p^2(p-1)^2(p+1)}\right ),
\end{equation}
where $\mathcal{C}(E)$ is a correction factor, see \cite{LS14}, \cite{BJ17}. The predicted number of elliptic primes $p \nmid 6N$ such that 
$\# E(\mathbb{F}_p)$ is prime should be have the asymptotic formula
\begin{equation}
\pi(x,E)=\delta(12,E)\int_9^x \frac{1}{\log(t+1)- \log 12} \frac{dt}{\log t}.
\end{equation}

\begin{table}
\begin{center}
\begin{tabular}{||l|r c l||} 
		\hline
		\textbf{Invariant} & \textbf{Value}&&   \\ [1ex]  
		\hline\hline
		Discriminant  &$\Delta$&=&$-16(4a^3+27b^2)=-2^4 \cdot 3^3$\\ 
		\hline
		Conductor& $N$&=&$2^2 \cdot 3^2 $\\
		\hline
		j-Invariant &$j(E)$&=&$ (-48b)^3/\Delta=0$\\
		\hline
		Rank &$\rk(E)$&=& $0$ \\ 
		\hline
Special $L$-Value &$L(E,1)$&$\approx$&$ .701091050663$\\
	\hline
		Regulator &$R$&=&$ 1$ \\ 
		\hline
		Real Period &$\Omega$&=&$4.20654631598$ \\ 
		\hline
		Torsion Group &$E(\mathbb{Q})_{\text{tors}}$&$\cong$&$\mathbb{Z}_6 $ \\ 
		\hline
		Integral Points &$E(\mathbb{Z})$&=&$ \{ \mathcal{O} ,(-1,0);(0,1), (2,3)\}$ \\
\hline
Rational Group &$E(\mathbb{Q})$&$=$&$ E(\mathbb{Q})_{\text{tors}}$ \\		
		\hline
		Endomorphims Group &$End(E)$&=&$\mathbb{Z}[\rho]$, CM \\
\hline
Integer Divisor  &$d_E$&=&$12$ \\
		\hline
\end{tabular}
\end{center}
\caption{\label{959} Data for $y^2=x^3+1$.}
\end{table}	

The associated weight $k=2$ cusp form, and $L$-function are
\begin{equation}
f(s)=\sum_{n \geq 1}a_n q^n=q-4q^7+2q^{13}+q^{13}+8q^{19}+ \cdots ,
\end{equation}

and
\begin{eqnarray}
L(s)&=&\sum_{n \geq 1}\frac{a_n}{n^s} \nonumber \\ &=& \prod_{p |N} \left ( 1-\frac{a_p}{p^s} \right )^{-1} \prod_{p \nmid N} \left ( 1-\frac{a_p}{p^s} +\frac{1}{p^{2s-1}}\right )^{-1} \\ 
&=& 1-\frac{4}{7^s}+\frac{2}{13^s}+\frac{8}{19^s}+ \cdots  \nonumber ,
\end{eqnarray}
where $q=e^{ i 2 \pi} $, respectively. The coefficients are generated using $a_p=p+1-\# E(\mathbb{F}_p)$, and the formulas
\begin{enumerate}
\item $a_{pq}=a_pa_q$  if $\gcd(p,q)=1$;
\item $a_{p^{n+1}}=a_{p^n}a_p-pa_{p^{n-1}}$  if $n \geq 2$.
\end{enumerate}

The corresponding functional equation is 
\begin{equation}
\Lambda(s)=\left ( \frac{\sqrt{N}}{2 \pi} \right ) ^s \Gamma(s) L(s)  \qquad \text{ and } \qquad \Lambda(s)=\Lambda(2-s),
\end{equation}
where $N=36$, see \cite[p.\ 80]{KN93}.
\end{exa}

\newpage		
\section{Problems}
\begin{exe} \normalfont
Assume random elliptic curve $E:f(x,y)=0$ has no CM, $P$ is a point of infinite order. Is the integer divisor $d_E < \infty$ bounded? This parameter is bounded for CM elliptic curves, in fact $d_E <24$, reference: \cite{JJ08}.\end{exe}

\begin{exe} \normalfont Assume the elliptic curve $E:f(x,y)=0$ has no CM, $P$ is a point of infinite order, and the density $\delta(1,E)>0$. What is the least prime $p\nmid 6N$ such that the integer divisor $d_E=1$? Reference: \cite{CA03}.\end{exe}

\begin{exe} \normalfont Fix an  elliptic curve $E:y^2=x^3+ax+b$ of rank $\rk(E)>0$, and CM; and $P$ is a point of infinite order. Let the integer divisor $d_E=4$. Assume the densities $\delta(1,E)>0$, $\delta(2,E)>0$, and $\delta(4,E)>0$ are defined. What is the arithmetic relationship between the densities?\end{exe}

\begin{exe} \normalfont Fix an  elliptic curve $E:y^2=x^3+ax+b$ of rank $\rk(E)>0$, and CM; and $P$ is a point of infinite order. Does $\#E(\mathbb{F}_p)=n$ prime for at least one large prime $p \geq 2$ implies that the integer divisor $d_E=1$?\end{exe}

%cccccccccccccccccccccccccccccccccccccccccccc
\chapter{Elliptic Cyclic Groups of Points} \label{c18}
This chapter consider the cyclic structure of the groups of points for elliptic curves over finite fields. A nonzero proportion of the elliptic groups $E(\mathbb{F}_p)$ are cyclic grouips. In fact, the short interval $[p-2\sqrt{p},p-2\sqrt{p}]$ contains $(6 \pi^{-2})(4\sqrt{p})+o(\sqrt{p})$ squarefree integers, and each each squarefree integer $n \geq$ supplies an elliptic cyclic group $E(\mathbb{F}_p) $ of cardinality $n$.

\section{Cyclic Groups of Points}
The early numerical analysis and conjectures are undertaken in \cite{BP75}. These authors computed the algebraic structures of the groups of points 
\begin{equation}
E(\mathbb{F}_p) \cong \mathbb{Z}_d \times \mathbb{Z}_{e},
\end{equation}
with $d |e$, and $n=de \geq 1$, for a large sample of elliptic curves. And observed the early version of the cyclic group conjecture. \\

\begin{lem}
	The group of rational points $E(\mathbb{F}_p)$ is cyclic if and only if $\gcd(\#E(\mathbb{F}_p),p-1)=1$.
\end{lem}

Given a fixed nonsingular elliptic curve, the subset of cyclic primes is defined by
\begin{equation}
\mathcal{C}(E)= \{p \in \mathbb{P}: p \not | N  \text{ and }  E(\mathbb{F}_p) \text{ is cyclic} \}
\end{equation}
with $\mathbb{P}=\{2,3,5, \ldots \}$ being the set of primes. The corresponding counting function is given by 
\begin{equation}
\pi(x,E)= \#\{p \leq x:p \not | N  \text{ and }  E(\mathbb{F}_p) \text{ is cyclic} \},
\end{equation}
here $N$ is the conductor of the ellipctic curve.

\begin{thm} \label{thm17.1}  {\normalfont (\cite{CA03}) } 
	Let $E$ be a CM elliptic curve over the rational numbers $\mathbb{Q}$ of conductor $N$ and let $\mathcal{K}=\mathbb{Q}\sqrt{-D}$, where $D>0$ is squarefree. Then, as $x \to \infty$,
	\begin{equation}
	\pi(x,E)= C(E) C_0 \li(x)+O\left ( \frac{x}{(\log x)(\log \log \log x)} \right ).
	\end{equation}	
\end{thm}

The constant $C(E)$, which depends on the fixed elliptic curve $E$, and the average cyclic elliptic density $C_0>0$ have analytic formulas discussed in Chapter 6, Section 6.5.1. A more general result for any elliptic curve, but having a weaker asymptotic main term, was proved in \cite {GM90}. This latter result claims that if an elliptic curve has no rational torsion group $E[2]=\{ \mathcal{O}\}$, then the corresponding group $E(\mathbb{F}_p)$ is cyclic for infinitely many primes. \\

\begin{thm} \label{thm 17.2}  { \normalfont (\cite{GM90})  }
	The density of primes such that a fixed elliptic curve $E$ has a cyclic group $E(\mathbb{F}_p)$ is $\delta(E)>0$ if and only if $\mathbb{Q}(E[2]) \ne \mathbb{Q}$	
\end{thm}

\section{Elementary Proof}
A different method here provides an elementary proof of the asymptotic order for the counting function of elliptic cyclic primes.

\begin{thm} \label{thm17.3}   
	Let $E$ be a nonsingular curve over the rational numbers $\mathbb{Q}$ of conductor $N$ and let the rank $\rk(E)>0$. Then, as $x \to \infty$,
	\begin{equation}
	\pi(x,E) \gg \frac{x}{\log^2 x} \left (1 +O\left ( \frac{x}{\log x} \right ) \right ).
	\end{equation}	
\end{thm}

\begin{proof} This is a corollary of Theorem \ref{thm16.1}.  \end{proof}

\section{Problems}
\begin{exe} \normalfont
	Let $E:f(x,y)=0$ be a nonsingular elliptic curve, and let $\delta(E)$ be the density of elliptic cyclic primes. Show that the finite sum $$\sum_{p\leq x, \text{E cyclic}} \frac{1}{p}=\delta(E)\log \log x +\beta_E+O(1/\log x),$$ where $\beta_E \in \mathbb{R}$ is a constant. The term $E$ cyclic refers to the relation $E(\mathbb{F}_p) \cong \mathbb{Z}_n$. 
\end{exe}

\begin{exe} \normalfont
	Let $E:y^2=x^3-x-1$, which is a nonsingular elliptic curve, and let $\delta(E)=\prod_{p \geq 2} (1-(p^2-1)^{-1}(p^2+1)^{-1})=.8137519 \ldots $ be the average density of elliptic cyclic primes. Compute or estimate the elliptic Merten constant $$\beta_E=\sum_{p\leq x, \text{E cyclic}} \frac{1}{p} -\delta(E)\log \log x $$
	for $p \leq 10^3$. The term $E$ cyclic refers to the relation $E(\mathbb{F}_p) \cong \mathbb{Z}_n$. 
\end{exe}

\begin{exe} \normalfont
	Compute number of cyclic groups for $p\leq 1000$ for several fixed curves of rank $\rk(E)=1,2,3$.
\end{exe}

\begin{exe} \normalfont
	Compute number of cyclic groups for $p\leq 1000$ for several fixed curves of increasing conductors $N\leq 1000$.
\end{exe}

\begin{exe} \normalfont
	Is there a cyclic group $E(\mathbb{F}_p)$, which does not have a primitive point P? For example, $<P> \ne E(\mathbb{F}_p)$.
\end{exe}

%%%%%%%%%%%%%%%%%%%%%%%%%kkkkkkkkkkkkkkkkkkkkkkkkkkkkkkkkkkkkkk
\chapter{Elliptic Groups of Squarefree Orders} 
This chapter considers some of the arithmetic structures of the elliptic groups of squarefree orders. This is a step closer to the extreme case of prime orders, which is more difficult to prove. \\

\section{Description of Groups of Squarefree Orders}
For fixed nonsingular elliptic curve $E:f(x,y)=0$ over the rational numbers $\mathbb{Q}$, of rank $\rk(E)>0$, and a point $P \in E(\mathbb{Q})$ of infinite order, the subset of elliptic squarefree primes is defined by
\begin{equation}
	\mathcal{Q}(E)=\{p \in \mathbb{P}: \#E(\mathbb{F}_p) \text{ is squarefree}  \},
\end{equation}
and the corresponding counting function is defined by
\begin{equation}
	\pi_{sf}(x,E)=\#\{p \leq x:\#E(\mathbb{F}_p) \text{ is squarefree}   \}.
\end{equation}
The finite number of prime divisors $p\,|\, N$ of the conductor $N$ of the elliptic curve are not included in the count.\\

The elliptic groups $E(\mathbb{F}_p)$ of squarefree orders are cyclic groups. This information is readily visible on the isomorphism
\begin{equation}
	E(\mathbb{F}_p) \cong \mathbb{Z}_d \times \mathbb{Z}_{e},
\end{equation}
with $d|e$. A conditional proof appears in \cite[Theorem 4]{CA04}. \\

A simpler unconditional proof for the subset 
\begin{equation}
	\mathcal{Q}(E,P)=\{p \in \mathbb{P}:  E(\mathbb{F}_p)=<P> \text{ and }  \#E(\mathbb{F}_p) \text{ is squarefree}  \},
\end{equation}
is given here. Let the corresponding counting function be defined by
\begin{equation}
	\pi_{sf}(x,E,P)=\# \{p \leq x: E(\mathbb{F}_p) =<P> \text{ and }  \#E(\mathbb{F}_p) \text{ is squarefree}  \}.
\end{equation}

\begin{thm} \label{thm48.1}   
	Let $E:f(x,y)=0$ be a nonsingular elliptic curve over the rational numbers $\mathbb{Q}$ of rank $\rk(E(\mathbb{Q})>0$. Then, as $x \to \infty$,
	\begin{equation}
		\pi_{sf}(x,E,P)= \delta_{sf}(E,P)\frac{x}{\log x} )+O\left ( \frac{x}{\log^2x}\right) ,
	\end{equation}  	
	where $\delta_{sf}(E)$ is the density constant.
\end{thm}

The density $\delta_{sf}(E,P)$ seems to be unknown, there is no theoretical formula yet. This analysis is focused on the asymptotic part of the problem, and the proof is given in the last subsection.

\section{Evaluations Of The Main Terms}
The exact asymptotic form of main term $M(x)$ in the proof of Theorem \ref{thm48.1} is evaluated here. The analysis uses the standard asymptotic formulas
\begin{equation} \label{el15502}
	\sum_{ n\leq x} \frac{\varphi(n)}{n}=\frac{6}{\pi^2}x +O(\log x) ,
\end{equation} 
and 
\begin{equation} \label{el5503}
	\sum_{ n\leq x} \mu(n)^2=\sum_{ n\leq x} \nu(n)=\frac{6}{\pi^2}x +O\left (x^{1/2}e^{-c \sqrt{\log x}} \right) ,
\end{equation} 
where $c>0$ is an absolute constant, see \cite[Theorem 3.11]{TG15}. The latter can be evaluated in term of the discrete squarefree measure  
\begin{equation} \label{el5505}
	\nu(t)=\frac{6}{\pi^2} +O\left (\frac{1}{t^{1/2}}e^{-c \sqrt{\log t}} \right) ,
\end{equation} 
or by other methods. The error term is not the best possible, but it is sufficient for this application. In fact there are significantly better result for squarefree integers in short intervals, see \cite{FT90}.\\

\begin{lem} \label{lem48.1}
	For any large number \(x\geq 1\), 
	\begin{equation} \label{el5959}
		\sum_{x\leq  n\leq 2x} \frac{\varphi(n)}{n}\mu(n)^2 =\frac{36}{\pi^4} \frac{x}{\log x}+ O\left (x^{1/2}e^{-c \sqrt{\log x}} \right) ,
	\end{equation} 
	where $c>0 $ is an absolute constant.
\end{lem}

\begin{proof} Rewritten in term of the squarefree discrete measure $\nu(t)=6/\pi^{2} +O\left (t^{-1/2}e^{-c \sqrt{\log t}} \right)$ over the integers $\mathbb{Z}$, the summatory function is
\begin{eqnarray} \label{el15505}
	\sum_{x \leq  n\leq 2x} \frac{\varphi(n)}{n}\mu(n)^2 
	&=&\sum_{ x \leq n\leq 2x }  \frac{\varphi(n)}{n} \left (     \frac{6}{\pi^2} +O\left (\frac{1}{n^{1/2}}e^{-c \sqrt{\log n}} \right)   \right )  \nonumber \\
	&= &\frac{6}{\pi^2} \sum_{x \leq  n\leq 2x}  \frac{\varphi(n)}{n}  +O\left ( \frac{1}{x^{1/2}}e^{-c \sqrt{\log x}}\sum_{x \leq  n\leq 2x }1 \right)\\
	&=&\frac{36}{\pi^4} x +O\left (x^{1/2}e^{-c \sqrt{\log x}} \right)  \nonumber ,
\end{eqnarray} 
where $c>0$ is an absolute constant. 
\end{proof}

\begin{lem} \label{lem48.2}
	For any large number \(x\geq 1\), 
	\begin{equation} \label{el5979}
		\sum_{x \leq p \leq 2x} \frac{1}{4 \sqrt{p} }  \sum_{p-2 \sqrt{p}\leq n\leq p+2\sqrt{p} } \mu(n)^2\sum_{\gcd(m,n)=1}\frac{1}{n} =c_0 \frac{x}{\log x}+  O  \left ( \frac{x}{\log^2 x}  \right ),
	\end{equation} 
	where $c_0=36/\pi^4$ is a constant.	
\end{lem}

\begin{proof} The value of the phi function $\varphi(n)=\#\{1\leq m<n:\gcd(m,n)=1\} $. Thus, main term can be rewritten in the form
\begin{eqnarray} \label{el15500}
	M(x)&=&\sum_{x \leq p \leq 2x} \frac{1}{4 \sqrt{p} }  \sum_{p-2 \sqrt{p}\leq n\leq p+2\sqrt{p}} \mu(n)^2 \sum_{\gcd(m,n)=1}\frac{1}{n} \\
	&=&\sum_{x \leq p \leq 2x} \frac{1}{4 \sqrt{p} }  \sum_{p-2 \sqrt{p}\leq n\leq p+2\sqrt{p}} \frac{\varphi(n)}{n}\mu(n)^2 \nonumber .
\end{eqnarray} 
Applying Lemma \ref{lem17.1} yields
\begin{eqnarray} \label{el15503}
	\sum_{x \leq p \leq 2x} \frac{1}{4 \sqrt{p} }  \sum_{p-2 \sqrt{p}\leq n\leq p+2\sqrt{p}} \frac{\varphi(n)}{n}\mu(n)^2
	&=&\sum_{x \leq p \leq 2x} \frac{1}{4 \sqrt{p} }  \left ( \frac{36}{\pi^4} \left  ( 4\sqrt{p} +O\left (p^{1/4}e^{-c \sqrt{\log p}} \right) \right )  \right ) \nonumber \\
	&=&\frac{36}{\pi^4} \sum_{x \leq p \leq 2x}1 + O  \left (x^{-1/4} e^{-c \sqrt{\log x}} \sum_{x \leq p \leq 2x} 1  \right )\\
	&=&\frac{36}{\pi^4} \frac{x}{\log x}+ O  \left ( \frac{x}{\log^2 x}  \right ) +O  \left ( x^{3/4} e^{-c \sqrt{\log x}}  \right ) \nonumber \\
	&=&\frac{36}{\pi^4} \frac{x}{\log x}+ O  \left ( \frac{x}{\log^2 x}  \right )  \nonumber .
\end{eqnarray}
This completes the evaluation.  
\end{proof}

The result can also be evaluated in term of the logarithm integral in the following way.
\begin{eqnarray} \label{el15508}
	\frac{36}{\pi^4} \sum_{x \leq p \leq 2x}1 +O  \left (x^{-1/4} e^{-c \sqrt{\log x}} \sum_{x \leq p \leq 2x} 1  \right )
	&=&\frac{36}{\pi^4} \left (\li(2x)-\li(x) \right ) +  O  \left (  xe^{-c\sqrt{\log x}}  \right ) \nonumber \\
	&=&\frac{36}{\pi^4} \li(x)  +  O  \left (  xe^{-c\sqrt{\log x}}  \right ),
\end{eqnarray}
where $\li(x)=\int_2^x(\log t)^{-1} dt$ is the logarithm integral, and $c>0$ is an absolute constant.

\section{Estimate For The Error Term}
The upper bounds for the error term $E(x)$ in the proof of Theorem \ref{thm17.1} is based on standard analytic methods. 
\begin{lem} \label{lem48.3}
	Let $E:y^2=f(x)$ be a nonsingular elliptic curve over rational number, let $P \in E(\mathbb{Q})$ be a point of infinite order. Suppose that $P \in E(\mathbb{F}_p)$ is not a primitive point for all primes $p\geq 2$, then
	
	\begin{equation} \label{el15400}
	\sum_{x \leq p \leq 2x} \frac{1}{4 \sqrt{p} }  \sum_{p-2 \sqrt{p}\leq n\leq p+2\sqrt{p}} \mu(n)^2 \sum_{d|n}\frac{\mu(d)}{d} 
	\sum_{ \substack{\ord(\chi)=d \\ \chi \ne 1}} \chi(P) =O\left ( x e^{-c\sqrt{\log x}} \right ),
	\end{equation} 
	where \(c >0\) is an absolute constant.
\end{lem}

\begin{proof} Let $n=md$, and change the order of summation in the second and third inner finite sums below. 
\begin{eqnarray} \label{el15404}
E(x)&=&\sum_{x \leq p \leq 2x } \frac{1}{4 \sqrt{p} } \sum_{p-2\sqrt{p}\leq n\leq p+2\sqrt{p}} \mu(n)^2 \sum_{d|n}\frac{\mu(d)}{d} 
\sum_{ \substack{\ord(\chi)=d \\ \chi \ne 1}} \chi(P)\nonumber \\
&=&\sum_{x \leq p \leq 2x}  \frac{1}{4 \sqrt{p} }  \sum_{d \leq p+2\sqrt{p}}\frac{\mu(d)}{d} \sum_{\frac{p-2\sqrt{p}}{d}\leq m\leq \frac{p+2\sqrt{p}}{d}}  \mu(n)^2
\sum_{ \substack{\ord(\chi)=d \\ \chi \ne 1}} \chi(P).
\end{eqnarray}

The double inner sum has the trivial upper bound
\begin{eqnarray} \label{el15405}
\sum_{\frac{p-2\sqrt{p}}{d}\leq m\leq \frac{p+2\sqrt{p}}{d}}  
\mu(n)^2\sum_{ \substack{\ord(\chi)=d \\ \chi \ne 1}} \chi(P) 	
&=& \sum_{ \substack{\ord(\chi)=d \\ \chi \ne 1}} \chi(P) \sum_{\frac{p-2\sqrt{p}}{d}\leq m \leq  \frac{p+2\sqrt{p}}{d}} 1 
\nonumber\\
&=&  4 \sqrt{p}\sum_{ \substack{\ord(\chi)=d \\ \chi \ne 1}} \frac{\chi(P) }{d}.
\end{eqnarray} 

This information is sufficient to obtain a nontrivial upper bound.
\begin{eqnarray} \label{el15414}
E(x)&=& \sum_{x \leq p \leq 2x}  \frac{1}{4 \sqrt{p} }  \sum_{d \leq p+2\sqrt{p}}\frac{\mu(d)}{d}  \left (\sum_{\frac{p-2\sqrt{p}}{d}\leq m\leq \frac{p+2\sqrt{p}}{d}}  
\mu(n)^2\sum_{ \substack{\ord(\chi)=d \\ \chi \ne 1}} \chi(P) \right ) \nonumber \\ 
&= &\sum_{x \leq p \leq 2x}  \frac{1}{4 \sqrt{p} }\sum_{d \leq p+2\sqrt{p}} \frac{\mu(d)}{d}  \left ( 4 \sqrt{p}\sum_{ \substack{\ord(\chi)=d \\ \chi \ne 1}} \frac{\chi(P) }{d} \right )  \nonumber\\
&= &\sum_{x \leq p \leq 2x} \sum_{d \leq p+2\sqrt{p}} \frac{\mu(d)}{d^2}\sum_{ \substack{\ord(\chi)=d \\ \chi \ne 1}} \chi(P) \\
&=& O\left (  \sum_{x \leq p \leq 2x}e^{-c\sqrt{\log p}} \right )  \nonumber\\
&=&O\left ( x e^{-c\sqrt{\log x}} \right ) \nonumber.
\end{eqnarray}
Here, the inner finite sum in the third line is
\begin{equation} \label{el15415}
\sum_{d \leq p+2\sqrt{p}} \frac{\mu(d)}{d}= O\left ( e^{-c\sqrt{\log p}} \right ),
\end{equation} 
where $c>0$ is an absolute constant.   
\end{proof}

\section{Squarefree Orders}
The characteristic function for primitive points in the group of points $E(\mathbb{F}_p)$ of an elliptic curve $E:f(x,y)=0$ has the representation
\begin{equation}
	\Psi_E(P)=
	\left \{\begin{array}{ll}
		1 & \text{ if the multiplicative order } \ord_E (P)=n,  \\
		0 &  \text{ if the multiplicative order } \ord_E (P) \ne n. \\
	\end{array} \right.
\end{equation} 
The parameter $n=\# E(\mathbb{F}_p)$ is the size of the group of points, and the exact formula for $\Psi_E (P)$ is given in Lemma \ref{lem3.6}.\\ 

Since each order $n$ is unique, the weighted sum
\begin{equation} \label{el157703}
	\frac{1}{4 \sqrt{p}}\sum _{p-2\sqrt{p}\leq n\leq p+2\sqrt{p}}\mu(n)^2 \Psi_E (P)
\end{equation}
is a discrete measure for the density of elliptic primitive primes $p \geq 2$ such that $P \in E(\mathbb{\overline{Q}})$ is a primitive point of squarefree order $n \in [p-2\sqrt{p}, p+2\sqrt{p}]$.\\

\begin{proof} (Proof of Theorem \ref{thm48.1}.)  Let $x \geq x_0 \geq 1$ be a large number, and suppose that \(P\not \in E_{\text{tors}}(\mathbb{Q}) \) is not a primitive point in $E(\mathbb{F}_p)$ for all primes \(p\geq x\). Then, the sum of the elliptic primes measure over the short interval \([x,2x]\) vanishes. Id est, 
\begin{equation} \label{el57703}
	0=\sum _{x \leq p\leq 2x} \frac{1}{4 \sqrt{p}}\sum _{p-2\sqrt{p}\leq n\leq p+2\sqrt{p}} \mu(n)^2\Psi_E (P).
\end{equation}
Replacing the characteristic function, Lemma \ref{lem3.7}, and expanding the nonexistence inequality (\ref{el57703}) yield
\begin{eqnarray} \label{el15704}
	0&=&\sum _{x \leq p\leq 2x}  \frac{1}{4 \sqrt{p}}\sum _{p-2\sqrt{p}\leq  n\leq p+2\sqrt{p}} \mu(n)^2\Psi_E (P) \nonumber\\
	&=&\sum _{x \leq p\leq 2x}  \frac{1}{4 \sqrt{p}} \sum _{p-2\sqrt{p}\leq  n\leq p+2\sqrt{p}}  \mu(n)^2 \left (\sum_{\gcd(m,n)=1}\frac{1}{n} 
	\sum_{ 0 \leq r \leq n-1} \chi ((mQ-P)r ) \right ) \nonumber \\
	&=&a_{sf}(E,P)\sum _{x \leq p\leq 2x}  \frac{1}{4 \sqrt{p}} \sum _{p-2\sqrt{p}\leq  n\leq p+2\sqrt{p}} \frac{\mu(n)^2}{n}\sum_{\gcd(m,n)=1}1
	\\
	& & \qquad+ \sum _{x \leq p\leq 2x}  \frac{1}{4 \sqrt{p}} \sum _{p-2\sqrt{p}\leq  n\leq p+2\sqrt{p}}  \frac{\mu(n)^2}{n}\sum_{\gcd(m,n)=1} 
	\sum_{ 0 \leq r \leq n-1} \chi ((mQ-P)r )\nonumber \\
	&=&M(x) + E(x) \nonumber,
\end{eqnarray} 
where $a_{sf}(E,P)\geq0$ is a constant depending on both the fixed elliptic curve $E:f(x,y)=0$ and the fixed point $P$. \\

The main term $M(x)$ is determined by a finite sum over the principal character \(\chi =1\), and the error term $E(x)$ is determined by a finite sum over the nontrivial multiplicative characters \(\chi \neq 1\).\\

Applying Lemma \ref{lem48.2} to the main term, and Lemma \ref{lem48.3} to the error term yield
\begin{eqnarray} \label{el147715}
	\sum _{x \leq p\leq 2x} \frac{1}{4 \sqrt{p}}\sum _{p-2\sqrt{p}\leq n\leq p+2\sqrt{p}} \mu(n)^2\Psi_E (P)
	&=&M(x) + E(x) \\
	&= & \delta_{sf}(E,P) \left ( \frac{x}{ \log x}+O \left (\frac{x}{\log^2 x} \right ) \right )+O\left ( x e^{-c\sqrt{\log x}} \right ) \nonumber \\
	&=&  \delta_{sf}(E,P) \frac{x}{ \log x} +O \left (\frac{x}{\log^2 x} \right ) \nonumber,
\end{eqnarray} 

where $\delta_{sf}(E,P)=(36 \pi^{-4}) a_{sf}(E,P) \geq0$ is the density constant, and $\varepsilon>0$ is an arbitrary small number. But for all large numbers $x \geq x_0$, the expression
\begin{eqnarray} \label{el17740}
	\sum _{x \leq p\leq 2x} \frac{1}{4 \sqrt{p}}\sum _{p-2\sqrt{p}\leq n\leq p+2\sqrt{p}} \mu(n)^2\Psi_E (P)
	&=&  \delta_{sf}(E,P) \frac{x}{ \log x} +O \left (\frac{x}{\log^2 x} \right ) \nonumber\\
	&>&0,
\end{eqnarray} 
contradicts the hypothesis  (\ref{el57703}) whenever $\delta(E,P)>0$. Ergo, the number of primes $p\geq x $ in the short interval $[x,2x]$ for which a fixed elliptic curve of rank $\rk(E)>0$ and a fixed primitive point $P$ of infinite order is infinite as $x \to \infty$.  Lastly, the number of elliptic primitive primes has the asymptotic formula
\begin{eqnarray} \label{el8741}
	\pi_{sf}(x,E,P)&=&\sum _{ \substack{p \leq x \\ \ord_E(P)=n}} \mu(n)^2 \nonumber \\
	&=&\sum _{p\leq x} \frac{1}{4 \sqrt{p}}\sum _{p-2\sqrt{p}\leq n\leq p+2\sqrt{p}}\mu(n)^2 \Psi_E (P)    \\
	&=&\delta_{sf}(E,P) \frac{x}{ \log x}+O \left (\frac{x}{\log^2x}\right )\nonumber,
\end{eqnarray} 
which is obtained from the summation of the elliptic primitive primes density function over the interval $[1,x]$. 
\end{proof}

%ccccccccccccccccccccccccccccccccccccccccccccccccccc1919
\chapter{Smooth And Nonsmooth Numbers} \label{c19}
For $x \geq 1$, and $2 \leq z \leq x$, the subset of smooth numbers is defined by
\begin{equation}
\mathcal{S}(z)=\{n \geq 1: P(n) \leq z\}
\end{equation}
where $P(n)= \max \{p|n\}$ is the largest prime divisor of $n \geq 1$. The corresponding counting function is given by
\begin{equation}
\Psi(x,y)=\#\{n \leq x: P(n)  \leq y\} \sim \rho x.
\end{equation}
The density of smooth numbers 
\begin{equation}\label{2000-105}
\rho(u)= \lim_{x \to \infty}\frac{\Psi(x,y)}{x}
\end{equation}
is a continuous function on $(0,\infty)$ specified by 
\begin{align}
\rho&=1, & 0<u\leq 1,\\
u \rho^{`}(u)&=\rho(u-1), & \text{ for } u>1 \nonumber.
\end{align}

\section{Results For Smooth Integers}

The exclusion-inclusion principle leads to the expression
\begin{eqnarray}
\Psi(x,z)&=&[x]-\sum_{z<p \leq x}\left [ \frac{x}{p} \right]+\sum_{z<p<q \leq x}\left [ \frac{x}{pq} \right]-\sum_{z<p<q<r \leq x}\left [ \frac{x}{pqr} \right ]+ \cdots \nonumber\\
&=& x \prod_{z<p \leq x} \left ( 1-\frac{1}{p} \right) +O(2^w)\\
&=& \rho(u)x+o(x) \nonumber
\end{eqnarray}
for $y=x^{1/u}$, $u=\log x/\log z$, and $w=O(\log x)$.

\section{Smooth And Nonsmooth Numbers In Short Intervals}
For a large number $x \geq 1$, and $2 \leq z \leq x$, the subset of smooth numbers is defined by
\begin{equation}
\mathcal{S}(z)=\{n \geq 1: P(n) \leq z\}
\end{equation}
where $P(n)= \max \{p\,|\,n\}$ is the largest prime divisor of $n \geq 1$. The corresponding counting function is given by
\begin{equation}
\Psi(x,y)=\#\{n \leq x: P(n)  \leq y\} \sim \rho x,
\end{equation}
where $\rho(u) \geq 0$ is the density of smooth numbers. Some interesting applications of smooth numbers in short intervals call for estimates of the forms
\begin{equation}
\Psi(x+x^{\beta},x^{\alpha})- \Psi(x,x^{\alpha}) \sim \rho(1/\alpha)x^{\beta},
\end{equation}
with $1 \leq x^{\alpha} \leq x^{\beta} \leq x$, as $ x \to \infty$, and $\alpha>0$ and $\beta>0$. A rich assortment of results for various ranges of the parameters $x,y,z$ in $\Psi(x+y,z)-\Psi(x,z) >0$ are proved in the literature. \\

For parameters of nonexponential sizes there are very few results, see \cite[Conjecture 1]{FL87}. Neither conditional nor unconditional results for the extreme cases
\begin{equation}
\Psi(x+\log^{B} x,\log^{C} x)- \Psi(x,\log^C x) \gg \log^{D} x,
\end{equation}
where $B, C, D >0$, are expected to be proved in the near future. The interested reader should refer to \cite{GR13} for a comprehensive introduction to this topic. 

\begin{thm} \label{thm250.4}  { \normalfont (\cite{ET84})}  Let $x \geq 1$ be a large number, and let $\beta \geq \alpha>0$. Then, 
	\begin{equation}
	\Psi(x+x^{\beta},x^{\alpha})- \Psi(x,x^{\alpha}) \leq \frac{1}{\alpha} x^{\beta}+ c_0 \pi_(x^{\alpha}).
	\end{equation}	
	as $x \to \infty$.
\end{thm}

\begin{thm} \label{thm250.6}  { \normalfont (\cite[Theorem 2.4]{FL87})}  There exist nonnegative absolute constants $\nu$ and $c$ such that 
	\begin{equation}
	\Psi(x,y)- \Psi(x-z,y) >c z
	\end{equation}	
	holds uniformly for all $y,z$ such that 
	\begin{enumerate}
		\item $y=x^{\alpha}$ with $1/2-\nu\leq \alpha \leq 1-\nu$; 
		\item $yz \geq x^{1-\nu}$ for all sufficiently large $x\geq 1$. 
	\end{enumerate}
\end{thm}

A more recent result, which is basically a special case of the last theorem, has a fixed parameter $z=x^{1/2}$ for the size of the short interval.\\

\begin{thm} \label{thm250.8}  { \normalfont (\cite{HG99})}  Let $\alpha >1/4 \sqrt{e}$. If $x \geq 1$ is sufficiently large then, 
	\begin{equation}
	\Psi(x+x^{1/2},x^{\alpha})- \Psi(x,x^{\alpha}) \gg  x^{1/2}.
	\end{equation}	
\end{thm}

\begin{lem} \label{lem250.20}   Let $\alpha >1/4 \sqrt{e}$. Then, there exists a constant $c_{\alpha}>0$ such that 
	\begin{equation}
	\Psi(x+x^{1/2},x^{\alpha})- \Psi(x,x^{\alpha})= c_{\alpha}  x^{1/2}+o(x^{1/2})
	\end{equation}
	for all sufficiently large $x \geq 1$. 	
\end{lem}

\begin{proof}
	Combining theorems \ref{thm250.4} and \ref{thm250.8} yield
	\begin{equation}
	x^{1/2}\ll\Psi(x+x^{1/2},x^{\alpha})- \Psi(x,x^{\alpha}) \leq \frac{1}{\alpha} x^{1/2}+ c_0 \pi_(x^{\alpha}).
	\end{equation}
	This immediately implies the claim.
	
\end{proof}

The counting function
\begin{equation}
\Theta(x,y,z)=\#\{n \leq x: p\,|\,n \Rightarrow y \leq p  \leq z\}
\end{equation}
is studied in \cite{FJ76}. In terms of the better known results
it has the form
\begin{eqnarray} \label{250-600}
\Theta(x,x^{\alpha},x^{\beta})&=&\Psi(x,x^{\beta}) \prod_{ p \leq x^{\alpha}} \left ( 1-\frac{1}{p} \right ) \nonumber \\
&\gg&\frac{x^{\beta}}{\log x^{\alpha}},
\end{eqnarray}
see \cite[Corollary 1]{FT91}. For certain parameters such as $\alpha> 1/4\sqrt{e}$ and $\beta>\alpha$, this lower bound can be improved.\\

\begin{lem} \label{lem250.60}   Let $x \geq 1$ be a large number, and let $x^{\alpha} <x^{\beta}$ with $1/4\sqrt{e} \leq \alpha \leq \beta <1$. Then, interval $[x+x^{1/2},x]$ contains 
	\begin{eqnarray}
	\Theta(x,x^{\alpha},x^{\beta})
	&=&\# \{ x\leq n \leq x+x^{1/2}: p\,|\,n \Rightarrow  x^{\alpha} \leq p \leq x^{\beta} \}  \nonumber \\
	&=&c(\alpha,\beta)x^{1/2} +o(x^{1/2 })
	\end{eqnarray}	
	nonsmooth integers $x \to \infty$, with $c(\alpha,\beta)> 0$ constant.
\end{lem}

\begin{proof}  Fix the parameters $ \alpha > 1/4\sqrt{e}$ and $\beta=1-\varepsilon$, with $\varepsilon>0$ an arbitrary small number. By Lemma \ref{lem250.20}
	the short interval $[x, x+x^{1/2}]$ contains 
	\begin{equation}
	\Psi(x+x^{1/2},x^{1-\varepsilon})- \Psi(x,x^{1-\varepsilon}) =c_{\beta} x^{1/2} +o(x^{1/2}),
	\end{equation}	
	$x^{1-\varepsilon}$-smooth integers, and 
	\begin{equation}
	\Psi(x+x^{1/2},x^{\alpha})- \Psi(x,x^{\alpha}) =c_{\alpha} x^{1/2} +o(x^{1/2}),
	\end{equation}	
	$x^{\alpha}$-smooth integers, respectively. Taking the difference yields
	\begin{equation}
	\# \{ x\leq n \leq x+x^{1/2}: p\,|\,n \Rightarrow  x^{\alpha} \leq p \leq x^{\beta} \}  
	=c(\alpha,\beta) x^{1/2}+o(x^{1/2}),
	\end{equation}
	where $c(\alpha,\beta)=c_{\beta}-c_{\alpha}>0$ is a constant. 
\end{proof}

%ccccccccccccccccccc202020
\chapter{Prime Numbers Theorems} \label{c20}
A survey of the important tools form the field of prime number theorems are recorded in this chapter. These theorems are often used in the proofs of result in this monograph. the  

\section{Primes In Short Intervals}
There are many unconditional results for the existence of primes in short intervals $[x,x+y]$ of subsquareroot length $y\leq x^{1/2}$ for almost all large numbers $x\geq 1$. One of the earliest appears to be the Selberg result for $y=x^{19/77}$ with $x\leq X$ and 
$O\left (X (\log X)^{2} \right )$ exceptions, see \cite[Theorem 4]{SA43}. Recent improvements based on different analytic methods are given in \cite{WN95}, \cite{HG07}, et alii. One of these results, but not the best, has the following claim. This result involves the weighted prime indicator function (vonMangoldt), which is defined by
\begin{equation}\label{800-12}
\Lambda(n)=
\begin{cases}
\log n &
\text{ if } n=p^k, k \geq 1,\\
0 &\text{ if } n \ne p^k, k \geq 1,
\end{cases}
\end{equation}
where $p^k,k\geq 1$ is a prime power.\\  
 
\begin{thm} {\normalfont (\cite{HG07})} \label{thm400.1}
Given $\varepsilon>0$, almost all intervals of the form $[x-x^{1/6+\varepsilon},x]$ contains primes except for a set of $x \in [X,2X]$ with measure $O\left (X (\log X)^{C-1} \right )$, and $C>1$ constant. Moreover, 
\begin{equation}\label{key}
\sum_{x-y 
	\leq n \leq x}\Lambda(n) >\frac{y}{2}.
\end{equation}
\end{thm}

\begin{proof}
Based on zero density theorems, \textit{ibidem}, p. 190.
\end{proof}
An introduction to zero density theorems and its application to primes in short intervals is given in \cite[p. 98]{KA93}, \cite[p.\ 264]{IK04}, et cetera. \\

The same result works for larger intervals $[x-x^{\theta +\varepsilon},x]$, where $\theta \in (1/6,1/2)$ and $(1-\theta)C<2$, with exceptions $O\left (X (\log X)^{C-1} \right ), C>1$.

%cccccccccccccccccccccccccc
\chapter{Euler Totient  Function}\label{c21}
A few average orders for the Euler totient function are computed, and a few are recorded without proofs in this chapter. 

\section{Euler Function}
Various average orders of the Euler phi function are requided in the analysis of the elliptic groups of points of elliptic curves. A few of the simpler are computed here. The basic forms are involve the normalized totient function
\begin{equation} \label{2100}
\frac{\varphi(n)}{n}=\sum_{d|n}\frac{\mu(d)}{d}=\prod_{p|n}(1-\frac{1}{p}),
\end{equation} 
where $\mu(n)$ is the Mobious function. The corresponding average orders for various parameters are recorded below.\\

\section{Sums Over The Integers}
\begin{lem} \label{lem40.1}
Let \(x\geq 1\) be a large number, and let \(\varphi (n)\) be the Euler totient function. Then
\begin{equation}
\sum_{n\leq x} \varphi(n)=\frac{3}{\pi^2}
x^2+
O(x\log x).
\end{equation} 
\end{lem}

\begin{lem} \label{lem40.2}
	Let \(x\geq 1\) be a large number, and let \(\varphi (n)\) be the Euler totient function. Then
\begin{equation}
\sum_{n\leq x} \frac{\varphi(n)}{n}=\frac{6}{\pi^2}
x+
O(\log x).
\end{equation}
\end{lem}

The proofs are widely available in the literature.

\section{Sums Over Arithmetic Progressions}
The proofs for the same results, but over arithmetic progression is significantly more involved. 

\begin{thm} \label{thm40.1}
{ \normalfont (\cite[p.\ 191]{PA88})} Let \(x\geq 1\) be a large number, let $1 \leq a \leq q$ be integers, and let \(\varphi (n)\) be the Euler totient function. Then
	\begin{equation}
	\sum_{\substack{n\leq x \\ n \equiv a \bmod q}} \varphi(n)=\frac{3}{\pi^2} \frac{1}{q}\prod_{p|q}\left ( 1-\frac{1}{p^2}\right)^{-1}\prod_{p|\gcd(a,q)}\left ( 1-\frac{1}{p}\right)
	x^2+
	O(x\log x).
	\end{equation}
\end{thm}

\begin{lem} \label{lem40.3}
	Let \(x\geq 1\) be a large number, let $1 \leq a \leq q$ be integers, and let \(\varphi (n)\) be the Euler totient function. Then
	\begin{equation}
	\sum_{\substack{n\leq x \\ n \equiv a \bmod q}} \frac{\varphi(n)}{n}=\frac{6}{\pi^2} \frac{1}{q}\prod_{p|q}\left ( 1-\frac{1}{p^2}\right)^{-1}\prod_{p|\gcd(a,q)}\left ( 1-\frac{1}{p}\right)
	x+
	O(\log x).
	\end{equation}
\end{lem} 

\begin{proof} This follows form Theorem \ref{thm40.1} by partial summation. \end{proof}

The expression $\log_k x$ denote the $k$-iteration of the logarithm, and the expression 
\begin{equation}
a=\prod_{\log_2 x<r\leq 2\log_2 x}r
\end{equation} 
is the product of all the consecutive primes $r$ within the stated range.\\

\begin{lem} \label{lem40.4}   Let $x \geq 1$ be a large number, and let $a=\prod_{\log_2 x<r\leq 2\log_2 x}r$ and $q=\prod_{r\leq 2\log_2 x}r$. Then, 
\begin{equation}
\sum_{\substack{n\leq x \\ n \equiv a \bmod q}} \frac{\varphi(n)}{n}=C(q,a)x+
O(\log x).
\end{equation}	
as $x \to \infty$, where $C(q,a)>0$ is a constant depending on $a$,and $q$.
\end{lem}

\begin{proof} This the same as Lemma \ref{lem40.3}. Accordingly, it is sufficient to calculate the constant. Toward this end, observe that
\begin{eqnarray}
C(q,a)&=&\frac{6}{\pi^2} \frac{1}{q}\prod_{p|q}\left ( 1-\frac{1}{p^2}\right)^{-1}\prod_{p|\gcd(a,q)}\left ( 1-\frac{1}{p}\right) \\
&=&\frac{6}{\pi^2} \frac{1}{q}\prod_{p \leq \log \log x}\left ( 1-\frac{1}{p^2}\right)^{-1}\prod_{\log \log x<p \leq 2 \log \log x}\left ( 1+\frac{1}{p}\right).
\end{eqnarray}
Here, $\gcd(a,q)=\prod_{\lg \log x<p \leq 2 \log \log x}r$, and 
\begin{equation}
\prod_{p|\gcd(a,q)}\left ( 1-\frac{1}{p}\right)=\prod_{\lg \log x<p \leq 2 \log \log x}\left ( 1-\frac{1}{p}\right),
\end{equation}
which cancels the corresponding upper parts of the first product. The remaining tail product
\begin{equation}
\prod_{\lg \log x<p \leq 2 \log \log x}\left ( 1+\frac{1}{p}\right)^{-1}=c_0+O\left(\frac{1}{\log \log x}\right),
\end{equation} 
where $c_0>0$ is a constant, which is a routine application of Mertens theorem. Therefore,
\begin{eqnarray}
C(q,a)&=&
\frac{6}{\pi^2} \frac{1}{q}\prod_{p \leq \log \log x}\left ( 1-\frac{1}{p^2}\right)^{-1}\prod_{\log \log x<p \leq 2 \log \log x}\left ( 1+\frac{1}{p}\right) \nonumber\\
&=& \frac{6}{\pi^2} \left( \frac{\pi^2}{6}+O\left(\frac{1}{\log \log x}\right) \right )\left ( c_0+O\left(\frac{1}{\log \log x}\right)\right ) \\
&=&c_0+O\left(\frac{1}{\log \log x}\right)\nonumber.
\end{eqnarray}
This completes the proof.  \end{proof}

The summatory function for Euler psi function was extended to polynomial arguments by Schwarz and other authors. An important case is for the simplest irreducible quadratic polynomial.

\begin{lem} \label{lem40.6}   Let $x \geq1$ be a real number, and let $f(x)=x^2+1$ be a polynomial over the integers. Then  
\begin{equation}
\sum_{n \leq x} \frac{\varphi(n^2+1)}{n^2+1}=\frac{1}{2} \prod_{p \equiv 1 \bmod 4} \left (1-\frac{2}{p^2} \right )x+O\left (\log^{2} x \right ).
\end{equation}

\end{lem}

\begin{proof} There are many proofs of this result, see \cite[p. \ 192]{PA88}.
\end{proof}

\section{Sums Over The Primes}
A finite sum over the shifted primes is computed in the next result.\\

\begin{lem} \label{lem40.5}   Let $x \geq1$ be a real number, and let $1 \leq a <q$, $\gcd(a,q)=1$ be integers with $q=O(\log^C x)$. Then  
\begin{equation}
\sum_{\substack{p \leq x \\ p \equiv a \bmod q}} \frac{\varphi(p-1)}{p-1}=\frac{c_q}{\varphi(q)}\frac{x}{\log x}+O\left (\frac{x}{\log^{B-1} x} \right ),
\end{equation}
where $c_q>0$, and $B> C>0$ are constants depending on $q>1$.
\end{lem}

\begin{proof} Use identity (\ref{2100}) to rewrite the summatory function as
\begin{eqnarray}
\sum_{\substack{p \leq x \\ p \equiv a \bmod q}} \frac{\varphi(p-1)}{p-1} &=&  \sum_{\substack{p \leq x \\ p \equiv a \bmod q}} \sum_{d\,|\,p-1} \frac{\mu(d)}{d} \nonumber\\
&=&  \sum_{d \leq x} \frac{\mu(d)}{d} \sum_{\substack{p \leq x \\ p \equiv a \bmod q \\ p \equiv 1 \bmod d}} 1 \nonumber.
\end{eqnarray}
The prime number theorem on arithmetic progression leads to
\begin{eqnarray}
\sum_{d \leq x} \frac{\mu(d)}{d} \sum_{\substack{p \leq x \\ p \equiv a \bmod q \\ p \equiv 1 \bmod d}} 1 &=& \sum_{d \leq x} \frac{\mu(d)}{d} \cdot  \pi(x,dq,b) \nonumber \\
& = & \sum_{d \leq x} \frac{\mu(d)}{d} \left ( \frac{1}{\varphi(dq)}\frac{x}{\log x} +O\left (\frac{x}{\log^B x} \right ) \right ) \\
& = & \frac{x}{\log x}\sum_{d \leq x} \frac{\mu(d)}{d \varphi(dq)} +O\left (\frac{x}{\log^B x} \sum_{d \leq x} \frac{1}{d } \right ) , \nonumber\\
\end{eqnarray}
where $b \equiv 1 \bmod d$, and $b \equiv a \bmod q$ is a residue class modulo $dq$. Choose a constant $B>2$, and use the identity
\begin{equation}
\frac{\varphi(dq)}{dq}=\frac{\varphi(d)}{d}\frac{\varphi(q)}{q} \prod_{r\,|\, \gcd(d,q)} \left (1-\frac{1}{r} \right )^{-1} , 
\end{equation}
where $r\geq 2$ is prime, to rewrite it as
\begin{eqnarray}
& &\frac{x}{\log x}\sum_{d \leq x} \frac{\mu(d)}{d \varphi(dq)} +O\left (\frac{x}{\log^B x} \sum_{d \leq x} \frac{1}{d } \right ) \nonumber\\
&=& \frac{1}{\varphi(q)}\frac{x}{\log x}\sum_{d \leq x} \frac{\mu(d)}{d \varphi(d)} \prod_{r\,|\, \gcd(d,q)} \left (1-\frac{1}{r} \right )+O\left (\frac{x}{\log^{B-1} x} \right ) . \nonumber\\
&=& \frac{c_q}{\varphi(q)}\frac{x}{\log x}+O\left (\frac{x}{\log^{B-1} x} \right ) . \nonumber\\
\end{eqnarray}
The constant $c_q>0$ is specified by the finite sum, which has a product expansion as
\begin{eqnarray}
\sum_{d \leq x} \frac{\mu(d)}{d \varphi(d)} \prod_{r\,|\, \gcd(d,q)} \left (1-\frac{1}{r} \right ) &=& \prod_{p \geq 2} \left (1-\frac{1}{p(p-1)} \prod_{p\,|\, q} \left (1-\frac{1}{p} \right )\right ) + O\left (\frac{1}{x} \right ). \nonumber\\
\end{eqnarray}
\end{proof}

\appendix

\chapter{}
Let $n \in \mathbb{N}$ be a nonzero integer. The period of the $\ell$-adic expansion $1/n=0.\overline{x_{d-1}x_{d-2},...,x_1x_0}$, where $0 \leq x_i \leq \ell -1$ is defined as the least integer such that $\ell^d-1\equiv 0 \bmod n$.\\
		
		Given a nonzero integer $n \in \mathbb{N}$, the binary expansion has the form
		\begin{equation}\label{1820}
		\frac{1}{n}=0.\overline{x_dx_{d-1} \ldots x_1x_0},
		\end{equation}
		where $x_i\in \{0,1\}$. The period is totally determined by the prime decomposition $n=p_1^{v_1}p_2^{v_2} \cdots p_t^{v_t}$, where $p_i$ is priem, and $v_i\geq 0$. Accordingly, it is sufficient to consider the prime powers. In addition, almost every prime power $p^k, k \geq 1$ has a period of the form $dp^{k-1}$, where $d \,|\, p-1$. \\

The repeating $\ell$-adic fractions representation 
	\begin{equation}
	\frac{1}{p}=\frac{m}{\ell^d}+\frac{m}{\ell^{2d}}+ \cdots= m \sum_{n \geq 1} \frac{1}{\ell^{dn}}=\frac{m}{\ell^d-1}
	\end{equation}
	has the maximal period $d=p-1$ if and only if $\ell \ne \pm 1, v^2$ has order $\ord_p(\ell)=p-1$ modulo $p$. This follows from the Fermat little theorem and $\ell^d-1 =mp$.\\

 The exceptions are known as Abel-Wieferich primes. The subset of Abel-Wieferich primes satisfy the congruence 
		\begin{equation}\label{1825}
		\ell^{p-1}-1\equiv 0 \bmod p^2,
		\end{equation}
		confer \cite[p.\ 333]{RP96} for other details. Numerical data are given in \cite{PA09} and by other sources. The $\ell$-adic expansions of these prime powers have irregular patterns. \\
		
		An application of $\ell$-adic expansions $1/p^k$ of prime powers to normal numbers is given in \cite{ST76}. \\
		
		Some examples of these expansions for prime powers are given below.\\
		
		\section{Binary Base} 
		\textbf{1.} The binary base $\ell=2$ is a primitive root $\text{mod }3^k$ for $k \geq 2$, so the period is maximal and has a uniform formula $d_k=\varphi(3^k)=2,2\cdot 3,2\cdot 3^2, \ldots$. \\ 
		
		\begin{tabular}{ll}
		Period & Binary Expansion \\[12pt]
		
		$\varphi(3)=2$, &  $\displaystyle \frac{1}{3}=0.\overline{01}$\\[12pt] 
		
		$\varphi(3^2)=2\cdot 3$, &  $\displaystyle \frac{1}{3^2}=0.\overline{000111}$  \\[12pt] 
		
		$\varphi(3^3)=2\cdot 3^2,$ &  $\displaystyle \frac{1}{3^3}=0.\overline{000010010111101101}$  \\ [12pt]
		
		\vdots &\vdots
		\end{tabular} \\

		\textbf{2.} The binary base $\ell=2$ is a primitive root $\text{mod }5^k$ for $k \geq 2$, so the period is maximal and has a uniform formula $d_k=\varphi(5^k)=2,4\cdot 5,4\cdot 5^2, \ldots$. \\ 
		
		\begin{tabular}{l l}
		Period & Binary Expansion \\[12pt]
		$\varphi(5)=4$, &  $\displaystyle \frac{1}{5}=0.\overline{0011}$\\[12pt] 
		
		$\varphi(5^2)=4\cdot 5$, &  $\displaystyle \frac{1}{5^2}=0.\overline{00001010001111010111}$  \\[12pt] 
		
		$\varphi(5^3)=4\cdot 5^2,$ &  $\displaystyle \frac{1}{5^3}=0.\overline{0000001000001100010010011\ldots}$  \\ [12pt]
		\vdots&\vdots 
		\end{tabular} \\
		
		\textbf{3.} The binary base $\ell=2$ is not a primitive root $\bmod 7^k$ for $k \geq 2$, so the period is not maximal. But it has a uniform formula $d_k=\varphi(7^k)=3,3\cdot 7,3\cdot 7^2, \ldots$. \\ 
		
		\begin{tabular}{ll}
		Period & Binary Expansion \\[12pt]
		
		$\varphi(7)=3$, &  $\displaystyle \frac{1}{7}=0.\overline{001}$\\[12pt] 
		
		$\varphi(7^2)=3\cdot 7$, &  $\displaystyle \frac{1}{7^2}=0.\overline{000001010011100101111}$  \\[12pt] 
		
		$\varphi(7^3)=3\cdot 7^2,$ &  $\displaystyle \frac{1}{7^3}=0.\overline{00000001101101100100\ldots}$  \\ [12pt]
		
		\vdots & \vdots
		
		\end{tabular}
		
		\textbf{4.} The prime $p=1093$ is an irregular case since $2^{p-1}-1 \equiv 0 \bmod p^2$. Accordingly, the binary base $\ell=2$ is a primitive root $\text{mod } p$ and $\text{mod } p^k$ for $k \geq 3$. But fails to be a primitive root $\text{mod }p^2$, so the period is not maximal for $k=2$, and has a nonuniform formula $d_k=\varphi(1093^k)$ for $k\ne 2$, and $d_k=1092$. This phenomenon also occurs for $p=3511$.\\
		
		\begin{tabular}{ll}
		Period & Binary Expansion \\[12pt]
		
		$\varphi(1093)=1092$, &  $\displaystyle \frac{2^{10}}{1093}=0.\overline{1110111111010110110001110001010111001 \ldots}$\\[12pt] 
		
		$\varphi(1093)=1092$, &  $\displaystyle \frac{2^{20}}{1093^2}=0.\overline{111000001011001010111011111011000001 \ldots}$  \\[12pt] 
		
		$\varphi(1093^3)=1092\cdot 1093^2,$ &  $\displaystyle \frac{2^{30}}{1093^3}=0.\overline{110100101000001101100001100110001001\ldots}$  \\ [12pt]
		
		\vdots &\vdots
		
		\end{tabular}

		\section{Decimal Base}
		\textbf{1.} The decimal base $\ell=10$ is a primitive root $\text{mod }7^k$ for $k \geq 2$, so the period is maximal and has a uniform formula $d_k=\varphi(7^k)=6,6\cdot 7,6\cdot 7^2, \ldots$. \\ 
		
		\begin{tabular}{ll}
		Period & Decimal Expansion \\[12pt]
		
		$\varphi(7)=6$, &  $\displaystyle \frac{1}{7}=0.\overline{142857}$\\[12pt] 
		
		$\varphi(7^2)=6\cdot 7$, &  $\displaystyle \frac{1}{7^2}=0.\overline{02040816326530612244897959183673 \ldots}$  \\[12pt] 
		
		$\varphi(7^3)=6\cdot 7^2,$ &  $\displaystyle \frac{1}{7^3}=0.\overline{0029154518950437317784256559766\ldots}$  \\ [12pt]
		
		\vdots &\vdots
		\end{tabular}

		\textbf{2.} The decimal base $\ell=10$ is not a primitive root $\text{mod }11^k$ for $k \geq 2$, so the period is not maximal. However, it has a uniform formula $d_k=\varphi(11^k)=2,2\cdot 11,2\cdot 11^2, \ldots$. \\ 
		
		\begin{tabular}{l l}
		Period & Decimal Expansion \\[12pt]
		$\varphi(11)=2$, &  $\displaystyle \frac{1}{11}=0.\overline{09}$\\[12pt] 
		
		$\varphi(11^2)=2\cdot 11$, &  $\displaystyle \frac{1}{11^2}=0.\overline{0082644628099173553719}$  \\[12pt] 
		
		$\varphi(11^3)=2\cdot 11^2,$ &  $\displaystyle \frac{1}{11^3}=0.\overline{00075131480090157776108189331329827\ldots}$  \\ [12pt]
		\vdots&\vdots
		
		\end{tabular}
		
		\textbf{3.} The decimal base $\ell=13$ is not a primitive root $\bmod 13^k$ for $k \geq 2$, so the period is not maximal. But it has a uniform formula $d_k=\varphi(13^k)/2=6,6\cdot 13,6\cdot 13^2, \ldots$. \\ 
		
		\begin{tabular}{ll}
		Period & Decimal Expansion \\[12pt]
		
		$\varphi(13)=6$, &  $\displaystyle \frac{1}{13}=0.\overline{076923}$\\[12pt] 
		
		$\varphi(13^2)=6\cdot 13$, &  $\displaystyle \frac{1}{13^2}=0.\overline{00591715976331360946745562130178\ldots}$  \\[12pt] 
		
		$\varphi(13^3)=6\cdot 13^2,$ &  $\displaystyle \frac{1}{13^3}=0.\overline{00045516613563950842057350933090578\ldots}$  \\ [12pt]
		
		\vdots & \vdots
		
		\end{tabular}
		
		\textbf{4.} The prime $p=487$ is an irregular case since $10^{p-1}-1 \equiv 0 \bmod p^2$. Accordingly, the decimal base $\ell=10$ is a primitive root $\text{mod } p$ and $\text{mod } p^k$ for $k \geq 3$. But fails to be a primitive root $\text{mod }p^2$, so the period is not maximal for $k=2$, and has a nonuniform formula $d_k=\varphi(487^k)$ for $k\ne 2$, and $d_2=486$. \\
		
		\begin{tabular}{ll}
		Period & Decimal Expansion \\[12pt]
		
		$\varphi(487)=486$, &  $\displaystyle \frac{1}{487}=0.\overline{00205338809034907597535934291581 \ldots}$\\[12pt] 
		
		$\varphi(487)=486$, &  $\displaystyle \frac{1}{487^2}=0.\overline{000004216402649587425000737870463\ldots}$  \\[12pt] 
		
		$\varphi(487^3)=486\cdot 487^2,$ &  $\displaystyle \frac{1}{487^3}=0.\overline{000000008657910984779106777695832\ldots}$  \\ [12pt]
		
		\vdots &\vdots
		
		\end{tabular}

\listoftables


\newpage
\begin{thebibliography}{9999}
\bibitem{AP98} Apostol, Tom M. Introduction to analytic number theory. Undergraduate Texts in Mathematics. Springer-Verlag, New York-Heidelberg, 1976.
		
\bibitem{AS15} Akbary, Amir; Scholten, Keilan. Artin prime producing polynomials. Math. Comp.  84  (2015),  no. 294, 1861-1882.
		
\bibitem{AM14} Ambrose, Christopher. Artin primitive root conjecture and a problem of Rohrlich. Math. Proc. Cambridge Philos. Soc. 157(2014), no. 1, 79-99.
		
		\bibitem{BC11} Balog, Antal; Cojocaru, Alina-Carmen; David, Chantal. Average twin prime conjecture for elliptic curves. Amer.
		J. Math. 133,(2011), no. 5, 1179-1229.
		
		\bibitem{BC02} Bailey, D. H.; Crandall, R. E. Random generators and normal numbers, Experimental Mathematics, 11 (4): 527-546, 2002.
		
		\bibitem{BH62} Bateman, P. T., Horn, R. A. (1962), A heuristic asymptotic formula concerning the distribution of prime numbers, Math. Comp. 16 363-367.
		
		\bibitem{BJ05} Baier, Stephan; Jones, Nathan A refined version of the Lang-Trotter conjecture. Int. Math. Res. Not. IMRN  2009,  no. 3, 433-461. 
		
		\bibitem{BJ07} Bourgain, J. Exponential sum estimates in finite commutative rings and applications. J. Anal. Math. 101, (2007),
		325-355.

\bibitem{BJ17}  Julio Brau, Character sums for elliptic curve densities, arXiv:1703.04154.

		\bibitem{BM09} Bullynck, Maarten, Decimal periods and their tables: a German research topic (1765-1801). Historia Math.  36  (2009),  no. 2, 137-160.
		
		\bibitem{BN04} Bourgain, Jean. New bounds on exponential sums related to the Diffie-Hellman distributions. C. R. Math. Acad. Sci. Paris, 338,
		(2004), no. 11, 825-830.
		
		\bibitem{BR07} Baier, Stephan; Zhao, Liangyi. Primes in quadratic progressions on average. Math. Ann. 338 (2007), no. 4, 963-982.
\bibitem{BP75} Borosh, I.; Moreno, C. J.; Porta, H. Elliptic curves over finite fields. II. Math. Comput. 29 (1975), 951-964. 

\bibitem{CA05} Cojocaru, Alina Carmen. Reductions of an elliptic curve with almost prime orders. Acta Arith. 119 (2005), no. 3, 265-289. 

\bibitem{CA04} Cojocaru, Alina Carmen. Questions about the reductions modulo primes of an elliptic curve.  Number theory,  61-79, CRM Proc. Lecture Notes, 36, Amer. Math. Soc., Providence, RI, 2004. 

 \bibitem{CA03} Cojocaru, Alina Carmen. Cyclicity of CM elliptic curves modulo p. Trans. Amer. Math. Soc. 355 (2003), no. 7, 2651-2662.

 \bibitem{CM90} Coleman, M. D. The distribution of points at which binary quadratic forms are prime. Proc. London Math. Soc. (3) 61 (1990), no. 3, 433-456. 		
		\bibitem{CC09} Cristian Cobeli, On a Problem of Mordell with Primitive Roots, arXiv:0911.2832.
		
		\bibitem{CI09} Ivan Cherednik, A note on Artin's constant, arXiv:0810.2325v4.
		
		\bibitem{CS73} Chowla, P.; Chowla, S. On Hirzebruch sums and a theorem of Schinzel. Acta Arith.  24  (1973), 223-224. 
		
		\bibitem{CZ98} Cobeli, Cristian; Zaharescu, Alexandru. On the distribution of primitive roots mod p . Acta Arith. 83, (1998), no. 2, 143-153.
		
		\bibitem{CJ07} Joseph Cohen, Primitive roots in quadratic fields, II, Journal of Number Theory 124 (2007) 429-441.

\bibitem{CP16}  Pete L. Clark, Paul Pollack, The truth about torsion in the CM case, II, arXiv:1612.06318.

		\bibitem{DH37} H. Davenport, On Primitive Roots in Finite Fields, Quarterly J. Math. 1937, 308-312. 
		
		\bibitem{DL66} Dickson, Leonard Eugene. History of the theory of numbers. Vol. I: Divisibility and primality.  Chelsea Publishing Co., New York 1966.

\bibitem{DL12} De Koninck, Jean-Marie; Luca, Florian Analytic number theory. Exploring the anatomy of integers. Graduate Studies in Mathematics, 134. American Mathematical Society, Providence, RI, 2012.


\bibitem{DS12} Rainer Dietmann, Christian Elsholtz, Igor E. Shparlinski, On Gaps Between Primitive Roots in the Hamming Metric, arXiv:1207.0842.
		 
\bibitem{ET84} Erdos, P.; Turk, J. Products of integers in short intervals. Acta Arith.  44  (1984),  no. 2, 147-174.		
\bibitem{ES57} Paul Erdos, Harold N. Shapiro, On The Least Primitive Root Of A Prime, 1957, euclidproject.org.

\bibitem{FK13} Tristan Freiberg, Par Kurlberg, On the average exponent of elliptic curves modulo p, arXiv:1203.4382.

\bibitem{FJ76}  Friedlander, John B. Integers free from large and small primes. Proc. London Math. Soc. (3)  33  (1976), no. 3, 565-576.	

\bibitem{FS02} Friedlander, John B.; Konyagin, Sergei; Shparlinski, Igor. E. Some doubly exponential sums over Zm. Acta Arith. 105, (2002), no. 4, 349-370.
 
\bibitem{FS01}Friedlander, John B.; Shparlinski, Igor E. Double exponential sums over thin sets. Proc. Amer. Math. Soc.  129  (2001),  no. 6, 1617-1621.		
 \bibitem{FS00} Friedlander, John B.; Hansen, Jan; Shparlinski, Igor E. Character sums with exponential functions. Mathematika  47  (2000),  no. 1-2, 75-85 (2002).			
\bibitem{FG91}  Friedlander, John; Granville, Andrew. Limitations to the equi-distribution of primes. IV. Proc. Roy. Soc. London Ser. A  435  (1991),  no. 1893, 197-204.
		
\bibitem{FI10} Friedlander, John; Iwaniec, Henryk. Opera de cribro. American Mathematical Society Colloquium Publications, 57. American Mathematical Society, Providence, RI, 2010.
\bibitem{FL87} Friedlander, J. B.; Lagarias, J. C. On the distribution in short intervals of integers having no large prime factor. J. Number Theory  25  (1987),  no. 3, 249-273. 
\bibitem{FJ76} Friedlander, John B. Integers free from large and small primes. Proc. London Math. Soc. (3)  33  (1976), no. 3, 565-576.

\bibitem{FT91} Fouvry, E.; Tenenbaum, G. Entiers sans grand facteur premier en progressions arithmetiques. (French) [Integers free of large prime factors in arithmetical progressions] Proc. London Math. Soc. (3)  63  (1991),  no. 3, 449-494.


\bibitem{FT90} Filaseta, Michael Short interval results for squarefree numbers. J. Number Theory  35  (1990),  no. 2, 128-149.
\bibitem{FT91} foubry tenebaum 1991 acta arth
		\bibitem{FX11} Adam Tyler Felix, Variations on Artin Primitive Root Conjecture, PhD Thesis, Queen University, Canada, August 2011.
		
		\bibitem{GC86} Gauss, Carl Friedrich. Disquisitiones arithmeticae. Translated and with a preface by Arthur A. Clarke. Revised by William C. Waterhouse, Cornelius Greither and A. W. Grootendorst and with a preface by Waterhouse. Springer-Verlag, New York, 1986.
		
		\bibitem{GC92} Cristian D. Gonzalez, Class numbers of quadratic function fields and continued fractions, Journal of Number Theory, Volume 40, Issue 1, January 1992, Pages 38-59.
		
		\bibitem{GK94} Girstmair, Kurt A "popular'' class number formula. Amer. Math. Monthly  101  (1994),  no. 10, 997-1001.
		
		\bibitem{GKS94}  Girstmair, Kurt The digits of 1/p  in connection with class number factors. Acta Arith.  67  (1994),  no. 4, 381-386.
		
		\bibitem{GL10} Luca Goldoni, Prime Numbers And Polynomials, Phd Thesis, Universita` Degli Studi Di Trento, 2010.
		
		\bibitem{GM68} Goldfeld, Morris. Artin conjecture on the average. Mathematika 15, 1968, 223-226.
		
\bibitem{GM84} Gupta, Rajiv; Murty, M. Ram. A remark on Artin's conjecture. Invent. Math. 78 (1984), no. 1, 127-130. 
\bibitem{GM86} Gupta, Rajiv; Murty, M. Ram Primitive points on elliptic curves. Compositio Math.  58  (1986),  no. 1, 13-44.
 
\bibitem{GM90} Gupta, Rajiv; Murty, M. Ram Cyclicity and generation of points mod p on elliptic curves. Invent. Math. 101 (1990), no. 1, 225-235. 

\bibitem{GM08}  Yves Gallot, Pieter Moree, Huib Hommersom; Value distribution of cyclotomic polynomial coefficients, arXiv:0803.2483. 

\bibitem{GR13} Granville, Andrew Smooth numbers: computational number theory and beyond.  Algorithmic number theory: lattices, number fields, curves and cryptography,  267-323, Math. Sci. Res. Inst. Publ., 44, Cambridge Univ. Press, Cambridge, 2008.

\bibitem{GZ05} Garaev, M. Z. Double exponential sums related to Diffie-Hellman distributions. Int. Math. Res. Not.  2005,  no. 17, 1005-1014. 

\bibitem{GK05}  M. Z. Garaev, A. A. Karatsuba, New estimates of double trigonometric sums with exponential functions, arXiv:math/0504026. 

\bibitem{GZ97} Gebel, J.; Petho, A.; Zimmer, H. G. On Mordell's equation. Compositio Math.  110  (1998),  no. 3, 335-367.

\bibitem{HB86} D. R. Heath-Brown, Artins conjecture for primitive roots, Quart. J. Math. Oxford Ser. (2) 37, No. 145, 27-38, 1986.
		
\bibitem{HC67} C. Hooley, On Artins conjecture, J. Reine Angew. Math. 225, 209-220, 1967.
		
\bibitem{HF73} Hirzebruch, Friedrich E. P. Hilbert modular surfaces. Enseignement Math. (2)  19  (1973), 183-281.  
		
\bibitem{HG07} Harman, Glyn. Prime-detecting sieves. London Mathematical Society Monographs Series, 33. Princeton University Press, Princeton, NJ, 2007. 

\bibitem{HG99}  Harman, Glyn. Integers without large prime factors in short intervals and arithmetic progressions. Acta Arith. 91 (1999), no. 3, 279-289.
\bibitem{HG91} Harman, Glyn. Short intervals containing numbers without large prime factors. Math. Proc. Cambridge Philos. Soc.  109  (1991),  no. 1, 1-5.

\bibitem{HJ86}  Hinz, Jurgen G. Character sums and primitive roots in algebraic number fields. Monatsh. Math.  95  (1983),  no. 4, 275-286.

\bibitem{HW08} Hardy, G. H.; Wright, E. M. An introduction to the theory of numbers. Sixth edition. Revised by D. R. Heath-Brown and J. H. Silverman. With a foreword by Andrew Wiles. Oxford University Press, 2008.
		
		
\bibitem{IR90} Ireland, Kenneth; Rosen, Michael. A classical introduction to modern number theory. Second edition. Graduate Texts in Mathematics, 84. Springer-Verlag, New York, 1990.
		
\bibitem{IK04} Iwaniec, Henryk; Kowalski, Emmanuel. Analytic number theory. AMS Colloquium Publications,
		53. American Mathematical Society, Providence, RI, 2004.
		
\bibitem{IJ08} Iwaniec, Henryk; Jimenez Urroz, Jorge. Orders of CM elliptic curves modulo p with at most two primes. Ann. Sc. Norm. Super. Pisa Cl. Sci. (5) 9 (2010), no. 4, 815-832.	

\bibitem{JJ08} Jimenez Urroz, Jorge. Almost prime orders of CM elliptic curves modulo p. Algorithmic number theory, 74-87, Lecture Notes in Comput. Sci., 5011, Springer, Berlin, 2008.

\bibitem{JN10} Jones, Nathan. Almost all elliptic curves are Serre curves. Trans. Amer. Math. Soc. 362 (2010), no. 3, 1547-1570.


\bibitem{JN09} Jones, Nathan. Averages of elliptic curve constants. Math. Ann. 345 (2009), no. 3, 685-710. 

\bibitem{JN10} Jones, Nathan. Almost all elliptic curves are Serre curves. Trans. Amer. Math. Soc. 362 (2010), no. 3, 1547-1570.
		
\bibitem{JW03} M. J. Jacobson and H. C. Williams, New quadratic polynomials with high densities of prime values, Math. Comp., 72 (2003), 499-519.

\bibitem{KA93} Anatolij A. Karacuba, Basic Analytic Number Theory, Springer-Verlag, 1993.

\bibitem{KN88} Koblitz, Neal Primality of the number of points on an elliptic curve over a finite field. Pacific J. Math.  131  (1988),  no. 1, 157-165.

\bibitem{KN93} Koblitz, Neal. Introduction to elliptic curves and modular forms. Second edition. Graduate Texts in Mathematics, 97. Springer-Verlag, New York, 1993.
		
\bibitem{KR16} Keating, Jonathan; Rudnick, Zeev. Squarefree polynomials and Mobius values in short intervals and arithmetic progressions. Algebra Number Theory  10  (2016),  no. 2, 375-420.
		
\bibitem{KS12} Konyagin, Sergei V.; Shparlinski, Igor E. On the consecutive powers of a primitive root: gaps and exponential sums. Mathematika  58  (2012),  no. 1, 11-20.

\bibitem{KS00} Kohel, David R.; Shparlinski, Igor E. On exponential sums and group generators for elliptic curves over finite fields.  Algorithmic number theory (Leiden, 2000),  395-404, Lecture Notes in Comput. Sci., 1838, Springer, Berlin, 2000.

\bibitem{LH77} Lenstra, H. W., Jr. On Artin's conjecture and Euclid's algorithm in global fields. Invent. Math. 42 (1977), 201-224.

\bibitem{LT77} Lang, S.; Trotter, H. Primitive points on elliptic curves. Bull. Amer. Math. Soc. 83 (1977), no. 2, 289-292		
		
\bibitem{LR12} Lemke Oliver, Robert J. Almost-primes represented by quadratic polynomials. Acta Arith.  151  (2012),  no. 3, 241-261.
		
\bibitem{LM11} H. W. Lenstra Jr, P. Moree, P. Stevenhagen, Character sums for primitive root densities,
		arXiv:1112.4816.
		
\bibitem{LD63} Lehmer, D. H. A note on primitive roots. Scripta Math.  26  1963 117-119.

\bibitem{LMFDB} The L-functions and modular forms database, www.lmfdb.org.		
		
\bibitem{LN97} Lidl, Rudolf; Niederreiter, Harald. Finite fields. With a foreword by P. M. Cohn. Second edition. Encyclopedia of Mathematics and its Applications, 20. Cambridge University Press, Cambridge, 1997.
		
\bibitem{LS14} Lenstra, H. W., Jr.; Stevenhagen, P.; Moree, P. Character sums for primitive root densities. Math. Proc. Cambridge Philos. Soc. 157 (2014), no. 3, 489-511.

\bibitem{MA09} Matomaki, Kaisa. A note on primes of the form p=aq2+1. Acta Arith. 137 (2009), no. 2, 133-137.
		
\bibitem{MK86} McCurley, Kevin S. The smallest prime value of $^n+a$. Canad. J. Math. 38 (1986), no. 4, 925-936. 

\bibitem{MK84} McCurley, Kevin S. Prime values of polynomials and irreducibility testing. Bull. Amer. Math. Soc. (N.S.)  11  (1984),  no. 1, 155-158.
		
\bibitem{MK76}  Matthews, K. R. A generalisation of Artin's conjecture for primitive roots. Acta Arith.  29  (1976), no. 2, 113-146.
		
\bibitem{MJ12}  Maynard, James. On the Brun-Titchmarsh theorem. Acta Arith.  157  (2013),  no. 3, 249-296.

\bibitem{MG15} Giulio Meleleo, Questions Related To Primitive Points On Elliptic Curves And Statistics For Biquadratic Curves Over
Finite Fields, Thesis, Universita degli Studi, Roma Tre, 2015.

\bibitem{ML49} Mirsky, L. The number of representations of an integer as the sum of a prime and a $k$-free integer. Amer. Math. Monthly  56,  (1949). 17-19.
		
\bibitem{MP07} Moree, P. Artin prime producing quadratics. Abh. Math. Sem. Univ. Hamburg  77  (2007), 109-127.
		
\bibitem{MP04} Moree, P. Artin prime producing quadratics. Abh. Math. Sem. Univ. Hamburg  77  (2007), 109-127 or arXiv:math/0412262.
		
\bibitem{MP99}  Moree, Pieter On primes in arithmetic progression having a prescribed primitive root. J. Number Theory  78  (1999),  no. 1, 85-98.	
		
\bibitem{MT06} Miller, Steven J.; Takloo-Bighash, Ramin. An invitation to modern number theory. With a foreword by Peter Sarnak. Princeton University Press, Princeton, NJ, 2006. 
		
\bibitem{MT10} Murty, M. Ram; Thangadurai, R. The class number of $Q(\sqrt{-p})$  and digits of $1/p$ . Proc. Amer. Math. Soc.  139  (2011),  no. 4, 1277-1289.
		
\bibitem{MV07} Montgomery, Hugh L.; Vaughan, Robert C. Multiplicative number theory. I. Classical theory. Cambridge University Press, Cambridge, 2007.

\bibitem{MV09} Miret, J.; Moreno, R.; Rio, A.; Valls, M. Computing the l -power torsion of an elliptic curve over a finite field. Math. Comp.  78  (2009),  no. 267, 1767-1786.

\bibitem{NW00} Narkiewicz, W. The development of prime number theory. From Euclid to Hardy and Littlewood. Springer Monographs in Mathematics. Springer-Verlag, Berlin, 2000. 
		
		\bibitem{PA09} Paszkiewicz, A. A new prime $p$ for which the least primitive root $\tmod p$ and the least primitive root $\tmod p^2$ are not equal. Math. Comp.  78  (2009),  no. 266, 1193-1195

\bibitem{PA88} A. G. Postnikov, Introduction to analytic number theory, Translations of Mathematical Monographs, vol. 68, American Mathematical Society, Providence, RI, 1988.
		
		\bibitem{PJ09} Pintz, Janos. Landau's problems on primes. J. Theory. Nombres Bordeaux 21 (2009), no. 2, 357-404. 
		
		\bibitem{PS95}  Pappalardi, Francesco; Shparlinski, Igor. On Artin's conjecture over function fields. Finite Fields Appl.  1  (1995),  no. 4, 399-404. 
		
		\bibitem{PS03} Pappalardi, Francesco; Saidak, Filip; Shparlinski, Igor E. Square-free values of the Carmichael function. J. Number Theory 103 (2003), no. 1, 122-131.
		
		\bibitem{PF06} Pappalardi, Francesco;Pappalardi, Francesco; Simultaneous Primitive Roots, Preprint, 2006.

\bibitem{RH02} Roskam, Hans Artin's primitive root conjecture for quadratic fields. J. Theory. Nombres Bordeaux 14 (2002), no. 1, 287-324. 

\bibitem{RH00} Roskam, Hans A quadratic analogue of Artin's conjecture on primitive roots. J. Number Theory 81 (2000), no. 1, 93-109.
		
		\bibitem{RD96} Redmond, Don. Number theory. An introduction. Monographs and Textbooks in Pure and Applied Mathematics, 201. Marcel Dekker, Inc., New York, 1996. 
		
		\bibitem{RP96} Ribenboim, Paulo, The new book of prime number records, Berlin, New York: Springer-Verlag, 1996.
		
		\bibitem{RI15} Igor Rivin, Some experiments on Bateman-Horn,  arXiv:1508.07821.
		
		\bibitem{RN99} Rosen, Michael. Number theory in function fields. Graduate Texts in Mathematics, 210. Springer-Verlag, New York, 2002.
		
		\bibitem{RH02} Roskam, Hans. Artin primitive root conjecture for quadratic fields. J. Theory Nombres Bordeaux,
		14, (2002), no. 1, 287-324.

\bibitem{RS02} Rubin, Karl; Silverberg, Alice Ranks of elliptic curves. Bull. Amer. Math. Soc. (N.S.)  39  (2002),  no. 4, 455-474.

\bibitem{SA43} Selberg, Atle. On the normal density of primes in small intervals, and the difference between consecutive primes. Arch. Math. Naturvid. 47, (1943). no. 6, 87-105.

\bibitem{SA74} Schinzel, A. On two conjectures of P. Chowla and S. Chowla concerning continued fractions. Ann. Mat. Pura Appl. (4) 98 (1974), 111-117.

\bibitem{SD61} Shanks, Daniel On numbers of the form $n^4 +1$ . Math. Comput.  15,  1961, 186-189.

\bibitem{SD59} Shanks, Daniel A sieve method for factoring numbers of the form $n^2 +1$ . Math. Tables Aids Comput.  13,  1959, 78-86.
		
		\bibitem{SJ14} Jordan Schettler, Using Continued Fractions to Compute Iwasawa Lambda Invariants of Imaginary Quadratic Number Fields, arXiv:1403.3946.

\bibitem{SJ09} Silverman, Joseph H. The arithmetic of elliptic curves. Second edition. Graduate Texts in Mathematics, 106. Springer, Dordrecht, 2009.

\bibitem{SJ05} Silverman, Joseph H. p-adic properties of division polynomials and elliptic divisibility sequences. Math. Ann.  332  (2005),  no. 2, 443-471.

\bibitem{SM09} Siguna Muller, Elliptic Pseudoprimes, Mathematics of Computation, 2009.

\bibitem{SM02} Stoll, Michael On the arithmetic of the curves y 2 =x l +A . II. J. Number Theory  93  (2002),  no. 2, 183-206.

\bibitem{SP03} Stevenhagen, Peter. The correction factor in Artin's primitive root conjecture. Les XXII emes Journees Arithmetiques (Lille, 2001). J. Theor. Nombres Bordeaux 15 (2003), no. 1, 383-391.
		
		\bibitem{SP69} Stephens, P. J. An average result for Artin conjecture. Mathematika 16, (1969), 178-188.
		
		\bibitem{ST76}  Stoneham, R. G. Normal recurring decimals, normal periodic systems $(j,\epsilon)$-normality, and normal numbers. Acta Arith.  28  (1975/76), no. 4, 349-361.
\bibitem{SM17}Stadnik, M. E. A multiquadratic field generalization of Artin's conjecture. J. Number Theory 170 (2017), 75-102.
\bibitem{SV11} Shparlinski, Igor E.; Voloch, Jose Felipe. Generators of elliptic curves over finite fields. Preprint 2011.
\bibitem{SZ03} Schmitt, Susanne; Zimmer, Horst G. Elliptic curves. A computational approach. With an appendix by Attila Petho. de Gruyter 
Studies in Mathematics, 31. Walter de Gruyter $\&$ Co., Berlin, 2003. 

		\bibitem{TG15} Tenenbaum, Gerald. Introduction to analytic and probabilistic number theory. Translated from the Third French edition. American Mathematical Society, Rhode Island, 2015.

\bibitem{UW00} Urbanowicz, Jerzy; Williams, Kenneth S. Congruences for $L$-functions. Mathematics and its Applications, 511. Kluwer Academic Publishers, Dordrecht, 2000.
		
\bibitem{VR73} Vaughan, R. C. Some applications of Montgomery's sieve. J. Number Theory 5 (1973), 64-79.

\bibitem{VV12} Virdol, Cristian Artin's conjecture for abelian varieties. Kyoto J. Math.  56  (2016),  no. 4, 737-743.
		
\bibitem{ZD09} Zywina, David. A refinement of Koblitz's conjecture, arXiv:0909.5280.

\bibitem{ZD08} Zywina, David. The Large Sieve and Galois Representations, arXiv:0812.2222.

\bibitem{WN95} Watt, N. Short intervals almost all containing primes. Acta Arith. 72 (1995), no. 2, 131-167.

\bibitem{WS82} Wagstaff, Samuel S., Jr. Pseudoprimes and a generalization of Artin's conjecture. Acta Arith. 41 (1982), no. 2, 141-150. 

\bibitem{WJ61} Wrench, John W., Jr. Evaluation of Artin's constant and the twin-prime constant. Math. Comp. 15 1961 396-398.				

\bibitem{WR01} Winterhof, Arne. Character sums, primitive elements, and powers in finite fields. J. Number Theory 91, 2001, no. 1, 153-163.
		
		
\end{thebibliography}
\end{document}